\documentclass[11pt,reqno]{amsart}
\usepackage{fullpage}
\usepackage[mathscr]{eucal}
\usepackage{mathrsfs}
\usepackage{amsfonts}
\usepackage{amsmath}
\usepackage{color}
\usepackage{amsthm}
\usepackage{amssymb}
\usepackage{latexsym}
\usepackage[all]{xy}

\usepackage[mathscr]{eucal}
\usepackage{palatino}
\usepackage{fullpage}
\usepackage{verbatim}
\usepackage{xr-hyper}
\usepackage{hyperref}

\usepackage[dvipsnames]{xcolor}

\theoremstyle{plain}
\newtheorem{thm}[equation]{Theorem}
\newtheorem{cor}[equation]{Corollary}
\newtheorem{prop}[equation]{Proposition}
\newtheorem{lem}[equation]{Lemma}

\theoremstyle{definition}
\newtheorem{defn}[equation]{Definition}

\newtheorem{claim}[equation]{Claim}
\newtheorem{conv}[equation]{Convention}

\theoremstyle{remark}

\newtheorem{examp}[equation]{Example}
\newtheorem{rem}[equation]{Remark}

\makeatletter
\renewcommand{\subsection}{\@startsection{subsection}{2}{0pt}{-3ex
plus -1ex minus -0.2ex}{-2mm plus -0pt minus
-2pt}{\normalfont\bfseries}} \makeatother

\numberwithin{equation}{subsection}
\newcommand{\Lmod}[1]{#1\text{-}{\mathsf{mod}}}
\newcommand{\LMod}[1]{#1\text{-}{\mathsf{Mod}}}

\newcommand{\mon}[1]{(#1)\text{-}{\mathsf{mon}}}

\newcommand{\idot}{{\:\raisebox{2pt}{\text{\circle*{1.5}}}}}

\newcommand{\mof}{{(\htrig_\kappa:\C[\TSL]^W)\text{-}{\mathsf{mod}}}}
\DeclareMathOperator{\Sp}{{\mathrm{Sp}}}

\DeclareMathOperator{\Ext}{\mathrm{Ext}}
\DeclareMathOperator{\sym}{\mathrm{Sym}}

\DeclareMathOperator{\im}{\mathrm{Im}}
\DeclareMathOperator{\supp}{\mathrm{Supp}}

\DeclareMathOperator{\Ker}{\mathrm{Ker}}

\DeclareMathOperator{\End}{\mathrm{End}}
\renewcommand{\sp}{{\operatorname{sp}}}

\DeclareMathOperator{\gr}{\mathrm{gr}}

\renewcommand{\SS}{{\operatorname{SS}}}
\DeclareMathOperator{\Hamp}{{^{\dagger}\mathbb{H}}}
\DeclareMathOperator{\Hamr}{{\mathbb{H}^{\dagger}}}
\DeclareMathOperator{\Ham}{\mathbb{H}}

\DeclareMathOperator{\Lie}{\mathrm{Lie}}
\DeclareMathOperator{\Tr}{\mathrm{Tr}}

\DeclareMathOperator{\Supp}{\mathrm{Supp}}
\DeclareMathOperator{\Char}{\mathrm{SS}}

\DeclareMathOperator{\Ad}{\mathrm{Ad}}
\DeclareMathOperator{\ad}{\mathrm{ad}}

\newcommand{\dis}{\displaystyle}
\newcommand{\oper}{\operatorname }
\newcommand{\beq}{\begin{equation}\label}
\newcommand{\eeq}{\end{equation}}

\DeclareMathOperator{\Spec}{\mathrm{Spec}}
\DeclareMathOperator{\pr}{pr}

\newcommand{\iso}{{\;\stackrel{_\sim}{\to}\;}}
\newcommand{\cd}{\!\cdot\!}
\DeclareMathOperator{\Hom}{\mathrm{Hom}}

\def\ccirc{{{}_{\,{}^{^\circ}}}}
\newcommand{\mgg}{\TSL^{\circ}}
\newcommand{\ak}{ V' }

\newcommand{\zero}{{\boldsymbol{0}_{Z(\LSL)^{\circ}}}}
\newcommand{\ck}{{\mathcal K}}

\newcommand{\FT}{{\mathfrak{T}}}

\newcommand{\sll}{{\mathfrak{sl}}}
\newcommand{\bg}{{\mathfrak{G}}}
\newcommand{\XX}{X }

\newcommand{\bo}{\mbox{$\bigotimes$}}
\renewcommand{\o}{\otimes }

\newcommand{\bad}{{\mathsf{BAD}}}
\newcommand{\sh}{{\mathbb{S}}}
\newcommand{\E}{{\mathsf{eu}_\X}}
\newcommand{\bplus}{\mbox{$\bigoplus$}}
\newcommand{\cc}{{\wt{\scr C}}}
\renewcommand{\ss}{{\scr C}}
\newcommand{\kap}{{\kappa}}

\newcommand{\si}{\sigma }

\newcommand{\OO}{\mc{O}}

\renewcommand{\t}{{\mathfrak t}}
\newcommand{\td}{{T_\Delta}}
\newcommand{\TT}{{\mathbb T}}
\renewcommand{\si}{\sigma }
\newcommand{\one}{{\boldsymbol{1}}}
\newcommand{\fz}{{\mathfrak z}}
\newcommand{\wt}{\widetilde }

\newcommand{\eu}{{\operatorname{\mathsf{eu}}}}

\newcommand{\fin}{\mathrm{fin}}

\newcommand{\Id}{{\operatorname{Id}}}

\newcommand{\fl}{{\mathfrak l}}

\newcommand{\fc}{{\TSL_L / N_L}}

\newcommand{\horo}{{\mathscr{Y}}}

\newcommand{\reg}{{\operatorname{reg}}}

\newcommand{\htrig}{{\mathsf{U}}}
\newcommand{\scr}[1]{\mathscr{#1}}
\newcommand{\bi}{{\mathbf{i}}}
\newcommand{\bj}{{\mathbf{j}}}

\def\ccirc{{{}_{^{\,^\circ}}}}

\newcommand{\la}{\lambda}

\newcommand{\dd}{{\mathscr{D}}}
\newcommand{\ddd}{{\mathcal{D}}}

\newcommand{\oo}{{\mathcal{O}}}
\newcommand{\wh}{\widehat }

\newcommand{\U}{{\mathcal{U}}}

\newcommand{\h}{{{\mathfrak h}}}
\newcommand{\bnu}{\overline{\nu}}

\newcommand{\mggW}{{\mgg/W}}
\newcommand{\pa}{\partial }

\newcommand{\NN}{{\mathsf N}}

\newcommand{\ny}{\NN_{X/Y}}
\newcommand{\arr}{\overset{{\,}_\to}}
\newcommand{\mc}{\mathcal}

\newcommand{\C}{\mathbb{C}}
\newcommand{\g}{\mathfrak{g}}
\renewcommand{\b}{\mathfrak{b}}

\newcommand{\La}{\Lambda }

\newcommand{\op}{{\operatorname{op}}}

\newcommand{\inv}{^{-1}}

\newcommand{\Z}{{\mathbb Z}}
\newcommand{\en}{{\enspace}}

\newcommand{\st}{{\operatorname{unst}}}
\newcommand{\sset}{\subset}
\newcommand{\sminus}{\smallsetminus}
\newcommand{\into}{\,\hookrightarrow\,}
\newcommand{\too}{\,\longrightarrow\,}
\newcommand{\mto}{\mapsto}
\newcommand{\onto}{\,\twoheadrightarrow\,}
\newcommand{\N}{{\mathcal{N}}}

\newcommand{\Ga}{\Gamma }

\newcommand{\nV}{{\mathcal V}}

\DeclareMathOperator{\Irr}{\mathrm{Irr}}

\newcommand{\ds}{{\dots}}

\newcommand{\s}{\mathfrak{S}}
\newcommand{\id}{\mathrm{id}}

\newcommand{\Q}{\mathbb{Q}}
\newcommand{\Cs}{\C^\times}
\newcommand{\Ind}{\mathrm{Ind}}

\newcommand{\mf}{\mathfrak}

\newcommand{\Stab}{\mathrm{Stab}}
\newcommand{\CG}{\mathcal{G}}
\newcommand{\FZ}{{\mathfrak{Z}}}
\newcommand{\tit}{\textit}
\newcommand{\X}{\mathfrak{X}}
\newcommand{\nc}{\newcommand}

\nc{\BM}{\mathbf{M}} 
\nc{\BG}{\mathbf{G}}
\nc{\BQ}{\mathbf{Q}}
\nc{\BL}{\mathbf{L}}
\nc{\BT}{\mathbf{T}}
\nc{\BH}{\mathbf{H}}

\nc{\LL}{{L'}}
\nc{\LSL}{\mathrm{L}}
\nc{\MSL}{\mathrm{M}}
\nc{\MM}{{M'}}
\nc{\tSL}{\mf{t}}
\nc{\TSL}{\mathrm{T}}

\nc{\To}{{T}^{\circ}}
\renewcommand{\u}{\mc{U}}

\newcommand{\xcyc}{\X^{\oper{cyc}} }

\nc{\tcL}{\t_{\LGL}} 
\nc{\tsL}{\t^{\symL}} 
\nc{\tsLo}{(\t^{\symL})^{\circ}}
\nc{\tLo}{\t^{\circ}} 
\nc{\Xo}{\X^{\circ}}
\nc{\Lo}{L^{\circ}}
\nc{\Mo}{M^{\circ}}
\nc{\ZMo}{Z(M)^{\circ}}
\nc{\paX}{\partial \X}

\newcommand{\rightsim}{\stackrel{\sim}{\longrightarrow}}
\newcommand{\ms}{\mathscr}
\newcommand{\Res}{\mathrm{Res}}
\newcommand{\Red}{\mathrm{Red}}

\nc{\twist}{\mathrm{twist}}
\nc{\Span}{\mathrm{Span}}
\nc{\Ps}{\mathbb{P}}

\renewcommand{\H}{\mathsf{H}}

\nc{\Mod}{\mathrm{Mod} \,}
\nc{\ras}{\Lambda}
\nc{\Loc}{\mathsf{Loc}}

\newcommand{\mn}{{\mathbf{U}}}

\newcommand{\Om}{\Omega }
\newcommand{\vv}{{v}}
\newcommand{\Nnil}{{\mathbb{M}}_{\mathsf{nil}}}
\newcommand{\brho}{{\bar\rho}}

\newcommand{\bN}{{\mathbf U}}
\newcommand{\sL}{{\mathsf L}}

\newcommand{\vi}{${\sf {(i)}}\;$}
\newcommand{\vii}{${\sf {(ii)}}\;$}
\newcommand{\viii}{${\sf {(iii)}}\;$}

\newcommand{\trig}{{\operatorname{trig}}}

\renewcommand{\H}{\mathsf{H}}
\nc{\by}{\mathbf{y}}
\nc{\bx}{\mathbf{x}} 
\nc{\bz}{\mathbf{z}} 
\nc{\Y}{\mathcal{Y}}
\nc{\bY}{\overline{\Y}}

\nc{\IC}{\mathrm{IC}}

\nc{\Thetasph}{\Theta^{\mathrm{spher}}}

\nc{\Orb}{\mathbb{O}}

\nc{\bs}{\mathbf{s}}
\nc{\bp}{\mathbf{p}}
\nc{\bq}{\mathbf{q}}
\nc{\bc}{\mathbf{c}}
\nc{\br}{\mathbf{r}}
\nc{\Sing}{\mathrm{Sing}}

\nc{\cyc}{\mathrm{cyc}}
\nc{\ol}{\overline}
\nc{\abT}{\mathbb{T}}
\nc{\bw}{\mathbf{w}}
\nc{\mm}{\ms{M}}
\nc{\Cas}{\Pi}
\nc{\Jac}{\mathbb{J}}
\nc{\SSL}{{\operatorname{SL}}}

\begin{document}

\title{{\textbf{Hamiltonian reduction and nearby cycles  for\\
Mirabolic $\pmb{\dd}$-modules}}}

\author{Gwyn Bellamy}
\address{School of Mathematics and Statistics, University of Glasgow, University Gardens, Glasgow G12 8QW}
\email{gwyn.bellamy@glasgow.ac.uk}

\author{Victor Ginzburg}\address{
Department of Mathematics, University of Chicago,  Chicago, IL 
60637, USA.}
\email{ginzburg@math.uchicago.edu}

\begin{abstract}
We study holonomic $\dd$-modules on
$SL_n(\C)\times \C^n$, called {\em mirabolic  modules}, analogous to Lusztig's character sheaves. We describe the supports of simple mirabolic modules. We show that a mirabolic module is killed  by the  functor of  Hamiltonian reduction from  the category of mirabolic modules to the category of representations of the trigonometric Cherednik algebra if and only if the characteristic variety of the module is contained in the unstable locus.

We introduce an analogue of Verdier's specialization functor for representations of Cherednik algebras which agrees, on category $\mathcal O$, with the restriction functor of Bezrukavnikov and Etingof. In type $A$, we also consider a Verdier specialization functor on mirabolic $\dd$-modules. We show that Hamiltonian reduction intertwines specialization functors on mirabolic $\dd$-modules with the corresponding functors on representations of the Cherednik algebra. This allows us to apply known Hodge-theoretic purity results for nearby cycles in the setting considered by  Bezrukavnikov and Etingof.
\end{abstract}

%We study holonomic $\dd$-modules on $SL_n(\C)\times \C^n$, called {\em mirabolic  modules}, analogous to Lusztig's character sheaves. We prove that a mirabolic module is killed by the  functor of  Hamiltonian reduction from  the category of mirabolic modules to the category of representations of the trigonometric Cherednik algebra   if and only if the characteristic variety of the module is contained in the unstable locus. We also show that Hamiltonian reduction commutes with shift functors and that it intertwines a Verdier specialization functor on $\dd$-modules with an analogous specialization functor on representations of the Cherednik algebra.

%We describe the supports of simple mirabolic modules and apply this to study submodules and quotients of the {\em mirabolic  Harish-Chandra ${\dd}$-module}. We show that, for generic parameters, the mirabolic  Harish-Chandra ${\dd}$-module is the minimal extension of a local system on the regular locus. However, unlike the classical case of the Harish-Chandra ${\dd}$-module on the group $SL_n$ itself, the  mirabolic  Harish-Chandra ${\dd}$-module does have submodules and quotients supported away from the regular locus when the parameters are not generic.

\maketitle

\centerline{\it To Boris Feigin, on the occasion of his 60th Birthday.}

{\small
\tableofcontents
}

\section{Introduction}\label{sec:intro}
%sss

\subsection{} In this paper, we  study mirabolic
$\dd$-modules,
following earlier works \cite{GG}, \cite{CherednikCurves}, and 
\cite{MirabolicCharacter}.
Mirabolic $\dd$-modules form an interesting category of
regular holonomic $\dd$-modules on the variety $SL_n(\C)\times\C^n$.
This category has a ``classical'' counterpart,
a certain category of {\em admissible $\dd$-modules} on
an arbitrary
complex reductive group, which was studied in \cite{AdmissibleModules}.  
Similar to that  ``classical'' case, there are two different definitions
of mirabolic  $\dd$-modules. The first definition involves
characteristic varieties, while the second definition
involves an action
of the enveloping algebra.
The first definition is more geometric and it can be used
to establish a connection with perverse sheaves via
the Riemann-Hilbert correspondence. The resulting perverse 
sheaves are ``mirabolic analogues'' of Lusztig's {\em character
sheaves} on a reductive group. Furthermore, a conjectural
classification of simple mirabolic  $\dd$-modules,
modeled on Lusztig's classification of character sheaves,
was suggested in \cite{FGT}.

The second definition of mirabolic  $\dd$-modules is more
algebraic and it is better adapted, in a sense,
for applications to Cherednik algebras, see below.
In particular, there is an important mirabolic  $\dd$-module,
the {\em mirabolic Harish-Chandra $\dd$-module}. This 
 $\dd$-module has a very natural algebraic definition,
while its geometric definition is not completely understood
so far, see however \cite{MirabolicCharacter}, Theorem 5.1.2.

Our first result establishes an equivalence of the
two definitions of mirabolic  $\dd$-modules.
Although this result was, in a way, implicit in  \cite{MirabolicCharacter},
its actual proof turns out to be more complicated than 
one could have expected. The idea of the proof is
 similar to the one used in  \cite{AdmissibleModules}
in the ``classical'' case. However, certain steps of the argument
do not work in the mirabolic setting and require a different
approach.

\subsection{}\label{setup} In order to state our results in more detail,
we introduce some  notation.  Throughout the paper, we work over $\C$,
the field of complex numbers. We write $\dd_X$ for the sheaf of
algebraic differential operators on a smooth algebraic 
variety $X$ and let $\dd(X)=\Ga(X, \dd_X)$ denote
the algebra of global sections. We write $T^*X$ for the cotangent bundle
of $X$ and
$\Char (\ms{M})\sset T^*X$ for the characteristic variety of a coherent
$\dd_X$-module
$\mm$.

\begin{conv}\label{convent} 
Throughout the paper, we let $V=\C^n$ be an $n$-dimensional vector space, and write $\SSL:=SL(V)=SL_n$.  Thus, $\sll=\sll(V)=\mf{sl}_n$ is the Lie algebra of the group $\SSL$.

We will explicitly write $SL_m$, resp. $\sll_m$, in all cases where $m\neq n$.
\end{conv}

Let $\mc{N} \subset \sll $ be the nilpotent cone. Let $\FZ$ be the algebra of bi-invariant differential operators on the
group $\SSL$. This algebra is nothing but the centre of $\U(\sll)$,
the enveloping algebra of the Lie algebra $\sll $.

We set $\X = \SSL \times V$, and put $\ddd  := \dd(\X)= \dd(\SSL) \otimes \dd(V)$.
We view a bi-invariant differential operator $z$ on $\SSL$ as an element $z \o1 \in \dd(\SSL) \otimes \dd(V)$. 
We identify the cotangent bundle $T^* \X$ with $\SSL \times \sll  \times V \times V^*$, where we have used the trace
pairing to identify the vector space $\sll$ with its dual.

We also consider the group $G = GL(V)$ with Lie algebra
$\g = \mf{gl}(V)$.
The group $G$ acts naturally on $V$ and
acts on $\SSL$ by conjugation (this action clearly factors
through an action of $PGL(V)$). 
We let $G$ act diagonally on $\X$. The $G$-action on $\X$ induces a morphism of Lie algebras $\mu: \g\to\ddd$. The $G$-action on $\X$ also
induces a
 Hamiltonian action of $G$ on the
cotangent bundle $T^* \X$ with moment map $\mu_\X:\ (g,Y,i,j) \mto g Y
g^{-1} - Y + i \o j$.
Following \cite{GG}, we define the following subvariety of $T^*\X$:
\begin{equation}\label{eq:defineNnilSL}
\Nnil(\SSL) = \{ (g,Y,i,j) \in T^*\X=\SSL \times \mc{N} \times V \times
V^* \ | \ g Y g^{-1} - Y + i \o j = 0 \}.
\end{equation}

We recall the following, c.f. \cite[Definition 4.5.2]{CherednikCurves}:
\begin{defn}\label{eq:defineNnilG}
A coherent $\dd_{\X}$-module $\ms{M}$ is called a {\em mirabolic
  $\dd$-module} if 
one has $\Char (\ms{M}) \subseteq \Nnil(\SSL)$ and,
moreover, $\mm$ has regular singularities. 
\end{defn}

It was shown in \cite{GG} that $\Nnil(\SSL)$ is a Lagrangian subvariety
of
$T^*\X$. Hence
the inclusion $\Char (\ms{M}) \subseteq \Nnil(\SSL)$
ensures that $\mm$ is holonomic, so that the condition
that $\mm$ has regular singularities makes sense.

The following result, which was  implicit in \cite{MirabolicCharacter},
will be proved  in section \ref{complete} below.

\begin{thm}\label{thm:admissible2}
For a  coherent $\dd_{\X}$-module $\ms{M}$, the following are equivalent:

\vi $\ms{M}$ is a mirabolic $\dd$-module.

\vii Both the $\FZ$-action and the $\mu(\g)$-action on $\Gamma(\X,\ms{M})$ are locally finite.
\end{thm}

Let $\one\in\g=\mf{gl}(V)$ denote the identity map. Thanks to the above theorem,  the action of the
element $\one$ on any mirabolic $\dd(\X)$-module $M$
is locally finite. Thus, one has a vector space
decomposition  $ M=\oplus_{c\in\C}\  M^{(c)}$, into  generalized
eigenspaces of $\one$,
where 
$$
 M^{(c)}:=\{m\in  M\mid (\mu(\one) - c)^\ell(m)=0\en\text{for  some}\en \ell=\ell(m)\gg0\}.$$ 

Let  $\cc$ be the  category of all mirabolic $\dd$-modules
and, for each $q\in\C^*$, let $\cc_q$ be the full subcategory of $\cc$
formed by the mirabolic $\dd$-modules $ M$ such that one has
$$
 M=\bigoplus_{ \exp(2 \pi \sqrt{-1} c) = q} \  M^{(c)}.
$$
The objects of the category  $\cc_q$ may be viewed as 
having `monodromy $q$' along the $\C^*$-orbits of the dilation action on $V$,
the second factor in $\X=\SSL\times V$. 
The category $\cc_q$ is a Serre subcategory of $\cc$.

Let $\ss_q$ be the full subcategory of $\cc_q$ formed by the mirabolic $\dd$-modules $ M$
such that the $\mu(\one)$-action on $ M$ is semisimple.

\subsection{The functor of Hamiltonian reduction}\label{sec:unstable}

Let $W$ denote the symmetric group on $n$ letters and write  $e\in\C[W]$ for the averaging idempotent in the group algebra of $W$. Let $\TSL$ be  the standard maximal
torus of $\SSL$ formed by diagonal matrices,
and write $\t=\Lie \TSL$.
The group $W$ acts naturally
on $\TSL$ and on $\t$ by permutation of coordinates.

Associated with a complex number  $\kap\in\C$,
there is an algebra  $\mathsf{H}_{\kappa}^{\mathrm{trig}} (\SSL)$ ,
the trigonometric Cherednik algebra of type $\SSL$ at parameter
$\kappa$, see \S\ref{app:shift1} for a precise
definition.
Let $\htrig_{\kappa} :=e \mathsf{H}_{\kappa}^{\mathrm{trig}} (\SSL) e$
denote the  spherical subalgebra of
$\mathsf{H}_{\kappa}^{\mathrm{trig}} (\SSL)$.
The algebra $\htrig_{\kappa}$ contains the algebra $(\sym\t)^W\cong \C[\t^*]^W$ as a commutative subalgebra.

 An important role in the representation theory
of  Cherednik algebras is played by category $\OO_\kappa$.
By definition, this is the category of finitely generated left $\htrig_{\kappa}$-modules
which are locally finite as $(\sym \t)^W$-modules, see \S\ref{ham_sec}.
The link between  mirabolic $\dd$-modules and  representations
of the Cherednik algebra is provided by  the functor
 of  Hamiltonian
reduction introduced in \cite{GG}.
We recall the construction of this functor.

First, one associates with $\kappa\in \C$ two other complex parameters, $q$ and $\kap$,
defined by the formulas:
\beq{kap}
\kappa = 1 - c, \qquad q =\exp(2 \pi \sqrt{-1} c).
\eeq
 Let $\g_{c}$ be the Lie subalgebra of $\ddd$ defined as the image of
the map $\g\to\ddd,\ a\mto \mu(a)- c \Tr(a)$, which is a Lie algebra homomorphism.
For any $\dd_\X$-module $\mm$, we put
\[
\Ham_{c}(\ms{M})\  :=\  \Gamma(\X,\ms{M}) ^{\g_c}\ =\ \{ m \in \Gamma(\X,\ms{M}) \ | \ \g_{c} \cd m = 0 \}.
\]

Further,  let   $\OO_\kappa$
be  the category of finitely generated left $\htrig_{\kappa}$-modules
which are locally finite as $(\sym \t)^W$-modules, see \S\ref{ham_sec}.
Then, it was shown in \cite[sectcion 6]{GG}
that the functor $\mm\ \mapsto\ \Ham_{c}(\ms{M})$ restricts to 
an exact functor $\Ham_{c}:\ \ms{C}_q \rightarrow \OO_\kap$ and, moreover,
this yields an equivalence $\ms{C}_q/\Ker(\Ham_c)\ \iso\ \OO_\kap$.

In this paper, we give a description of the kernel of the functor $\Ham_c$.
Our description  involves a stability condition in the
sense of Geometric Invariant Theory.
To formulate the  stability condition, 
one equips the trivial line bundle over $T^* \X$
with  a $G$-equivariant structure using the determinant character.
Explicitly, we let $G$ act on $T^*\X\times \C$ by
$g\cdot (x,t) = (g \cdot x,\, \det (g)^{-1} t)$ for all $x \in T^* \X, g \in
G$ and $t \in \C$. Write $(T^* \X)^{\st,+}$ for the set of unstable
points with respect to this line bundle
and  $(T^* \X)^{\st,-}$ for the set of unstable points with
respect to the inverse line bundle,
which corresponds to the character $\det^{-1}$.

To state our result, we also need
the following
\begin{defn}\label{bad} Put
$\bad = \{ \frac{a}{b} \ | \ a,b \in \Z, \ 1 \le b \le n \}$. We say
that the parameter $c\in \C$ is \textit{good} if $c \notin \bad \cap (0,1)$. We say that $c$ is  \textit{admissible} if $c \notin \bad
\cap \Q_{> 0}$ and, additionally, if $n = 2$ then $c \notin \bad$. 
\end{defn}

\begin{thm}\label{cor:unstable} 
Assume $c \in \C$ is admissible and let $\ms{M}\in\cc_q$ be a mirabolic module on $\X$. Then, we have $\Char (\ms{M}) \subset (T^*\X)^{\st,+}$ if and only if $\Ham_{c}(\ms{M}) = 0$. 
\end{thm}

The above result, to be proved in \S\ref{ham_shift}, was anticipated,
in one form or the other, by 
a number of people. A similar result is expected  to
hold in the much more general framework of noncommutative algebras
obtained from an algebra of differential operators via
quantum Hamiltonian reduction (e.g. quantizations of
Nakajima's quiver varieties). Most of the results of \S\ref{sec:bimodules}, that
are used
in the proof of Theorem \ref{cor:unstable},  can
 be extended to this more general setting without much
difficulty. However, one of the obstacles to 
finding a generalization of our theorem is lack
of a plausible  general condition that
should replace 
the admissibility assumption on the parameter $c$. The latter  assumption is quite essential
for the statement of Theorem  \ref{cor:unstable} to be true.

An interesting special case where a replacement of the admissibility condition has been found and an analogue
of the above theorem has been established is the case of hypertoric varieties studied in the recent preprint by K. McGerty and T. Nevins \cite{McN}.

Theorem \ref{cor:unstable}, 
together with Proposition \ref{thm:irrsupport} below imply that:

\begin{cor}
For any $c \in \C\sminus\bad$, the functor $\Ham_{c} : \ss_q \rightarrow \oo_{\kappa}$ is an equivalence. 
\end{cor}

\begin{rem}
One can ask if a result similar to Theorem \ref{cor:unstable} holds for the negative stability condition. In order to prove such a result, one must consider $c \in \Q_{>0}$, and in particular values of $c$ which are not admissible. The arguments used in the proof of Theorem \ref{cor:unstable} are not applicable because of the failure of localization for $\dd$-modules on $\SSL \times \mathbb{P}(V)$. However, Lemma \ref{gwyn} implies: if $c > 0$ then  $\Char (\ms{M}) \subset (T^* \X)^{\st,-}$ implies that $\Ham_{c}(\ms{M}) = 0$.
\end{rem}

Our next result says that the functor of Hamiltonian reduction commutes
with shift functors. When studying  shift functors, it is more convenient to work
with {\em twisted} $\dd_c$-modules  on $X = \SSL  \times \mathbb{P}(V)$
rather than ordinary  $\dd$-modules  on $\X = \SSL  \times V$,
 see section \ref{ham_shift}. An advantage of working with twisted
$\dd_c$-modules is that the category $\ss_q$, of mirabolic $\dd$-modules on $\X$,
gets replaced by a category $\ss_c$, of mirabolic twisted
$\dd_c$-modules on $X$.

Tensoring with the line bundle $\oo(n)$, on $\mathbb{P}(V)$, gives a natural geometric shift functor $\ss_{c} \rightarrow \ss_{c-1},\ \mm \mto \mm(n)$. On the other hand, based on Opdam's theory of shift operators,
there is a shift functor $\sh : \Lmod{\htrig_{\kappa}} \rightarrow
\Lmod{\htrig_{\kappa+1}}$. 

In \S\ref{twist_pf}, we prove the  following strengthening of
an earlier result established in \cite{GGS}.

\begin{thm}\label{twistshift} 
Assume that $c$ is admissible. Then, the following diagram commutes
$$
\xymatrix{
\ss_{c}\ \ar[rr]^<>(0.5){\mm \mto \mm(n)}_<>(0.5){\sim}
\ar[d]_<>(0.5){\Ham}&&\ \ss_{c-1}\ \ar[d]_<>(0.5){\Ham}\\
\oo_{\kappa} \ \ar[rr]^<>(0.5){\sh}_<>(0.5){\sim}&&\ \oo_{\kappa+1}
}
$$
\end{thm}

\subsection{The support of simple mirabolic $\dd$-modules}
The space $\X$ has a partition into a finite union of
smooth locally-closed strata, see section \ref{strat} for details. 
Every irreducible component of the Lagrangian variety 
$\Nnil(\SSL)$ is the closure of the conormal bundle
to a certain stratum.
The strata which arise in this way are called \textit{relevant}. These
relevant strata, $\X(\lambda,\mu)$, are labeled by the bi-partitions
$(\lambda,\mu)$ of $n$.
The support of any  simple
 mirabolic module is the closure of a relevant stratum.

Let $\xcyc \sset \X$ be the open subset formed by the pairs $(g,v)$ such that the vector $v$ is {\em cyclic} for $g$, i.e. such that we have
$\C[g] \cd v = V$. Further, let  $\X^{\reg} \sset \xcyc$ be the open
subset formed by the pairs $(g,v) \in \xcyc$, such that $g \in \SSL$, is
a regular semi-simple element - that is, a matrix with determinant one, and with $n$ pairwise distinct eigenvalues.
The set $\X^\cyc$ is a union of relevant strata and the set $\X^\reg$ is the
unique open stratum in $\X$.

The following result provides some information about the support of simple mirabolic modules. This information is an important ingredient in the proof of Theorem \ref{cor:unstable}, as well as in the analysis of the mirabolic Harish-Chandra $\ddd$-module, see Theorem \ref{main}.

\begin{prop}\label{thm:irrsupport}
Let $\ms{M} \in \ss_{q}$ be a simple mirabolic module. If $q$ is a primitive $m$-th root of unity, where $1 \le m \le n$, then 
$$
\Supp \ms{M} = \ol{\X((m^{v},1^w),(m^{u}))}
$$
for some $u,v,w \in \mathbb{N}$ such that $n = (u+v) m + w$. Otherwise, $\Supp \ms{M} = \X$.   
\end{prop}

Independent of Theorem \ref{cor:unstable} above, one can ask whether a
simple mirabolic $\dd$-module supported on the closure of a given
stratum is killed by the functor of Hamiltonian reduction. The following
proposition gives a partial answer to that question.  
First, introduce the following sets of rational numbers:
$$
\Sing_- = \left\{ \frac{r}{m} \in \Q_{\le 0} \ | \ 2 \le m \le n, \ (r,m) = 1 \right\}, \quad \Sing_+ = \{ - c + 1 \ | \ c \in \Sing_- \},
$$
where $(r,m)$ denotes the highest common factor of $r$ and $m$, and we set $\Sing_0 = \Z$. Let $\Sing \subset \C$ be the union of these three sets.

\begin{prop}\label{prop:Hamkill}
Let $c \in \C$ and set $q = \exp (2 \pi \sqrt{-1} c)$. Let $\ms{M} \in \ss_{q}$ be a simple mirabolic module such that $\Supp \ms{M} \subset \X \sminus \X^{\mathrm{reg}}$. 

\begin{enumerate}
\item If $\Supp \ms{M} = \ol{\X((m^{v},1^w),\emptyset)}$ for some $v,w
  \in \mathbb{N}$ such that $n = v m + w$ then\newline
  $\Ham_{c}(\ms{M}) \neq 0$ implies that $c = \frac{r}{m}_{_{}} \in \Sing_-$.

\item If $\Supp \ms{M} = \ol{\X((1^w),(m^u))}$ for some $u,w \in \mathbb{N}$ such that $n = u m + w$ then\newline $\Ham_{c}(\ms{M}) \neq 0$ implies that $c = \frac{r}{m}_{_{}} \in \Sing_+$.  

\item If $\Supp \ms{M} = \ol{\X(\emptyset,(1^n))}$ then $\Ham_{c}(\ms{M}) \neq 0$ implies that $c \in \Z_{>0}$. 
\end{enumerate}
For all other $\ms{M}$, one has $\Ham_{c}(\ms{M}) = 0$ for all $c$.
\end{prop}

The proof of Propositions \ref{thm:irrsupport} and
\ref{prop:Hamkill} are given in section \ref{sec:proofs}. We note that
Proposition \ref{prop:Hamkill} can be used to classify the possible
supports of simple modules in category $\mc{O}_{\kappa}$
for the algebra $\htrig_{\kappa}$.  

\subsection{The mirabolic Harish-Chandra $\dd$-module}
In the seminal paper \cite{HottaKashiwara}, Hotta and Kashiwara
have defined, for any complex semisimple group $\BG$,
 a holonomic $\dd(\BG)$-module
that they called
{\em the Harish-Chandra $\dd$-module}.
This  $\dd$-module is 
 important, for instance, because of its  close relation
to the system of partial differential
equations on the group $\BG$ introduced by Harish-Chandra around 1960 in his
study of irreducible characters of infinite dimensional
representations of the group $\BG$.

The definition of the Harish-Chandra $\dd$-module 
involves a choice of ``central character'',
a closed point $\la\in \Spec \FZ$,
where $\FZ$ is the algebra of
bi-invariant differential operators on $\BG$.
Write $\FZ_\la$ for the
corresponding maximal ideal in $\FZ$
and put ${\mathfrak G}:=\Lie \BG$.
The adjoint action of the group $\BG$ on itself
 induces a Lie algebra
map $\ad:\ {\mathfrak G} \to \dd(\BG)$, c.f.  \S\ref{setup}.
Then, following  Hotta and Kashiwara
in \cite{HottaKashiwara2},
the  Harish-Chandra $\dd$-module at parameter $\la$
is defined to be
\beq{J}
{\mathcal J}_\la\ :=\ \dd(\BG)/(\dd(\BG)\cdot\ad({\mathfrak G}) + \dd(\BG)\cdot\FZ_\la ).
\eeq

It is not difficult to show that ${\mathcal J}_\la|_{\BG^\reg}$,
the restriction of ${\mathcal J}_\la$ to the open set  $\BG^\reg$ of regular, semi-simple elements,
is  a local system of rank $|W|$, the order of the Weyl
group of $\BG$. Furthermore, one of
the main results proved by Kashiwara, \cite{KashInvHolo}, using a famous theorem of Harish-Chandra on regularity of invariant eigen-distributions says
that one has ${\mathcal J}_\lambda \simeq j_{!*} ({\mathcal J}_\lambda |_{\BG^\reg})$, i.e. ${\mathcal J}_\lambda$ is the minimal extension
with respect to  the natural open embedding $j : \BG^\reg \into \BG$
of the local system ${\mathcal J}_\lambda |_{\BG^\reg}$.

We now return to the setting of \S\ref{setup}
and  let $\BG=\SSL=SL(V) $. We identify $\Spec \FZ$ with $\tSL/W$ via the Harish-Chandra homomorphism. Motivated by formula \eqref{J},
in \cite{GG}, the authors  introduced, for each 
 pair $(\la,c) \in \tSL^* /W \times \C$,
the following $\dd$-module on the space $\X=\SSL \times V$
$$
\CG_{\lambda,c} := \ddd  / (\ddd  \cd \g_c + \ddd 
\cd \FZ_\lambda),
$$
called the
 {\em mirabolic Harish-Chandra $\dd$-module} with parameters $(\la,c)$.

The mirabolic Harish-Chandra $\dd$-module is an
example of a mirabolic $\dd$-module and
it  was further studied in \cite{MirabolicCharacter}.
It is known, in particular, that 
the restriction of $\CG_{\lambda,c}$ to $\X^{\reg}$ is 
 a local system of rank $n!$.

One of the motivations for the present work was
our desire to understand whether or not one
has an isomorphism
$\CG_{\lambda,c}\cong\jmath_{!*}(\CG_{\lambda,c}|_{\X^\reg})$,
where $\jmath: \X^\reg\into\X$ denotes the open embedding.  
It turns out that if $c$ is generic then the  isomorphism holds.
However, for non-generic values of 
the parameter $c$, the isomorphism may fail.
In other words, for certain
values of  $c$, the $\ddd$-module  $\CG_{\lambda,c}$ may have
either nonzero simple quotients or submodules (or both)
supported on $\X\sminus\X^\reg$.

Our main result about the
 mirabolic Harish-Chandra $\dd$-module  describes 
the possible supports of  simple quotients,
resp. submodules, of  $\CG_{\lambda,c}$ as follows.

\begin{thm}\label{main}
Let $c \in \C$. For any $\la \in \tSL^* /W$, the following holds:
\begin{enumerate}
\item If $c \notin \Sing$ then $\CG_{\lambda,c}$ has no submodules or quotients supported on $\X\sminus\X^{\reg}$. 

\item If $c = \frac{r}{m} \in \Sing_-$ then:\ The simple quotients of
  $\CG_{\lambda,c}$ are supported on the closure of the strata
$\X((m^v,1^w),\emptyset)$ where $v,w \in \mathbb{N}$ such that $n = v
m + w$.

\noindent
The simple submodules of $\CG_{\lambda,c}$ are supported on the closure
 of  the strata
$\X((1^w),(m^u))$, where $u,w \in \mathbb{N}$ such that $n = u m + w$. 

\item If $c = \frac{r}{m} \in \Sing_+$ then:\ The simple quotients of
  $\CG_{\lambda,c}$ are supported on the closure of  the strata
$\X((1^w),(m^u))$, where $u,w \in \mathbb{N}$ such that $n = u m + w$;

\noindent
The simple submodules of $\CG_{\lambda,c}$ are supported on the closure of the strata $\X((m^v,1^w),\emptyset)$, where $v,w \in \mathbb{N}$ such that $n = v m + w$. 
\end{enumerate}
\end{thm}

\begin{rem}
The behaviour of $\CG_{\lambda,c}$, when $c \in \Z = \Sing_0$, is quite different from
(and, in a sense, less interesting than) the case where $c \notin \Z$ and it
will not be considered in this article. When $c \in \Sing$, it is likely that $\CG_{\lambda,c}$ will also have subquotients supported on the closure of the strata $\X((m^v,1^w),(m^u))$, where $n = (v + u)m + w$. 
\end{rem}

\begin{cor}\label{cor:main}
Let $c \in \C$ and $\la \in \tSL^* /W$.
\begin{enumerate}
\item If $c \notin \Sing$ then the Harish-Chandra $\dd$-module is the minimal extension of its restriction to $\X^\reg$. 

\item If $c \in \Sing_-$, resp. $c \in \Sing_+$, then $\CG_{\la,c}$ has no quotient modules, resp. submodules, supported on $\X^\cyc\sminus\X^\reg$.
\end{enumerate}
\end{cor}

\subsection{Specialization}
P. Etingof introduced certain  `sheafified' versions  of Cherednik
algebras.
Specifically, 
associated to any smooth quasi-projective variety $X$ and finite group
$W$ of automorphisms of $X$,
Etingof \cite{ChereSheaf} defines  a sheaf   $\mc{H}_{\kappa}(X,W)$ of
associative algebras on the quotient $X/W$. Here $\kappa$ is a (multi-) parameter
and the family of the algebras
$\mc{H}_{\kappa}(X,W)$ is
 a flat deformation of  the algebra $\dd_X \rtimes W$. 
%It follows from the construction that the
%natural action of $\dd_X \rtimes W$ on  $\mc{O}_X$
%gets deformed to a well-defined $\mc{H}_{\kappa}(X,W)$-action  on
%$\mc{O}_X$,
%for each $\kappa$.
In this paper, we restrict ourselves to the (most interesting) case where the action
of $W$ on $X$ is `generated by pseudo-reflections' in the sense
of \S\ref{sec:Cheredefn} below.

A  subgroup of $W$ is called a {\em
  parabolic subgroup}
if it is equal to the stabilizer of some point $x\in X$.
Fix a  parabolic  subgroup   $W' \subset W$ and a
 connected component $Y$ of the set of points in $X$ with stabilizer
 $W'$. Let  $\NN_{X/Y}$ denote  the normal bundle to $Y$ in $X$ and $W_Y := \{ w
 \in W \ | \ w(Y) = Y \}$. 
Mimicing a construction of Kashiwara, we introduce
a canonical  $\Z$-filtration  on $\mc{H}_{\kappa}(X,W)$,
which we call  the \textit{$V$-filtration}.
We show that the associated graded of $\mc{H}_{\kappa}(X,W)$ with
respect to the $V$-filtration is isomorphic to
$\mc{H}_{\kappa'}(\NN_{X/Y},W_Y)$, the Cherednik algebra associated
to  the variety
 $\NN_{X/Y}$ and the group $W_Y$
that acts on $\NN_{X/Y}$ by vector bundle automorphisms.
This allows one to define
a functor
$$
\Sp_{X/WY} : \Lmod{\mc{H}_{\kappa}(X/WY)} \to \Lmod{\mc{H}_{\kappa'}(\NN_{X/Y},W_Y)},
$$
on an appropriate  category of specializable
$\mc{H}_{\kappa}(X,W)$-modules.
This  is a Cherednik algebra analogue of 
 Verdier specialization of $\dd$-modules, as defined by Kashiwara.
Our functor $\Sp_{X/WY}$ 
enjoys the expected properties of specialization,
in particular, it is an exact functor and it comes equipped with
a canonical monodromy automorphism.

Next, we return to the case of trigonometric 
Cherednik algebras of type $A$. Thus,
let  $\TSL$ be the maximal torus of $\SSL$. Given a Levi subgroup $\LSL \subset \SSL$, let
$W_L\sset W=S_n$ be the Weyl group of $L$ and
 $Y:=\TSL^{W_L}_{\circ}$  the set of points in $\TSL$ with stabilizer $W_L$. The group $W$ acts on the Lie algebra $\mf{t}$ of $\TSL$ as a reflection group. Bezrukavnikov and Etingof, \cite{BE}, defined for each $b \in \mf{t}$ such that $W_b = W_L$, a restriction functor $\Res_b : \mc{O}(W) \rightarrow \mc{O}(W_L)$ from category $\mc{O}$ for the \textit{rational} Cherednik algebra $\mc{H}_{\kappa}(\mf{t},W)$ to category $\mc{O}$ for $\mc{H}_{\kappa}(\mf{t}_L,W_L)$. We construct an analogous restriction functor $\Res_Y : \Lmod{\mc{H}_{\kappa}(\TSL,W)} \to \Lmod{\mc{H}_{\kappa'}(Y \times \mf{t}_L,W_Y)}$. We show that, on the category of $\mc{H}_{\kappa}(\TSL,W)$-modules that are
coherent over $\mc{O}_{\TSL/W}$, the restriction functor and specialization functors agree: 

\begin{thm}\label{prop:specialcoherent}
Let $\ms{M}$ be a ${\mathcal H}_{\kappa}(\TSL,W)$-module, coherent as an $\mc{O}_{\TSL/W}$-module. Then, $\ms{M}$ is specializable along $W Y$ and $\Sp_{\TSL/WY}(\ms{M}) = \Res_Y(\ms{M})$. 
\end{thm}

Our general construction of $V$-filtration associates
to the data $(W_L,\,Y)$ the canonical $V$-filtration on 
the algebra  $\mc{H}_{\kappa}(\TSL,W)$. The latter induces, by restriction,
 a similar filtration on $\htrig_{\kappa}=e \mc{H}_{\kappa}(\TSL,W) e$,
 the spherical subalgebra.
On the other hand, 
we consider a certain Zariski open subset  $Z(\LSL)^{\circ}$  of
$Z(\LSL)$ and put $\mc{Y}:=\SSL\cdot(Z(\LSL)^{\circ}\times\{0\})$,
the $\SSL$-saturation of the set
$Z(\LSL)^{\circ}\times\{0\}\sset\SSL\times V$.
Then, $\mc{Y}$ is a $G$-stable locally-closed subvariety of
$\X$. Therefore, there is an associated Kashiwara 
$V$-filtration on the algebra $\dd(\X)$.
We show that  the Kashiwara filtration on  $\dd(\X)$ 
goes, via quantum
Hamiltonian reduction, to our $V$-filtration on $\htrig_{\kappa}$.
In more detail, the quantum
Hamiltonian reduction construction provides an isomorphism
of   $\htrig_{\kappa}$ with a quotient
of the algebra $\dd(\X)^G$. The 
$V$-filtration on  $\dd(\X)$ gives, by restriction,
a filtration on  $\dd(\X)^G$ and the latter induces
a quotient filtration on the Hamiltonian reduction.
We prove that the resulting filtration on  $\htrig_{\kappa}$
equals the filtration obtained by restricting
the $V$-filtration on $\mc{H}_{\kappa}(\TSL,W)$ to
 the spherical subalgebra.

There is also a parallel story for module categories over
the algebras in question. To explain this, 
let  $\mf{l}'$ be the Lie algebra of $\LL := [L,L]$ and let
$\NN_L := Z(\LSL)^{\circ} \times \mf{l}'
\times V$ be a trivial  vector bundle over $Z(\LSL)^{\circ}$, which plays the role of 
a normal bundle. 
In section \ref{sec:nearby}, we use the standard
specialization functor on $\dd$-modules
to construct
an exact functor $\sp_L : \ss_c \rightarrow \ss_{\NN_L,c}$, that takes mirabolic modules
 on $\X$ to mirabolic modules on $\NN_L$.
The main result of section \ref{10} says that
%A similar construction for $\dd$-module on the flag variety was described in \cite{ENV}. 
the functor of Hamiltonian reduction intertwines the specialization
functor for mirabolic modules with the specialization functor for
modules in category $\mc{O}$ for the trigonometric Cherednik algebra, 
that is, we have

\begin{thm}\label{thm:admissiblecommute}
Assume that $c$ is admissible. Then, the following diagram commutes
$$
\xymatrix{
\ss_c \ar[rr]^{\Ham_c} \ar[d]_{\sp_L} & & \mc{O}_{\kappa} \ar[d]^{\Sp_{W_L}} \\
\ss_{\NN_L,c} \ar[rr]_{\Ham_L} & & \mc{O}_{\kappa}(W_L).
}
$$
\end{thm}

As a consequence of Theorem \ref{thm:admissiblecommute}, one can exploit
Saito's theory of mixed Hodge modules to prove non-trivial results about the {\mbox Bezrukavnikov-Etingof} restriction functor, c.f. Corollary \ref{cor:mixedhodge}. 

\subsection{Outline of the article}

In section \ref{sec:mirabolic}, we first recall the definition of
$G$-monodromic $\dd$-modules and describe some of their basic
properties. Then, we give two definitions of mirabolic modules and prove
Theorem \ref{thm:admissible2} from the introduction, which claims that
these two definitions are equivalent. We use this theorem  to construct the spectral decomposition of the category of mirabolic modules. 

In section \ref{strat}, we introduce a  stratification of the space $\X$ and remind the reader of the definition of the \textit{relevant} strata in this stratification. Then, cuspidal mirabolic modules on Levi subgroups of $\SSL$ are studied in section \ref{chetyre}. We explicitly describe all cuspidal modules. 

In order to be able to use this classification of cuspidal modules, we study the restriction of mirabolic modules to a relevant stratum in section \ref{sec:reduction}. This produces, for each simple mirabolic module supported on the closure of a given relevant stratum $\X(L,\Omega)$, a cuspidal module associated to $L$. 

In section \ref{ham_sec}, we study the functor $\Ham$ of Hamiltonian
reduction, which is one of the main objects of interest in the
article. We show that $\Ham$ possess both a left and right
adjoint. Moreover, $\Ham$ is shown to be compatible with the reduction
functor, in an obvious sense. Combining this with
 an explicit description of the cuspidal mirabolic modules, we are able
 to deduce considerable information about the kernel of the functor $\Ham$. 

The proof of the main results, Theorems \ref{cor:unstable} and \ref{twistshift}, are contained in section \ref{ham_shift}. Here, we study the compatibility of $\Ham$ with shift functors on $\ss_c$ and $\oo_{\kappa}$. 

Then, in section \ref{jj}, we define the $V$-filtration on sheaves of Cherednik algebras. Using this definition, we construct the corresponding specialization functor for $\mc{H}_{\kappa}(X,W)$-modules. It is show that the specialization functor agrees with the restriction functor defined by Bezrukavnikov-Etingof when restricted to sheaves coherent over $\mc{O}_{X/W}$. 

Similarly, in section \ref{sec:nearby}, a specialization functor $\sp_L$ is defined on the category of mirabolic modules. We give a formula for the characteristic variety of $\sp_L(\mm)$ in terms of the characteristic variety of $\mm$. It is show in the final section that the functor of Hamiltonian reduction intertwines the specialization functor on category $\OO$ for the trigonometric Cherednik algebra of section \ref{jj} with the specialization functor for mirabolic modules of section \ref{sec:nearby}.

The appendix contains some details on the radial parts map and summarizes the results of \cite{GGS} (in the trigonometric case).  

\subsection{Acknowledgments}

The authors would like to thank Iain Gordon for stimulating discussions.
We are grateful to Kevin McGerty and Tom Nevins for
showing us a draft of their preprint \cite{McN} before it was made public.
 The first author is supported by the EPSRC grant EP-H028153. The research of the second author was supported in part by the NSF award DMS-1001677.

\section{Mirabolic modules}\label{sec:mirabolic}

In this section we give two definitions of mirabolic $\dd$-modules. Our main result, Theorem \ref{thm:admissible2}, says that these definitions are equivalent. Before we do that, we recall some of the basic properties of monodromic $\dd$-modules. 

\subsection{Monodromic $\dd$-modules} 
Let $ G$ be a connected linear algebraic group with Lie algebra
$\mf{g}$.  We fix a smooth $G$-variety $X$ and let $\mu : T^* X
\rightarrow \g^*$ be the moment map for the induced action on the
cotangent bundle of $X$. By a local system 
 we always mean an algebraic vector bundle equipped with an integrable connection that has {\textbf{regular singularities}}.

\begin{defn} A coherent $\dd_X$-module $\ms{M}$ on $X$ is said to be $G$-\textit{monodromic} if $\Char(\ms{M}) \subset \mu^{-1}(0)$.
Let  $\mon{\dd_X,G}$ denote the category of $G$-monodromic  $\dd_X$-modules.
\end{defn}

We begin with a general result, which is probably
well-known. We provide a proof  since we were unable
to find the statement in the existing literature.

\begin{prop}\label{prop:monodromicfinite}
Let $X$ and $Y$ be smooth quasi-projective $G$-varieties.

\vi For any $G$-equivariant morphism $f: X\to Y$ and $k\in\Z$, the functors
$f_*,\,f_!,\, f^*,f^!$ (more precisely, the  cohomology in every degree
of the derived functor in question), as well as Verdier duality,
 preserve the categories of 
$G$-monodromic, regular
holonomic $\dd$-~modules.

\vii Let $\ms{M}$ be a $G$-monodromic,
regular holonomic $\dd$-module on $X$. Then the action of $\g$ on $H^k(X,\ms{M})$ is
locally finite. 

\viii Conversely, let $\ms{M}$ be a coherent $\dd_X$-module generated by a finite dimensional $\g$-stable subspace of $\Ga(X,\mm)$. Then $\ms{M}$ is $G$-monodromic. 
\end{prop}

Part (ii) of the Proposition was also stated
(without proof) as Proposition B1.2 in
\cite{AdmissibleModules}. 

\begin{proof} 
The proof of (i) mimics the standard argument \cite[Theorem 6.1.5]{HTT}
in the non-monodromic case. First of all, Verdier duality doesn't
affect the characteristic variety, so it takes $G$-monodromic modules
into $G$-monodromic modules. Next, 
given a map $f: X\to Y$,
one has the standard diagram, c.f. \cite[\S 2.4]{HTT}:
$$
\xymatrix{
T^* X \ && \ X \times_Y T^* Y\ \ar[ll]_<>(0.5){\rho_f}
\ar[rr]^<>(0.5){\omega_f} &&\  T^* Y.
}
$$ 

Let $\mm$ be a $G$-monodromic $\dd_X$-module.
In the special case where $f$ is a proper morphism we have $f_*=f_!$
and, thanks to a result of Kashiwara, \cite[Remark 2.4.8]{HTT}, one has
$\SS(R^kf_*\mm)\sset \omega_f(\rho_f\inv(\SS \mm))$. This yields the
statement of part (i) involving the functors $f_*,\,f_!$. A similar
estimate, \cite[Theorem 2.4.6]{HTT}, yields the statement 
involving the functors $f^!$ and $f^*$ in the case where $f$ is a smooth morphism.

Next, let $f$ be an {\em affine} open embedding. Then, all four functors  $f_*,\,f_!,\,f^*,\,f^!$
are exact. The statement of part (i)
involving  pull-back functors is clear. 
Thus, thanks to duality, it remains to 
consider the functor $f_*$. To this end, we put
$Z:=Y\sminus f(X)$. Since $f$ is assumed to be affine, $Z$ is a $G$-stable divisor in $Y$. Then, one shows that
there exists a $G$-equivariant line bundle  on $Y$,
a section $s$ of that line bundle, and  a character $\chi\in \g^*$ such
that $\mu(a)(s)=\chi(a)\cdot s$ holds for all $a\in\g$
(i.e. $s$ is a semi-invariant section) and, moreover, set-theoretically
one has
$Z=s\inv(0)$, see \cite{MirabolicCharacter}, Lemma 2.2.1. 
The semi-invariance property implies 
that for any $t\in\C$, we have the equality $\mu(\mu\inv(0)+t\cdot d\log s)=t\cdot \chi$ in $\g^*$, c.f. \cite[\S2]{MirabolicCharacter} for a discussion of the relevant geometry.

Now, let $\mm$ be a $G$-monodromic,
holonomic $\dd_X$-module with regular singularities.
Then, using \cite[Theorem
6.3]{GinzburgBiChern} in the first equality below, we find
$$\SS(f_*\mm)=\underset{^{t\to0}}\lim\
[\SS(\mm) + t \cdot d\log s]\ \sset\
\underset{^{t\to0}}\lim\ [\mu\inv(0) + t \cdot d\log s]\ 
\sset \underset{^{t\to0}}\lim\ \mu\inv(t\cdot
\chi) =\mu\inv(0).$$
We conclude that  $f_*\mm$ is a
$G$-monodromic $\dd_Y$-module. 

Finally, let $f: X\to Y$ be an arbitrary $G$-equivariant morphism. Then,  by Sumihiro's Theorem, \cite[Theorem 5.1.25]{CG}, the variety
$X$ has  a $G$-equivariant completion, i.e., there exists a projective $G$-variety ${\bar X}$ and a $G$-equivariant open embedding $j: X\into {\bar X}$, with dense image. Applying, if necessary, the equivariant version of Hironaka's Resolution of Singularities Theorem, \cite{BierstoneMilman}, we may ensure that the variety ${\bar X}$ is smooth and  that ${\bar X}\sminus X$ is a divisor in ${\bar X}$. The morphism $f$ factors as a composition
\beq{graph}
\xymatrix{
X\ \ar[rr]^<>(0.5){\text{graph}(f)}&&
\ X\times Y\
 \ar[rr]^<>(0.5){j\boxtimes\id_Y} &&\ {\bar X} \times Y \
\ar[r]^<>(0.5){pr_Y}&\ Y},
\eeq
where the first map is a closed embedding via the graph of $f$.

The statement of part (i) involving direct image functors holds for
each of the three maps in \eqref{graph},
thanks to
the special cases 
considered above. Therefore,
it holds for the map $f$ itself.
A similar argument yields the  statement of part (i) involving 
pull-back functors provided the statement holds for the first map
in \eqref{graph}.
More generally, let $i: X \into Y$ be  an arbitrary closed embedding
and let  $\ms{N}$ be  a
$G$-monodromic, regular holonomic $\dd_Y$-module.
Put
$U:=Y\sminus i(X)$; let $j:\ U\into Y$ be the natural open
embedding and let $\mm:=j^*\ms{N}$.
Then, there is a canonical
exact triangle $i_!i^!\ms{N}\to \ms{N}\to j_*\mm$.  We know that both $\ms{N}$
and $j_*\mm$ are $G$-monodromic and regular holonomic. 
It follows that $i^!\ms{N}$ is also
$G$-monodromic. This completes the proof of part (i).
 
To prove (ii), we choose a $G$-equivariant completion
$j: X\into {\bar X}$, as above, and let $\ms{N}:=j_*\mm$.
Thus, $\ms{N}$ is a $G$-monodromic,  regular holonomic
$\dd$-module.
Such a $\dd$-module has a canonical good filtration $\mc{F}_i \ms{N}$
such that the support of the associated graded sheaf is reduced,
\cite[Corollary 5.1.11]{KashiwaraKawai}. We know that the
characteristic variety of $\ms{N}$ is contained in $\mu^{-1}(0)$, where
$\mu$ is the moment map for the action of $G$ on $T^*{\bar X}$. This means
that if $\sigma_a $ is the vector field on ${\bar X}$ corresponding to $a
\in \g$ then multiplication by the symbol of $\sigma_a$ is the zero map
$\mc{F}_i \ms{N} / \mc{F}_{i-1} \ms{N} \rightarrow \mc{F}_{i+1} \ms{N} /
\mc{F}_{i} \ms{N}$. Hence the $\g$-action
respects  the filtration. 
Note that, for any $k$ and $i$,  the space 
$H^k({\bar X}, \mc{F}_{i} \ms{N})$ is finite dimensional,
since $\mc{F}_{i} \ms{N}$ is a coherent sheaf on
${\bar X}$ and ${\bar X}$ is a projective variety.

Finally, since $j$ is an affine embedding, we get
$$
H^k(X,\mm)\ =\ H^k({\bar X}, j_*\mm)\ =\ H^k({\bar X}, \ms{N})\ =\  \underset{^{\stackrel{\too}i}}\lim\
H^k({\bar X}, \mc{F}_{i} \ms{N}).$$
We deduce that the $\g$-action on $H^k(X,\mm)$ is locally
finite.

\viii Let $M_0 \subset \Gamma(X, \ \mm)$ be a finite dimensional, $\g$-stable subspace that generates $\mm$ as a $\dd$-module. If $\dd_{X, \le i}$ for $i \in \Z$ denotes the order filtration on $\dd_X$, then $\mc{F}_i := \dd_{X, \le i } \cdot M_0$ defines a $\g$-stable, good filtration on $\mm$. Since $\mu(\g) \subset \dd_{X, \le 1} \sminus \dd_{X, \le 0}$, this implies that $\Char (\mm) \subseteq \mu^{-1}(0)$ i.e. $\mm$ is $G$-monodromic. 
\end{proof}

\subsection{$(G, \mathbf{q})$-monodromic $\dd$-modules}\label{BQ}
From now on, we assume that $G$ is a {\em reductive}, connected  group
with Lie algebra $\g$. Let $\Hom(G,\Cs)$ be the character lattice of $G$.
Taking the differential of a  character at the identity element of 
the group $G$
yields an embedding of the
character lattice into the vector space $(\g / [\g,\g])^*$,
of Lie algebra characters.
We put $\BQ(G) = (\mf{g}/[\mf{g},\mf{g}])^*/ \Hom(G,\Cs)$.

Given an element $\chi \in (\mf{g}/[\mf{g},\mf{g}])^*$, let
$\mc{O}_{ G}^{\chi} := \dd( G)/ \dd(G)\cdot\{x-\chi(x),\ x\in\g\}$. 
This  $\dd$-module  is a rank one
local system on $G$. It is easy to show that every rank one local system on $G$ is isomorphic to
$\mc{O}_{ G}^{\chi}$, for some $\chi \in (\mf{g}/[\mf{g},\mf{g}])^*$. 
Any two local systems $\mc{O}_{ G}^{\chi}$ and $\mc{O}_{ G}^{\psi}$ are isomorphic if and
only if the image of $\chi$ and $\psi$ in $\BQ(G)$ are the
same. Therefore, given $\bq \in\BQ(G)$, the local system
$\mc{O}_{ G}^{\bq}$ is well-defined up to isomorphism.

Next, fix a smooth $ G$-variety $X$ and write $a :  G \times  X \rightarrow  X$ for the action map.

\begin{defn}
Let $\bq \in \BQ(G)$. A coherent $\dd$-module on $X$ is said to be $(G,\bq)$-\textit{monodromic} if there is an isomorphism $a^* \ms{M} \stackrel{\sim}{\rightarrow} \mc{O}_{ G}^{\bq} \boxtimes \ms{M}$ of $\dd$-modules on $G \times X$.
\end{defn}

It is clear that for $\bq = 1$ we have $\mc{O}_{ G}^{\bq} =\mc{O}_{G}$. Thus, since we have assumed that $G$ is connected, being
$(G,1)$-monodromic is the same thing as being $ G$-equivariant,
c.f. \cite[\S 1.9]{CharSheaves5}. If $G$ is connected and semisimple
then the only linear character of $\g$ is $\chi = 0$, so any
$(G,\bq)$-monodromic $\dd$-module is $G$-equivariant. In general, any
$(G,\bq)$-monodromic $\dd$-module is a  $G$-equivariant,
quasi-coherent
$\mc{O}_X$-module.  

Further, any $(G,\bq)$-monodromic $\dd$-module is clearly $G$-monodromic. 
An extension of two  $(G,\bq)$-monodromic
$\dd$-modules is a $G$-monodromic $\dd$-module which is not necessarily
$(G,\bq)$-monodromic, in general.

Let  $\mon{\dd_X, G,\bq}$
be the full subcategory of the
category $\mon{\dd_X, G}$, of
$G$-monodromic  $\dd_X$-modules, whose objects are
 $(G,\bq)$-monodromic $\dd_X$-modules.
It is clear that $\mon{\dd_X, G}$  is an abelian category
and $\mon{\dd_X, G,\bq}$ is an abelian subcategory (which is not, however, a Serre subcategory of $\mon{\dd_X, G}$, in
general).

The above is complemented by the following result.

\begin{lem}\label{lem:monoimpliesq}
\vi Let $\ms{M}$ be a simple, holonomic $G$-monodromic $\dd$-module with regular singularities. Then, there exists some $\bq \in \BQ(G)$ such that $\ms{M}$ is $(G,\bq)$-monodromic.

\vii Assume there exists some $\bq \in \BQ(G)$ such that $\ms{M}$ is
$(G,\bq)$-monodromic and moreover $G$ acts on $X$ with finitely many orbits. Then $\ms{M}$ has regular singularities. 

\viii Let $\mm$ be a coherent $\dd$-module, generated by a $\mu(\g)$-semi-invariant global section  of weight $\chi\in(\g/[\g,\g])^*$. Then $\mm$ is $(G,\bq)$-monodromic, where $\bq$ is the image
of $\chi$ in $\BQ(G)$.
 \end{lem}

\begin{proof}

\vi We may view $\supp\mm$, the support of $\mm$, as a closed subvariety of the zero section of the cotangent bundle $T^*X$ (that is, if $\Supp \mm$ is not dense in $\X$, then we identify $\Supp \mm$ with the closure of the conormal to $(\Supp \mm)_{\reg}$ in $T^* \X$). Thus, $\supp\mm$ is an irreducible component of $\SS(\mm)$, a $G$-stable Lagrangian subvariety in $T^*X$. Let $\SS'(\mm)$ be the union of all the irreducible components of $\SS(\mm)$ but  $\supp\mm$. Thus, $\SS'(\mm)$ is a $G$-stable closed subvariety of $T^*X$. Hence, $U:=\supp\mm\sminus \SS'(\mm)$ is a  $G$-stable open  subvariety of $\supp\mm$. Replacing $U$ by its smooth locus if necessary, we may assume in addition that $U$ is smooth.
Let $j: U\into X$ denote the (locally closed) embedding
and let $\mathsf{L}=j^!\mm$. Applying Kashiwara's theorem, c.f. \cite[Theorem 1.6.1]{HTT}, we find that the characteristic variety of $\mathsf{L}$ equals the zero section of $T^*U$. It follows, since $\mm$ is assumed to have regular singularities, that $\mathsf{L}$ is a local system on $U$. Further, since $\mm$ is simple, we deduce that the canonical morphism $\mm\to j_*j^!\mm$ yields an isomorphism $\mm\iso  j_{!*} \mathsf{L}$. Moreover, the local system  $\mathsf{L}$ must be irreducible.

Let $a :  G \times U \rightarrow U$ be the action map. Since
$\mathsf{L}$ is simple, $a^* \mathsf{L}$ is a simple local system on $ G
\times U$ and thus, thanks to the Riemann-Hilbert correspondence
and an isomorphism $\pi_1(G\times U)
=\pi_1(G)\times \pi_1(U)$, equals $\mathsf{L}_1 \boxtimes \mathsf{L}_2$
for some simple local systems $\mathsf{L}_1$ on $ G$ and $\mathsf{L}_2$
on $U$. If $i : U \rightarrow  G \times U$, $x \mapsto (e,x)$, then $i
\circ a = \mathrm{Id}_U$ implies that $\mathsf{L}_2 \simeq i^*
(\mathsf{L}_1 \boxtimes \mathsf{L}_2) = \mathsf{L}$ (here we have used
the fact that the rank of $\mathsf{L}_1$ is one). Hence, by the
functorality of minimal extensions, $a^* \ms{M} \simeq \mathsf{L}_1
\boxtimes \ms{M}$, and (i) follows.
Part \vii is \cite[Lemma 2.5.1]{MirabolicCharacter}.

\viii Let $u\in\Ga(X,\mm)$ be a  semi-invariant
section of weight $\chi \in (\g / [ \g, \g])^*$ 
that generates $\mm$. Let $\mu_X : \g \rightarrow \dd_X$ be the differential of the action map $a : G \times X \rightarrow X$ and $\mu_G : \g \rightarrow \dd_{G}$ the embedding as \textit{right invariant} vector fields (i.e. the map obtained by differentiating the action of left multiplication by $G$ on itself). Then $a$ is $G$-equivariant, where $G$ acts on $G \times X$ by acting on $G$ by left multiplication and acts trivially on $X$. This implies that 
$$
\mu_G(A) \o m = 1 \o \mu_X(A) \cdot m \in \dd_{G \times X \rightarrow X} \o_{a^{-1} \dd_X} a^{-1} \mm \quad \forall \ A \in \g, \ m \in \mm.
$$
Hence, there is a nonzero homomorphism $\phi : \mc{O}_{}^{\chi} \boxtimes \mm \rightarrow a^* \mm$, given by $\one_{G} \boxtimes u \mapsto 1 \o
u$. This is surjective since $a^* \mm$ is generated by $1 \o u$. Therefore, we have a short exact sequence
\[
\xymatrix{
0\ar[r]& \ms{K} \ar[r] &\mc{O}_G^{\chi} \boxtimes \mm \ar[r]^<>(0.5){\phi} & a^* \mm \ar[r] & 0.}
\]
For each $g \in G$, let $j_g : X \hookrightarrow G \times X$ be given by $j_g(x) = (g,x)$. The module $\mc{O}_G^{\chi} \boxtimes \mm$ is non-characteristic for $j_g$, hence so too are $\ms{K}$ and $a^* \mm$. Thus, the sequence 
\[
\xymatrix{
0\ar[r]&j^*_g \ms{K} \ar[r] &\mm \ar[rr]^<>(0.5){j^*_g(\phi)} && j^*_g a^* \mm \ar[r] & 0}
\]
is exact. Since $j^*_g (\phi) = \mathrm{Id}_{\mm}$, 
we have $j^*_g \ms{K} = 0$. The fact that $\ms{K}$ is non-characteristic for $j_g$ and $j_g^* \ms{K} = 0$ for all $g \in G$ implies that $\ms{K} = 0$.  
\end{proof}

The special case $G=GL(V)$ will be most important for  us. In that case, we have $\g={\mathfrak{gl}}_n$ so the vector space $(\g/[\g,\g])^*$ is one dimensional with $\Tr$, the trace function, being the natural base element. Similarly, the character lattice  $\Hom(G,\Cs)$ is, in this case, a free abelian group with generator $\det$, the determinant character. The above mentioned canonical embedding $\Hom(G,\Cs)\into (\g / [\g,\g])^*$ sends  $\det$ to $\Tr$. Thus, we have $\BQ(G) \cong \Cs$. Explicitly,
one has $\C \iso (\g / [\g,\g])^*$ by $c \mapsto c \Tr$, and hence  
\beq{eq:indentify}
\C \iso (\g /[\g,\g])^* \rightarrow \BQ(G) \iso \Cs, \quad c \mapsto q := \exp ( 2 \pi \sqrt{-1} c ).  
\eeq

\subsection{$(H, \chi)$-monodromic $\dd$-modules}

Later, we will need to consider non-connected subgroups $H$ of $GL(V)$. In this case, we use the notion of $(H,\chi)$-monodromicity. Therefore, let $G$ be a connected, reductive group and $H$ a Levi subgroup. Recall that $a : G \times X \to X$ is the action map and $m : G \times G \to G$ the multiplication map. Let $s : X \to G \times X$, $s(x) = (e,x)$.    

Choose $\chi \in (\g / [\g,\g])^*$ and let $\bq$ be its image in $\BQ(G)$. Recall that a coherent $\dd$-module $\mm$ is said to be $(G,\bq)$-monodromic if there exists an isomorphism $\theta : \mc{O}_G^{\chi} \boxtimes \mm \iso a^* \mm$. We say that the $(G,\bq)$-monodromic module $\mm$ satisfies the \textit{cocycle} condition if $s^* \theta = \id_{\mm}$ and $\theta$ can be chosen so that the following diagram is commutative:
\beq{eq:cocycle}
\xymatrix{
\mc{O}^{\chi}_G \boxtimes \mc{O}^{\chi}_G \boxtimes \mm \ar[rr]^{\id_{G} \times \theta} \ar[d]_{=} &&  \mc{O}^{\chi}_G \boxtimes a^* \mm \ar[d]^{=} \\
(m \times \id)^* (\mc{O}_G^{\chi} \boxtimes \mm) \ar[d]_{(m \times \id_X)^* \theta\en} &&  (\id_G \times a)^*(\mc{O}_G^{\chi} \boxtimes \mm) \ar[d]^{\en(\id_G \times a)^* \theta} \\
(m \times \id_X)^* a^* \mm \ar[rr]^= && (\id_G \times a)^*a^* \mm  
}
\eeq

The cocyle condition allows us to extend the definition of monodromicity to non-connected groups. 

\begin{defn}\label{defn:strongmonodormic}
Let $H$ be a (not necessarily connected) reductive group acting on $X$. Then, a coherent $\dd$-module is said to be $(H,\chi)$-monodromic if there is a \textit{fixed} isomorphism $\theta : \mc{O}_H^{\chi} \boxtimes \mm \to a^* \mm$, satisfying the cocycle condition.
\end{defn}

The category of all $(H,\chi)$-monodromic modules on $X$ is denoted $\mon{\dd_X, H,\chi}$. The following proposition is well-known in the equivariant case, and the proof in the mondromic is identical.   

\begin{prop}\label{prop:monoimpliesweakG}
Let $\mm$ a $(G,\bq)$-monodromic, regular holonomic $\dd$-module. Then $\mm$ can be endowed with a (non-unique) $(G,\chi)$-monodromic structure. 
\end{prop}

\begin{rem}
The proposition implies that every regular holonomic $(G,\bq)$-monodromic module is weakly $G$-equivariant, and hence the cohomology groups $H^i (X, \mm)$ are rational $G$-modules. It also implies that the forgetful functor $\mon{\dd_X, G,\chi}_{\reg} \to \mon{\dd_X, G,\bq}_{\reg}$ is essentially surjective. 
\end{rem}

It is still useful to work with the weaker notion of $(G,\bq)$-monodromicity because we will encounter local systems which are $(G,\bq)$-monodromic, but have \textit{no} canonical $(G,\chi)$-monodromic structure. The strong notion of monodromicity given in Definition \ref{defn:strongmonodormic} agrees with the definition commonly found in the literature c.f. \cite[\S 2.5]{MirabolicCharacter} and \cite[\S 4]{KasAsterisque}. 

\begin{examp}\label{examp:one}
There are three natural left actions of $G$ on itself, multiplication on the left, multiplication on the right and the adjoint action. The corresponding action maps are denoted $a_L,a_R,a_{\mathrm{Ad}} : G \times G \to G$. If we stipulate that the generator $\mathbf{1}$ of $\mc{O}_G^{\chi}$ is $G$-invariant, then there are \textit{canonical} isomorphisms
%\begin{align*}
$$\mc{O}_G^{\chi} \boxtimes \mc{O}_G^{\chi}  \iso a_L^* \mc{O}_G^{\chi},\qquad
\mc{O}_G^{-\chi} \boxtimes \mc{O}_G^{\chi}  \iso a_R^* \mc{O}_G^{\chi},\qquad
\mc{O}_G \boxtimes \mc{O}_G^{\chi}  \iso a_{\mathrm{Ad}}^* \mc{O}_G^{\chi}.
$$
%\end{align*}
In particular, $\mc{O}_G^{\chi}$ is $G$-equivariant for the adjoint action of $G$ on itself. 
\end{examp}
    
\begin{cor}
Let $H \subset G$ be a (\textit{not} necessarily connected) closed subgroup of $G$ and $\mm$ a regular holonomic, $(G,\bq)$-monodromic module. Then $\mm$ can be endowed with a $(H,\chi|_{\h})$-monodromic structure. 
\end{cor}

\begin{proof}
Let $i_H : H \hookrightarrow G$ be the inclusion. Apply $(i_H \times i_H \times \id_X)^*$ to the cocycle diagram for $\mm$. 
\end{proof}

%Under this identification, the quotient $(\g /[\g,\g])^* \rightarrow
%\BQ(G)$ 
%is given by
%becomes the map 

\subsection{} Fix $T \subset B \subset G$, a maximal torus contained
inside a Borel subgroup of $G$, and let $U$ be the unipotent radical of
$B$. 
Write $\t\sset\b\sset\g$ and $\u$,
for the corresponding Lie algebras.
The group $T$ acts freely on the affine base space $G / U$ by multiplication on
the right, and $\mc{B}:=G/B=(G / U ) / T$ is the flag manifold. 
Let $\td\sset T\times T$ be the diagonal and $\TT:=(T\times T)/\td$.
Let $\horo = (G / U \times G / U) /\td$,  the horocycle space,
be the quotient of $G / U \times G / U$ by the
$T$-diagonal action  on the right.
Thus, the  $(T\times T)$-action on $G / U \times G / U$ on the right
 descends to a free $\TT$-action on $\horo$.

Consider the diagram 
\beq{diag1}
\xymatrix{
 & G \times \mathcal{B} \ar[dl]_{\mathsf{p}} \ar[dr]^{\mathsf{q}} & \\
G &  & \horo,
}
\eeq
where $\mathsf{p}$ is projection along $\mc{B}$ and $\mathsf{q}(g,F) = (F,g(F))$.

Given a holonomic  $\dd$-module $\mm$ on $G$ and $i\in\Z$, let $\ms{H}^i(\mathsf{q}_! \mathsf{p}^*\ms{M})$ denote the $i$-th cohomology $\dd_\horo$-module of the complex $\mathsf{q}_! \mathsf{p}^*\ms{M}$.

\begin{lem}\label{lem:keylemma}
Let $\ms{M}$ be a holonomic $\dd$-module on $G$. Then, for any $i\in\Z$,
one has

\vi If  $\Char (\ms{M}) \subseteq G \times \mc{N}$ 
and $\mm$ has regular singularities, then 
  $\ms{H}^i(\mathsf{q}_! \mathsf{p}^*\ms{M})$ is 
a $\TT$-monodromic $\dd$-module.

\vii Let $H$ be a subgroup of $G$ and let $\mm$ be  $H$-monodromic
with respect to the adjoint action of $H$ on $G$. If
 $\mm$ has  regular singularities, then
 $\ms{H}^i(\mathsf{q}_! \mathsf{p}^*\ms{M})$ is an $H$-monodromic $\dd$-module
with respect to the $H$-diagonal action on $\horo$ on the left.

\viii If the $\FZ$-action on $\Ga(G,\mm)$ is locally finite then, for
any 
$T$-stable open subset $U\sset \horo$, the  $\U(\Lie \TT)$-action on $\Gamma(U, \ \ms{H}^i(\mathsf{q}_! \mathsf{p}^*\ms{M}))$ is locally finite. 
\end{lem}

\begin{proof}

\vi Since $\mathsf{p}$ is smooth, the characteristic variety of $\mathsf{p}^* \ms{M}$ will be contained in $G \times \mc{N} \times \mc{B} \times \{ 0 \}$. The map $\mathsf{q}$ is also smooth and the fiber of $\mathsf{q}$ over $(h_1 U , h_2 U)$ can be identified with $h_2 U h_1^{-1}$. Moreover, one can show that $\mathsf{q}$ makes $G \times \mc{B}$ a locally trivial (in the Zariski topology) fiber bundle over $\horo$. Therefore the result \cite[Proposition B.2]{MirkovicVilonen}, based on \cite[Theorem 3.2]{CharCycles}, shows that 
\begin{equation}\label{eq:charinc}
\Char (\mathsf{q}_! \mathsf{p}^*\ms{M}) \subseteq \overline{\omega_{\mathsf{q}} (\rho_{\mathsf{q}}^{-1}(G \times \mc{N} \times \mc{B} \times \{ 0 \}))}.
\end{equation}
We identify $\g$ with left invariant vector fields on $G$. Then $T^* (G / U)$ is naturally identified with $G \times_U \mf{b}$. The differential of $\mathsf{q}$ is given by 
$$
(d_{(g,hU)} \mathsf{q})(A,B) = (B, B - \Ad_h(A)), \quad \forall \ A \in \g, \ B \in \g / \mf{b}.
$$
The moment map for the right $T \times T$-action on $T^*_{(hU,h'U)}( G
/ U \times G/U )$ is the map 
$$
(G \times_U \mf{b}) \times (G \times_U \mf{b}) \rightarrow \mf{t} \oplus \mf{t}, \quad (g_1,X,g_2,Y)\mapsto (X \,\operatorname{mod}\, \mc{U}, Y  \,\operatorname{mod}\,\mc{U}).
$$
Therefore, if $(X,Y) \in T^*_{(g_1 U,g_2 U)} \horo$ then $(X + Y) \equiv 0  \,\operatorname{mod}\, \mc{U}$. We have  

$$
\rho_{\mathsf{q}}((g,hU),(X,Y)) = (g,- \Ad_{h^{-1}}(X),hU,X + Y) 
$$
Hence, if $(h U,X,ghU,Y)$ is contained in the right-hand side of
(\ref{eq:charinc}), then $- \Ad_{h^{-1}}(X) \in \mc{N}$ implies that $X$
is nilpotent. But then $(X + Y) \equiv 0 \,\operatorname{mod}\, \mc{U}$ implies that
$Y$ is nilpotent too.
 This implies that $(X,Y) \in \mc{U} \oplus \mc{U} = \mu^{-1}_{T \times
   T}(0)_{(hU,ghU)}$. Thus, since $\mu_{T \times T}^{-1}(0)$ is closed
 in $T^* \horo$, the right-hand side of (\ref{eq:charinc}) is contained
 in $\mu_{T\times T}^{-1}(0)$.
Part (i) follows.

\vii Define a $G$-action on $G\times \mathcal{B}$ by
$h: (g_1, g_2B)\mto (hg_1h\inv, hg_2B)$. With respect to this action,
each of the maps $\mathsf p$ and $\mathsf q$ is 
$G$-equivariant. In particular, these maps are $H$-equivariant,
for any subgroup $H\sset G$.
Therefore, part (ii) follows from  Proposition \ref{prop:monodromicfinite}(i).

\viii For $U=\horo$, this is stated in  \cite[Proposition 8.7.1]{AdmissibleModules}
and proved in \cite[page 158]{GinzburgIndRes}.
The proof given in {\em loc. cit.} works in the more general case
of an arbitrary $T$-stable open subset $U\sset \horo$ as well.
\end{proof}

\subsection{}\label{ck} Given an action of the abelian Lie algebra $\t$ on a vector space $E$ and an element 
$\theta\in \t^*$,
let $E^{(\theta)}$ denote the corresponding generalized
weight space - the subspace of $E$ formed by all elements annihilated
by a sufficiently high power of the maximal ideal $\Ker \theta$ of the algebra $\U(\t)$.
Similarly, given 
an action of the Lie algebra
$\t\oplus\t$ on $E$, and an element
$(\theta_1,\theta_2)\in \t^*\oplus\t^*$, we write $E^{(\theta_1,\theta_2)}$ for the corresponding generalized
weight space.

The projection $\gamma: G/U \to (G/U)/T=G/B$, by the $T$-action on the right, 
is a principal $T$-bundle, in particular, it is an
affine morphism. Let $\t\to \dd(G/U)$ be the Lie algebra
embedding induced by the $T$-action.
Let $\rho\in\t^*$ be the half-sum of positive roots.

Later on, we will use the following consequence of the
Localization Theorem  of Beilinson and Bernstein.
\begin{thm}\label{BBthm}
Let $\ck$ be an object of the bounded derived category
of coherent
$\dd$-modules on $G/U$ such that, for every $j\in\Z$,  the $\t$-action on 
the sheaf $\gamma_\idot\ms{H}^j(\ck)$ is locally finite.
Let $\theta\in\t^*$ be such that
$(\gamma_\idot\ms{H}^k(\ck))^{(\theta)}\neq0$
for some $k\in \Z$ and let  $\nu\in\t^*$
be such that $\nu-\theta\in\Hom(T,\C^*)$
and  $\nu+\rho$ is a regular element in $\t^*$.
Then, one has

\vi There exists $\ell\in\Z$
such that $R^\ell\Ga(G/U, \ck)^{(\nu)}\neq0$.

\vii Assume in addition that the object $\ck$ is concentrated in degree zero,
i.e. is a $\dd$-module, and the weight
 $\nu + \rho \in\t^*$ is dominant. Then, one has $\Ga(G/U, \ck)^{(\nu)}\neq0$.
\end{thm}

Here, part (ii) follows from \cite{BB1} and part (i)
is a reformulation of a result from \cite{Casselman}.

\subsection{}\label{complete} From now on, we use the setup of \S\ref{setup}, see Convention \ref{convent}. Thus, $V=\C^n$ and $G=GL(V)$. 
The group $G$ acts on $\SSL=SL(V)$ by conjugation and acts naturally on $V$. We equip the space $\X:=\SSL\times V$ with the $G$-diagonal action.

Let $\si:\ \SSL\times
V\to \SSL\times V$ be the map given by the
assignment $(g,v)\mto (g, g(v))$. It is immediate to check 
that $\si$ is a $G$-equivariant automorphism of $\X$.

\begin{lem}\label{si} The pull-back functors
$\mm\mto\si^{*}\mm, \si^! \mm$ and the push-forward functors $\mm \mto \si_! \mm, \si_* \mm$ give auto-equivalences of the category $\ms{C}$.
\end{lem}

\begin{proof} 
Let $d^* \si$ be the automorphism of $T^*\X$ induced by $\si$. In both cases, proving the claim amounts to showing that $d^* \si(\Nnil)\sset\Nnil$. The map $\si$ being $G$-equivariant, it follows, on general grounds, that one has $d^* \si(\mu_\X\inv(0))\sset \mu_\X\inv(0)$. It remains to show that if  $(g,Y,i,j)\in \Nnil\sset T^*\X=\SSL\times\sll\times V\times V^*$
and $d^* \si(g,Y,i,j)=(g',Y',i',j')$, then $Y'$ is a nilpotent matrix. To this end, one first calculates that $d_{(g,i)} \sigma (X,v) = (X, g X (i) + g(v))$. Hence the map $d^* \si$ is given by the following explicit formula
$$
d^* \si:\ (g,\,Y,\,i,\,j) \mto (g,\, Y-i\o j,\, g(i),\, g^*(j)).
$$
Hence, for $(g,Y,i,j)\in \Nnil$, from the equation in the right hand side of \eqref{eq:defineNnilSL}, we deduce that $Y-i\o j=-gYg\inv$. Thus,  $Y\in \mc{N}$ implies that $Y-i\o j\in\mc{N}$, and we are done.
\end{proof}

It will be convenient for the proof of Theorem 
\ref{thm:admissible2} to consider the group $\BG := \C^\times\times \SSL$. We make $\X$ into a $\BG$-variety by letting the group $\SSL$ act through its embedding into $G$ and letting $\C^\times$ act on $\X = \SSL\times V$ via the natural action by dilations on the second factor. It is clear that a $\dd(\X)$-module is $G$-monodromic iff it is $\BG$-monodromic. The latter condition is also equivalent to being $\SSL$-equivariant and $\C^\times$-monodromic. Further, let $\pr :\ \C^\times \times \X \to \X$ denote the projection map. We have an identification $\C^\times\times\X=\C^\times\times \SSL\times V=\BG\times V$. Replacing $\ms{M}$ by $\pr^* \ms{M}$, $\X$ by $\BG \times V$ and $\g$ by $\mf{G} := \Lie (\BG)$, it suffices to prove the analogue of Theorem \ref{thm:admissible2} in this setting.

We let the group $\SSL$  act
diagonally on the variety $\horo\times V$,
resp.  on the variety $\SSL\times (\SSL/U) \times V$. We let $\C^\times$ act on $\horo
\times V$,
 resp.  on  $\SSL\times (\SSL/U) \times V$,
via its action on $V$, the last factor. Thus, we obtain a
$\BG$-action on $\horo\times V$,
 resp.  on  $\SSL\times (\SSL/U) \times V$. Further, let the torus $\TT$ act on
$\horo\times V$ through its action on $\horo$, the first factor.
The actions of $\BG$ and $\TT$ commute, making $\horo\times V$ a
$\BG\times \TT$-variety. It will be crucial for us that the
number of $\BG\times \TT$-orbits on $\horo \times V$ is known to be
finite, \cite{MWZ}.

Our arguments below involve the following analogue of diagram \eqref{diag1}:
\beq{diag2}
\xymatrix{
 & \BG \times \mathcal{B} \times V\ar[dl]_{p} \ar[dr]^{q} &\\
\BG \times V&  & \horo\times V
}
\eeq
where $p$ is projection along $\mc{B}$ and $q(\lambda,g,F,v) = (F,g(F),v)$ for all $\lambda \in \Cs$, $g \in \SSL$, $F \in \SSL/U$, and $v \in V$.
Each of the maps $p$ and $q$ is easily seen to be $\BG$-equivariant.

Fix $\ms{M}$, a $\dd$-module on $\BG\times V$ with regular
singularities such that $\Char (\ms{M})
\subseteq \Nnil(\BG)$. It is $\BG$-monodromic and holonomic. Put $M :=
q_!p^*\ms{M}$ and, for each integer $k$, let $\ms{H}^k(M)$ denote the
$k$-th cohomology $\dd$-module of $M$. The characteristic variety of
$p^* \ms{M}$ is contained in $\Nnil( \BG ) \times T^*_{\mc{B}}
\mc{B}$. Therefore it follows from Lemma \ref{lem:keylemma}(i) that
$\ms{H}^k(M)$ is a $\TT$-monodromic $\dd$-module on $\horo\times V$. 
Furthermore, the  maps $p$ and $q$ being 
 $\BG$-equivariant, part (ii) of Lemma \ref{lem:keylemma}
ensures that
 $\ms{H}^k(M)$ is a $\BG$-monodromic $\dd$-module.

\begin{rem} The following diagram, similar to \eqref{diag2}, has  been
considered in \cite[\S4.3]{MirabolicCharacter}:
$$
\xymatrix{\SSL  \times{\mathbb P}(V)\  &
\  \SSL \times \mathcal{B} \times{\mathbb P}(V)\
\ar[l]_<>(0.5){\wt p} 
\ar[r]^<>(0.5){\wt q} &\ \horo\times {\mathbb P}(V),}$$
Here, the map $\wt p$
is the projection along $\mathcal{B}$ and
the  map ${\wt q}$ is given
by the formula
${\wt q}(g,F,v) = (F,g(F),  g(v))$. Let $\widehat{q} : \BG \times \mc{B} \times V \to \horo \times V$ be similarly defined by $\widehat{q}(g,F,v) = (F, g(F), g(v))$. The relation between these two settings is provided,
essentially, by the
automorphism $\si$ of Lemma \ref{si}.
Specifically, write $\varpi:\ V\sminus\{0\}\to {\mathbb P}(V)$
for the canonical projection and
$\bar q: \
\SSL \times \mathcal{B} \times (V\sminus\{0\}) \to \horo
\times(V\sminus\{0\})$
for the map induced by the map $q$ in diagram  \eqref{diag2}.
Then, one
clearly has: 
$$
\widehat{q} = q \ccirc (\si \times \mathrm{Id}_{\mc{B}}) \quad \textrm{and} \quad {\wt q} \ccirc (\Id_{\SSL \times  \mathcal{B}}\times\varpi) =(\Id_{\horo} \times \varpi)\ccirc \bar q\ccirc (\si\times  \Id_{\mathcal{B}}).
$$
In \textit{loc. cit.} the pair of adjoint functors (between the appropriate derived categories) $\mathsf{hc}( - ) = {\wt q}_! \ccirc {\wt p}^* ( - )[\dim \mc{B}]$ and $\mathsf{ch}( - ) = {\wt p}_* \ccirc {\wt q}^! ( - )[-\dim \mc{B}]$ were studied. Analogously, we define the functors  $\mathrm{hc}( - ) = {q}_! \ccirc {p}^* ( - )[\dim \mc{B}]$ and $\mathrm{ch}( - ) = {p}_* {q}^! \ccirc ( - )[-\dim \mc{B}]$. Then, 
$$
\mathrm{hc} = q_! \ccirc p^* \ccirc \si_! ( - )[\dim \mc{B}], \quad \mathrm{ch} = \si^! \ccirc p_* \ccirc q^! ( - )[-\dim \mc{B}].
$$
Moreover, the results of \S \ref{ham_shift_sec} imply that $\mathsf{hc} \ccirc F_c = {\bar F}_c \ccirc \mathrm{hc}$ and $\mathsf{ch} \ccirc F_c = {\bar F}_c \ccirc \mathrm{ch}$, where ${\bar F}_c : \Lmod{\dd_{\horo \times V}} \to \Lmod{\dd_{\horo \times \mathbb{P}(V)}}$ is defined by ${\bar F}_c(\mm) = \Ker (\mathsf{eu}_V - c; (\mathrm{Id}_{\horo} \times \varpi)_{\idot} j^* \mm)$ and $j$ is the inclusion $\horo \times (V \sminus \{ 0 \}) \hookrightarrow \horo \times V$. 
\end{rem}

\subsection{Proof of the implication (i) $\Rightarrow$ (ii) of  Theorem \ref{thm:admissible2}}\label{sec:onetwo}

Let $\mm$ be a $G$-monodromic, holonomic
$\dd_\X$-module with regular
singularities. Proposition \ref{prop:monodromicfinite} 
 implies that the action of
$\mu(\mf{G})$ on $\Gamma(\BG \times V,\ms{M})$ is locally finite.

Thus, we must show that the $\FZ$-action on $\Gamma(\BG \times
V,\ms{M})$ is also locally finite. To this end, we will need to apply 
the  results  of section
\ref{ck}
in a slightly different setting. In more detail, recall the
notation 
$\horo=(\SSL/ U\, \times\, \SSL / U)/\td$.
We consider the following diagram:
\beq{Y} Y:=\SSL/ U \times \SSL / U \times V
\ \stackrel{\varpi}\too \
\horo\times V\ 
\stackrel{\pi}\too \
\mc{B} \times \mc{B}\times V,
\eeq
where the first map is a torsor of the group $\td$
and the second map is  a torsor of the group $(T_n\times T_n)/\td$.
We put $\ms{N}=\varpi^*(q_!p^*\ms{M})$ and write
$\ms{N}^j:=\ms{H}^j(\ms{N})$ for the $j$-th cohomology $\dd$-module
of $\ms{N}$, an object of some derived category.

The group  $T_n\times T_n$ acts along the fibers 
of the map $\pi\ccirc\varpi$.
We claim first that the induced action of the Lie algebra
$\t_n\oplus\t_n$ on $(\pi\ccirc\varpi)_\idot\ms{N}^j$,
is locally finite, for any $j\in\Z$.
To prove this, let  $\ms{V}\sset \mc{B} \times \mc{B}\times
V$ be an open subset.
The morphism $\pi\ccirc\varpi$ being affine,
we have $\Ga((\pi\ccirc\varpi)\inv(\ms{V}),\
\ms{N}^j)=
\Ga(\ms{V}, (\pi\ccirc\varpi)_\idot\ms{N}^j)$.
Observe that $\ms{N}^j$
is a $T_n\times T_n$-monodromic $\dd$-module,
by Proposition \ref{prop:monodromicfinite}(i).
Therefore, part (ii) of the same proposition implies that
the $(\t_n\oplus\t_n)$-action on
the space $\Ga((\pi\ccirc\varpi)\inv(\ms{V}),\
\ms{N}^j)$ is locally finite.
Hence, the $(\t_n\oplus\t_n)$-action on
$\Ga(\ms{V}, (\pi\ccirc\varpi)_\idot\ms{N}^j)$
is also locally finite, and our  claim follows.

We conclude  that 
there exists an integer $j\in\Z$ and 
an element $(\theta_1,\theta_2)\in \t_n^*\oplus\t_n^*$
such that we have
$((\pi\ccirc\varpi)_\idot\ms{N}^j)^{(\theta_1,\theta_2)}\neq0$.
Observe  that, the morphism $\vartheta$ being
smooth and affine, the projection formula yields
$\varpi_\idot \ms{N}=\varpi_\idot\varpi^*(q_!p^*\ms{M})=
\varpi_\idot\oo_Y\o q_!p^*\ms{M}$,
where the tensor product is taken over
$\oo_{\horo\times V}$.
It follows, that the action 
on $(\pi\ccirc\varpi)_\idot\ms{N}^j$
of
the Lie subalgebra  $\Lie\td\sset \t_n\oplus\t_n$
can be exponentiated to an action of the torus
$\td$. Therefore, we must have
$\theta_1+\theta_2\in \Hom(T_n,\C^*)$.
This implies that there exists $\theta\in\t_n^*$
such that $\theta_1-\theta,\ \theta_2+\theta+2\rho\in
\Hom(T_n,\C^*)$ and, moreover,
the weights $\theta+\rho$ and $-(\theta+2\rho)+\rho$
are both regular.

Next, we are going to  apply the  results  of section
\ref{ck}
in a slightly different setting, where
the variety $G/U$ is replaced by
$Y$, resp. $\mc{B}$ is
 replaced by $\mc{B} \times \mc{B}\times V$,
and the map $\gamma$ is
 replaced by the map $\pi\ccirc\varpi$.
Theorem \ref{BBthm} has an obvious analogue in this setting. Furthermore, we have
shown  that the object $\ms{N}$ satisfies the
assumptions of that analogue of Theorem \ref{BBthm}.
Thus, applying the corresponding version of
part (i) of the
theorem, we deduce that 
there exists an integer
 $k\in\Z$ such that one has
$R^k\Ga(Y, \ms{N})^{(\theta,-\theta-2\rho)}\neq0$.

Fix  $\theta$ as above and let $\la=\la(\theta)$ be the image of 
 $\theta$ in $\t_n^*/ W$. We may (and
will) identify $\la(\theta)$ with a point in $\Spec\FZ$ via the
Harish-Chandra isomorphism. Then, according to \cite[formula
(4.5.2)]{MirabolicCharacter}, which is based on \cite[Theorem
1]{HottaKashiwara2}, one has a canonical isomorphism
\beq{fg}
[R^k\Gamma(Y,\ \ms{N})]^{(\theta,- \theta - 2 \rho)}\ \cong\ [R^k\Gamma(\BG\times V, \ms{M})]^{(\la(\theta))}.
\eeq
We know, by our choice of $\theta$, that the left hand side of this formula is nonzero.
On the other hand, the variety $\BG\times V$ being affine,
the right  hand side of \eqref{fg} vanishes for any $k\neq 0$.
Thus, 
we conclude that  $k$ must be equal to zero and, then, we get
$[\Gamma(\BG\times V, \ms{M})]^{(\la(\theta))}
\neq0$.
\smallskip

To complete the proof of the implication (i) $\Rightarrow$ (ii),
it suffices to consider
the case where  $\ms{M}$ is a nonzero simple $\dd$-module.
 Assuming this, we will use the above  to prove that the action of $\FZ
 \subset \dd(\BG \times V)$ on $\Gamma(\BG \times V,\ms{M})$ is locally
 finite. 
Recall that $\bg = \Lie \BG$ and identify $\U(\bg)$ with the algebra of
left-invariant differential operators on $\BG$. Then, one has a
$\bg$-module isomorphism $\dd(\BG)=\C[\BG]\otimes \U(\bg)$. Therefore,
we get a $\bg$-module isomorphism $\dd(\BG \times V)=\dd(\BG)\otimes
\dd(V)=\C[\BG]\otimes (\U(\bg) \otimes \dd(V))$,
where the $\bg$-action on the tensor factor $\dd(V)$ is taken to be the
trivial action. Thus, since $\ms{M}$ is simple, for any nonzero element
$m\in \Gamma(\BG\times V,\ms{M})$ we have 
$$
\Gamma(\BG\times V,\ms{M})=\dd(\BG \times V) \cdot m= \big[\C[\BG]\otimes (\U(\bg) \otimes \dd(V))\big] \cdot m.
$$
We let $E:=(\U(\bg) \otimes \dd(V)) \cdot m$, which is a $\bg$-stable
subspace. We conclude that $\Gamma(\BG\times V,\ms{M})$ is isomorphic,
as a $\bg$-module, to a quotient of  $\C[\BG] \otimes E$.

We now use the fact that there exists $\la \in \Spec \FZ$ such that
$\Gamma(\BG\times V,\ms{M})^{(\la)} \neq 0$. Thus, we may (and will)
choose  our nonzero element $m$ so that $(\FZ_\la)^k m=0$, for some $k >
0$. Then, since $\FZ\otimes 1$ is a central subalgebra of the algebra
$U(\bg) \otimes \dd(V)$, we deduce that the ideal $(\FZ_\la)^k$
annihilates the $\bg$-module $E$. On the other hand, the $\bg$-action on
$\C[\BG]$ is clearly locally finite. It follows, by a result due to
Kostant \cite{Ko}, that the $\FZ$-action on $\Gamma(\BG\times V,\ms{M})=
\C[\BG] \otimes E$ is
locally finite.

\subsection{Proof of the implication
(ii) $\Rightarrow$ (i) of  Theorem \ref{thm:admissible2}}
Fix $\ms{M}$ such that both the $\FZ$-action and the $\mu(\mf{G})$-action
on $\Gamma(\BG \times V,\ms{M})$ are locally finite. 
This implies that $\Char (\ms{M}) \subset \Nnil(\BG)$ and hence $\ms{M}$
is a holonomic $\dd$-module. Moreover, it is
$\BG$-monodromic by the last statement of Proposition
\ref{prop:monodromicfinite}(ii). Therefore, it remains to show that
$\ms{M}$ has regular singularities. By \cite[Proposition
4.3.2(ii)]{MirabolicCharacter}, it suffices to show that $q_!p^*\ms{M}$
has regular singularities. 
This amounts to proving that, for any $j\in\Z$,
the $\dd_Y$-module $\ms{N}^j=\ms{H}^j(\varpi^*q_!p^*\ms{M})$
has regular singularities, c.f. \eqref{Y}.

To this end, we put  $Y:=\SSL/U \times \SSL/U\times V$.
We equip $Y$ with the $\SSL$-diagonal action ``on the left''
and with a $(T_n\times T_n)$-action induced by the
$(T_n\times T_n)$-action on $\SSL/U \times \SSL/U$ on the right.
 Also, put
 $\FT:= (\Lie \C^\times) \oplus \mathfrak t\oplus
\mathfrak t$.

\begin{claim}\label{claim_fin2} 
For any $k\in\Z$, the $\dd$-module $\ms{N}^k$ is
$\SSL$-equivariant. 
Moreover, the $\FT$-action on $(\pi\ccirc\varpi)_\idot\ms{N}^k$ is locally finite. 
\end{claim}

\begin{proof} Observe first that the group $\SSL$ is simply
  connected. Hence the locally finite $\sll$-action on $\Gamma(\BG
  \times V, \ \ms{M})$ can be exponentiated to an algebraic
  $\SSL$-action. It follows that $\ms{M}$ is an $\SSL$-equivariant
  $\dd$-module. This implies, by functoriality, that
  $\ms{H}^k(\varpi^*q_!p^*\ms{M})$ is a an $\SSL$-equivariant,
 $\dd$-module.

The second statement of the claim follows  from Lemma \ref{lem:keylemma}(iii). 
\end{proof}

The proof of the implication (ii) $\Rightarrow$ (i) is now completed by the following claim.

\begin{claim}\label{claim_fin3} 
Let $\ms{H}$ be an $\SSL$-equivariant, holonomic 
$\dd$-module on $Y$ such that  the $\FT$-action  on $(\pi\ccirc\varpi)_\idot\ms{H}$
is locally finite. Then,  $\ms{H}$ has regular singularities.
\end{claim}

\begin{proof} 
Using that any holonomic module has finite length and that the direct image functor $(\pi\ccirc\varpi)_\idot$ is exact
on the category $\dd$-modules,
one reduces the claim to the special case of simple $\dd$-modules.
Hence, we will assume (as we may)
that $\ms{H}$ is simple. 
Thus $\ms{H}$ is a holonomic, 
$\SSL$-equivariant 
 simple 
$\dd_Y$-module such that   the action of the Lie algebra $\FT$
on $(\pi\ccirc\varpi)_\idot\ms{H}$ is locally finite.
Therefore, applying an appropriate version of
 the Beilinson-Bernstein Theorem \ref{BBthm}(ii),
we deduce
that there exists $\Theta\in\FT^*$
such that $\Gamma(Y, \ms{H})^{(\Theta)}\neq0$.
It follows that one can find a nonzero element $u\in \Gamma(Y, \ms{H})$ such that one has $a(u)=\Theta(a)\cdot u$ for all $a\in\FT^*$.
The element $u$ generates $\ms{H}$ since  $\ms{H}$
is a simple module. Therefore we may apply Lemma \ref{lem:monoimpliesq}(iii) and conclude that $\ms{H}$ is $(\C^\times\times T_n\times T_n,{\mathbf q})$-monodromic where ${\mathbf q}$ is the image of $\Theta$. 

Next, we may extend $\Theta$ to a linear function on $(\Lie\BG)
\times\mathfrak t_n\times\mathfrak t_n$
 that vanishes on the subalgebra $\sll \sset  \Lie\BG$.
Abusing notation, we write ${\mathbf q}$ for the analogue of the corresponding element for the
group $\BG\times  T_n\times T_n$. 
Combining the $\SSL$-equivariant and
$(\C^\times\times T_n\times T_n,{\mathbf q})$-monodromic structures on
 $\ms{H}$ together makes $\ms{H}$  a $(\BG\times 
T_n\times T_n,{\mathbf q})$-monodromic $\dd_Y$-module. The number of
$\BG\times T_n\times T_n$-orbits on $Y$ being finite,
we conclude that  $\ms{H}$ has regular singularities,
thanks to Lemma \ref{lem:monoimpliesq}(ii). 
\end{proof}

\begin{rem}
The proof of the analogue of Theorem \ref{thm:admissible2} for mirabolic modules on $\mf{sl} \times V$ is more straight-forward due to the fact that one can make use of the Fourier transform, see \cite[Proposition 5.3.2]{GG}. 
\end{rem}

\begin{cor}\label{HCmir} 
The mirabolic Harish-Chandra $\ddd$-module $\CG_{\lambda,c}$ is a mirabolic $\dd$-module.
\end{cor}

\begin{proof} 
By Theorem \ref{thm:admissible2}, it suffices to show that the action of
$\FZ$ and $\mu(\g)$ on $\Gamma(\X,\CG_{\lambda,c})$ are both locally
finite. The adjoint action of $\mu(\g)$ on $\dd(\X)$ is locally finite,
therefore it is clear that $\mu(\g)$ acts locally finitely on
$\Gamma(\X,\CG_{\lambda,c})$.

To show that the action of $\FZ$ is
locally finite on $\Gamma(\X,\CG_{\lambda,c})$, we 
let $\mathbf{m}$ be the canonical generator of $\CG_{\lambda,c}$. By definition,
$\FZ_{\lambda} \cdot \mathbf{m}=0$. Now one just repeats the argument in
the proof of the implication (i) $\Rightarrow$ (ii) of  Theorem \ref{thm:admissible2}, section \ref{sec:onetwo}, based on Kostant's result \cite{Ko} (with the section $m$ replaced by $\mathbf{m}$).
\end{proof}

\subsection{Spectral decomposition}\label{sec:spectral} 

Theorem \ref{thm:admissible2} implies, as in \cite{AdmissibleModules}
and \cite{MirabolicCharacter}, that there is a \textit{spectral
  decomposition} of $\cc$. Let $P_0 \subset \tSL^*$ be the root lattice
of $\TSL$ (thought of as the abstract torus associated with $\SSL$). We
write $W_{\mathrm{aff}} = P_0 \rtimes W$ for the affine Weyl group. The
weights of $\dd(\X)$ under the adjoint action of $\mu(\tSL)$ are
contained in $P_0$. Let $\mm \in \cc$ and $\mm = \oplus_{\lambda \in
  \tSL^*/W} \mm^{(\lambda)}$ the decomposition of $\mm$ into generalized
eigenspaces with respect to the action of $\FZ$. 

For any $\Theta \in \tSL^* / W_{\mathrm{aff}}$, we put
$\mm \langle \Theta \rangle:=\bigoplus_{\lambda \in \Theta}
\mm^{(\lambda)}$.
This is a $\ddd$-submodule of $\mm$. Let $\cc_q\langle \Theta \rangle$
be the full subcategory of the category $\cc_q$ formed
by the mirabolic modules $\mm$ such that one has
$\mm\langle \Theta \rangle=\mm$.

\begin{rem}
Note that $\CG_{\lambda,c}$, the mirabolic Harish-Chandra $\ddd$-module,
is an object of category
$ \cc_q \langle \Theta \rangle$, where
$\Theta$ is the image of $\lambda$ in $\tSL^* / W_{\mathrm{aff}}$ and $q
= \exp (2 \pi \sqrt{-1} c)$. 
\end{rem}

One has the following simple result whose proof is left to
the reader.

\begin{prop} We have $\cc_q=\oplus_{\Theta \in \tSL^* /
    W_{\mathrm{aff}}}\
\cc_q\langle \Theta \rangle$, i.e. any object $\mm\in\cc_q$
has a canonical $\ddd$-module direct sum  decomposition 
$$
\mm = \bigoplus_{\Theta \in \tSL^* / W_{\mathrm{aff}}} \mm \langle
\Theta \rangle,
\qquad \mm \langle
\Theta \rangle\in \cc_q\langle \Theta \rangle.
$$

Moreover, for any pair $\Theta,\Theta'\in \tSL^* / W_{\mathrm{aff}}$
such that $\Theta\neq\Theta'$, and any mirabolic modules
$\mm\in \cc_q\langle \Theta \rangle$ and
$\mm'\in \cc_q\langle \Theta' \rangle$, one has
$\Hom_\ddd(\mm,\mm')=0$.
\end{prop}

\begin{rem} Using arguments analogous to those in \cite[Remark 1.3.3]{AdmissibleModules}, the last statement of the above proposition
may be strengthened by showing that, with the same assumptions as above, one
has $\Ext_{\ddd}^k(\mm,\mm')=0$ for all $k\geq0$.
\end{rem}

\section{A stratification}\label{strat}

\subsection{}\label{sec:defineMcirc} The following notation will be
fixed, without further mention, \textit{throughout} the article. Given a group $H$,
we write $Z(H)$ for the center of $H$ and $Z_H(h)$
for the centralizer of   an element $h\in H$. Given a subgroup $K\sset H$,
we let $N_H(K)$ denote the normalizer of $K$ in $H$.
Let $\TSL$ be a maximal torus in $\SSL$ and $T$ 
 the corresponding torus in $G$. The symmetric group on $n$ letters is denoted $W$.

Finally,
let $z_1, z_2 \in \SSL$ be semi-simple elements and put $\LSL =Z_{\SSL}(z_1), \MSL = Z_{\SSL}(z_2)$,
resp.  $L = Z_{G}(z_1), M = Z_{G}(z_2)$.
These are Levi subgroups of $\SSL$ and $G$ respectively, with corresponding Lie algebras $\mf{l}, \mf{m}, \mf{l}(\g)$ and $\mf{m}(\g)$. The derived subgroups of $\LSL$ and $\MSL$ are denoted $\LL$ and $\MM$. By $\TSL$ we mean a
maximal torus in $\SSL$ and $T$ will denote the corresponding torus in $G$. The symmetric group on $n$ letters is denoted $W$.

Given $g \in \MSL$, write  $g_s$ for the semi-simple part of  $g$. Then,
 $Z_{\SSL}(g_s)$ is a connected group. Let $\MSL^{\circ}$ denote the Zariski open
subset of $\MSL$ formed by the elements $g \in \MSL$, such that the
group $Z_{\SSL}(g_s)$ is contained in $\MSL$. Let $Z(\MSL)^{\circ} = \MSL^{\circ} \cap
Z(\MSL)$. If $\MSL$ is a proper subgroup of $\SSL$ then $Z(\MSL)$ is a
connected torus. The set $Z(\MSL)^{\circ}$ is a dense open subset of $Z(\MSL)$. Note that $\SSL^\circ = \SSL$. Let $\mn\sset \SSL$ be the unipotent variety, and let $\mn^{\MSL}:=\mn\cap \MSL$  be the unipotent variety of the Levi
subgroup $\MSL$. The group $M$ acts diagonally on $\mn^{\MSL} \times{V}$
with finitely many orbits.

\subsection{}\label{stratdef} In the paper \cite{CherednikCurves} a partition of the variety $\X$ is given, based on a stratification of $\sll 
\times V$ given in \cite[\S 4.2]{GG}. We recall this partition now. The strata of the partition are labeled by elements of the
{\em finite} set of $G$-conjugacy classes of pairs $(M,\Om)$, where $M \sset G$ is a Levi subgroup and $\Om$ is an $M$-diagonal
orbit in $\mn^{\MSL}\times{V}$. Given such a pair $(M,\Om)$, we
define 
$$
\X (M,\Om):=\big\{(g, \vv)\in \SSL \times{V} \mid (g,\vv) \in G (z
\cd u,\vv),\enspace \text{for some} \enspace z\in
Z(\MSL)^\circ,\,(u,\vv)\in\Om\big\}.
$$
Here and below, we  write $G (g,\vv)$ for the $G$-diagonal orbit of an
element $(g,\vv)\in {\SSL \times V}$;
 we also use similar notation for $M$-diagonal orbits. In the extreme case $\MSL = \SSL$, we have $Z(\SSL)^\circ = Z(\SSL)$, is a finite set. In that case, the connected components of a stratum $\X(\SSL,\Om)$ are labeled by the elements of $Z(\SSL)$. Abusing notation, we regard each connected component as a separate stratum. This way, any stratum becomes a smooth, connected, locally-closed $G$-stable subvariety of $\X$. 

\begin{defn}\label{distinguished}
A pair $(g, \vv)\in \X$, resp. the stratum $\X(M,\Om)\sset{X}$ that contains a pair $(g, \vv),$ is said to be 
{\em  relevant} if $\dim Z_{\SSL}(g)=n-1$ (i.e. $g$ is regular) and, moreover, the subspace $\C[g]\vv\sset V$ has a $g$-stable complement.
\end{defn}

As shown in \cite[Lemma 4.2.1]{McGertyKZ}, there is a natural
parameterization of the relevant strata in $\X$ by bi-partitions of
$n$. It is described as follows. Let $g \in \SSL$. The action of $g_s$
on $V$ defines a decomposition $V = V_{1} \oplus \ds \oplus V_{k}$
where, after reordering summands if necessary, $\nu = (\dim V_1 \ge \ds \ge \dim V_k)$ is a partition of $n$. If
$(g,\vv)$ is a relevant pair then $\vv = \vv_1 + \ds + \vv_k$ with
$\vv_i \in V_{i}$ and either $\vv_i = 0$ or $\C[g]\vv_i = V_i$. This
dichotomy defines a bi-partition $(\lambda,\mu)$ of $n$ such that
$\lambda + \mu = \nu$. Under this parameterization, the Levi
subgroup $M_n$ is specified up to conjugacy by the partition $\lambda +
\mu$. The stratum labeled by the pair $(\lambda,\mu)$ is denoted
$\X(\lambda,\mu)$.

%\subsection{} Let $\mathsf{vol}$ be a volume form on $V$ and define $\bs : \X \rightarrow \C$ by 
%$$
%\bs(g,v) = \langle \mathsf{vol}, v \wedge g \cdot v \wedge \ds \wedge g^{n-1} \cdot v \rangle. 
%$$
%Then $\X^{\mathrm{cyc}} = (\bs \neq 0)$ and hence the embedding $\beta : \X^{\mathrm{cyc}} \hookrightarrow \X$ is affine. The set of regular semi-simple elements in $\SSL$ is the set of points where the discriminant polynomial $\mathsf{disc}$ does not vanish. Therefore $\X^{\reg}$ is the set of all points in $\X$ where $\mathbf{f} = \mathsf{disc} \cdot \bs$ does not vanish. Hence the embedding $\jmath$ is also affine. The following proposition explains the ``relevance'' of relevant strata. 

\begin{prop}[\cite{GG}, Theorem 4.3]\label{relevant} 
The conormal bundle to a stratum $\X(M,\Om)$ is contained in $\Nnil(\SSL)$ if and only if $\X(M,\Om)$ is a relevant stratum.
\qed
\end{prop}

\subsection{} Let $\LSL$ be a Levi subgroup  of $\SSL$. We put 
\beq{not}
 \X_{\LSL} = \LSL \times V,\quad
\Xo_{\LSL} = \LSL^{\circ} \times V,\quad
 \X(\LSL) = \cup_{\Omega} \ \X(\LSL,\Omega),\quad
\paX(\LSL)=\ol{\X(\LSL)} \sminus \X(\LSL),
\eeq
 where in the third equation
the union is over all $L$-orbits $\Omega$ in $\mn^{\LSL} \times V$.

Later, we will use the following two lemmas.

\begin{lem}\label{trans}
For any Levi subgroup $\LSL \subseteq \SSL$, each stratum $\X(M,\Om) \subseteq \X$ meets the set $\Xo_{\LSL}$ transversely.
\end{lem}

\begin{proof} 
Pick a point $p:=(g,v)\in \Om \cap \Xo_{\LSL}$. An element of the vector space $T^*_p \X$ is a pair $(Y,w)\in \mf{sl} \times V^*$. Write $\mathsf{N}_p \sset \mf{sl} \times V^*$ for the conormal space at $p$ to the stratum $\X(M,\Om)$.

The conormal space to $\Xo_{\LSL}$ at $p$ equals $\mf{l}^\perp \times \{0\}\sset \mf{sl} \times V^*$. Thus, the statement of the lemma amounts to the claim that $\mathsf{N}_p \cap (\mf{l}^\perp\times \{0\})=0$. Equivalently, we must prove that the only element $Y \in \mf{l}^\perp$ such that $(Y,0)\in \mathsf{N}_p$, is the element $Y = 0$.

To prove the latter claim, let $\fz_{\mf{m}}$ be the centre of the Lie algebra $\mf{m}$. According to (\ref{eq:defineNnilSL}), we have $\mathsf{N}_p= (\fz_{\mf{m}}^\perp \times  V^*) \cap \mu\inv(0)$, where $\mu(Y,j)= g\cdot Y\cdot g\inv - Y + v\otimes w$. This way, we are reduced to showing that, for a point of the form $(Y,0)\in \mf{sl} \times \{0\}$, we have
\begin{equation}\label{yy}
Y \in \mf{l}^\perp\cap\fz_{\mf{m}}^\perp \quad \textrm{ and } \quad g\cd Y\cd g\inv-Y=0 \quad \Longrightarrow \quad Y=0.
\end{equation}

To see this, write $g=z\cdot u$, where $z\in Z(\MSL)^\circ$, and $u\in
\MSL$ is a unipotent element. Observe that $\fz_{\mf{sl}}(z)$, the centralizer of $z$ in $\mf{sl}$, equals $\mf{m}$. Further, since $g \in \LSL^{\circ}$, we also have $z\in \LSL^{\circ}$. We deduce $\mf{m} = \fz_{\mf{sl}}(z)\sset \mf{l}$. Hence, we have $Y \in \mf{l}^\perp \sset \mf{m}^\perp$. It is also clear that $Y \in \fz_{\mf{sl}}(g) \sset \fz_{\mf{sl}} (z) = \mf{m}$. Therefore, the conditions on $Y$ imposed on the left hand side of \eqref{yy} imply that $Y \in \mf{l}^\perp\cap \fz_{\mf{sl}}(g)\sset \mf{m}^\perp\cap \mf{m} = 0$, and \eqref{yy} is proved.
\end{proof}   

\begin{lem}\label{lem:stratainfo} For a 
Levi subgroup $\LSL$ of
  $\SSL$, one has
$\ol{\X(\LSL)} \cap \Xo_{\LSL} = Z(\LSL)^{\circ} \times \mn^{\LSL} \times V$.
\end{lem}

\begin{proof}
Let $Y$ be the union of all strata $\X(M,\Omega)$ such that $\MSL$ is conjugate to some Levi subgroup of $\SSL$ containing $\LSL$. The set $\ol{\X(\LSL)}$ is contained in $Y$. If $g \cdot (z \cdot u, v)$ belongs to the intersection $\X(M,\Omega) \cap \Xo_{\LSL}$ for some Levi subgroup $\MSL$ of $\SSL$, then $g z u g^{-1} \in \LSL^{\circ}$ and hence $zu \in (g^{-1} \cdot \LSL)^{\circ}$. By definition, this means that $Z_{SL}(z) \subseteq g^{-1} \cdot \LSL$. However, $z \in Z(\MSL)^{\circ}$ which means that $Z_{\SSL}(z) = \MSL$. Thus $g \cdot \MSL \subseteq \LSL$. Therefore we have shown that 
$$
\ol{\X(\LSL)} \cap \Xo_{\LSL} = \X_{\LSL} \cap \LSL^{\circ} \times V.
$$
Finally, let $g \cdot (zu,v) \in \X(\LSL,\Omega) \cap (\LSL^{\circ} \times V)$ for some $\Omega$. Since $gzug^{-1} \in \LSL^{\circ}$, the element $l := g z g^{-1}$ is contained in $\LSL^{\circ}$ too. Hence $z = g^{-1} \cdot l \in Z(\LSL)^{\circ}$. This implies that $Z_{\SSL}(g^{-1} \cdot l) = \LSL$ and hence $Z_{\SSL}(l) = g \cdot \LSL$. However, $Z_{\SSL}(l) = \LSL$ implies that $g \in N_{\SSL}(\LSL)$ and thus $g \cdot Z(\LSL) = Z(\LSL)$. Therefore $l \in Z(\LSL)^{\circ}$ as required. 
\end{proof}

\section{Cuspidal $\dd$-modules}\label{chetyre} 

In this section we define, and classify, the cuspidal mirabolic modules on $\X_{\LL}$ for $\LSL$ a Levi subgroup of $\SSL$. As in the classical case, \cite{LusztigInterCoh}, there are relatively few Levi subgroups of $\SSL$ that support cuspidal mirabolic modules.

\subsection{}\label{sec:monodromydefinition} The centre $Z(\SSL)$ of the
group $\SSL$ is a cyclic group, the group of scalar matrices of the form
$z\cdot\Id$, where $z\in\C$ is an $n$-th root of unity.
Let $\jmath: \bN^\reg\into\bN$ be the open embedding of the conjugacy class formed by the regular unipotent elements. The fundamental group of $\bN^\reg$ may be identified canonically with $Z(\SSL)$. For each integer $r = 0,1,\ldots,n-1,$ there is a group homomorphism $Z(\SSL) \to \C^\times$ given by $z \cdot \Id \mto z^r$, and a corresponding rank one $\SSL$-equivariant local system $\sL_r$ on $\bN^{\reg}$ with monodromy $\theta = \exp \left( \frac{2 \pi \sqrt{-1} r}{n} \right)$.

From now on, we assume that $(r,n) = 1$, ie. that $\theta$ is a primitive $n$-th root of unity. Then the local system $\sL_r$ is known to be {\em clean}, that is, for $\dd$-modules on $\SSL$, one has, c.f. \cite{CharSheaves5}:
$$
\jmath_!\sL_r \stackrel{\sim}{\rightarrow} \jmath_{!*}\sL_r \stackrel{\sim}{\rightarrow} \jmath_*\sL_r.
$$

Given a central element $z \in Z(\SSL)$, we have the conjugacy class  $z\bN^\reg$, and we let $z \ms{L}_r := z\jmath_!\sL_r$ denote the corresponding translated $\dd$-module supported on the closure of  $z\bN^\reg$. According to Lusztig \cite{CharSheaves5}, $z\jmath_!\sL_r$ is a cuspidal mirabolic $\dd$-module on $\SSL$.

\subsection{} Let $H$ be an arbitrary reductive subgroup of $G$
and $\mathfrak{h} = \Lie H$. By restricting the trace form on $\g$ to $\mathfrak{h}$, we identify $\mathfrak{h}$ with its dual. Write $\mu_H : T^* (H \times V) \rightarrow \mathfrak{h}$ for the moment map of the Hamiltonian action of $H$ on $T^*(H \times V)$. Then, the variety 
\begin{equation}\label{eq:nnilLSL}
\Nnil(H) := \{ x = (g,Y,i,j) \in H \times \mathfrak{h} \times V \times V^* \ | \ \mu(x) = 0 \ \textrm{ and $Y$ nilpotent.} \},
\end{equation}
is a Lagrangian subvariety in $H \times \mathfrak{h} \times V \times V^*$. Each Levi subgroup $\MSL$ of $\LL$ and $M$-orbit $\Omega$ in $\bN^{\MSL} \times V$ defines a stratum $\X_{\LL}(\MSL,\Omega)$ of $\X_{\LL}$. The set of all strata forms a partition of $\X_{\LL}$. Since $\LL$ is a product of special linear groups, \cite[Theorem 4.4.2]{GG} implies that 

\begin{lem}
We have 
$$
\Nnil(\LL) = \bigcup_{(\MSL,\Omega)} \overline{T^*_{\X_{\LL}(\MSL,\Omega)} \X_{\LL}},
$$
where the union is over all $\LSL$-conjugacy classes of Levi subgroup $\MSL$ of $\LSL$ and relevant $M$-orbits $\Omega$ in $\bN^{\MSL} \times V$. 
\end{lem}

\begin{defn}
An $\LSL$-equivariant, regular holonomic $\dd$-module on $\X_{\LL}$ with support in $\bN^{\LL} \times V$ and characteristic variety contained in $\Nnil(\LL)$ is called an $\LSL$-\textit{cuspidal} mirabolic module.
\end{defn}

\subsection{} The Levi subgroups of $\SSL$ are parameterized up to conjugacy by partitions of $n$ (we say that $\LSL$ is a Levi of \textit{type} $\lambda \vdash n$). Recall that a partition of $n$ is a multiset $\{ \{ \lambda_1, \ds, \lambda_s \} \}$ of positive integers such that $\sum_{i = 1}^s \lambda_i = n$. However, we will think of a partition as being an ordered tuple of positive integers $(\lambda_1,\ds, \lambda_s)$ such that $\lambda_1 \ge \cdots \ge \lambda_s > 0$. Clearly the two notions are equivalent but the second will be more convenient when dealing with signed partitions. The set of all relevant strata in $\X_{\LL}$ is labeled by the set of all bi-partitions 
$$
(\boldsymbol{\mu},\boldsymbol{\nu}) = ((\mu^{(1)},\nu^{(1)}), \ds,(\mu^{(s)},\nu^{(s)}))
$$
of $\lambda$, where each $(\mu^{(i)},\nu^{(i)})$ a bi-partition of $\lambda_i$. The relevant strata contained in $\bN^{\LL} \times V$ are labeled by those bi-partitions $(\boldsymbol{\mu},\boldsymbol{\nu})$ such that each $(\mu^{(i)},\nu^{(i)})$ is either of the form $((\lambda_i),\emptyset)$ or $(\emptyset,(\lambda_i))$. 

\begin{defn}
A \textit{signed partition} of $n$, $\lambda^{\pm} \vdash n$, is a partition $(\lambda_1, \ds, \lambda_s)$ of $n$ such that each $\lambda_i$ is assigned a sign $\mathrm{sgn}(\lambda_i) \in \{ +, - \}$. If $\lambda^{\pm}$ is a signed partition then $\lambda^+$ is the tuple $(\lambda_1^+, \ds, \lambda_s^+)$ such that $\lambda_i^+ = \lambda_i$ if $\mathrm{sgn}(\lambda_i) = +$ and $\lambda_i^+ = 0$ otherwise. The tuple $\lambda^-$ is defined similarly. We write $\lambda^{\pm} = (\lambda^+,\lambda^-)$. 
\end{defn}

Thus, the relevant strata contained in $\bN^{\LL} \times V$ can be labeled $\X_{\LL}(\lambda^{\pm})$, by signed partitions whose underlying partition is $\lambda$. The following signed partitions will play an important role in the classification of $\LSL$-cuspidal sheaves. Choose $2 \le m \le n$ and $u,v,w \in \mathbb{Z}_{\ge 0}$ such that $n = (u+v)m + w$. Let $\lambda$ be the partition $(m^{u+v}, 1^w)$ of $n$. We associate to $\lambda$ the signed partitions $\lambda^{\pm}(u;v,w) = (\lambda^+(v,w),\lambda^-(u))$, where $\lambda^+(v,w)$ is a tuple with $v$ entries equal to $m$, $w$ equal to $1$ and the other $u$ are $0$ (hence $\lambda^-(u)$ has $u$ entries equal to $m$ and all other entries are zero). There are ${ u + v \choose u }$ such signed partitions.  

\subsection{} Given $t \in \C / \Z$, define the simple $\dd_{\C}$-module $\ms{E}_t$ to be the minimal extension of the local system on $\Cs$ with monodromy $\theta = \exp (2 \pi \sqrt{-1} t)$. If $L$ is of type $(m^{u + v}, 1^w)$, then the center of $L$ defines a canonical decomposition $V = V^{\oplus (u+v)}_m \oplus \C^{w}$, where $V_m \subseteq V$ is $m$-dimensional. For each $0 \le r \le m-1$ with $(r,m) = 1$, we define the $\dd_{\X_{\LL}}$-module
\begin{equation}\label{eq:mruvw}
\ms{M}(r,u,v,w) := \left[(\ms{L}_{r} \boxtimes \C[V_{m}])^{\boxtimes v} \right] \boxtimes \left[ \boxtimes_{i = 1}^w \ms{E}_{-\frac{r}{m}} \right] \boxtimes \left[ (\ms{L}_{r} \boxtimes \delta_{V_{m}})^{\boxtimes u} \right]. 
\end{equation}
Its support is the closure of a relevant stratum labeled by a signed partition of the type $\lambda^{\pm}(u;v,w)$.  

\begin{prop}\label{prop:classification1}
Let $\LSL \neq \TSL$ be a Levi subgroup of $\SSL$ of type $\lambda$ ($\neq (1^n)$). 
\begin{enumerate}

\item If $\lambda = (m^{u + v},1^w)$ for some $2 \le m \le n$ and $\lambda^{\pm}(u;v,w)$ is some signing of $\lambda$, then, for each $0 \le r \le m - 1$ with $(r,m) = 1$, the $\dd$-module $\ms{M}(r,u,v,w)$ is a simple, $\LSL$-cuspidal mirabolic module whose support is the closure of $\X_{\LL}(\lambda^{\pm}(u;v,w))$. Up to isomorphism, these are all simple, $\LSL$-cuspidal mirabolic modules whose support is the closure of $\X_{\LL}(\lambda^{\pm}(u;v,w))$.
 
\item For any other signed partition $\lambda^{\pm}$, there are no simple, $\LSL$-cuspidal mirabolic modules whose support is the closure of $\X_{\LL}(\lambda^{\pm})$. 
\end{enumerate}
\end{prop}

The classification when $\LSL = \TSL$ is a torus is slightly different. It is given in (\ref{sec:torus}). 

\begin{proof}
We begin by fixing some $G$-orbits in $\bN \times V$. Firstly, there is the unique open, dense orbit, which we denote $\Orb^{(0)}$. Then there is the orbit $\Orb^{(1)} = \bN^{\reg} \times \{ 0 \} \subset \bN \times \{ 0 \}$. Finally, we let $\Orb^{(2)}$ be the $G$-orbit consisting of all pairs $(X,v)$ such that $X \in \bN^{\reg}$ and $v \in \mathrm{Im} (X - 1) \sminus \mathrm{Im} (X-1)^2$. It is the unique codimension one $G$-orbit in $\bN \times V$. 

Fix a signing $(\lambda^+,\lambda^-)$ of $\lambda$. Without loss of generality we may assume that 
$$
\lambda^+ = (\lambda_1, \ds, \lambda_k,0,\ds,0), \quad \lambda^- = (0,\ds,0,\mu_1, \ds, \mu_l),
$$
where, for clarity, we have set $\mu_i = \lambda_{i + k}$. Let $k = v + w$ such that $\lambda_i = 1$ if and only if $v + 1 \le i \le k$ and $u \le l$ such that $\mu_j = 1$ if and only if $u + 1 \le i \le l$. 

The relevant stratum in $\bN^{\LL} \times V$ labeled by $\lambda^{\pm}$ is the $L$-orbit
$$
\Orb_{\lambda^{\pm}} = \Orb^{(0)}_{\lambda_1} \times \cdots \times \Orb^{(0)}_{\lambda_k} \times \Orb^{(1)}_{\mu_1} \times \cdots \times \Orb^{(1)}_{\mu_l}.
$$
If there exists a simple, $\LSL$-cuspidal mirabolic module $\ms{M}$
supported on the closure of the stratum
$\X_{\LL}(\lambda^{\pm})$ then, since the orbit $\Orb_{\lambda^{\pm} }$ is a dense, open subset of the smooth locus of this closure, $\ms{M}$ will be the minimal extension of some simple, $\LSL$-equivariant local system on $\Orb_{\lambda^{\pm}}$. Since the fundamental group of $\Orb^{(0)}$ is $\Z$ and the fundamental group of $\Orb^{(1)}$ is $\Z_n$, the fundamental group of $\Orb_{\lambda^{\pm}}$ is $\Z^k \times \prod_{i = 1}^l \Z_{\mu_i}$. For each $1 \le i \le v$, swapping the orbit $\Orb^{(0)}_{\lambda_i}$ for $\Orb^{(2)}_{\lambda_i}$ we get a $L$-orbit $\Orb_{\lambda^{\pm}}^{(i)}$. It is a codimension one orbit in $\overline{\Orb_{\lambda^{\pm}}}$. Let $\theta = (\theta_{\lambda},\phi_{\lambda},\theta_{\mu},\phi_{\mu})$ be our choice of local system on $\Orb_{\lambda^{\pm}}$, where $\theta_{\lambda}$ is the local system on $\Orb^{(0)}_{\lambda_1} \times \cdots \times \Orb^{(0)}_{\lambda_{v}}$, $\phi_{\lambda}$ the local system on $\Orb^{(0)}_{\lambda_{v + 1}} \times \cdots \times \Orb^{(0)}_{\lambda_k}$ and similarly for $\mu$. Note that $\Orb_{\mu_i}^{(1)} = \{ 0 \}$ for $i > u$ and hence $\phi_{\mu}$ is the trivial local system on a point. The Lagrangian $\ol{T^*_{\Orb_{\lambda^{\pm}}^{(i)}} \X_{\LL}}$ is not a component of $\Nnil(\LL)$ because $\Orb_{\lambda^{\pm}}^{(i)}$ is not a relevant stratum. Therefore the monodromy around the divisor $\Orb_{\lambda^{\pm}}^{(i)}$, which is given by $(\theta_{\lambda})_i^{\lambda_i}$, must be trivial i.e. $(\theta_{\lambda})_i^{\lambda_i} = 1$. Also, taking monodromy around lines in $V^{o}_{\lambda_i}$ is given by $(\theta_{\lambda})_i^{\lambda_i} = 1$. Therefore if we define 
$$
\widetilde{\Orb_{\lambda^{\pm}}} = \Orb^{(1)}_{\lambda_1} \times \cdots \times \Orb^{(1)}_{\lambda_{v}} \times \Orb^{(0)}_{\lambda_{v+1}} \times \cdots \times \Orb^{(0)}_{\lambda_{k}} \times \Orb^{(1)}_{\mu_1} \times \cdots \times \Orb^{(1)}_{\mu_l}.
$$
then we see that $\theta$ is the pullback, via the natural projection $\Orb_{\lambda^{\pm}} \rightarrow \widetilde{\Orb_{\lambda^{\pm}}}$ of a local system $\eta$ on $\widetilde{\Orb_{\lambda^{\pm}}}$. This implies that  
$$
\ms{M} \simeq \left[ \boxtimes_{i = 1}^{v} (\ms{L}_{\theta_{\lambda_i}} \boxtimes \C[V_{\lambda_i}]) \right] \boxtimes \left[ \boxtimes_{i = 1}^{w} \ms{E}_{\phi_{\lambda_i}} \right] \boxtimes \left[ \boxtimes_{j = 1}^{u} (\ms{L}_{\theta_{\mu_j}} \boxtimes \delta_{V_{\mu_j}}) \right] \boxtimes \left[\boxtimes_{j = 1}^{l-u} \delta_{V_{\mu_j}} \right].
$$
If $\ms{M}$ is a mirabolic $\dd$-module then it is $\LSL$-equivariant. Therefore $\eta$ must be an $\LSL$-equivariant local system. Choose some $x \in \widetilde{\Orb_{\lambda^{\pm}}}$. The local system $\eta$ must come from a representation of the group 
$$
\Stab_{\LSL}(x) / \Stab_{\LSL}(x)^{\circ} \simeq \Stab_{Z(\LSL)}(x) / \Stab_{Z(\LSL)}(x)^{\circ}  \simeq \Z_m,
$$
where $m$ is the greatest common divisor of $\lambda_1, \ds, \lambda_{v}, \mu_1, \ds, \mu_l$. However, for $\ms{M}$ to be a mirabolic module, it must have the correct characteristic variety. In particular, this implies that each $\ms{L}_{\theta_{\lambda_i}}$ and $\ms{L}_{\theta_{\mu_j}}$ must be a cuspidal mirabolic module for $SL(V_{\lambda_i})$ and $SL(V_{\mu_j})$ respectively. As explained in section \ref{sec:monodromydefinition}, this is equivalent to requiring that the restrictions $\theta_{\lambda_i}$ and $\theta_{\mu_j}$, of $\eta$, to the centers $\Z_{\lambda_i}$ of $SL(V_{\lambda_i})$ and $\Z_{\mu_j}$ of $SL(V_{\mu_j})$ must be primitive representations. Thus, $\lambda_i$ and $\mu_j$ divide $m$. Hence we find that $\lambda_i = \mu_j = m$ for all $1 \le i \le v$ and $1 \le j \le l$. Moreover, thinking of $\eta$ as an integer $0 \le \eta \le m-1$, it must be coprime to $m$. Note also that this shows that if $l > u$ then $m = 1$ and $\nu = (1^n)$, contradicting our initial assumptions. Thus $l = u$.

Finally, for $u+1 < i \le k$, we must calculate the monodromy of $\eta$ along a loop in $\Orb^{(0)}_{\lambda_i} \simeq \Cs$ in order to determine the monodromy $\phi_{\lambda_i}$. But it follows from the identification of $\pi_1(L) = \Z^{u + v}_m \times \Z^w$ that $\theta_{\lambda_i}$ is the representation $\Z \rightarrow \Z_m \rightarrow \Cs$ which, identifying $\Z_m$ with the $m$-th roots of unity in $\Cs$, is given by $z \mapsto z^r$. See the proof of Corollary \ref{cor:summarycuspidal} for details. 
\end{proof}

\subsection{}\label{sec:torus} When $\LSL = \TSL$ is a maximal torus
inside $\SSL$, 
we have $\bN^{\TSL} \times V = V$ and the relevant strata, again labeled by signings of $(1^n)$, are precisely the $T$-orbits in $V$. The $\TSL$-cuspidal mirabolic modules are those $\TSL$-equivariant, regular holonomic $\dd_V$-modules whose characteristic variety is contained in 
$$
\Nnil(T') = \{ (i,j) \in V \times V^* \ | \ i_t \cdot j_t = 0, \ \forall \ t = 1, \ds, n \}.
$$
Denote by $\mc{O}_{T}^{q}$ the simple local system on $T$ which is obtain from pulling back along the map $\det : T \rightarrow \Cs$ the rank one local system on $\Cs$ with monodromy $q$. The local systems $\mc{O}_T^q$ are precisely the $\TSL$-equivariant, simple local systems on $T$. There is a unique open $T$-orbit in $V$. It is a free $T$-orbit and choosing a base point in that orbit gives an open embedding $j : T \hookrightarrow V$. 

\begin{prop}\label{lem:cuspidalkillT}
The simple $\TSL$-cuspidal mirabolic modules on $V$ are:
$$\ms{M}(q) := j_{*!} \mc{O}_{T}^{q}, \en \forall \ q \in \Cs,
\quad\text{and also}\quad
\ms{M}(k) := \C[V_k] \boxtimes \delta_{V_{n-k}}, \quad \forall \ 0 \le k \le n - 1,
$$
where $\dim V_k = k$.
\end{prop}

\begin{proof}
Let $\ms{M}$ be a simple mirabolic module supported on the closure of a stratum of type $\lambda^{\pm}(k;n-k)$, where $\lambda = ((1^k),(1^{n-k}))$ and $k \neq n$. Then, this stratum is a $\TSL$-orbit $\Orb \subset V$. The group $\TSL$ is connected and the stabilizer of a point $v \in \Orb$ is a connected subgroup of $\TSL$. Therefore the only simple $\TSL$-equivariant local system on $\Orb$ is the trivial local system. Its minimal extension to $V$ is $\ms{M}(k)$. This $\dd$-module is a mirabolic $\dd$-module. Now assume that $\Orb = T$ is the open stratum. It is easy to see that the $\TSL$-equivariant, simple local systems on $T$ are precisely the
modules $\mc{O}_{T}^{q},\ q \in \Cs$.    Therefore, we conclude  that $\ms{M} = j_{!*} \mc{O}_{T}^{q}$, for some $q \in \Cs$.   
\end{proof}

\subsection{} In all cases, the simple $\LSL$-cuspidal mirabolic sheaves supported on $\bN^{\LL} \times V$ have a natural $L$-monodromic structure. 

\begin{cor}\label{cor:summarycuspidal}

\vi The $\dd$-module $\ms{M}(r,u,v,w)$ defined in (\ref{eq:mruvw}) is $(L,q)$-monodromic, where $q = \exp( \frac{2 \pi \sqrt{-1} r}{m})$. 

\vii The $\dd_V$-modules $\ms{M}(k)$ are $T$-equivariant and $\ms{M}(q)$ is $(T,q)$-monodromic. 

\end{cor}

\begin{proof}
We explain the $L$-monodromic structure on the module $\ms{M}(r,u,v,w)$; statement \vii is straight-forward. Recall that $\ms{M}(r,u,v,w)$ is the minimal extension of a local system $\mathsf{L}_{\theta}$ defined on the orbit $\Orb_{\lambda^{\pm}}$. This local system is, in turn, the pull-back of a local system $\mathsf{L}_{\eta}$ on $\Orb := \widetilde{\Orb_{\lambda^{\pm}}}$. The spaces $\Orb_{\lambda^{\pm}}$ and $\Orb$ are $L$-orbits such that the projection $\Orb_{\lambda^{\pm}} \twoheadrightarrow \Orb$ is equivariant. Therefore, if $\mathsf{L}_{\eta}$ has the structure of a $(L,q)$-monodromic $\dd$-module for some $q$ it follows by functorality, Proposition \ref{prop:monodromicfinite}, that both $\mathsf{L}_{\theta}$ and $\ms{M}(r,u,v,w)$ are also $(L,q)$-monodromic. 

Let $a : L \times \Orb \rightarrow \Orb$ be the action map. Choose $x \in \Orb$ and let $b : L \rightarrow \Orb$ be defined by $g \mapsto g \cdot x$. This induces a group homomorphism $b_* : \pi_1(L) \rightarrow \pi_1(\Orb)$. Since $\mathsf{L}_{\eta}$ is a simple local system and $L$ is connected, $a^* \mathsf{L}_{\eta}$ is a simple local system on $L \times \Orb$. Since the rank of $\mathsf{L}_{\eta}$ is one, the local system $a^* \mathsf{L}_{\eta}$ is isomorphic to $\mathsf{L}' \boxtimes \mathsf{L}_{\eta}$, where $\mathsf{L}' \simeq b^* \mathsf{L}_{\eta}$. 

Let $p = u + v + w$. We have 
$$
\pi_1(\LSL) = \left\{ (a_i) \in \Z^p \ \big| \ \sum\nolimits_i\ a_i = 0 \right\}, \quad \pi_1(L) = \Z^p, \quad \pi_1(\Orb) = \Z_m^{u + v} \times \Z^w.
$$
The image of $\pi_1(\LSL)$ in $\pi_1(\Orb)$ equals $\{ (a_i) \in \Z_m^{u + v} \times \Z^w \ | \ \sum_i a_i = 0 \textrm{ in } \Z_m \}$ so that the quotient of $\pi_1(\Orb)$ by $\mathrm{Im} \ (\pi_1(\LSL))$ is isomorphic to $\Z_m$. Recall that $\eta$ is the representation of $\Z_m$ defined by $z \cdot \mathrm{Id} \mapsto z^r$. Therefore the representation of $\pi_1(L)$ corresponding to the local system $\mathsf{L}_{\eta}$ is
$$
\pi_1(L) \longrightarrow \pi_1(\Orb) / \mathrm{Im} \ (\pi_1(\LSL)) \simeq \Z_m \stackrel{\eta}{\longrightarrow} \Cs
$$
which, under the convention (\ref{eq:indentify}), is given by $\exp( \frac{2 \pi \sqrt{-1} r}{m})$.
\end{proof}

\section{Reduction to a Levi}\label{sec:reduction}

In this section we study the restriction of mirabolic modules to the locally closed subsets $\X_{\LSL}^{\circ} = \LSL^{\circ} \times V$.  

\subsection{} Let $\Upsilon$ denote the locally closed embedding $\Xo_{\LSL} \hookrightarrow \X$. 

\begin{prop}\label{prop:wonderfulinclusion}
We have $\rho_{\Upsilon}(\omega_{\Upsilon}^{-1}(\Nnil(\SSL))) \subseteq \Nnil(\LSL^{\circ})$. 
\end{prop}

\begin{proof}
Firstly, $\omega_{\Upsilon}^{-1}(\Nnil(\SSL)) = \Nnil(\SSL) \cap T^* \X^{\circ}_{\LSL}$. Write $x=(g,i) \in \X^{\circ}_{\LSL} = \LSL^{\circ} \times V$ and let $(Y,j)$ be a covector in $\mc{N}_x$. We may identify: $T^*_x \X = \mf{sl} \times V^*$ and $T^*_x \X^{\circ}_{\LSL} = \mf{l} \times V^*$. Further, we have a decomposition $\mf{sl} = \mf{l} \oplus \mf{l}^\perp$, with respect to the trace pairing on $\mf{sl}$. For any $Y \in \mf{sl}$, we thus have $Y=Y' + Y''$, the corresponding decomposition. The map $p$ is then given by $(Y,j) \mapsto (Y',j)$. Now, the moment map equation on $(g,Y,i,j)$ is $g Y g^{-1} - Y + i \circ j = 0$. Let $V = V_1 \oplus \cdots \oplus V_m$ be the decomposition of $V$ with respect to the action of $Z(\LSL)$. Since $g \in \LSL$, we have $g Y' g^{-1} - Y' + \sum_{k = 1}^m i_k \circ j_k = 0$, where $i_k$ is the component of $i$ in $V_k$ and similarly for $j$. Since $(g,Y,i,j) \in \Nnil(\SSL)$, there exists a complete flag $\mc{F}_{\idot} = (0 = \mc{F}_0 \subset \cdots \cdot \mc{F}_n = V)$ in $V$ that is stabilized by $g$ and $Y$, \cite[Lemma 12.7]{EG}. Therefore, writing $g= s.u$ for the Jordan decomposition of $g$, we deduce that $\mc{F}$ is also $s$-stable. Let $\mf{b} = \Stab_{\mf{sl}}(\mc{F})$ be the Borel in $\mf{sl}$ corresponding to the flag $\mc{F}$. Thus, $Y,s \in \mf{b}$. 

Since $g\in \LSL^{\circ}$, we deduce that the centralizer $\mf{m}$ of $s$ in $\mf{sl}$ is contained in $\mf{l}$. It follows that $\mf{z}_{\mf{l}} \subset \mf{z}_{\mf{m}}$. But the centre of $\mf{m}$ is contained in $\mf{b}$ (to see this, we may diagonalize $s$ so that $\mf{b}$ are the upper triangular matrices in $\mf{sl}$). Thus, $\mf{z}_{\mf{l}} \subset \mf{b}$ and hence $\mf{z}_{\mf{l}}$ stabilizes $\mc{F}$. This implies that the flag $\mc{F}$ is compatible with the decomposition $V = V_1 \oplus \cdots \oplus V_m$ in the sense that $\mc{F}_i = \oplus_{k = 1}^m \mc{F}_i^{(k)}$ with $\mc{F}_i^{(k)} \subset V_k$. The set of all flags admitting such a decomposition can be identified with the flag variety $\mf{l} / \mf{b}'$ for some Borel $\mf{b}'$ in $\mf{l}$. This Borel is 
$$
\Stab_{\mf{l}}(\mc{F}) =  \Stab_{\mf{sl}}(\mc{F}) \cap \mf{l} = \mf{b} \cap \mf{l}.
$$
It follows that $\mf{l} \cap \mf{b}$ is a Borel in $\mf{l}$. Let $\mf{t} \subset (\mf{l} \cap \mf{b})$ be a Cartan subalgebra. Then, $\mf{b} = \mf{t} \oplus \mf{n}$, and $\mf{l} \cap \mf{b} = \mf{t} \oplus \mf{n}_0$ for some $\mf{t}$-stable subspaces $\mf{n}_0 \subset \mf{n}$. The space $\mf{n}$ has a \textit{unique} decomposition $\mf{n} = \mf{n}_0 \oplus \mf{n}_1$ as $\mf{t}$-modules. Since the decomposition $\mf{sl} = \mf{l} \oplus \mf{l}^{\perp}$ is $\mf{t}$-stable, the uniqueness of the above decomposition implies that $\mf{l}^{\perp} \cap \mf{b} = \mf{n}_1$. Thus, the assumption that $Y$ is nilpotent in $\mf{sl}$ implies that $Y \in \mf{n}$ and the decomposition $Y = Y'+Y''$ corresponds to the image $Y$ in $\mf{n}_0$ and $\mf{n}_1$ respectively, under the projections from $\mf{n}$. In particular, $Y' \in \mf{n}_0 \subset \mf{n}$ implies that $Y'$ is nilpotent. 
\end{proof}

A regular holonomic $(N_G(L),c)$-monodromic $\dd$-module on $\X_{\LSL}^{\circ}$, whose singular support is contained in $\Nnil(L^{\circ})$, is called a mirabolic module. The category of all mirabolic modules on $\X_{\LSL}^{\circ}$ is denoted $\ss_{\X_{\LSL}^{\circ},c}$. Associated with the embedding $\Upsilon$, there is a (right exact) \textit{underived}
restriction functor
$$
\Upsilon^* : \Lmod{\dd_{\X}} \rightarrow \Lmod{\dd_{\Xo_{\LSL}} }. 
$$

\begin{cor}\label{prop:restrictionfunctor}
The image of $\ss_c$ under $\Upsilon^*$ is contained in $\ss_{\X_{\LSL}^{\circ},c}$ and $\Upsilon^* : \ \ss_c \rightarrow \ss_{\X_{\LSL}^{\circ},c}$ is exact, commutes with Verdier duality and preserves semi-simplicity.
\end{cor}

\begin{proof}
Lemma \ref{trans} implies that every mirabolic module on $\X$ is non-characteristic for $\Upsilon$ and hence $\Upsilon^* : \ss_c \rightarrow \Lmod{\dd_{\Xo_{\LSL}}}$ is exact.  It also implies that $\Upsilon^*$ commutes with Verdier duality, \cite[Theorem 2.7.1]{HTT}, and  preserves semi-simplicity. Proposition \ref{prop:wonderfulinclusion}, together with \cite[Theorem 2.7.1]{HTT}, implies that $\Char(\Upsilon^* (\ms{M}))$ is contained in $\Nnil(\LSL^{\circ})$ (as defined in (\ref{eq:nnilLSL})), for all $\ms{M} \in \ss_c$. 
\end{proof}

\subsection{} We write $\zero$ for the zero section of $T^*(Z(\LSL)^{\circ})$. We denote by $\FZ_{\LL}$ the subalgebra of $L'$ bi-invariant differential operators in $\dd(\LL)$. Let $\varsigma : \dd(\LL) \rightarrow \dd(Z(\LSL)^{\circ} \times \X_{\LL})$ be the natural embedding.

\begin{lem}\label{lem:locallyfactorL}
Let $\ms{M}$ be a simple, regular holonomic, $(L,q)$-monodromic $\dd$-module on $Z(\LSL)^{\circ} \times \X_{\LL}$. We assume that the support of $\mm$ is contained in $Z(\LSL)^{\circ} \times \mn^{\LL}
\times V$ and its characteristic variety is contained in
$\zero \times \Nnil(\LL)$.

\vi There exists a local system $\mathsf{L}$ on $Z(\LSL)^{\circ}$ and simple $\dd$-module $\ms{N}$ on $\X_{\LL}$ such that $\ms{M} \simeq \mathsf{L} \boxtimes \ms{N}$. 

\vii The action of $\varsigma(\FZ_{\LL})$ on $\Gamma(Z(\LSL)^{\circ} \times \X_{\LL},\ms{M})$ is locally finite.
\end{lem}

\begin{proof}
\vi The support of $\ms{M}$ equals $Z(\LSL)^{\circ} \times \overline{\Orb}$,
where $\Orb$ is a relevant stratum in $\mn^{\LL} \times
V$. This relevant stratum $\Orb$ is an $L$-orbit. Since $L$ is
connected, $\Orb$ is irreducible. There is some open subset $U$ of
$Z(\LSL)^{\circ} \times \Orb$ and simple, $\LSL$-equivariant local
system $\mathsf{M}$ on $U$ such that $\ms{M} = \IC(U,\mathsf{M})$. The
set $U$ is necessarily $L$-stable. Choose some $(z,x) \in U$. Then
$Z(\LSL)^{\circ} \times \{ x \} \cap U$ is open and non-empty, hence
dense, in $Z(\LSL)^{\circ}$ and thus equals $U_x \times \{x \}$ for some
open subset $U_x$ of $Z(\LSL)^{\circ}$. This implies that 

$$
L \cdot (U_x \times \{ x \}) = U_x \times \Orb
$$
is contained in $U$. The closure of $U_x \times \Orb$ equals $Z(\LSL)^{\circ} \times S$, hence, if $\mathsf{M}'$ is the restriction of $\mathsf{M}$ to $U_x \times \Orb$, one has $\ms{M} = \IC(U_x \times \Orb, \mathsf{M}')$. The local system $\mathsf{M}'$ is an irreducible representation of $\pi_1 (U_x) \times \pi_1(\Orb)$ and hence is isomorphic to $\mathsf{L}' \boxtimes \mathsf{N}$. Thus
$$
\ms{M} = \IC (U_x \times \Orb,\mathsf{M}') \simeq \IC (U_x,\mathsf{L}') \boxtimes \IC (\Orb,\mathsf{N}). 
$$
The constraints on the characteristic variety of $\ms{M}$ imply that $\mathsf{L} = \IC (U_x,\mathsf{L}')$ is a local system.

\vii The module $\ms{M}$ has finite length. Therefore, by induction on length, it suffices to assume that $\ms{M}$ is simple. By part \vi, such a simple module is isomorphic to $\mathsf{L} \boxtimes \ms{N}$, for some local system $\mathsf{L}$ on $Z(\LSL)^{\circ}$ and simple $\dd$-module $\ms{N}$. Then, since $\LL$ is a product of $SL_m$'s, Theorem \ref{thm:admissible2} implies that $\FZ_{\LL}$ acts locally finitely on $\Gamma(\X_{\LL},\ms{N})$.  
\end{proof}

We are now in a position to prove Proposition \ref{thm:irrsupport}. 

\begin{proof}[Proof of Proposition \ref{thm:irrsupport}]
Assume that we are given a simple mirabolic module $\ms{M} \in \ss_{q}$ such that $\Supp \ms{M} \cap \X^{\reg} = \emptyset$. Let $c \in \C$ such that $\exp(2 \pi \sqrt{-1} c) = q$. By Proposition \ref{prop:monoimpliesweakG}, we can endow $\mm$ with the structure of a $(G,c)$-monodromic module i.e. $\mm \in \ss_c$. Choose a proper Levi subgroup $\LSL$ of $\SSL$ and relevant stratum $\X(\LSL,\Omega)$, whose closure is $\Supp \ms{M}$. 

The variety $Y := Z(\LSL)^{\circ} \times \mn^{\LSL} \times V$ is a closed subspace of both $\Xo_{\LSL}$ and $Z(\LSL)^{\circ} \times \X_{\LL}$, which are open subsets of $\X_{\LSL}$. Therefore, Kashiwara's Theorem implies that the category of coherent $\dd_{\Xo_{\LSL}}$-modules supported on $Y$ can be canonically identified with the category of coherent $\dd_{Z(\LSL)^{\circ} \times \X_{\LL}}$-modules supported on $Y$. Lemma \ref{lem:stratainfo} implies that $\Upsilon^* \ms{M}$ is supported on $Y$. The fact that $\mm$ is non-characteristic for $\Upsilon$ and $\Supp \mm \cap \Xo_{\LSL} \neq \emptyset$ implies that $\Upsilon^* \ms{M}$ is non-zero $\dd$-module on $Z(\LSL)^{\circ} \times \X_{\LL}$. 

Let $\ms{N}$ be a simple, $(L,q)$-monodromic submodule of $\Upsilon^* \ms{M}$. By Lemma \ref{lem:locallyfactorL}, $\ms{N}$ is isomorphic to $\mathsf{L} \boxtimes \ms{N}'$ for some simple local system $\mathsf{L}$ on $Z(\LSL)^{\circ}$ and $(L,q)$-monodromic, $\LSL$-cuspidal mirabolic module on $\bN^{\LL} \times V$. Corollary \ref{cor:summarycuspidal} implies that $\LSL,\Omega$ and $q$ are as in the statement of Proposition \ref{thm:irrsupport}. 
\end{proof}

\subsection{} In order to use the restriction functor to study the Harish-Chandra $\dd$-module, we need to understand how it relates to taking invariants.

\begin{prop}\label{prop:keyfact}
Let $\ms{M}$ be a simple mirabolic module whose support is
$\ol{\X(L,\Omega)}$. The natural map
\begin{equation}\label{eq:keyfact}
\Gamma(\X,\ms{M})^{\SSL} \rightarrow \Gamma(\Xo_{\LSL},\Upsilon^* \ms{M})^{N_{\SSL}(\LSL)}
\end{equation}
is an embedding.
\end{prop}

\begin{proof}
Consider the standard diagram 
$$
\Xo_{\LSL} \ \stackrel{j}{\longrightarrow} \ \SSL \times \Xo_{\LSL} \ \stackrel{\pi}{\longrightarrow}  \ Z := \SSL \times_{N_{\SSL}(\LSL)} \Xo_{\LSL}\  \stackrel{p}{\longrightarrow} \  \X.
$$
where $p([g,x]) = g \cdot x$ and $j(v) = (e,v)$. The image of $p$ is a (non-affine) open subset $\Xo$ of $\X$. Since $\ms{M}$ is simple, $\Supp m = \Supp \ms{M}$ for all non-zero sections $m \in \Gamma(\X,\ms{M})$. Therefore, since $\Xo \cap \overline{\X(L,\Omega)} \neq \emptyset$, the natural map $\Gamma(\X,\ms{M}) \rightarrow \Gamma(\Xo,\ms{M})$ is an embedding (its kernel consists of all section with support in the complement of $\Xo$). Hence $\Gamma(\X,\ms{M})^{\SSL} \hookrightarrow \Gamma(\Xo,\ms{M})^{\SSL}$. Now take a non-zero section $m \in \Gamma(\Xo,\ms{M})^{\SSL}$. We wish to show that $p^* m$ is a non-zero section of $\Gamma(Z,p^* \ms{M}))$. 

Although the map $p$ is not finite, it is locally strongly \'etale,
see  Lemma \ref{lem:strongetale} below. It follows that
there exists an affine open, $\SSL$-saturated, covering $\{ \Xo_i \}_{i \in I}$ of $\Xo$ such that $\{ U_i = p^{-1}(\Xo_i) \}_{i \in I}$ is an affine open, $\SSL$-saturated, covering of $Z$ and $U_i \simeq U_i /\!/ \SSL \times_{\Xo_i /\!/ \SSL} \Xo_i$ for all $i$. The section $m$ may be written $m = (m_1, \ds, m_k)$, where $m_i$ is the image of $m$ in $\Gamma(\Xo_i,\ms{M})$. Since each $\Xo_i$ is open in $\X$, either the restriction of $\ms{M}$ to $\Xo_i$ is simple or is zero. Either way, we may assume that $m_i$ generates $\ms{M} |_{\Xo_i}$ for all $i$. Hence $p^* \ms{M} |_{\Xo_i} = \dd_{U_i} \cdot p^* m_i$. The restriction of $p$ to each $U_i$ is a finite morphism. Therefore $\C[U_i]$ is a locally free $\C[\Xo_i]$-module and $\ms{M} |_{\Xo_i} \neq 0$ implies that $p^* \ms{M} |_{\Xo_i} \neq 0$. Hence $m_i \neq 0$ implies that $p^* m_i \neq 0$ and we have shown that $p^* m$ is a non-zero section of $p^* \ms{M}$ (which actually locally generates the module). 

Therefore, $\Gamma(\X,\ms{M})^{\SSL}$ embeds inside $\Gamma(Z,p^* \ms{M})^{\SSL}$. Since $\Upsilon = p \circ \pi \circ j$, Proposition \ref{prop:monoequiv} \vii implies that $\Gamma(Z,p^* \ms{M})^{\SSL} \iso \Gamma(\Xo_{\LSL},p^* \ms{M})^{N_{\SSL}(\LSL)}$ as required. 
\end{proof}

 A $G$-equivariant morphism $f : X \rightarrow Y$ between affine varieties $X$ and $Y$ is said to be \textit{strongly \'etale} if the induced map $f' : X/\!/G \rightarrow Y/\!/G$ is \'etale and $f'$ induces an isomorphism $X \simeq Y \times_{Y/\!/G} X/\!/G$ such that $f$ corresponds to projection onto $Y$; see \cite{Luna}. By \textit{locally strongly \'etale}, we mean that for each point $x$ in $Z$, there exists an affine, $\SSL$-saturated open neighborhood of $x$ in such that the restriction of $p$ to this neighborhood is strongly \'etale.

\begin{lem}\label{lem:strongetale}
The morphism $p :\ Z \rightarrow \Xo:=p(\X)$
 is locally strongly \'etale.
\end{lem}

\begin{proof}
We begin by showing that $p$ is quasi-finite. Let $(g_1,(x_1,v_1))$ and $(g_2, (x_2,v_2))$ be points of $Z$ such that $(g_1 \cdot x_1,g_1(v_1)) = (g_2 \cdot x_2,g_2(v_2))$ in $\X$. If the Jordan decomposition of $x_i$ is $s_i u_i$ with $s_i$ semi-simple, then $s_i \in \LSL^{\circ}$ and $g_1 \cdot s_1 = g_2 \cdot s_2$. Thus, $s_1$ is conjugate in $\SSL$ to $s_2$. Since we only want to show the map is quasi-finite, using the fact that the intersection $\SSL \cdot s_1 \cap \LSL$ consists of finitely many $\LSL$-orbits, we may assume that $s_1 = s_2$. Thus, we are reduced to showing that number of cosets of $Z_{\SSL}(s_1) \cap N_{\SSL}(\LSL)$ in $Z_{\SSL}(s_1)$ is finite. Since $s_1 \in \LSL^{\circ}$, we actually have $Z_{\SSL}(s_1) \subset \LSL$, hence $Z_{\SSL}(s_1) \cap N_{\SSL}(\LSL) = Z_{\SSL}(s_1)$. 

Since $p$ surjects onto $\Xo$ and is quasi-finite, it is proper. Therefore Zariski's main theorem implies that it is finite. The map $p$ will be \'etale if its differential is everywhere surjective. It suffices to show that the differential of $q = p \circ \pi : \SSL \times \Xo_{\LSL} \rightarrow \X$ is surjective. If we identify $\g$ with \textit{left} invariant vector fields on $G$, then for $g \in \SSL, l \in \Lo_n$ and $v \in V$, 
$$
d_{(g,l,v)} q : (X,y,w) \mapsto (gyg^{-1} + g[X ,l]g^{-1}, g(w) + X(v)).
$$
This is surjective: one can check this at $g = e$. 

Next we show that it is locally strongly \'etale.  Since both domain and image of $p$ are smooth varieties, \cite[Lemme Fondamental]{Luna} says that it suffices to check that, for each $x \in Z$ such that the orbit $\SSL \cdot x$ is closed, we have $p(\SSL \cdot x)$ is closed in $\X$ and $p |_{\SSL \cdot x}$ is injective. The closed orbits of $\SSL$ in $Z$ are all of the form $\SSL \cdot (1,z,0)$, where $z \in \Lo_n \cap \SSL^{ss}$ (here $\SSL^{ss}$ is the set of all semi-simple elements in $\SSL$). Since the orbit of $(g \cdot z,0)$ is closed in $\X$ for any $g \in \SSL$, $p(\SSL \cdot x)$ is closed in $\X$. Now assume that $(g,n \cdot z,0)$ is mapped to $(z,0)$ under $p$. Then $(gn) \cdot z = z$ and hence $gn \in Z_{\SSL}(z)$. Since $z \in \Lo_n$ this implies that $gn \in \LSL$ and hence $(g,n \cdot z,0) = (1, gn \cdot z, 0 ) = (1, z,0)$. Thus $p |_{\SSL \cdot x}$ is injective. 
\end{proof}

\section{The functor of Hamiltonian reduction}\label{ham_sec}

\subsection{}\label{sec:qhr}
Recall that $\mathsf{H}_{\kappa}^{\mathrm{trig}} (\SSL)$ denotes the
trigonometric Cherednik algebra of type $\SSL$, as defined in \S\ref{app:shift1}, and $\htrig_{\kappa}=e \mathsf{H}_{\kappa}^{\mathrm{trig}} (\SSL) e$
is the corresponding spherical subalgebra. As in \S\ref{app:shift1}, by using the Dunkl embedding, we think of $\htrig_{\kappa}$ as subalgebra of the simple ring $\htrig^{\reg} = \dd(\TSL^{\reg})^{W}$. The algebra $\htrig_{\kappa}$ has two obvious commutative subalgebras: the subalgebra $\C[\TSL]^W$ of $W$-invariant regular functions on $\TSL$, and the subalgebra $(\sym \tSL)^W$.

The algebra $\htrig_{\kappa}$ is noetherian and we let $\Lmod{\htrig_{\kappa}}$ denote
the abelian category  of {\em finitely generated}
left $\htrig_{\kappa}$-modules. 
In this paper, we will also consider the following full subcategories
of the category $\Lmod{\htrig_{\kappa}}$:
\begin{itemize}

\item Category $\mof$, whose objects are $\htrig_{\kappa}$-modules
which are finitely generated over the subalgebra $\C[\TSL]^W\sset
\htrig_{\kappa}$.

\item Category $\OO_\kappa$, whose objects are finitely generated $\htrig_{\kappa}$-modules, such that the action of $\C[\tSL]^W$ is locally finite. 
\end{itemize}

The fact that the natural map $(T^* \TSL) / W \twoheadrightarrow \TSL / W \times \tSL / W$ is a finite morphism implies that $\OO_\kappa$ is a full subcategory of $\mof$. As in (\ref{sec:spectral}), we also have a spectral decomposition, c.f. ~\cite{VVQuantumAffineSchur},
$$
\OO_{\kappa} = \bigoplus_{\Theta \in \tSL^* / W_{\mathrm{aff}}} \OO_{\kappa} \langle \Theta \rangle.
$$
Recall from section \ref{sec:intro} that $\g_{c} = (\mu - c \Tr)(\g)$, $\ddd  = \dd(\X)$ and $\kappa = -c + 1$ . We consider the left $\ddd $-module $\ddd  / \ddd  \cdot \g_{c}$, resp. the right $\ddd $-module $\ddd  / \g_{c} \ddd $. According to \cite{GG}, the space $\ddd  / \ddd  \cdot \g_{c}$ has the natural structure of a (weakly) $G$-equivariant $(\ddd ,\ \htrig_\kappa)$-bimodule. Similarly, the space $\ddd  / \g_{c} \ddd $ has the natural structure of a (weakly) $G$-equivariant $(\htrig_\kappa,\ \ddd )$-bimodule. One has an (infinite) direct sum decomposition
\begin{equation}\label{eq:Gdecomp}
\ddd  / \g_{c} \ddd  =\bplus_{\sigma\in\Irr G}\ (\ddd  / \g_{c} \ddd )^{(\sigma)}
\end{equation}
into $G$-isotypic components. The left $\htrig_\kappa$-action on $\ddd /\g_c\ddd $ commutes with the $G$-action and $\gr(\ddd /\g_c\ddd) $, the associated graded of $\ddd /\g_c\ddd $ with respect to the order filtration, is a finitely generated $\C[T^* \X]$-module. Hence, a well-known 
result of  Hilbert, \cite[Zusatz 3.2]{Kraft}, to be referred to as
`{\em Hilbert's Theorem}' in the future,
implies that  each isotypic component of $\gr(\ddd /\g_c\ddd)$
is a finitely generated module over the subalgebra of $G$-invariants, 
that is, $$
\gr [(\ddd  / \g_{c} \ddd )^{(\sigma)}] = [\gr (\ddd  / \g_{c} \ddd )]^{(\sigma)}
$$
is a finitely generated $\gr [(\ddd  / \g_{c} \ddd )^G] = \gr \htrig_\kappa$-module.

\subsection{} The functor of Hamiltonian reduction $\Ham_{c} : \mon{\ddd ,G,q} \rightarrow \Lmod{\htrig_\kappa}$ is defined by 
$$
\Ham_{c}(\ms{M}) = \{ m \in \Gamma(\X,\ms{M}) \ | \ \arr{u} \cdot m = c\Tr(u) m, \ \forall \ u \in \g \}
$$

Next, using the fact that the adjoint action of $\mu(\g)$ on $\ddd /\ddd \g_c$ and $\ddd /\g_c\ddd$ is locally finite, we introduce a pair of functors $\Lmod{\htrig_\kappa} \to \mon{\ddd ,G}$ as follows
\begin{align*}
\Hamp(E)&:=\ddd /\ddd \g_c\,\ \bo_{\htrig_\kappa}\,\ E,\\
\Hamr(E)&:=\bplus_{\sigma\in\Irr G}\
\Hom_{\htrig_\kappa}((\ddd /\g_c\ddd )^{(\sigma)},\ E).
\end{align*}
It is clear that the functor $\Hamp$ is right exact
and the functor $\Hamr$ is left exact. 

\begin{lem}\label{adjoint} 

\vi The functor $\Hamp$ is a left adjoint, resp. $\Hamr$ is a right adjoint,
of the functor $\Ham_{c}$. Each of the canonical adjunctions 
$\Ham_{c} \circ \Hamp (E) \to E$ and $E\to \Ham_{c} \ccirc \Hamr(E)$ is an isomorphisms.

\vii For any $E\in\Lmod{\htrig_{\kappa}}$, the module
 $\Hamp(E)$ has no quotient modules, resp. 
 $\Hamr$  has no submodules, annihilated by $\Ham_{c}$.
\end{lem}

\begin{proof} For any $(G,q)$-monodromic $\ddd $-module $\mm$, Proposition \ref{prop:monodromicfinite} implies that one has a canonical isomorphism $\mm^{\g_c}$ $\iso \mm/\g_c\mm$. Therefore, for each $\mm \in \mon{\ddd ,G,q}$, we obtain
\begin{multline*}
\Hom_{\htrig_\kappa}(\Ham_{c}(\mm), E)=\Hom_{\htrig_\kappa}(\mm/\g_c\mm, E)\\
=\Hom_{{\htrig_\kappa}}(\ddd /\g_c\ddd \o_{\ddd }\mm, E)
=\Hom_{\ddd }(\mm,\ \Hom_{{\htrig_\kappa}}(\ddd /\g_c\ddd , E)).
\end{multline*}
Let $f: \mm\to \Hom_{{\htrig_\kappa}}(\ddd /\g_c\ddd , E)$
be a $\dd$-module morphism.
The $G$-action on $\mm$ being locally finite,
any element $m\in\mm$ is contained in a 
{\em finite} sum of $G$-isotypic components.
Hence, $f(m)$ is also  contained in a
finite sum of $G$-isotypic components.
It follows that the morphism
$f$ factors through a map
$\mm\to \bplus_{\sigma\in\Irr G}\
\Hom_{\htrig_\kappa}((\ddd /\g_c\ddd )^{(\sigma)},\ E)$.
Thus, $\Hamr$ is a right adjoint of $\Ham_{c}$.
Furthermore,
we clearly have
\begin{multline*}
\Ham_{c}(\Hamr(E))=\big[\bplus_{\sigma\in\Irr G}\
\Hom_{\htrig_\kappa}((\ddd /\g_c\ddd )^{(\sigma)},\ E)\big]^G\\
=
\Hom_{\htrig_\kappa}((\ddd /\g_c\ddd )^G,\ E)
=\Hom_{\htrig_\kappa}({\htrig_\kappa}, E)=E.
\end{multline*}
This proves the statements in (i) concerning
the functor $\Hamr$.

Now, let $\mm$ be a module annihilated by $\Ham_{c}$. We get
$$
\Hom_\ddd (\mm,\Hamr(E))=\Hom_{\htrig_\kappa}(\Ham_{c}(\mm),E) =0.
$$
We deduce that $\Hamr(E)$ can not have a submodule annihilated by $\Ham_{c}$. This proves (ii) for ~$\Hamr$.

The statements concerning the functor $\Hamp$ are proved similarly.
\end{proof}

Recall from (\ref{eq:defineNnilSL}) the Lagrangian subvariety $\Nnil$ of $T^* \X$. To state our next result, we need to introduce the category $\Lmod{(\ddd,\Nnil)}$, whose objects are coherent $\ddd$-modules $\mm$ such that $\SS(\mm)\sset \Nnil$. Any object
$\mm\in \Lmod{(\ddd,\Nnil)}$ is a $G$-monodromic,
holonomic $\dd_\X$-module. Such a module is
a mirabolic $\dd$-module if and only if it has
regular singularities. 
Thus, one has a strict inclusion $\cc_q \sset \Lmod{(\ddd,\Nnil)}$.

\begin{prop}\label{many} The  functors $(\Hamp, \Ham)$ induce the following pairs of adjoint functors
$$
\xymatrix{
{\mathsf{(1)}}\en\Lmod{(\ddd,\Nnil)}\  \ar@<0.5ex>[r]^<>(0.5){\Ham_c}&
\ \mof \ar@<0.5ex>[l]^<>(0.5){\Hamp}& 
 {\mathsf{(2)}}\en \ss_q \langle \Theta \rangle \ \ar@<0.5ex>[r]^<>(0.5){\Ham_c}&
\ \OO_\kappa \langle \Theta \rangle,\en \forall \ \Theta \in \tSL^* /
W_{\mathrm{aff}}. 
\ar@<0.5ex>[l]^<>(0.5){\Hamp}
}
$$
\end{prop}

\begin{proof} For the left pair above, the  result was proved in the course of the proof of \cite{CherednikCurves}, Proposition 4.6.2 (although  part (ii) of that proposition, claiming that the functor $\Hamp$ sends category $\oo({\mathsf A}_{\kappa,\psi})$ to $\ss_{\psi,c}$, is incorrect, as stated). 

Recall next that the radial parts homomorphism $\mathfrak{R} : \ddd^{G}\to \htrig_\kappa$ of Theorem \ref{thm:radialparts} yields an algebra isomorphism $\FZ\iso (\sym \tSL)^W$. This result, combined with  Theorem \ref{thm:admissible2}, implies that the functor $\Ham_c$ sends the category $\cc_q \langle \Theta \rangle$ to category $\OO_\kappa \langle \Theta \rangle$.

For each $N \in \mathbb{N}$, let $\CG_{\lambda,c}^{(N)}$ be a \textit{generalized Harish-Chandra} $\dd$-module, where we quotient out
by $\FZ_{\lambda}^N$ instead of $\FZ_{\lambda}$. The module
$\CG_{\lambda,c}^{(N)}$ has a finite filtration whose subquotients are
quotients of $\CG_{\lambda,c}$. Thus, Corollary \ref{HCmir} implies that
$\CG_{\lambda,c}^{(N)} \in \cc_q \langle \Theta \rangle$, where $\Theta$
is the image of $\lambda$ in $\tSL^* / W_{\mathrm{aff}}$. Since the adjoint action of $\mu(\g)$ on $\dd(\X)$ is semi-simple, $\CG_{\lambda,c}^{(N)}$ actually belongs to the full subcategory $\ss_q \langle \Theta \rangle$ of $\cc_q \langle \Theta \rangle$. Let $m_1, \ds, m_k$ be generators of $M \in \OO_\kappa$. Then $\Hamp(M)$ is generated by $1 \o m_1, \ds, 1 \o m_k$. For each $i$, arguing as in the proof of Corollary \ref{HCmir}, one can find $\lambda_i \in \tSL^*$ and $N \gg 0$ such that $\CG_{\lambda,c}^{(N)} \twoheadrightarrow \ddd \cdot (1 \o m_i)$. Thus, $\Hamp(M)$ belongs to
$\ss_q$, since the latter category is closed under submodules and quotients. Moreover, if $M \in \OO_{\kappa} \langle \Theta \rangle$, then
 the image of $\lambda_i$ in $\tSL^* / W_{\mathrm{aff}}$ will equal $\Theta$ for all $i$. This proves the second statement. 
\end{proof}

\begin{rem} It is straight-forward to see that,
for  $E \in \oo_\kappa$, the $\FZ$-action on $\Hamr(E)$ is 
also locally finite.
\end{rem}

\subsection{} For any $E\in \Lmod{\htrig_\kappa}$, there is a canonical vector space 
embedding 
\[\imath:\ E=\Hom_{\htrig_\kappa}({\htrig_\kappa}, E)=
\Hom_{\htrig_\kappa}((\ddd /\g_c\ddd )^{\text{triv}},\ E)
\ \into\ \Hamr(E).\]
The  map $\imath$ induces a $ \ddd $-module 
morphism $\Hamp(E)=\ddd /\ddd \g_c\o_{\htrig_\kappa} E \to\Hamr(E)$. The latter morphism corresponds to the identity via the isomorphisms
$$
\id\in \Hom(E, E) = \Hom(E, \Ham_{c}(\Hamr(E))=\Hom(\Hamp(E), \Hamr(E)).
$$
We put $E_{!*}:=\im[\Hamp(E) \to\Hamr(E)]$. 
Equivalently, $E_{!*}$ is the $\ddd $-submodule
of $\Hamr(E)$ generated by the subspace $\imath(E)$.

It is clear from definitions
 that  $\Ham_{c}(E_{!*})=E$ and that the object $E_{!*}$ has neither
 quotient nor sub objects annihilated by $\Ham_{c}$.
This implies the following

\begin{cor}\label{top} 

For any simple object $E\in\OO_{\kappa}$, one has

\vi The mirabolic module $\Hamp(E)$ has a simple top.

\vii 
The top of  $\Hamp(E)$ is  $E_{!*}$ and  $\Ham_{c}(E_{!*})=E$.
\end{cor}

\begin{proof} 
It was shown in \cite{GG} that the functor $\Ham_{c}$ is an exact, quotient functor. Hence, for $E$ simple, the  $\dd$-module $\Hamp(E)$ has exactly one simple composition factor, call it $\mm$, such that $\Ham_{c}(\mm)=E$. Now, suppose $f: \Hamp(E)\onto \mm'$ is a nonzero morphism onto a simple module $\mm'$. The morphism $f$ gives, by adjunction, a  nonzero morphism $E \to \Ham_{c}(\mm')$. Hence, one must have $\Ham_{c}(\mm')\neq 0$. We see that, for any simple summand $\mm'$, of the top of $\Hamp(E)$, one has $\Ham_{c}(\mm')\neq0$. Thus, there is exactly one such summand and that summand must be the module $\mm$.
\end{proof}

Let $\alpha : \X^{\cyc} \hookrightarrow \X$ be the open, \textit{affine}, embedding. Combining Lemma \ref{adjoint} with Proposition \ref{prop:Hamkill} gives:

\begin{cor}\label{cor:extensions}
Assume that $c \notin \Sing_+$, and let $E \in \mc{O}_{\kappa}$. Then 

\vi $\Hamp (E)$ has no quotients supported on $\X \sminus \X^{\cyc}$, 
 
\vii $\Hamr (E)$ has no submodules supported on $\X \sminus \X^{\cyc}$; and

\viii $E_{!*}$ is the minimal extension $\alpha_{!*}( E_{!*} |_{\X^{\cyc}})$ of $E_{!*} |_{\X^{\cyc}}$. 

\end{cor}

For $\lambda \in \mf{t}_n^* / W$, let $\mf{m}_{\lambda}$ be the corresponding maximal ideal in $(\sym \tSL)^W \subset \htrig_{\kappa}$. Define $\mc{P}_{\kappa,\lambda} = \htrig_{\kappa} / \htrig_{\kappa} \cdot \mf{m}_{\lambda}$. Then, it follows from remark \ref{rem:radialiso} that $\Hamp (\mc{P}_{\kappa,\lambda}) = \CG_{\lambda,c}$ and thus $\Ham_c(\CG_{\lambda,c}) = \mc{P}_{\kappa,\lambda}$. Corollary \ref{cor:extensions} implies that, for $c \notin \Sing_+$, 
$$
(\mc{P}_{\kappa,\lambda})_{!*} = \alpha_{!*}[ (\mc{P}_{\kappa,\lambda})_{!*} |_{\X^{\cyc}}]
$$
is the minimal extension of its restriction to the cyclic locus. This is not true in general of the Harish-Chandra module $\CG_{\lambda,c}$ because it will have quotients supported on the complement $\X \sminus \X^{\cyc}$ when $c \in \Sing_-$. Thus, $(\mc{P}_{\kappa,\lambda})_{!*}$ is ``better behaved'' than $\CG_{\lambda,c}$. Note also that Proposition \ref{prop:Hamkill} implies that $(\mc{P}_{\kappa,\lambda})_{!*}$ has no submodule supported on the complement of $\X^{\reg}$. 

\subsection{}
We define an analogue of $\Ham_{c}$ for the Levi subgroup $\LSL$ of $\SSL$ by setting 
$$
\Ham_{c}(\ms{M}) = \{ m \in \Gamma(\X_{\LSL},\ms{M}) \ | \ \arr{u} \cdot m = c \Tr(u) m, \ \forall \ u \in \mf{l}(\g) \},
$$
where $\X_{\LSL} = \LSL \times V$ and $\mf{l}(\g) = \Lie L$. Since we have an explicit description of the $\LSL$-cuspidal mirabolic modules, we can determine which of them belong to the kernel of $\Ham^L_{c}$. 

\begin{prop}\label{prop:cuspidalkill}
Let $c \in \C$ and $\ms{M}(r,u,v,w)$ the module defined in (\ref{eq:mruvw}). If $c \notin \Q$ then $\Ham_{c}(\ms{M}(r,u,v,w)) = 0$. Otherwise, 

\vi if $c \in \Q_{\le 0}$ then 
$\dis
\Ham_{c}(\ms{M}(r,u,v,w)) \neq 0
$
if and only if $u = 0$ and $c - \frac{r}{m} \in \Z_{\le 0}$. 

\vii If $c \in \Q_{> 0}$ then 
$\dis
\Ham_{c}(\ms{M}(r,u,v,w)) \neq 0
$
if and only if $v = 0$ and $c + \frac{r}{m} \in \Z_{> 0}$. 

\end{prop}

\begin{proof}
If we decompose the Lie algebra $\mf{l}$ of $L$ as $\mf{gl}_{m}^{\oplus (u + v)} \oplus \mf{gl}_{1}^{\oplus w}$ then the character $c \Tr$ of $\g$ restricts to the character $\frac{m c}{n} \Tr$ on $\mf{gl}_{m}$ and $\frac{c}{n} \Tr$ on $\mf{gl}_{1}$. Noting that there is a minus sign in the definition of $\chi_r$ given in \cite[\S 9.8]{CEE} that does not appear in our definition of $\theta$ as given in section \ref{sec:monodromydefinition}, the result \cite[Theorem 9.8]{CEE} says that if $r$ is coprime to $n$ then $\Ham_{c}(\ms{L}_r \boxtimes \C[V])$ is nonzero if and only if $nc \in \Z_{\le 0}$ and $r \equiv -nc \ \mathrm{mod} \ n$. Similarly, $\Ham_{c}(\ms{L}_r \boxtimes \delta_V)$ is nonzero if and only if $-nc \in \Z_{\ge n}$ and $r \equiv -nc \ \mathrm{mod} \ n$. Therefore, if $\Ham_{c}^L(\ms{M}(r,u,v,w)) \neq 0$ then either $u = 0$ or $v = 0$.  

If $v \neq 0$, and hence $\lambda^- = (0,\ds,0)$, then we must have $- m c \in \mathbb{N}$ and $r \equiv - m c \ \mathrm{mod} \ m$ i.e. $c = -\frac{r}{m } - a$ for some $a \in \mathbb{N}$. In particular, $c \in \Q_{\le 0}$. To show that $\Ham_{c}(\ms{M}(r,0,v,w)) \neq 0$ we need to check that 
$$
\Ham_{c} \left( \ms{E}_{\frac{r}{m}} \right) = \Ham_{-\frac{r}{m} - a} \left( \ms{E}_{\frac{r}{m}} \right) \neq 0.
$$
Noting that $\frac{r}{m} \neq 0$ in $\C / \Z$, the module $\ms{E}_{\frac{r}{m}}$ equals $\C[z,z^{-1}] \cdot z^{\frac{r}{m}}$. Since $\mu_{\C}(\one) = - z \pa_z$, the nonzero section $z^{a + \frac{r}{n}}$ belongs to $\mathbb{H}_{c} \left( \ms{E}_{\frac{r}{m}} \right)$. On the other hand, if $u \neq 0$ we must have $c = \frac{r}{m} + a$ for some $a \in \Z_{\ge 1}$. This implies that $\lambda^+ = (1^u)$.  
\end{proof}

Recall from (\ref{sec:torus}) that the $\TSL$-cuspidal mirabolic modules on $V$ are $\ms{M}(k) := \C[V_k] \boxtimes \delta_{V_{n-k}}$, for $k = 0, \ds, n-1$ and $\ms{M}(q) = j_{!*} \mc{O}_{T}^{q}$, where $j : T \hookrightarrow V$ is the embedding of the open $T$-orbit and $q \in \Cs$.

\begin{prop}\label{prop:supportT}
\begin{enumerate}
\item If $k \neq 0,n$ then $\Ham_{c}(\ms{M}(k)) = 0 $ for all $c$. 
\item If $k = 0$ then $\Ham_{c}(\ms{M}(0)) \neq 0$ if and only if $c \in \Z_{> 0}$. 
\item If $k = n$ and $q \neq 1$ then $\Ham_{c}(\ms{M}(q)) \neq 0$ if and only if $q = \exp(2 \pi \sqrt{-1} c)$. If $q = 1$ then $\Ham_{c}(\ms{M}(q)) \neq 0$ if and only if $c \in \Z_{\le 0}$.    
\end{enumerate}
\end{prop}

\begin{proof}
(1)\en We fix a basis $x_1, \ds, x_n$ of $V^*$ such that $x_1, \ds, x_k$ is a basis of $V_k$ and $x_{k+1}, \ds, x_n$ a basis of $V_{n-k}$. Then 
\begin{equation}\label{eq:defnofH}
\Ham_{c}(\ms{M}(k)) = \left\{ m \in \ms{M}(k) \ \Big| \ \left(x_i \pa_i + c \right) \cdot m = 0, \ 1 \le i \le n \right\}.
\end{equation}
For $i \le k$, there exists some $0 \neq m \in \ms{M}(k)$ with
$(x_i \pa_i + c) \cdot m = 0$ if and only if $c \in \Z_{\le 0}$. Similarly, for $i > k$, there exists
some $0 \neq m \in \ms{M}(k)$ with $(x_i \pa_i + c) \cdot m = 0$ if and only if $c \in \Z_{\ge 1}$. Therefore $\Ham_{c}(\ms{M}(k)) = 0$. \vskip2pt

(2)\en 
 This follows from (\ref{eq:defnofH}), noting that $x_i \pa_i \cdot m = \epsilon m$ implies that $\epsilon \in \Z_{\ge 1}$ for $m \in \delta_V$. 
 \vskip2pt

(3)\en Let $x_1, \ds, x_n$ be a basis of $V$ such that $T = (x_i \neq 0, \ \forall \ i)$ and $\TSL = V(x_1 \cdots x_n = 1)$. If $q = 1$ then $\mc{O}_{T}^{q}$ is the trivial local system on $T$ and its minimal extension to $V$ is just $\C[V]$. As in part (2), $\mathbb{H}_c(\C[V]) \neq 0$ if and only if $c \in n \Z_{\le 0}$. If $q \neq 1$ then consider 
$$
\ms{M}'(q) = \left[ \dd_1 / \dd_1 ( x_1 \pa_1 - c) \right] \boxtimes \cdots \boxtimes \left[ \dd_n / \dd_n ( x_n \pa_n - c ) \right],
$$
where $\dd_i = \C \langle x_i,\pa_i \rangle$. The module $\ms{M}'(q)$ is a simple $\dd$-module whose restriction to $\BT$ is the quotient of $\dd_{T}$ by $x_i \pa_i - x_j \pa_j$ and $\mathsf{eu}_V - nc$. This is precisely the local system $\mc{O}_{T}^{q}$. Hence $\ms{M}'(q) = j_{!*} \mc{O}_{T}^{q}$.  
\end{proof}

\subsection{} 
The group $N_{G}(L)$ is not connected. Therefore, for a $(N_{G}(L),c)$-monodromic $\dd$-module $\ms{M}$ on $\X_{\LSL}$, we define 
$$
\Ham_{L}(\ms{M}) := \Gamma(\X_{\LSL}, \mm)^{N_{G}(L)} \subset \Ham_c(\mm).
$$ 

\begin{cor}\label{cor:keyfact}
Let $\ms{M} \in \ss_q$ be a simple mirabolic module whose support is $\overline{\X(\LSL,\Omega)}$. 
\begin{enumerate}
\item The natural map $\Ham_{c}(\ms{M}) \to \Ham_{L}(\Upsilon^* \ms{M})$ is an embedding.
\item If $\Ham_{c}(\ms{M}) \neq 0$ then there exists an $\LSL$-cuspidal mirabolic module $\ms{N}$ such that $\Ham_{c}(\ms{N}) \neq 0$.
\end{enumerate}
\end{cor}

\begin{proof}
We endow $\mm$ with the structure of a $(G,c)$-monodromic module. Then,
$$
\Gamma(\X,\mm)^G = \{ m \in \Gamma(\X,\mm)^{\SSL} \ | \ \arr{\mathbf{1}} \cdot m = c \Tr(\mathbf{1}) m \} = \Ham_c(\mm),
$$
and similarly, 
$$
\Ham_{L}(\Upsilon^* \ms{M}) = \{ m \in \Gamma(\X_{\LSL},\Upsilon^* \ms{M})^{N_{\SSL}(\LSL)} \ | \ \arr{\mathbf{1}} \cdot m = c \Tr(\mathbf{1}) m \},
$$
are independent of the choice of lift of $\mm$ to $\ss_c$. Therefore the first claim follows from Proposition \ref{prop:keyfact}. As shown in Theorem \ref{prop:restrictionfunctor}, $\Upsilon^* \ms{M}$ is a mirabolic $\dd$-module (which, as in the proof of Proposition \ref{thm:irrsupport}, we may think of $\Upsilon^* \ms{M}$ as a $\dd$-module on $Z(\LSL)^{\circ} \times \X_{\LL}$, supported on $Z(\LSL)^{\circ} \times \mn^{L} \times V$). Since $\Ham_{c}$ is a left exact functor, $\Ham_{c} (\ms{K}) \neq 0$ implies that there is a simple submodule $\ms{K}'$ of $\ms{K}$ such that $\Ham_{c}(\ms{K}') \neq 0$. Thus, if $\Ham_{L}(\Upsilon^* \ms{M}) \subset \Ham_c(\Upsilon^* \ms{M})$ is nonzero, then there exist a simple submodule $\ms{N}'$ of $\Upsilon^* \ms{M}$ such that $\Ham_{c}(\ms{N}')$ is nonzero. As noted in the proof of Lemma \ref{lem:locallyfactorL}, $\ms{N}'$ is isomorphic to $\mathsf{L} \boxtimes \ms{N}$, where $\ms{N}$ is an $\LSL$-cuspidal mirabolic module. The fact that $\Ham_{c}(\ms{N}) \neq 0$ implies that $\Ham_{c}(\ms{N}) \neq 0$ too.  
\end{proof}  

The above results allow us to prove Proposition \ref{prop:Hamkill}. 

\begin{proof}[Proof of Proposition \ref{prop:Hamkill}]
By Proposition \ref{thm:irrsupport}, the support of $\mm$ is the closure of some relevant stratum $S = \X((m^v,1^w),(m^u))$. Assume that $m \neq 1$, so that the associated Levi subgroup $\LSL$ is not a torus. If $\Ham_{c}(\ms{M}) \neq 0$ then Corollary \ref{cor:keyfact} implies that $\Ham_{c}(\ms{N}) \neq 0$ for some $\LSL$-cuspidal mirabolic module on $\X_{\LL}$. Proposition \ref{prop:cuspidalkill} implies that either $c = \frac{rn}{m} \in \Sing_+$ and $\ms{N}$ must be supported on $\X_{\LL}(\lambda^{\pm}(0;v,w))$ (in which case, $S = \X((m^v,1^w),\emptyset)$), or $c = - \frac{rn}{m} \in \Sing_-$ and $\ms{N}$ must be supported on $\X_{\LL}(\lambda^{\pm}(u;0,w))$ (in which case, $S = \X((1^w),(m^u)$). 

Similarly, if $m = 1$ then Proposition \ref{prop:supportT} implies that the support of $\mm$ must be the closure of $\X(\emptyset,(1^n))$, which equals $\SSL \times \{ 0 \}$ (recall that we have assumed in the statement of Proposition \ref{prop:Hamkill} that $\Supp \mm \neq \X$, so we disregard possibility (3) of Proposition \ref{prop:supportT}). 
\end{proof}

\subsection{Proof of Theorem \ref{main} and Corollary \ref{cor:main}}\label{sec:proofs} Let  $\ms{M}$
be a simple quotient of $\CG_{\lambda,c}$. Such a quotient must be generated by a global section $\mathbf{m}$ such that $\g_{c} \cdot \mathbf{m} = 0$. Hence $\Ham_{c}(\ms{M}) \neq 0$.

 First we consider the case where $\mm$ is supported on the complement of $\X^{\reg}$. 
Then, Proposition \ref{prop:Hamkill} implies that either 
\begin{itemize}
\item $c = \frac{r}{m} \in \Sing_-$ and the support of $\ms{M}$ is the closure of a stratum $\X((m^v,1^w),\emptyset)$, for some $v,w \in \mathbb{N}$ such that $n = v m + w$,
\item $c = \frac{r}{m} \in \Sing_+$ and the support of $\ms{M}$ is the closure of a stratum $\X((1^w),(m^u))$, for some $u,w \in \mathbb{N}$ such that $n = u m + w$.
\end{itemize}
Let $\mc{U}$ be the enveloping algebra of $\mf{sl}$ and denote by $\tau : \mc{U}^{\mathrm{op}} \stackrel{\sim}{\rightarrow} \mc{U}$ the isomorphism defined by $\tau(x) = - x$ for $x \in \g$. Since $\dd(\SSL) = \C[\SSL] \o \mc{U}$, $\tau$ extends to an isomorphism $\tau : \dd(\SSL)^{op} \iso \dd(\SSL)$, $\tau(f \o u) = (1 \o \tau(u)) \cdot (f \o 1)$. Then $\tau$ restricts to an automorphism of $\FZ$. For $\lambda \in \tSL^* / W$, write $\FZ_{\tau(\lambda)} := \tau(\FZ_{\lambda})$. Under the Harish-Chandra homomorphism, $\tau$ corresponds to the map $W \cdot \lambda \mapsto W \cdot (-\lambda-2 \rho)$. By \cite[Proposition 6.2.1]{MirabolicCharacter}, $\mathbb{D}(\CG_{\lambda,c}) \simeq \CG_{\tau(\lambda),-c + 1}$ (the parameter $c'$ of \textit{loc. cit.} is related to our parameter $c$ by $c = -\frac{c'}{n}$). Therefore, if $\CG_{\lambda,c}$ has a submodule supported on the closure of some stratum $\X(\mu,\nu)$, then $\CG_{\tau(\lambda),-c +1}$ will have a quotient supported on the closure of $\X(\mu,\nu)$. Hence if $\ms{N}$ is a simple submodule of $\CG_{\lambda,c}$ that is supported on the complement of $\X^{\reg}$ then either, in case (1), the support of $\ms{N}$ must be the closure of a stratum $\X((1^w),(m^u))$, for some $u,w \in \mathbb{N}$ such that $n = u m + w$ or, in case (2), the support of $\ms{N}$ must be the closure of a stratum $\X((m^v,1^w),\emptyset)$, for some $v,w \in \mathbb{N}$ such that $n = v m + w$. 

Corollary \ref{cor:main} now follows from the fact that $\X(\mu,\nu)$ is contained in $\X^{\mathrm{cyc}}$ if and only if $\nu = \emptyset$. 

\begin{rem} 
The $\dd$-module $\CG_{\la,0}$
has been studied earlier  by Galina and Laurent
\cite{GalinaLaurent2}, in connection with Kirillov's conjecture, c.f. \cite{Ba}. 
In \cite{GalinaLaurent2}, the  authors claim that the $\dd$-module $\CG_{\la,0}$ has no quotients supported on $\X\sminus \X^\reg$. However, the argument on page 17 of \cite{GalinaLaurent2} seems to have a serious gap. Specifically, the three options considered there do not exhaust all the possibilities since there may be other \textit{relevant} strata, see definition
  \ref{distinguished}, that need to be considered.
\end{rem}

\subsection{Finite generation of $\Hamr(E)$}

We expect, but can not prove, that $\Hamr(E)$ is a coherent $\ddd$-module, for any finitely generated $\htrig_\kappa$-module $E$. However, in the rational situation, it is possible to show that $\Hamr(E)$ is a coherent $\ddd$-module for $E$ in category $\mc{O}$. 

Therefore, for the remainder of this section, let $\X = \mf{sl} \times V$, $\g = \mf{gl}(V)$, and let $\eu_{\mf{sl}}$ denote Euler vector field along the factor $\mf{sl}$. Let $\htrig_{\kappa}$ be the spherical subalgebra of the rational Cherednik algebra $\H_{\kappa}(\mf{h},W)$. In the rational situation, $\mc{O}_{\kappa}$ is the full subcategory of $\Lmod{\htrig_{\kappa}}$ consisting of all modules on which the action of $\C[\h^*]^W_+$ is locally nilpotent. The functor of Hamiltonian reduction, $\Ham$, and its left (resp. right) adjoint $\Hamp$ (resp. $\Hamr$) are defined as in the trigonometric case. 

\begin{thm}\label{thm:fg}
For any $M$ in $\mc{O}_{\kappa}$, the $\dd$-module $\Hamr(M)$ is finitely generated. Hence, it is a mirabolic module. 
\end{thm}

A coherent, $G$-monodromic $\dd$-module $\ms{F}$ on $\X$ is said to be mirabolic if the action of $\sym (\mf{sl})_+^{\SSL} \subset \dd(\X)$ on $\Gamma(\X,\ms{F})$ is locally nilpotent and the action of $\eu_{\mf{sl}}$ on $\Gamma(\X,\ms{F})$ is locally finite. The category of all mirabolic modules is denoted $\cc$. Proposition \ref{prop:monodromicfinite} implies that the action of $\g$ on $\Gamma(\X, \ms{F})$ is locally finite.   

\begin{lem}\label{lem:isofd}
For any mirabolic module $\ms{F}$, any $(\g,\eu_{\mf{sl}})$-isotypic component in $\ms{F}$ is finite dimensional. 
\end{lem}

\begin{proof}
Passing to the associated graded, we are reduced to showing that any $(\g,\eu_{\mf{sl}})$-isotypic component of $\C[\Nnil(\mf{sl})]$ is finite dimensional. 
To see this, let ${\mathbb{M}}(\mf{sl})$ be the zero fiber of the moment map
$T^*\X=\mf{sl}\times\mf{sl}\times V\times V^*\to\mf{gl}$.
According to  \cite[Proposition 8.2.1]{GG} one
has an algebra isomorphism
$\C[{\mathbb{M}}(\mf{sl})]^G\cong \C[\h \times \h^*]^W$.
The variety $\Nnil(\mf{sl})$  is a $G$-stable closed subvariety of
${\mathbb{M}}(\mf{sl})$.
Furthermore,  the proof in \cite[Proposition 8.2.1]{GG} shows that
the restriction map $\C[{\mathbb{M}}(\mf{sl})]\to
\C[\Nnil(\mf{sl})]$ induces, via the isomorphism above, an algebra isomorphism
\beq{niliso}
\C[\Nnil(\mf{sl})]^{\g} = \C[\Nnil(\mf{sl})]^{G} = \C[\h \times \h^*]^W / \langle \C[\h^*]_+^W \rangle.
\eeq
The grading on $\C[\Nnil(\mf{sl})]^{\g}$  coming from the
action of the element $\eu_{\mf{sl}}$ goes, under \eqref{niliso}, to the grading on
$\C[\h \times \h^*]^W / \langle \C[\h^*]_+^W \rangle$ by polynomial
degree with respect the $\h$-variable.
It is clear that $\C[\h \times \h^*]^W / \langle \C[\h^*]_+^W \rangle$
is a finitely generated $\C[\h]^W$-module. 
Therefore, any homogeneous component of
$\C[\h \times \h^*]^W / \langle \C[\h^*]_+^W \rangle$ is finite
dimensional. It follows that any $\eu_{\mf{sl}}$-isotypic component component of
$\C[\Nnil(\mf{sl})]^{\g}$  is finite
dimensional as well.

Further, Hilbert's Theorem implies that 
any $\g$-isotypic component of
$\C[\Nnil(\mf{sl})]$ is a finitely generated
$\C[\Nnil(\mf{sl})]^{\g}$-module. We conclude 
that any $\eu_{\mf{sl}}$-isotypic component of an $\g$-isotypic component $\C[\Nnil(\mf{sl})]$ is finite dimensional.
\end{proof}

\begin{rem}\label{rem:comment}
The $(\mf{sl},\eu_{\mf{sl}})$-isotypic components of a mirabolic module $\ms{F}$ are \textit{not} in general finite dimensional.
\end{rem}

Write $\ms{F} = \oplus_{(\lambda,\alpha)} \ms{F}_{(\lambda,\alpha)}$ for the $(\g,\eu_{\mf{sl}})$-isotypic decomposition of a mirabolic module $\ms{F}$. We define the restricted dual of such an $\ms{F}$ by 
$$
\ms{F}^{\star} = \oplus_{(\lambda,\alpha)} \ms{F}_{(\lambda,\alpha)}^*.
$$
It is clear that $\ms{F}^{\star}$ is a right $\dd$-module. Furthermore, Lemma \ref{lem:isofd} implies that the functor $\ms{F} \mto \ms{F}^{\star}$ is exact and one has a canonical isomorphism $\ms{F} \simeq (\ms{F}^{\star})^{\star}$. 

\begin{lem}\label{lem:fgdual}
For any mirabolic module $\ms{F}$, $\ms{F}^{\star}$ is a mirabolic module. In particular, it has finite length and hence is finitely generated over $\dd(\X)$.
\end{lem}

\begin{proof}
First, the adjoint action of $\mu_{\X}(\g)$ is trivial on the algebra $\sym (\mf{sl})^{\SSL}$ and the adjoint action of $\eu_{\mf{sl}}$ on $\sym (\mf{sl})^{\SSL}$ induces the usual $\mathbb{N}$-grading on this algebra. Therefore, if $z \in \sym (\mf{sl})^{\SSL}$ is homogeneous of degree $k$ and $f \in \ms{F}_{(\lambda,\alpha)}$ then there is some $N \gg 0$ such that $z^N \cdot f \in \ms{F}_{(\lambda,\alpha + kN)}$ equals zero. Since the space $\ms{F}_{(\lambda,\alpha)}$ is finite dimensional, we may assume that $N$ is independent of $f$. This implies that $z^N \cdot \ms{F}_{(\lambda,\alpha)}^* = 0$. Thus, the action of $\sym (\mf{sl})^{\SSL}$ on $\ms{F}$ is locally nilpotent. 

Since the action of the pair $(\g,\eu_{\mf{sl}})$ on $\ms{F}^\star$ is locally finite, with finite dimensional $(\g,\eu_{\mf{sl}})$-isotypic components, any submodule $\ms{E}$ of $\ms{F}^\star$ also has this property. Thus, the restricted dual of $\ms{E}$ makes sense and $\ms{E}^\star$ is a quotient of $(\ms{F}^\star)^\star \simeq \ms{F}$. Since $\ms{F}$ is holonomic, it has finite length. This implies that $\ms{F}^\star$ has finite length (take a strictly ascending chain of $\dd$-submodules $\ms{E}_1 \subset \ms{E}_2 \subset \cdots $ of $\ms{F}^\star$, then we get a chain of quotients $\cdots \twoheadrightarrow \ms{E}_2^{\star} \twoheadrightarrow \ms{E}_1^{\star}$ of $\ms{F}$; the length of this chain is bounded by the length of $\mc{F}$). 
\end{proof}

Let $\eu_{\kappa}$ be the Euler element in $\htrig_{\kappa}$. For any $M \in \mc{O}_{\kappa}$, let $M = \bigoplus_{\beta} M_{\beta}$ be the $\eu_{\kappa}$ generalized eigenspace decomposition. The eigenspaces $M_{\beta}$ are finite dimensional. The standard duality functor on $\mc{O}_{\kappa}$ can be defined in the following equivalent ways: 
$$
M^{\star} := \oplus_{\beta} (M_{\beta})^* = (M^*)^{\mathrm{nil}},
$$
where $(M^*)^{\mathrm{nil}}$ denotes the set of functionals in $M^*$ that are locally nilpotent with respect to the (right) action of $\C[\h^*]^W_+$. To avoid the issues involved in making right $\htrig_{\kappa}$-modules into left $\htrig_{\kappa}$-modules (or right $\dd$-modules into left $\dd$-modules), we work with both left and right modules. So $\mc{O}_{\kappa}^{\op}$ will be the category of finitely generated right $\htrig_{\kappa}$-modules for which the action of $\C[\h^*]_+^W$ is locally nilpotent and similarly for $\ms{C}^{\op}$. Thus, by Lemma \ref{lem:fgdual}, we have contravariant equivalences $( -)^{\star} : \mc{O}_{\kappa} \iso \mc{O}_{\kappa}^{\op}$ and $( -)^{\star} : \ms{C} \iso \ms{C}_c^{\op}$. Put $K_r = \g_{c} \dd(\X) \backslash \dd(\X)$ and $K = \dd(\X) / \dd(\X) \g_{c}$ so that $\Hamp(M) = K \o_{\htrig_{\kappa}} M$ and $\Hamp_r(N) = N \o_{\htrig_{\kappa}} K_r$. 

\begin{prop}\label{prop:swapping}
There is an isomorphism of functors:
$$
\Hamr \circ ( - )^{\star} \simeq ( - )^{\star} \circ \Hamp_r : \mc{O}_{\kappa}^{\op} \to \ms{C}_c. 
$$
\end{prop}

\begin{proof}
We begin with the following claim. 

\begin{claim}
Let $N$ be a $(\eu_{\kappa}, \htrig_{\kappa})$-bimodule, finitely generated as a $\htrig_{\kappa}$-module and locally finite with finite dimensional eigenspaces for the adjoint action of $\eu_{\kappa}$. Let $M$ be a finitely generated left $\htrig_{\kappa}$-module that is locally finite with respect to $\eu_{\kappa}$ with finite dimensional eigenspaces. Then, as right $\C[\eu_{\kappa}]$-modules, $\Hom_{\htrig_{\kappa}}(N,M^{\star}) \simeq (N \o_{U_{\kappa}} M)^{\star}$. 
\end{claim}

\begin{proof}
Naive adjunction says that  
$$
\eta : \Hom_{\htrig_{\kappa}}(N,M^*) \iso (N \o_{\htrig_{\kappa}} M)^*,  
$$
where $\eta(\phi)(n \o m) = \phi(n)(m)$. Since $\Hom_{\htrig_{\kappa}}(N,M^{\star})$ is a subspace of $\Hom_{\htrig_{\kappa}}(N,M^*)$ and $(N \o_{U_{\kappa}} M)^{\star}$ a subspace of $(N \o_{\htrig_{\kappa}} M)^*$, it suffices to show that $\eta$ restricts to an isomorphism between these subspaces. Let $I$ be the subspace of $N \o_{\C} M$ such that $N \o_{\htrig_{\kappa}} M = (N \o_{\C} M) / I$. Then each $\eta(\phi)$ is a functional on $N \o_{\C} M$ vanishing on $I$. Let $N = \oplus_{\alpha} N_{\alpha}$ be the $\ad(\eu_{\kappa})$-decomposition of $N$. We choose $\alpha_1, \ds, \alpha_l$ such that $N$ is generated as a right $\htrig_{\kappa}$-module by $N' = N_{\alpha_1} \oplus \cdots \oplus N_{\alpha_l}$. The fact that $\eta(\phi)$ vanishes on $I$ implies that $\eta(\phi) \in (N \o_{U_{\kappa}} M)^{\star}$ iff $\eta(\phi) |_{N' \o_{\C} M_{\beta}} = 0$ for all but finitely many $\beta$. But this holds iff $\phi(N') \subset M^{\star}$ i.e. iff $\phi \in \Hom_{\htrig_{\kappa}}(N, M^{\star})$. 
\end{proof}  

As in the proof of the claim, for any $M \in \mc{O}_{\kappa}^{\op}$, we have  
$$
\eta : \Hom_{\htrig_{\kappa}}(K ,M^*) \iso (K \o_{\htrig_{\kappa}} M)^*.  
$$
Furthermore, $\Hom_{\htrig_{\kappa}}(K ,M^{\star})$ is clearly a subspace in $\Hom_{\htrig_{\kappa}}(K ,M^*)$. Therefore, $\Hamr(M^\star)$ is a subspace of $(K \o_{\htrig_{\kappa}} M)^*$. As a right $\htrig_{\kappa}$-module, $K = \oplus_{\sigma} K^{(\sigma)}$, where the decomposition is into $G$-isotypic components. Arguing as in the proof of Lemma \ref{lem:isofd}, each $K^{(\sigma)}$ is a $(\eu_{\kappa}, \htrig_{\kappa})$-bimodule, finitely generated as a $\htrig_{\kappa}$-module, and locally finite with finite dimensional eigenspaces for the adjoint action of $\eu_{\kappa}$. Therefore, the proposition follows from the claim, since it is clear that $\eta$ restricts to an isomorphism 
$$
\eta : \oplus_{\sigma} \Hom_{\htrig_{\kappa}}(K^{(\sigma)} ,M^*) \iso \oplus_{\sigma} (K^{(\sigma)} \o_{\htrig_{\kappa}} M)^*.  \qedhere
$$  
\end{proof}

There is also another proof of Proposition \ref{prop:swapping}, using the uniqueness of adjunctions.

\begin{proof}
Let $\ddd = \dd(\X)$ and $\mathsf{Ind} \,\cc$ be the category of (not necessarily finitely generated) $\ddd$-modules $F$ such the action of $\sym (\mf{sl})_+^{\SSL}$ on $F$ is locally nilpotent and the action of both $\eu_{\mf{sl}}$ and $\g$ on $F$ is locally finite with the eigenspaces $F_{\lambda,\alpha}$ finite dimensional. Similarly, let $\mathsf{Ind} \,\mc{O}_{\kappa}$ be the category of all (not necessarily finitely generated) $\htrig_{\kappa}$-modules such that the action of $\C[\h^*]_+$ is locally nilpotent and the action of $\eu_{\kappa}$ is locally finite with finite dimensional generalized eigenspaces. We also have the opposite categories $\mathsf{Ind} \,\cc^{\mathrm{op}}$ and $\mathsf{Ind} \,\mc{O}_{\kappa}^{\mathrm{op}}$. 

Then, the restricted dualities $( - )^{\star} : \mathsf{Ind} \,\cc \iso \mathsf{Ind} \,\cc^{\mathrm{op}}$ and $( - )^{\star} : \mathsf{Ind} \,\mc{O}_{\kappa} \iso \mathsf{Ind} \,\mc{O}_{\kappa}^{\mathrm{op}}$ are well-defined. Moreover, the functors $\Ham$ and $\Hamp, \Hamp_r$ are also well-defined in this setting. The same is true of $\Hamr : \mathsf{Ind} \,\mc{O}_{\kappa} \to \mathsf{Ind} \,\cc$. To see this, it suffices to show that the $\eu_{\mf{sl}}$-generalized eigenspaces of $\Hom_{\htrig_{\kappa}}(K^{(\sigma)}_r,E)$ are finite dimensional for all $E \in \mathsf{Ind} \,\mc{O}_{\kappa}$. The left $\htrig_{\kappa}$-module $K^{(\sigma)}_r$ is finitely generated, therefore we fix a finite dimensional, $\ad(\eu)$-stable subspace $K_0$ such that $\htrig_{\kappa} \cdot K_0 = K^{(\sigma)}_r$. Then, restriction defines an embedding $\Hom_{\htrig_{\kappa}}(K^{(\sigma)}_r,E) \hookrightarrow K_0^* \o_{\C} E$. Since $K_0$ is finite dimensional, it is clear that the $\eu_{\mf{sl}}$-generalized weight spaces in $K_0^* \o_{\C} E$ are finite dimensional. 

Now, the proposition follows from the fact that both $( - )^{\star} \circ \Hamp_r \circ (- )^{\star}$ and $\Hamr$ are right adjoints to $\Ham$. 
\end{proof}

Theorem \ref{thm:fg} follows from Lemma \ref{lem:fgdual} and Proposition \ref{prop:swapping}

\section{Hamiltonian reduction and shift functors}\label{ham_shift}

\subsection{}\label{sec:stabilty2} Recall from section \ref{sec:unstable} that the character $\det$ defines a $G$-equivariant structure on the trivial line bundle over $T^* \X$ by $g \cdot (x,t) = (g \cdot x, \det (g)^{-1} t)$ for all $x \in T^*\X, g \in G$ and $t \in \C$. The set of stable points with respect to this line bundle is 
\begin{equation}\label{eq:stabplus}
(T^* \X)^{\mathrm{ss},+} = \{ (g,Y,i,j) \in \SSL \times \sll  \times V \times V^* \ | \ \C \langle g,Y \rangle \cdot i = V \},
\end{equation}
and the set of points that are stable with respect to the inverse $\det^{-1}$ equals
$$
(T^* \X)^{\mathrm{ss},-} = \{ (g,Y,i,j) \in \SSL \times \sll  \times V \times V^* \ | \ j \cdot \C \langle g,Y \rangle = V^* \}.
$$ 
In both cases, a point is stable if and only if it is semi-stable. By definition, the complement $(T^*\X)^{\st,\pm}$ is the unstable locus.

Let $\mu : T^* \X \rightarrow \g^*$ be the moment map for the action
of $G$ on $T^* \X$.

\begin{lem}\label{gwyn} 
For any module $\ms{M} \in \cc$ the following holds:
\vskip2pt

\vi If  $c \le 0$ and $\Ham_{c}(\ms{M}) \neq 0 $ then
$\Char (\ms{M}) \cap (T^* \X)^{\mathrm{ss},+} \neq \emptyset$.
\vskip2pt

\vii If $c > 0$ and $\Ham_{c}(\ms{M}) \neq 0 $ then
$\Char (\ms{M}) \cap (T^* \X)^{\mathrm{ss},-} \neq \emptyset$.
\end{lem}
\begin{proof}
If $\Ham_c(\mm) \neq 0$ then there is some simple sub-quotient $\mm'$ of $\mm$ such that $\Ham_c(\mm') \neq0$. Therefore we may assume that $\mm$ is simple. Clearly $\Ham_c(\mm) \neq 0$ implies that $\mm \in \cc_q$. Thus, by simplicity, we may in fact assume that $\mm \in \ss_q$.

 \vi \en
By Proposition \ref{prop:Hamkill}, $\ms{M}$ must be supported on the closure of $\X((m^v,1^w),\emptyset)$ and hence the closure of the conormal to this stratum is contained in $\Char (\ms{M})$. It follows from \cite[Lemma 4.2.3]{McGertyKZ} that the $T^*_{\X(\lambda,\mu)} \X \subseteq (T^* \X)^{\mathrm{ss},+}$ if and only if $\mu = \emptyset$. Hence, the set $\Char (\ms{M}) \cap (T^* \X)^{\mathrm{ss},+}$ is nonempty.   
\vskip2pt

\vii \en 
By Proposition \ref{prop:Hamkill}, $\ms{M}$ must be supported on the closure of $\X((1^w),(m^u))$. If $(g,Y,i,j)$ is a point of $\mu^{-1}(0) \cap (T^* \X)^{\mathrm{ss},-}$ then \cite[Lemma 2.1.3]{GG} says that $i = 0$. Arguing again as in \cite[Lemma 4.2.3]{McGertyKZ}, one sees that if the intersection $(T^*_{\X(\lambda,\mu)} \X) \cap (T^* \X)^{\mathrm{ss},-}$ is nonempty, then $\lambda = \emptyset$. Therefore, if $w > 0$, the conormal to $\X((1^w),(m^u))$ is contained in the unstable locus, and we need some more information about $\Char (\ms{M})$ to conclude that $\Char (\ms{M}) \cap (T^* \X)^{\mathrm{ss},-} \neq \emptyset$. Let $L$ be the block diagonal Levi subgroup of $G$ consisting of $u$ blocks of size $m$ and $w$ blocks of size one. Then, Propositions \ref{prop:classification1} and \ref{prop:cuspidalkill} imply that the simple subquotients of $\Red^{\SSL}_{\LSL}(\ms{M})$ are of the form $\mathsf{L} \boxtimes \ms{M}(r,u,0,w)$ for some simple local system $\mathsf{L}$ on $Z(\LSL)^{\circ}$ and $r$ coprime to $m$. Recall that 
$$
\ms{M}(r,u,0,w) = \left[ \boxtimes_{i = 1}^w \ms{E}_{\frac{r}{m}} \right] \boxtimes \left[ (\ms{L}_{r} \boxtimes \delta_{V_{m}})^{\boxtimes u} \right]. 
$$
Since $\frac{r}{m} \neq 0$ in $\C/\Z$, the characteristic variety of $\ms{E}_{\frac{r}{m}}$ equals $(x = 0) \cup (y = 0)$ in $T^* \C$. Therefore, the closure of the conormal to $\X_{\LL}(\emptyset,(m^u,1^w))$ in $T^* \X_{\LL}$ is a component of $\Char (\mathsf{L} \boxtimes \ms{M}(r,u,0,w))$. This implies that it must also be a component of $\Red_{\LSL}^{\SSL} (\ms{M})$. Since $\Upsilon$ is non-characteristic for $\ms{M}$, this implies that the closure of the conormal to $\X(\emptyset,(m^u,1^w))$ in $T^* \X$ is a component of $\Char(\ms{M})$, as required. 
\end{proof}

\begin{rem}
The statements of Lemma \ref{gwyn} are false if the inequalities are removed. If we take $c = -\frac{r}{n} - k < 0$, where $(r,n) = 1$, and let $\ms{M}$ be the unique cuspidal mirabolic module whose support equals $\mn \times V$, then $\Ham_{c}(\ms{M}) \neq 0$. However, the characteristic variety of $\ms{M}$ is $\overline{T^*_{\X((n),\emptyset)} \X}$, which is contained in $(T^* \X)^{\mathrm{us},-}$. Similarly, if we take $c = \frac{r}{n} + k > 0$ and let $\ms{N}$ be the unique cuspidal mirabolic module whose support equals $\mn \times \{ 0 \}$, then $\Ham_{c}(\ms{M}) \neq 0$. The characteristic variety of $\ms{M}$ is $\overline{T^*_{\X(\emptyset,(n))} \X}$, which is contained in $(T^* \X)^{\mathrm{us},+}$.
\end{rem}

\subsection{}\label{ham_shift_sec} 

For each $c \in \C$, there exists a sheaf of twisted differential operators $\dd_{\Ps,c}$ on $\Ps := \Ps(V)$. In order to agree with our conventions on $G$-monodromic modules, (\ref{eq:indentify}), we parameterize these twisted differential operators so that, for each $m \in \Z$, the line bundle $\mc{O}_{\mathbb{P}}(m)$ is a $\dd_{\Ps,-\frac{m}{n}}$-module. This implies that the Euler vector field $\mathsf{eu}_V = \sum_{i = 1}^n x_i \pa_i$ acts as the scalar $-nc$ on the global sections $\Gamma(\Ps,\ms{M})$ of any $\dd_{\Ps,c}$-module $\ms{M}$. 

Set $\XX := \SSL \times \Ps(V)= \SSL\times\Ps$ and let $\dd_{c}$ be the sheaf of twisted differential operators on $\XX$ (the twist entirely in the $\Ps$ direction). Recall that we have defined $V^o = V \sminus \{ 0 \}$. Put $\wt X:=\SSL  \times V^o $. We have a natural diagram $\X \stackrel{j}\hookleftarrow \wt X \stackrel{p}\onto X$. Let $\Lmod{(\dd_\X, G, q)}_0$ be the full subcategory of $\Lmod{(\dd_\X, G, q)}$ whose objects have support contained in the subset $\SSL \times \{ 0 \} \sset \SSL \times V$.

Let $\Ga$ be the center of the group $\SSL $. Thus, $\Ga$ is a cyclic group of order $n$, with generator $\zeta:=\exp(-\frac{2\pi \sqrt{-1}}{n})\cdot\id\in \SSL $. The group $\Ga$ acts trivially on $\Ps$ and it also acts trivially on $\SSL$ by conjugation. Thus, the action of $\Ga$ on an $\SSL$-equivariant $\dd_{c}$-module commutes with the $\dd_{c}$-action. This gives a morphism of the group $\Ga$ into the automorphism group of the identity functor of the category $\Lmod{(\dd_{c}, \SSL)}$. Therefore, this category decomposes into a direct sum of subcategories $$\Lmod{(\dd_{c}, \SSL)}\ =\ \bigoplus_{k\in \Z/ n\Z}\ \Lmod{(\dd_{c}, \SSL; \zeta^k)},\quad $$ where the subcategory $\Lmod{(\dd_{c}, \SSL; \zeta^k)}$ consists of $\dd_{c}$-modules $M$ such that $\zeta$ acts on $M$ via multiplication by $\exp(-\frac{2\pi \sqrt{-1} k}{n})$. Notice that $\Lmod{(\dd_{c}, \SSL; \zeta^0)} = \Lmod{(\dd_{c}, \mathrm{PSL})}$, where $\mathrm{PSL} = \mathrm{PSL}_n(\C)$ is the projective special linear group. 

\begin{prop}\label{XX} 
The assignment $\mm\mto \ker(\mathsf{eu}_V + nc; p_\idot j^*\mm)$ gives an exact functor 
$$
F_{c} : \Lmod{(\dd_\X, G, c)}^{\mathrm{hol}} \to \Lmod{(\dd_{c}, \SSL)}^{\mathrm{hol}}.
$$
Furthermore, this functor kills the subcategory $\Lmod{(\dd_\X, G, c)}_0^{\mathrm{hol}}$ and induces an equivalence
$$
\Lmod{(\dd_\X, G, c)}^{\mathrm{hol}} / \Lmod{(\dd_\X, G, c)}_0^{\mathrm{hol}} \ \iso\ \Lmod{(\dd_{c},\mathrm{PSL})}^{\mathrm{hol}}.
$$
\end{prop}

\begin{proof}
Recall that we have set $\BG = \SSL \times \Cs$. Let $a : G \times \X \to \X$, resp. ${\wt a} : \BG \times \X \to \X$, be the action map for $G$, resp. for $\BG$, and write $\rho : \BG \twoheadrightarrow G$ for the quotient map. The kernel of $\rho$ is denoted ${\wt \Gamma}$, a cyclic group of order $n$ generated by $(\zeta,\exp(\frac{2\pi \sqrt{-1}}{n}))$. Then, the fact that ${\wt a} = a \circ \rho$ and $\rho^* \mc{O}_G^{c} \simeq \mc{O}_{\BG}^{c}$ implies that every $(G,c)$-monodromic module on $\X$ is a $(\BG,c)$-monodromic, such that the action of ${\wt \Gamma}$ is trivial. Conversely, let $\mm$ be a $(\BG,c)$-monodromic module such that the action of ${\wt \Gamma}$ on $\mm$ is trivial. By definition, we are given an isomorphism $\phi : \mc{O}_{\BG}^{c} \boxtimes \mm \iso {\wt a}^* \mm$ satisfying the cocycle condition. Then, one can check that the isomorphism 
$$
(\rho_{\idot} \phi)^{{\wt \Gamma}} : (\mc{O}_{\BG}^{c} \boxtimes \mm)^{{\wt \Gamma}} = \mc{O}_{G}^{c} \boxtimes \mm \stackrel{\sim}{\longrightarrow} a^* \mm = ({\wt a}^* \mm)^{{\wt \Gamma}} 
$$
satisfies the cocycle condition too. These rules define an equivalence between the category of $(G,\chi)$-monodromic modules on $\X$ and $(\BG,\chi)$-monodromic modules on $\X$ such that the action of ${\wt \Gamma}$ is trivial. The latter category will be denoted $\Lmod{(\dd_\X,\BG,{\wt \Gamma}, c)}$. Then, $j^*$ is an exact functor $\Lmod{(\dd_\X,\BG,{\wt \Gamma}, c)} \to \Lmod{(\dd_{\wt X},\BG,{\wt \Gamma}, c)}$, whose kernel is $\Lmod{(\dd_\X,\BG,{\wt \Gamma}, c)}_0$. In order for $j^*$ to be essentially surjective, we restrict to holonomic $\dd$-module. Then, given $\mm \in \Lmod{(\dd_{\wt X},\BG,{\wt \Gamma}, c)}^{\mathrm{hol}}$, $\mc{H}^0(j_* \mm) \in \Lmod{(\dd_{\X},\BG,{\wt \Gamma}, c)}^{\mathrm{hol}}$ and $j^* \mc{H}^0(j_* \mm) \simeq \mm$, implying that $j^*$ is essentially surjective. 

Since the action of $\SSL$ and $\Cs$ on ${\wt X}$ commute, the pull-back functor $p^* : \Lmod{(\dd_c,\SSL)}^{\mathrm{hol}} \to \Lmod{(\dd_{\wt X},\BG, c)}^{\mathrm{hol}}$ is an equivalence with quasi-inverse $\mm \mapsto \ker(\mathsf{eu}_V + nc; p_\idot \mm) = (p_{\idot} \mm)^{\Cs}$. If the action of ${\wt \Gamma}$ on $\mm$ is trivial then $\Gamma$ acts trivially on $(p_{\idot} \mm)^{\Cs}$ too. 
\end{proof}

Observe that the image of $\one \in \g$ under the quantum moment map $\mu_{\X}$ is $-\mathsf{eu}_V$. Therefore, given a $G$-monodromic $\dd_\X$-module $\mm$ and $m$ a section of $\Gamma(\X,\mm)$, we have $\mathsf{eu}_V \cdot m = -nc m$ if and only if $\mu_{\X}(\one) \cdot m = c \Tr (\one) m$. This justifies our unusual parameterization of twisted differential operators on $\mathbb{P}(V)$.  

From now on we assume that $c$ is admissible, c.f. Definition \ref{bad}. This assumption on $c$ ensures, thanks to the Beilinson-Bernstein localization theorem, that the functor $\Gamma(\Ps,-)$, of global sections, provides an equivalence between the categories of sheaves of quasi-coherent $\dd_{\Ps,c}$ and of $\Gamma(\Ps,\dd_{\Ps,c})$-modules, respectively. Admissibility also implies, \cite[Lemma 6.2]{GGS}, that  
$$
\dd_{c}(X) = \Gamma(X,\ (p_\idot\dd_{\wt X}/p_\idot\dd_{\wt X}(\E + nc))^{\C^\times} = (\ddd / \ddd  (\E + cn))^{\C^\times}.
$$
Therefore, for any $\ddd$-module $M$, the vector space 
$$
\ker(\E + nc; M):= \{m\in M\mid \E(m)= -nc\cdot m\}
$$
has a natural $\dd_{c}(X)$-module structure. In this case, we also have an equivalence between $\dd_{c}$-modules and $\dd_{c}(X)$-modules. Thus, below we will freely switch between the setting of sheaves of $\dd_{c}$-modules and that of $\dd_{c}(X)$-modules, whichever is more convenient.

A regular holonomic $\dd_{c}$-module $\mm\in \Lmod{(\dd_{c}, \mathrm{PSL})}$ will be called ``mirabolic'' if $\SS(\mm)$ is contained in $(\SSL \times \mathcal{N}) \times T^*\Ps$, where $\mc{N}$ is the nilpotent cone in $\mf{sl}$. Let $\ss_{c}$ be the full subcategory of $\Lmod{(\dd_{c}, \mathrm{PSL})}$ formed by all mirabolic $\dd_{c}$-modules. The reader may notice that we have now used $\ss_{c}$ to denote both the category of $(G,c)$-monodromic, mirabolic modules on $\X$ and the category of mirabolic modules on $X$. If it is not clear from the context which of these two categories is being referred to, we will write $\ss_c(\X)$ to denote the former and $\ss_c(X)$ the latter. Further, let $\ss_c(\X)_0 = \ss_c(\X)  \cap \Lmod{(\dd_\X, G, c)}_0$ be the category of mirabolic $\ddd$-modules supported on the subset $\SSL\times \{0\} \sset \SSL\times V$. Then, the equivalence of Proposition \ref{XX} induces an equivalence
\beq{cccc}
\ss_c(\X) / \ss_c(\X)_0\ \iso\ \ss_{c}(X).
\eeq
The functor of Hamiltonian reduction $\Ham : \ss_{c}(X) \rightarrow \mc{O}_{\kappa}$ is defined to be $\Ham(\mm) = \Gamma(X,\mm)^{\SSL}$. Proposition \ref{XX} implies that
\beq{eq:hamham}
\Ham_{c} = \Ham \circ F_{c},\quad\text{whenever $c$ is admissible.}
\eeq 

\subsection{}\label{sec:bimodules}
For $m \in \Z$, let $\mc{O}(m)$ be the pull-back of the standard line bundle $\mc{O}_{\mathbb{P}}(m)$ under the projection $G \times \mathbb{P} \rightarrow \mathbb{P}$. Tensoring by $\mc{O}(nm)$ defines an equivalence $\ss_{c} \stackrel{\sim}{\longrightarrow} \ss_{c - m}$, $\ms{M} 
\mapsto \ms{M}(nm)$. Define 
$$
{}_{c - m} \dd_{c} = \left[ p_{\idot} \dd_{\wt X}/p_\idot\dd_{\wt X}(\E + cn) \right]^{(nm)},
$$
where, for a $\Cs$-equivariant, quasi-coherent sheaf of $\mc{O}_{\wt X}$-modules $\ms{M}$, $(p_{\idot} \ms{M})^{(i)}$ is defined to be the sheaf of all sections $m$ such that $\lambda \cdot m = \lambda^i m$ for all $\lambda \in \Cs$. By \cite[Lemma 2.2]{GGS}, ${}_{c - m} \dd_{c}$ is a $(\dd_{c - m}, \dd_{c})$-bimodule (we remark that the ``$d$'' of \textit{loc. cit.} equals our parameter ``$c$''). 

\begin{lem}[\cite{GGS}, Lemma 6.7 (1)]\label{lem:6.7}
Assume that $c$ is admissible. For each positive integer $m$ and $\ms{M} \in \ss_{c}$, we have
$$
\mm(nm)= {}_{c - m}\dd_{c} \otimes_{\dd_{c}} \mm
$$
\end{lem}

In section \S\ref{app:shift} below (see Corollary \ref{cor:bimodules}),
 we will introduce natural $(\htrig_{\kappa},\htrig_{\kappa + 1})$,
 resp. $(\htrig_{\kappa+1},\htrig_{\kappa})$-bimodules ${}_{\kappa}
 \mathsf{P}_{\kappa + 1}$, resp. ${}_{\kappa+1} \mathsf{Q}_{\kappa}$,
 for all $\kappa \in \C$. We then define inductively
$$
{}_{\kappa} \mathsf{P}_{\kappa + m} = ({}_{\kappa} \mathsf{P}_{\kappa + 1}) \cdot ({}_{\kappa + 1} \mathsf{P}_{\kappa + m}), \quad {}_{\kappa + m} \mathsf{Q}_{\kappa} = ({}_{\kappa + m} \mathsf{Q}_{\kappa + m -1}) \cdot ({}_{\kappa + 1} \mathsf{Q}_{\kappa}),
$$
where the multiplication is defined inside $\htrig^{\reg}$, making ${}_{\kappa} \mathsf{P}_{\kappa + m}$, resp. ${}_{\kappa+m} \mathsf{Q}_{\kappa}$, into a $(\htrig_{\kappa},\htrig_{\kappa + m})$, resp. $(\htrig_{\kappa+m},\htrig_{\kappa})$-bimodule. Define the shift functor $\sh : \Lmod{\htrig_{\kappa}} \rightarrow \Lmod{\htrig_{\kappa+1}}$ by 
$$
\sh(M) = {}_{\kappa + 1} \mathsf{Q}_{\kappa} \o_{\htrig} M.
$$
Theorem \ref{thm:projbi} implies that $\sh$ that provides, for any $\kappa$ such that both $\kappa$ and $\kappa + 1$ are good, an equivalence $\sh:  \oo_\kappa \iso \oo_{\kappa+1}$. Recall that $\kappa= -c + 1$. We will use the following version of \cite[\S 6.4]{GGS}:

\begin{prop}\label{ggs} 
Assume that $c$ is admissible. 

\vi For any $\mm\in \ss_{c}$ there is a canonical morphism $\psi_{c} : \sh(\Ham(\mm))\to \Ham(\mm(n))$. 

\vii Let $M\in \oo_{\kappa}$ and $\mm:=\Hamp(M)$. Then, the morphism $\psi_{c} : \sh(M) = \sh(\Ham(\mm)) \to \Ham(\mm(n))$ is an isomorphism.
\end{prop}

\begin{proof}
By Lemma \ref{lem:6.7}, one has: $\dis \mm(n)={{}_{c - 1}}\dd_{c} \otimes_{\dd_{c}} \mm$. Taking $\SSL$-invariants, we obtain a natural map
$$
f:\ ({}_{c -1}\dd_{c})^{\SSL} \otimes_{(\dd_{c})^{\SSL}} \mm^{\SSL} \too \big({}_{c - 1}\dd_{c} \otimes_{\dd_{c}} \mm\big)^{\SSL}.
$$
We claim that the map $f$ descends to $[({}_{c-1}\dd_{c})^{\SSL}/({}_{c - 1}\dd_{c}\g_{c})^{\SSL}]\ \otimes_{(\dd_{c})^{\SSL}} \mm^{\SSL}$, a quotient space. To see this,  let $(-)_\sll$ denote  the functor of $\sll$-coinvariants, c.f. \S\ref{sec:qhr}. The
natural projection 
$$
h:\ \big({}_{c-1} \dd_{c} \otimes_{\dd_{c}} \mm\big)^{\SSL}\too \big({}_{c-1}\dd_{c} \otimes_{\dd_{c}} \mm\big)_\sll.
$$
is an isomorphism since the $\sll$-action involved is locally finite. Note that we have an equality $({{}_{c-1}} \dd_{c}\g_{c})^{\SSL} = (\g_{c-1}({{}_{c-1}}\dd_{c}))^{\SSL}$ and that the composite map $h\ccirc f$ clearly kills the subspace $(\g_{c-1}({{}_{c-1}}\dd_{c}))^{\SSL}\ \otimes_{\dd_{c}^{\SSL}} \mm^{\SSL}$. Hence, the map $f$ kills this subspace as well, and our claim follows.

We conclude that the map $f$ descends to a well defined map
$$
\psi_{c} : \ ({{}_{c-1}}\dd_{c}/{{}_{c-1}}\dd_{c}\g_{c})^{\SSL}\otimes_{(\dd_{c}/\dd_{c}\g_{c})^{\SSL}}
\mm^{\SSL}\too \big({{}_{c-1}}\dd_{c} \otimes_{\dd_{c}} \mm\big)^{\SSL}.
$$
Further, we have $(\dd_{c}/\dd_{c}\g_{c})^{\SSL}= \htrig_{\kappa}$ and also ${{}_{\kappa+1}{\mathsf{Q}}_{\kappa}} = ({{}_{c-1}} \dd_{c} / {{}_{c-1}}\dd_{c}\g_{c})^{\SSL}$, see part (ii) of \cite[Lemma 6.7]{GGS}. Thus, the map $\psi_{c}$ takes the following form
$$
\psi_{c} : \ \sh(\Ham(\mm)) = {{}_{\kappa+1}{\mathsf{Q}}_{\kappa}} \o_{\htrig_{\kappa}} \mm^{\SSL} \too \Ham(\mm(n)).
$$
This proves (i). The statement of part (ii) is \cite[Theorem 6.5]{GGS}.
\end{proof}

\subsection{} Recall that we defined two notions of stability in section \ref{sec:stabilty2}, one with respect to $\det$ and the other with respect to $\det^{-1}$. By analogy with the description of $(T^* \X)^{\mathrm{ss},+}$ given in (\ref{eq:stabplus}), we say that a point of $T^* X$ is \textit{semi-stable} if it is contained in 
$$
(T^* X)^{\mathrm{ss}} := \{ (g,Y,\ell,j) \in \SSL \times \sll  \times T^* \mathbb{P} \ | \ \C \langle g,Y \rangle \cdot \ell = V \}.
$$
The set of unstable points in $X$ is denoted $(T^* X)^{\st}$. For each $m \in \Z$ we denote by $\mc{O}(m)$ the line bundle on $T^* X$ obtained by pulling back $\mc{O}_{\mathbb{P}}(m)$ along the projection $T^* X \rightarrow X \rightarrow \mathbb{P}$. 

\begin{lem}\label{lem:GIT}
Let $\mc{F}$ be a $\SSL$-equivariant, coherent $\mc{O}_{T^* X}$-module. Then, $\Supp \mc{F} \cap (T^* X)^{\mathrm{ss}} \neq \emptyset$ if and only if $\Gamma(T^* X, \mc{O}(mn) \o \mc{F})^{\SSL} \neq 0$ for $m \gg 0$. 
\end{lem}

\begin{proof}
Let $Y = \mu_{\Cs}^{-1}(0) \subset T^* \X$, where $\mu_{\Cs}$ is the moment map for the Hamiltonian $\Cs$-action on $T^* \X$. Recall that $\BG = \SSL \times \Cs$. Let $\det : \Cs \rightarrow \Cs$ be the character $\lambda \mapsto \lambda^n$ and denote by the same symbol the corresponding character of $\BG$. Then $Y^{\mathrm{ss},\BG} \subseteq Y^{\mathrm{ss},\Cs}$, where stability is with respect to $\det$. The group $\Cs$ acts freely on $Y^{\mathrm{ss},\Cs}$ and the quotient space is $T^* X$. Therefore, in the notation of \cite[Proposition 7.4]{GGS}, there exists a coherent, $\BG$-equivariant sheaf $\mc{G}$ on $Y$ such that 
$$
\mc{F} = \mathbb{F}\left( \bigoplus_{m \ge 0} \Gamma(Y, \mc{G})^{(\Cs,\det^{-m})} \right) = \mathbb{F} \left( \bigoplus_{m \ge 0} \Gamma(T^* X, \mc{O}(nm) \o \mc{F}) \right).
$$
Thus, for $m \gg 0$, we have $\Gamma(Y, \mc{G})^{(\Cs,\det^{-m})} = \Gamma(T^* X, \mc{O}(nm) \o \mc{F})$ and hence 
$$
\Gamma(Y, \mc{G})^{(\BG,\det^{-m})} = \Gamma(T^* X, \mc{O}(nm) \o \mc{F})^{\SSL}.
$$
By \cite[Proposition 7.4 (2)]{GGS}, $\Gamma(Y, \mc{G})^{(\BG,\det^{-m})} \neq 0$ for $m \gg 0$ if and only if $\Supp \mc{G} \cap Y^{\mathrm{ss},G} \neq \emptyset$. Since $(T^* X)^{\mathrm{ss}}$ is just the image of $Y^{\mathrm{ss},\BG}$ in $T^* X$, and the coherent sheaf on $Y^{\mathrm{ss},\BG} / \BG$ corresponding by descent to $\mc{G}$ is also the sheaf corresponding by descent to $\mc{F}$, we conclude that $\Supp \mc{F} \cap (T^* X)^{\mathrm{ss}} \neq \emptyset$ if and only if $\Gamma(T^* X, \mc{O}(m) \o \mc{F})^{\SSL} \neq 0$ for $m \gg 0$.
\end{proof}

\begin{lem}\label{twist} 
Given $\ms{M}\in \ss_{c}$, there exists an integer $\ell(\mm) \gg 0$ such that, for all $k \geq \ell(\mm)$ one has

\vi The canonical morphism
$\psi_{c - k}: \sh(\Ham(\mm(kn)) \to \Ham(\mm((k+1)n))$
is surjective.

\vii The set $\SS(\ms{M})$  contains semistable points if and only if  $\Ham(\mm(kn))\neq 0$.
\end{lem}

\begin{proof} 
 Let $A^0 = \C[\TSL \times \mf{t}^*_n]^{W}$ be the invariants, resp. $A^1 \sset \C[\TSL \times \mf{t}^*_n]$ the $\mathsf{sign}$-isotypic component, of the diagonal $W$-action. For each $k \geq 1$, we let $A^k \sset \C[\TSL \oplus \mf{t}^*_n]$ be the $\C$-linear span of the set of elements of the form $a_1\cdot a_2\cdot\ldots\cdot a_k$, for $a_i\in A^1$. Thus, $\wh A:=\oplus_{k\geq 0}\ A^k$ is a commutative, graded algebra.

Since $\mm$ is a regular holonomic, $\SSL$-equivariant $\dd$-module on $T^* X$, we may choose, by \cite[Corollary 5.1.11]{KashiwaraKawai}, a good, $\SSL$-stable filtration on $\mm$ such that the associated graded sheaf is reduced. This gives a good, reduced filtration on $\mm(kn)$ for each $k \geq 0$. Let $\gr(\mm(kn))$ denote the associated graded space. Then, we have 
$$
(\gr\mm(kn))^{\SSL} = \gr(\mm(kn)^{\SSL}),
$$
and, if $\mu_X : T^* X \rightarrow \sll^*$ is the moment map, then
$\gr\mm(kn)$ is supported on $\mu_X^{-1}(0)$ and, the filtration being reduced, it
is in fact a coherent $\mc{O}_{\mu^{-1}_X(0)}$-module. If $Y = \mu^{-1}_{\Cs}(0)$ as in the proof of Lemma \ref{lem:GIT} and $p : Y^{\mathrm{ss},\Cs} \rightarrow T^* X$ the quotient map, then the closure of $p^{-1} (\mu_X^{-1}(0))$ in $T^* \X$ is contained in $\mu_{\X}^{-1}(0)$ since the action of $\SSL \times \Cs$ on $T^* \X$ factors through the action of $G$. Thus, the direct sum $\oplus_{k\geq 0}\ (\gr\mm(kn))$ is a finitely generated module for the graded algebra $\oplus_{k \geq 0}\ \C[\mu_{\X}^{-1}(0)]^{(\Cs,\det^{-k})}$. Since $\SSL$ is reductive, Hilbert's Theorem implies that the space $\oplus_{k\geq 0}\ (\gr\mm(kn))^{\SSL}$ is a finitely generated module for the algebra  
$$
\left( \bigoplus_{k \geq 0}\ \C[\mu_{\X}^{-1}(0)]^{(\Cs,\det^{-k})} \right)^{\SSL} = \bigoplus_{k \geq 0}\ \C[\mu_{\X}^{-1}(0)]^{(G,\det^{-k})} \simeq \wh A,
$$
where the isomorphism of the right hand side is given by \cite[Proposition A2]{GG}. It follows, since the algebra $\wh A$ is generated over $A^0$ by its degree one component $A^1$, that for all $k\gg0$, one has an equality $(\gr\mm((k+1)n))^{\SSL} = A^1 \cdot(\gr\mm(kn))^{\SSL}$.

One has an isomorphism $\gr({{}_{\kappa+1}{\mathsf{Q}}_{\kappa}}) = A$, see \cite[Theorem 5.3]{GGS} and also \cite[Lemma 6.9(2)]{GordonStaffordI}.
We deduce that the map $\gr({{}_{\kappa+1}{\mathsf{Q}}_{\kappa}})\o
\gr(\mm(kn)^{\SSL})\to \gr(\mm((k+1)n)^{\SSL})$ is surjective.
This implies that
 the map ${{}_{\kappa+1}{\mathsf{Q}}_{\kappa}}\o
(\mm(kn))^{\SSL}\to (\mm((k+1)n))^{\SSL}$
is surjective, proving (i).

To prove (ii), we apply Lemma \ref{lem:GIT} to $\gr\mm$, a coherent sheaf on $T^* X$. Specifically, Lemma \ref{lem:GIT} says that the set $\supp(\gr\mm) = \SS(\mm)$ contains semistable points if and only if $\Gamma(T^*\X, \oo(nk)\otimes\gr\mm)^{\SSL} \neq 0$ for all sufficiently large positive integers $k$. The latter holds if and only $\Gamma(\X, \oo(nk)\otimes \mm)^{\SSL}\neq0$.
\end{proof}

\subsection{}

The following is an analogue of Theorem \ref{cor:unstable} for mirabolic $\dd$-modules on $X$. 

\begin{thm}\label{cor:unstable2} 
Assume $c \in \C$ is admissible and let $\ms{M} \in \ss_c$. Then, $\Char (\ms{M}) \subset (T^* X)^{\st}$ if and only if $\Ham(\ms{M}) = 0$. 
\end{thm}

\begin{proof}
The functor $\Ham$ being exact, it suffices to prove the result for simple modules. Thus, let $\mm\in\ss_{c}$ be a simple module.
If the set $\SS(\ms{M})$ does not  contain semistable points then $\Ham(\mm)=0$ by the previous lemma.

So, assume that the set $\SS(\ms{M})$ contains semistable points. We choose and fix an integer $\ell=\ell(\mm)$ as in the statement of Lemma \ref{twist}. For any  $k\geq
\ell$, by part (ii) of the lemma
we have that $\Ham(\mm(k n))$
is a (nonzero) simple object. Furthermore, writing
$M_\ell:=\Ham(\mm(\ell n))$, we deduce from part (i)
of  Lemma \ref{twist} that 
the canonical map
$\psi^{k-\ell}: \sh^{k-\ell}(M_\ell)\onto \Ham(\mm(k n))$
is surjective for any $k\geq \ell$. 
Note that the object $\sh^{k-\ell}(M_\ell)$ is simple
since the functor $\sh^{k-\ell}$ is  an
equivalence. It follows that the above map
$\psi^{k-\ell}$ is in fact an isomorphism.
Thus, from  Lemma \ref{top}, we conclude that $\mm(k n)=[\sh^{k-\ell}(M_\ell)]_{!*}$ is the simple top of the object $\Hamp(\sh^{k-\ell}(M_\ell))$.

Observe next that the functor $\sh^\ell$ is also an equivalence. It follows that there is a unique simple object $M\in\oo_{\kappa}$ such that one has $M_\ell=\sh^\ell(M)$. Then, using Proposition \ref{ggs}, we find 
$$
\Ham(\mm(k n))=\sh^{k-\ell}(M_\ell) = \sh^k(M) = \Ham(\Hamp(M)(k n)),\quad k\geq \ell.
$$

We consider the following composition
$$u:\ \Hamp(\Ham(\mm(k n)))\too\Hamp(\Ham(\Hamp(M)(k n)))\too \Hamp(M)(k
n),$$
where the first map is obtained by
applying 
$\Hamp(-)$
to the composite isomorphism above and  the second map
is the canonical adjunction.
Note that, the morphism
$\Ham(u)$ induced by $u$ is, by construction,
 the identity map on $\sh^k(M)$. Let $\mm'=M_{!*}$ be the
simple top of $\Hamp(M)$. Hence, $\mm'(k n)$ is  the simple top of
$\Hamp(M)(k n)$. Note that since $\Ham(\mm')=M\neq 0$, the set
$\SS(\mm')$   contains semistable points, by Lemma
\ref{gwyn}. Therefore, for $k\gg0$ by Lemma \ref{twist}(ii) one has
$\Ham(\mm'(k n))\neq 0$. Thus, we choose (as we may) $k\geq\ell$
such that $\Ham(\mm'(k n))\neq 0$.
Recall that the object $\Ham(\Hamp(M)(k n))=\sh^k(M)$ is simple. Therefore,
the $\dd$-module $\Hamp(M)(k n)$ contains a single
composition factor that has a nonzero Hamiltonian reduction.
But, $\mm'(k n)$ is a simple  quotient of $\Hamp(M)(k n)$
that has this property. Thus, we deduce that 
$\Ham(\mm'(k n))=\Ham(\Hamp(M)(k n))=\sh^k(M)$.

To complete the proof, we consider a diagram
$$
\xymatrix{
\mm(k n)\ &
\ \Hamp(\Ham(\mm(k n)))\  \ar[r]^<>(0.5){u}\ar@{->>}[l]_<>(0.5){pr}&\
\  \Hamp(M)( k n)\ \ar@{->>}[r]^<>(0.5){pr'}&\
\mm'( k n),
}
$$
where the map $pr$, resp. $pr'$, is the projection of
$\Hamp(\Ham(\mm(k n)))$, resp. $\Hamp(M)( k n)$,
onto its simple top. 
Our construction implies that 
$\ker(pr)$ is the largest submodule of $\Hamp(\Ham(\mm(k n)))$
annihilated by the functor $\Ham$, resp. $\ker(pr')$ is the largest
submodule of $\Hamp(M)( k n)$
annihilated by the functor $\Ham$. 
It follows that the morphism $u$ in the above diagram 
maps $\ker(pr)$ into  $\ker(pr')$. Hence,  the morphism
descends to a map
$$\bar u:\ \mm(k n)\cong \Hamp(\Ham(\mm(k n)))/\ker(pr)\ \to\
\mm'( k n)\cong \Hamp(M)( k n)/\ker(pr').
$$

Note that the induced map $\Ham(\bar u): \Ham(\mm(k n))\to
\Ham(\mm'( k n))$ may be identified with the
composite isomorphism
$\Ham(\mm(k n))=\sh^k(M)=\Ham(\Hamp(M)( k n))$.
In particular, $\Ham(\bar u)$ is a nonzero map.
We conclude that $\bar u$ is a nonzero map between
two simple objects. Thus, this map is an isomorphism.
It follows that $\mm\cong\mm'$.

We deduce that $\Ham(\mm)=\Ham(\mm')= \Ham(M_{!*})=M\neq0$. Furthermore, we obtain a chain of isomorphisms $\sh^\ell(\Ham(\mm))=\sh^\ell(\Ham(\mm'))=\sh^\ell(M)= \Ham(\mm(\ell n))$. Using this and the fact that $\sh$ is an equivalence, one deduces by induction on $i$ that the canonical morphism $\sh^{\ell-i}(\Ham(\mm))\to \Ham(\mm(i n))$
must be an isomorphism for any $i=0,1,\ldots,\ell-1$.
This completes the proof.
\end{proof}

\subsection{}\label{twist_pf}  {\em Proof of Theorem \ref{cor:unstable}.}\en
As in the proof of Theorem \ref{cor:unstable2}, we may assume that $\mm \in \cc_{q}$ is simple. Hence $\mm$ belongs to the subcategory $\ss_q$. We lift $\mm$ to $\ss_c$. Clearly the characteristic variety of this lift is independent of any choice and, as shown in the proof of Corollary \ref{cor:keyfact}, so too is $\Ham_c(\mm)$. If $\mm \in (\ss_{c})_0$ then Proposition \ref{thm:irrsupport} implies that the support of $\mm$ is the closure of $\X(\emptyset,(m^u))$ for some $n = um$, hence $\Char(\mm) \subset (T^* \X)^{\st,+}$. Then Proposition \ref{prop:Hamkill}, together with our assumption that $c$ is admissible, implies that $\Ham_{c}(\mm) = 0$. Therefore we may assume that $F_{c}(\mm)$, the image of $\mm$ in $\ss_{c} / (\ss_{c})_0 \simeq \ss_{c}$ is nonzero. By Theorem \ref{cor:unstable2} and equation (\ref{eq:hamham}) it suffices to show that 
$$
\Char(\mm) \cap (T^* \X)^{\mathrm{ss},+} \neq \emptyset \ \Longleftrightarrow \ \Char(F_{c}(\mm)) \cap (T^* X)^{\mathrm{ss}} \neq \emptyset.
$$
The fact that $\Char(F_{c}(\mm)) \cap (T^* X)^{\mathrm{ss}} \neq \emptyset$ implies that $\Char(\mm) \cap (T^* \X)^{\mathrm{ss},+} \neq \emptyset$ follows from the fact that $(T^* \X)^{\mathrm{ss},+} \subset T^* \widetilde{X}$ and $\Char(F_{c}(\mm)) = (T^*_{\Cs} \widetilde{X} \cap \Char(\mm)) / \Cs$. 

On the other hand, if $\Char(\mm) \cap (T^* \X)^{\mathrm{ss},+} \neq \emptyset$ then, as noted in the proof of Lemma \ref{gwyn}, we can find some relevant stratum $\X(\lambda,\emptyset)$ such that the conormal $T^*_{\X(\lambda,\emptyset)} \X$ to this stratum has non-empty intersection with $(T^* \X)^{\mathrm{ss},+}$; in fact, the conormal is contained in $(T^* \X)^{\mathrm{ss},+}$. The intersection $T^*_{\X(\lambda,\emptyset)} \X \cap T^*_{\Cs} \widetilde{X}$ is dense in $T^*_{\Cs} \widetilde{X}$. Thus, $\Char(F_{d}(\mm)) \cap (T^* X)^{\mathrm{ss}} \neq \emptyset$ as required. 
\qed

\begin{proof}[Proof of Theorem \ref{twistshift}]
We are going to show that the canonical morphism $\psi_{c}:\sh(\Ham(\mm))\to \Ham(\mm(n))$
is an isomorphism for any $\mm\in\ss_{c}$. The functors
$\sh,\ \Ham$, and also the twist functor $(-)(n)$, all being exact,
it suffices to prove the result for simple mirabolic modules.
Thus, let $\mm\in\cc_{c}$ be a simple module.

Assume first that  the set
 $\SS(\ms{M})$ does not contain semistable points.
Since $\SS(\mm(n))=\SS(\mm)$, we deduce from  Lemma \ref{gwyn}
that $\Ham(\mm)=0$ and  $\Ham(\mm(n))=0$. Therefore, the map $\psi_c$
is, in this case, a map between two zero vector spaces
and we are done.

Assume now that  the set
 $\SS(\ms{M})$  contains semistable points.
Since $\SS(\mm(kn))=\SS(\mm)$, we deduce that
$\Ham( \mm(kn))\neq 0$, by Theorem \ref{cor:unstable}.
Thus, $\Ham( \mm)$ and $\Ham(\mm(n))$ are nonzero simple objects.
Therefore, to complete the proof it suffices to show that the 
 canonical
 map $\psi_c : \sh(\Ham(\mm))\to \Ham(\mm(n))$ is nonzero.

Choose an integer $\ell=\ell(\mm)$ as in Lemma \ref{twist}
and
use the notation of the proof of Theorem \ref{cor:unstable}.
In particular, by definition, we  have $\sh^\ell(\Ham(\mm'))=M_\ell=
\Ham(\mm(\ell n))$. Furthermore, we have shown in the course of the proof that,
in fact, one has an isomorphism $\mm'\cong\mm$.
Therefore, we get an isomorphism $\phi: \sh^\ell(\Ham(\mm))
\iso \Ham(\mm(\ell n))$. On the other hand, going through
the construction of the isomorphism  $\mm'\cong\mm$, 
one sees that the map $\phi$ may be factored as a composition
of the following chain of maps 
\begin{multline*}
\sh^\ell(\Ham(\mm))=
\xymatrix{
\sh^{\ell-1}(\sh(\Ham(\mm)))\ar[rr]^<>(0.5){\sh^{\ell-1}(\psi_c)}&&
\sh^{\ell-1}(\Ham(\mm(n)))=\sh^{\ell-2}(\sh(\Ham(\mm(n))))}\\
\xymatrix{\ar[rr]^<>(0.5){\sh^{\ell-2}(\psi_{c-1})}&&
\sh^{\ell-2}(\Ham(\mm(2n)))=\sh^{\ell-3}(\sh(\Ham(\mm(3n))))}\\
\xymatrix{
\ar[rr]^<>(0.5){\sh^{\ell-3}(\psi_{c-2})}&&\ldots\ar[r]&
\ \sh(\Ham(\mm(\ell-1)n))\ar[rr]^<>(0.5){\psi_{d-(\ell-1)}}&&
\Ham(\mm(\ell n))}.
\end{multline*}

We deduce that $\sh^{\ell-1}(\psi_c)$, the first map in the
chain, is nonzero. Hence, the map $\psi_c$ is itself nonzero, and we are done.
\end{proof}

\subsection{The KZ-functor}
Let $\varpi:\ \X=\SSL\times V\to \SSL\to\SSL/\!/\Ad\SSL\iso \TSL/W$
be a composition of the first projection, the adjoint quotient morphism,
and the Chevalley isomorphism.
We have $\varpi\inv(\TSL^{\reg}/W)=\X^{\reg}$
and the map $\varpi$ restricts to a $G$-torsor $\X^{\reg} \to \TSL^{\reg}/W$.
Thus, there is a short exact sequence
\beq{pi1}
1\to \pi_1(G)\to \pi_1(\X^{\reg}) \to \pi_1(\TSL^{\reg}/W)\to 1,
\eeq
of fundamental groups. A construction given in \cite[Section 4.4]{BFG}
produces
(although it was not  explicitly  stated  in {\em loc. cit.}
in this way), for each $c\in\C$,
a flat
connection on the sheaf $\oo_{\X^{\reg}}$ such that
the monodromy action of the canonical generator of
the group $\pi_1(G)=\Z$ is given by multiplication
by $q=\exp(2\pi\sqrt{-1} c)$. The connection gives
the sheaf $\oo_{\X^{\reg}}$ the structure of a 
$(G,q)$-monodromic $\dd$-module, to be denoted
$\oo_{\X^{\reg}}^c$. 
It follows that the functor
$\mm\mto [\varpi_\idot(\oo_{\X^{\reg}}^{-c}\o \mm)]^G$
gives an equivalence of the categories of
 $(G,q)$-monodromic $\dd_{\X^{\reg}}$-modules
and $\dd_{\TSL^{\reg}/W}$-modules, respectively.
Observe further that we have
\[\Gamma\big(\TSL^{\reg}/W,\ [\varpi_\idot(\oo_{\X^{\reg}}^{-c}\o \mm)]^G\big)=
\Gamma(\X^{\reg},\mm)^{\g_c}.\]
The variety $\TSL^{\reg}/W$ being affine,
we conclude that  the assignment $\mm\mto \Gamma(\X^{\reg},\mm)^{\g_c}$
gives an equivalence of the categories of
 $(G,q)$-monodromic $\dd_{\X^{\reg}}$-modules
and $\dd(\TSL^{\reg}/W)$-modules, respectively.

Next,
let $\delta$ denote the Weyl denominator, c.f. the Appendix, 
so $\delta^2 \in \C[\TSL]^W$. There is
a natural isomorphism
$\htrig_{\kappa}[\mbox{$\frac{1}{\delta^2}$}] \simeq \dd(\TSL^{\reg}/W)$.
Hence, for any $\htrig_{\kappa}$-module $M$ one can view
 $M[\mbox{$\frac{1}{\delta^2}$}]$, a localization of $M$, as a $\dd(\TSL^{\reg}/W)$-module. 

Thus, associated with a mirabolic module
$\mm\in\ss_q$, there are a pair of $\dd(\TSL^{\reg}/W)$-modules
defined as follows
\[H_{\operatorname{loc}}(\mm):=\big(\Ham_c(\mm)\big)[\mbox{$\frac{1}{\delta^2}$}]=
\big(\Gamma(\X,\
\mm)^{\g_c}\big)[\mbox{$\frac{1}{\delta^2}$}]
\quad\text{and}\quad H(\mm_{_{\operatorname{loc}}}):=
\Gamma(\X^{\reg},\ \mm|_{\X^{\reg}})^{\g_c}.
\]
This gives a pair of exact functors $\ss_q\to\Lmod{\dd(\TSL^{\reg}/W)}$. Furthermore,  there is a canonical morphism $H_{\operatorname{loc}}(\mm)\to
H(\mm_{_{\operatorname{loc}}})$.

\begin{prop}\label{prop:isof}
Assume that $c$ is admissible. Then, for any $\mm\in \ss_q$, 
 the canonical morphism $H_{\operatorname{loc}}(\mm)\to
H(\mm_{_{\operatorname{loc}}})$
is an isomorphism.
\end{prop}

\begin{proof} 
Both functors are exact.
 Therefore, by induction on length it suffices to 
prove the isomorphism of the proposition
when $\mm \in \ss_c$ is simple. If $\mm$ is supported on $\X \sminus
\X^{\reg}$ then 
$H(\mm_{\operatorname{loc}})=0$ and we must show that $H_{\operatorname{loc}}(\mm) = 0$. This is equivalent to showing that $\mm$ is supported on $\X \sminus \pi^{-1}( \TSL^{\reg}/W)$. The relevent strata contained in the locally closed set $\pi^{-1}(\TSL^{\reg}/W) \sminus \X^{\reg}$ are precisely those of the form $\X(1^m,1^{n-m})$ with $m < n$. Since $c$ is assumed to be admissible (and in particular not in $\Z$) Proposition \ref{thm:irrsupport} says that $\mm$ cannot be supported on the closure of any of the strata $\X(1^m,1^{n-m})$. Hence $\H_{\operatorname{loc}}(\mm) = 0$. 

Therefore, we are left with showing that if $\Supp \mm = \X$, then 
$H_{\operatorname{loc}}(\mm) \iso H(\mm_{\operatorname{loc}})$. Since $\mm$ is assumed to be simple, $\mm
|_{\X^{\reg}}$ is a simple local system and $\Ham_c$, being a quotient
functor, sends $\mm |_{\X^{\reg}}$ to a simple local system on
$\TSL^{\reg}/W$. Similarly, $\mm\mto
H_{\operatorname{loc}}(\mm)=\Ham_c(-)[\mbox{$\frac{1}{\delta^2}$}]$ is also a quotient functor,
sending simple modules to simple modules. Hence it suffices to show that
the map $H_{\operatorname{loc}}(\mm) \rightarrow H(\mm_{\operatorname{loc}})$ is non-zero. Theorem
\ref{cor:unstable} says that $\Ham_c(\mm)$ is non-zero. The simplicity
of $\mm$ implies that any non-zero section of $\Ham_c(\mm)$ is supported
on the whole of $\TSL/W$, hence $\Ham_c(\mm)$ embedds in $H_{\operatorname{loc}}(\mm)$. By
the same argument, $\Ham_c(\mm)$ embeds in $H(\mm_{\operatorname{loc}})$.
Then the map $H_{\operatorname{loc}}(\mm) \rightarrow H(\mm_{\operatorname{loc}})$ is just the localization
of the identity on $\Ham_c(\mm)$.  As we have explained, $\Ham_c(\mm)\neq
0$, hence, the localized map is nonzero.
\end{proof}

Next, write $\mathsf{DR}(\mathcal F)$ for the local system
associated with a {\em lisse} $\dd$-module $\mathcal F$.
For any $M\in \mc{O}_{\kappa}$, the $\dd(\TSL^{\reg}/W)$-module
$M[\mbox{$\frac{1}{\delta^2}$}]$ gives a local system 
$\mathsf{DR}(M[\mbox{$\frac{1}{\delta^2}$}])$, on $\TSL^{\reg}/W$.
Let $H_q^{\mathrm{aff}}(\s_n)$ denote the affine Hecke algebra of type
$A$ at parameter $q = \exp (2 \pi \sqrt{-1} c)$. The algebra
$H_q^{\mathrm{aff}}(\s_n)$ is a quotient of the group algebra of the
fundemental group  $\pi_1(\TSL^{\reg}/W)$.
It is shown in \cite[Theorem
4.1 (ii)]{VVQuantumAffineSchur} that for $M$ as above the 
monodromy action of $\pi_1(\TSL^{\reg}/W)$ on 
$\mathsf{DR}(M[\mbox{$\frac{1}{\delta^2}$}])$
factors through $H_q^{\mathrm{aff}}(\s_n)$. Thus, $\mathsf{DR}( - [\mbox{$\frac{1}{\delta^2}$}])$ gives a functor
$\mathsf{KZ}_{\kappa} : \mc{O}_{\kappa} \rightarrow
\Lmod{H_q^{\mathrm{aff}}(\s_n)}$,
 the trigonometric
KZ-functor.

On the other hand,
the discussion preceeding Proposition \ref{prop:isof} shows
that, for any $\mm\in \ss_q$,
 the monodromy action of $\pi_1(\X^{\reg})$ on $\mathsf{DR}(\oo^{-c})\o \mathsf{DR}(\mm|_{\X^{\reg}})$
factors through $\pi_1(\TSL^{\reg}/W)$,
c.f. \eqref{pi1}.
Then,
Proposition \ref{prop:isof} yields the following result

\begin{cor} 
Assume that $c$ is admissible and let $\mm\in\ss_q$. Then, the 
 monodromy representations
 of $\pi_1(\TSL^{\reg}/W)$ on $\mathsf{DR}(\oo^{-c})\o \mathsf{DR}(\mm|_{\X^{\reg}})$
is isomorphic to $\mathsf{KZ}_{\kappa}(\Ham_c(\mm)[\mbox{$\frac{1}{\delta^2}$}])$;
in particular, the $\pi_1(\TSL^{\reg}/W)$-action on
$\mathsf{DR}(\oo^{-c})\o \mathsf{DR}(\mm|_{\X^{\reg}})$
 factors through the affine Hecke algebra.
\end{cor}

%\begin{rem}
%In section \ref{sec:nearby} we will define a specialization functor $\sp_L$ for mirabolic modules. Theorem \ref{thm:admissiblecommute} implies that when $c$ is admissible, we have $\mathsf{KZ}_G \simeq \Ham_T \circ \ \sp_T$. 
%\end{rem}

\section{Specialization for Cherednik algebras}\label{jj}

In this section, we define a  Verdier type specialization functor 
on category $\mc{O}$ for the (sheaf of) Cherednik algebras introduced
by Etingof \cite{ChereSheaf}.

\subsection{Specialization}\label{sec:specialA}

Let $X$ be a smooth variety and $Y \hookrightarrow X$  a smooth, locally
closed subvariety.
We write  $\rho : \NN_{X/Y} \to Y$ for
 the normal bundle to $Y$ in $X$ and
 $\mc{I}$ for the sheaf of ideals in $\mc{O}_X$ defining $Y$. There is a
 canonical
graded algebra isomorphism $\rho_*\oo_{\NN_{X/Y}}\cong \oplus_{k\geq 0}\
\mc{I}^k/\mc{I}^{k+1}$.

One   defines an
ascending $\Z$-filtration on $\dd_X$,
to be called the $Y$-filtration, by setting 
\beq{eq:Vfiltr}
V^Y_k \dd_X = \{ D \in \dd_X \ | \ D( \mc{I}^r ) \subset \mc{I}^{r-k}, \ \forall \ r \in \Z \},
\eeq
where $\mc{I}^r = \mc{O}_X$ for $r \le 0$. It is known
\cite{KashVanishing} that $\gr^Y\dd_X:=\gr_{V^Y} \dd_X$, the associated
graded of $\dd_X$ with respect to the $Y$-filtration, is a sheaf of
algebras on $Y$ isomorphic to $\rho_{\idot} \dd_{\NN_{X/Y}}$.
The $\Cs$-action on $\NN_{X/Y}$ defines an Euler vector field $\eu$ on $\mc{O}_{\NN_{X/Y}}$. 
Locally on $X$, we can (and will) choose a lift $\theta\in V_0 \dd_X$ of the Euler vector field $\eu$ on $\mc{O}_{\NN_{X/Y}}$. 

An ascending filtration  $\mm_k,\ k\in\Z$, on a coherent $\dd$-module $\mm$ is said
to be a \textit{$Y$-good filtration} 
if each $\mm_k$  is a coherent $V^Y_0 \dd_X$-module and the following standard conditions hold
$$ (V^Y_k \dd_X) \cdot \mm_l \subset \mm_{k+l}, \quad  (V^Y_k \dd_X) \cdot \mm_l = \mm_{k + l}, \quad \cup_{k \in \Z} \mm_k = \mm, 
$$
where the second equality holds for any fixed $k \ge 0$ and all $l \gg l(k)>0$, resp. for
 any fixed $k \le 0$ and all $l\ll l(k)<0$. 
For an $Y$-good filtration $\mm_k,\ k\in\Z$,  the associated graded  $\gr \mm$ acquires the natural structure of  a coherent $\rho_{\idot} \dd_{\NN_{X/Y}}$-module.

\begin{rem}
 The $Y$-filtration on $\dd_X$, considered as a left $\dd_X$-module,
 is a $Y$-good filtration. 
\end{rem}

%A  filtration $\mm_k$ on a coherent $\dd$-module $\mm$ is said
%to be \textit{$V$-good filtration} 
%if the following holds
%\begin{align*}
% & (V_k \dd_X) \cdot \mm_l \subset \mm_{k+l}, \\%\label{eq:gf1}\\
% & (V_k \dd_X) \cdot \mm_l = \mm_{k + l} \textrm{ if $l \gg 0$ and $k \ge 0$ or $l \ll 0$ and $k \le 0$}, \\%\label{eq:gf2}\\
% & \textrm{$\mm_k$ is a coherent $V_0 \dd_X$-module}, \\%\label{eq:gf3} \\
% & \textrm{$\cup_{k \in \Z} \mm_k = \mm$}.% \label{eq:exhaust}
%\end{align*}

Let $C:=\{z\in\C\mid 0\leq \mathfrak{Re}(z) <1\}$, a set of coset representatives of $\Z$ in $\C$. 

Kashiwara \cite{KashVanishing} proved the following 

\begin{prop}\label{thm:vanishing}
Let $\mm$ be a coherent $\dd_X$-module. Assume that, locally on $X$, there exists a coherent $V^Y_0 \dd_X$-submodule $\ms{F}$ of $\mm$ and non-zero polynomial $d$ such that 

\vi $d(\theta) \cdot \ms{F} \subset (V^Y_{-1} \dd_X) \cdot \ms{F}$,

\vii $\dd_X \cdot \ms{F} = \mm$. 

Then, there is a unique (globally defined) $Y$-good filtration $V^Y_{\idot} \mm$ on $\mm$ 
and a non-zero polynomial $b$ with $b^{-1}(0) \subset C$ such that $b(\eu + k) \cdot (V^Y_k \mm/V^Y_{k-1}\mm) = 0$ for all $k \in \Z$.
\end{prop}

A coherent $\dd_X$-module $\mm$ is said to be {\em spealizable at} $Y$ if
the conclusions of the above proposition hold for $\mm$.
In this case, we refer to the filtration $V^Y_{\idot} \mm$ as {\em the} $Y$-filtration on
$\mm$ and write $\gr^Y\mm$ for the associated graded
  $\rho_{\idot} \dd_{\NN_{X/Y}}$-module. We let $\Psi_{X/Y}(\mm)$ denote the $\dd_{\NN_{X/Y}}$-module
such that $\rho_{\idot}\Psi_{X/Y}(\mm)=\gr^Y \mm$. Property (ii) of Proposition \ref{thm:vanishing}
ensures that  $\Psi_{X/Y}(\mm)$ is a $\C^*$-monodromic $\dd_{\NN_{X/Y}}$-module.

A deep result of Kashiwara and Kawai   \cite{KKAsymptotic} reads 
\begin{thm} Let $\mm$ be a holonomic $\dd$-module with regular singularities.
Then, $\mm$ is specializable along any smooth subvariety $Y\sset X$ and, moreover,
$\Psi_{X/Y}(\mm)$ is  a holonomic $\dd_{\NN_{X/Y}}$-module with regular singularities.
\end{thm}

It follows from the  theorem that  $\Psi_{X/Y}$ yields
 an exact functor, the \textit{specialization functor}, from the category of regular holonomic $\dd$-modules on $X$, to the category of regular holonomic, $\Cs$-monodromic $\dd$-modules on $\NN_{X / Y}$.

\subsection{Specialization for Cherednik algebras}\label{sec:Cheredefn}
We begin by recalling the definition of sheaves of Cherednik algebras as
given in \cite{ChereSheaf}. 

Let $X$ be a smooth, connected,
quasi-projective variety and $W \subset \mathsf{Aut}(X)$ a finite
group.  Given $x \in X$ we write $W_x$ for the isotropy group of
$x$ in $W$. The group $W$ is said to act on $X$ as a
\textit{pseudo-reflection group} if, for every $x \in X$ , the group $W_x$
acts on $T_x X$ as a pseudo-reflection group. From
now on we  assume that  $W$ acts on $X$ as a
pseudo-reflection group. By the
Chevalley-Shephard-Todd Theorem, this implies that $X/W$ is  smooth
and $\pi_{\idot} \mc{O}_X$ is a locally free $\mc{O}_{X/W}$-module of
rank $|W|$, where $\pi$ is the quotient map. 

Let $\mc{S}(X,W)$ denote the set of pairs $(w,Z)$, where $w \in W$ and $Z$ is a connected component of $X^w$ of codimension one. The group $W$ acts on the set $\mc{S}(X,W)$ by 
$$
g \cdot (Z,w) = (g(Z),gwg^{-1})
$$
and we fix $\kappa : \mc{S}(X,W) \rightarrow \C$ to be a $W$-equivariant
function. 

On any sufficiently small affine, $W$-stable, open subset $U$ of $X$, the closed subvarieties $Z \cap U$, for $(w,Z) \in \mc{S}(X,W)$, are defined by the vanishing of a function $f_Z$ say. The \tit{Dunkl-Opdam operator} associated to a vector field $v \in \Gamma(U,\Theta_X)$ is the rational section
\begin{equation}\label{eq:Dunkl}
D_v = \pa_v + \sum_{(Z,w) \in \mc{S}} \frac{2 \kappa(Z,w)}{1 -
  \lambda_{Z,w}}\cdot \frac{v(f_Z)}{f_Z}\ (w - 1) 
\end{equation}
of $\Gamma(U,\dd_{X} \rtimes W)$, where $\lambda_{Z,w}$ is the
eigenvalue of $w$ on each fiber of the conormal of $Z$ in $X$. On $U$,
the Cherednik algebra $\H_{\kappa}(U,W)$ is the affine $\C$-subalgebra
of $\End_{\C}(\C[U])$ generated by $\Gamma(U,\mc{O}_X
\rtimes W)$ and all Dunkl-Opdam operators. The algebras $\H_{\kappa}(U,W)$ glue to form \tit{the sheaf of Cherednik algebras} $\mc{H}_{\kappa}(X,W)$ on $X/W$. 

\subsection{}\label{sec:notation} We fix a subgroup $W'\sset W$ such that the set
$X':=\{x\in X\mid W_x=W'\}$ is nonempty.
Let $Y$ be a connected component of $X'$. Thus, $Y$ is  a smooth, locally closed subset of $X$. Let $W_Y:=\{w\in W\mid w(Y)=Y\}$. The group $W_Y$ contains
$W'$ as a normal subgroup.

The normal bundle $\rho: \ny \rightarrow Y$ is a $W_Y$-equivariant vector bundle on $Y$, equipped with
a linear $W'$-action on the fibers. If a point $v\in \ny$ is fixed by an element $w\in W_Y$ then 
so is the point $\rho(v)$. It follows that any element $w\in W_Y\sminus W'$
acts on $\ny$ without fixed points. 
Further, for any $w\in W'$, the fixed point set $\ny^w$ is vector sub-bundle
of the normal bundle on $Y$. We deduce that
the group
$W_Y$ acts on $\ny$ by pseudo-reflections. 

For any  pair $(w,C')$, where $w \in W'$
and $C'$ is a codimension one component of $\ny^w$,
there is a  unique connected component $C$ of $X^w$ such that 
 $Y\sset C$ and the normal bundle to $Y$ in $C$ equals $\ny^w$.
The assignment $(w,C')\mto (w,C)$ yields a canonical injective
map $\mc{S}(\ny,W_Y)=\mc{S}(\ny,W')\ \into\ \mc{S}(X,W)$. Moreover, since each $w \in W'$ acts linearly on the fibers of $\NN_{X/Y}$ and trivially on $Y$, the space $\ny^w$ is connected and hence the projection map $\mc{S}(\ny,W') \rightarrow W'$ is also an embedding. 

Given a $W$-invariant function $\kappa : \mc{S}(X,W) \rightarrow \C$, one obtains
via the above injection, a $W_Y$-invariant function $\kappa' : \mc{S}(\ny, W_Y) \rightarrow \C$.
Associated with the $W_Y$-action on the variety $\ny$ and the function $\kappa'$, one has
the sheaf of Cherednik algebras on $\ny/W_Y$. The morphism $\brho: \ny/W_Y\to Y/W_Y$ being
affine, we will abuse the notation and also write
 $\mc{H}_{\kappa'}( \ny, W_Y)$ for the direct image of the
sheaf of Cherednik algebras on $\ny/W_Y$ via $\brho$. Thus, we
view $\mc{H}_{\kappa'}( \ny, W_Y)$ as a sheaf of algebras on $Y/W_Y$.

Following \cite{BE}, we write $\mathrm{Fun}_{W_Y}(W,\mc{H}_{\kappa'}( \ny, W_Y))$ for the sheaf on $Y / W_Y$ consisting of all $W_Y$-equivariant functions $f : W \rightarrow \mc{H}_{\kappa'}( \ny, W_Y)$, where $W_Y$ acts on $\mc{H}_{\kappa'}( \ny, W_Y)$ by conjugation. Then the sheaf $\mathrm{Fun}_{W_Y}(W,\mc{H}_{\kappa'}( \ny, W_Y))$ is a right $\mc{H}_{\kappa'}( \ny, W_Y)$-module and hence a left module for 
$$
Z(W,W_Y, \mc{H}_{\kappa'}( \ny, W_Y)) := \End_{\mc{H}_{\kappa'}( \ny, W_Y)}(\mathrm{Fun}_{W_Y}(W,\mc{H}_{\kappa}( \ny, W_Y))).
$$
The sheaf $Z(W,W_Y, \mc{H}_{\kappa'}( \ny, W_Y))$ is Morita equivalent to $\mc{H}_{\kappa}( \ny, W_Y)$.

\subsection{}
Let $X^\circ$ be the set of points $x\in X$ such that
the group $W_x$ is conjugate (in $W$) to a subgroup of $W'$.
The set  $X^\circ$ is a $W$-stable Zariski open  subset of $X$.
We write $WY$ for the $W$-saturation of the set $Y$.
The set $WY$ is a  $W$-stable {\em closed} subvariety of $X^\circ$.
Furthermore, $WY$ is a
disjoint union of the 
subvarieties $w(Y)$ where $w$ runs over a set of coset representatives
of $W/W_Y$ in $W$.
Therefore, the image of $WY$ in $X^\circ/W$ equals $Y/W_Y$. We have a commutative diagram of natural maps:

\beq{eq:bigmapdiag2}
\xymatrix{
\NN_{X/Y} \ar@{->>}[rr]^{\rho} \ar[rrd]^{\bnu} \ar[d]_{\nu} & & Y \ar@{^{(}->}[rr] \ar[d]^{\pi_Y} & & X^\circ \ar@{->>}[d]^{\pi} \\
\ny / W_Y \ar@{->>}[rr]^{\brho} & & Y/W_Y \ar@{^{(}->}[rr] & & X^\circ/W %\ar[ur]_{\varpi} \ar@/_3pc/@{.>}[rruu]_{\Upsilon = \xi \circ \bk} 
}
\eeq

The canonical map $\NN_{X/WY} \twoheadrightarrow WY \twoheadrightarrow Y/W_Y$ is denoted $\eta$. Let  $\mc{I}_{WY}$ be the sheaf of ideals in $\mc{O}_{X^\circ}$ defining $W Y$, and let $\mc{I}:=\pi_\idot\mc{I}_{WY}$. We replace $X$  by $X^{\circ}$ and assume that $W Y$ \textit{is closed in} $X$. 

Since the algebra $\mc{H}_{\kappa}(X,W)$ acts on $\pi_{\idot} \mc{O}_X$, we may define a $V$-filtration on $\mc{H}_{\kappa}(X,W)$. 

\begin{defn}\label{defn:VCherednik}
The $WY$-filtration on $\mc{H}_{\kappa}(X,W)$ is defined to be
\beq{eq:Hfilt}
V^{WY}_m \mc{H}_{\kappa}(X,W) = \left\{ D \in \mc{H}_{\kappa}(X,W) \ | \ D( \mc{I}^k ) \subset \mc{I}^{k - m}, \ \forall \ k \in \Z \right\},
\eeq
where $\mc{I}^k := \pi_{\idot} \mc{O}_X$ for $k \le 0$. 
\end{defn}

One can show that  $V^{WY}_m \mc{H}_{\kappa}(X,W)$ defines a quasi-coherent subsheaf of $\mc{H}_{\kappa}(X,W)$. In particular, for all $x \in X/W$, the stalk of $V^{WY}_m \mc{H}_{\kappa}(X,W)$ at $x$ equals $\{ D \in \mc{H}_{\kappa}(X,W)_{x} \ | \ D(I_{x}^j) \subset I_{x}^{j-m} \ \forall \ j \}$. We write $\gr^{WY}\mc{H}_{\kappa}(X,W)$ for the corresponding associated graded algebra. Then, clearly $\gr^{WY}\mc{H}_{\kappa}(X,W)$ is supported on $Y / W_Y$. 

We have a grading $\bnu_\idot\oo_{\ny}=\bplus_{\ell\in\Z}\ (\bnu_\idot\oo_{\ny})^{(\ell)}$
by the homogeneity degree along the fibers of the vector
bundle $\ny\to Y$. 
The sheaf $\mc{H}_{\kappa'}( \ny, W_Y)$ acts naturally on
$\bnu_\idot\oo_{\ny}$. We put
$$\mc{H}_{\kappa'}^{(\ell)}( \ny, W_Y)\ :=\
\{h\in \mc{H}_{\kappa'}( \ny, W_Y)\en|\en h((\bnu_\idot\oo_{\ny})^{(k)}\sset
(\bnu_\idot\oo_{\ny})^{(\ell+k)},\en\forall\ k\in\Z\}.
$$
This gives a canonical $\Z$-grading
$\mc{H}_{\kappa'}( \ny, W_Y)=\bplus_{\ell\in\Z}\ \mc{H}_{\kappa'}^{(\ell)}( \ny, W_Y)$.

\begin{lem}\label{lem:globalEu}
The grading on the algebra $\mc{H}_{\kappa'}( \ny, W_Y)$ is {\em inner}, specifically, there is a unique element $\eu\in \Ga(Y/W_Y,\ \mc{H}_{\kappa'}( \ny, W_Y))$ such that for any $\ell\in\Z$ one has
$$ 
\eu(f)=\ell\cdot f \quad\text{\em and}\quad [\eu,h]=\ell\cdot h,\quad \forall f\in (\bnu_\idot\oo_{\ny})^{(\ell)},\en h\in 
\mc{H}_{\kappa'}^{(\ell)}( \ny, W_Y).
$$
\end{lem}

\begin{proof} 
Choose an open affine covering $Y=\cup U_i$ such that the restriction of the normal bundle is a trivial bundle
$\ny|_{U_i}\cong U_i\times \mf{h}$. Then, we have $\H_{\kappa'}(U_i,W')\cong
\dd(U_i)\otimes \H_{\kappa'}(\mf{h}, W')$. In the algebra $\dd(U_i)\otimes \H_{\kappa'}(\mf{h}, W')$ there is a unique element $\eu$ (necessarily belonging to $\H_{\kappa'}(\mf{h}, W')$) satisfying the properties stated in the lemma. Uniqueness implies that these elements glue to a global section of $\mc{H}_{\kappa'}( \ny, W_Y)$.   
\end{proof}

 Let $e$ denote the trivial idempotent in $\C W$ and $e_Y$ the trivial idempotent in $\C W_Y$. The main result of this section is the following (cf. \eqref{eq:bigmapdiag2} for the definition of $\brho$):

\begin{thm}\label{conj:Vfilt}
There is an isomorphism of graded sheaves of algebras on $Y / W_Y$, 
$$
\gr^{WY} \mc{H}_{\kappa}(X,W) \stackrel{\sim}{\longrightarrow} Z(W,W_Y, \brho_{\idot} \mc{H}_{\kappa'}( \ny, W_Y)),
$$
which restricts to an isomorphism $\gr^{WY} e \mc{H}_{\kappa}(X,W) e \stackrel{\sim}{\longrightarrow} \brho_{\idot}  (e_Y \mc{H}_{\kappa'}( \ny, W_Y) e_Y)$.
\end{thm}

\subsection{Proof of Theorem \ref{conj:Vfilt}}

The natural action of each of the algebras $\gr^{WY} \mc{H}_{\kappa}(X,W)$ and $Z(W,W_Y, \brho_{\idot} \mc{H}_{\kappa'}( \ny, W_Y))$ on the sheaf  $\eta_{\idot} \mc{O}_{\NN_{X/WY}}$ is faithful. Therefore, in order to show the existence of the isomorphism of Theorem \ref{conj:Vfilt}, it suffices to show that the sheaves are equal as subsheaves of the endomorphism sheaf $\mc{E}nd_{\C}(\eta_{\idot} \mc{O}_{\NN_{X/WY}})$. Being a local statement, we can check this on stalks. Fix some $y \in Y$ and denote its image in $Y / W_Y$ by $\mathbf{y}$.  

At $\mathbf{y}$, we have a decomposition of stalks of sheaves 
\beq{eq:decompstalk}
(\pi_{\idot} \mc{O}_X)_{\mathbf{y}} = \bigoplus_{w \in W/W'} \mc{O}_{X,w(y)}  \quad \textrm{and} \quad (\eta_{\idot} \mc{O}_{\NN_{X/WY}})_{\mathbf{y}} = \bigoplus_{w \in W / W'} \mc{O}_{\NN_{X/w(Y)},w(y)}.
\eeq
Analogous to the main result of \cite{BE}, the left hand equality in (\ref{eq:decompstalk}) implies that there is an isomorphism of algebras 
\beq{eq:Zalg1}
\mc{H}_{\kappa}(X,W)_{\mathbf{y}} \simeq Z(W,W', \mc{H}_{\kappa'}(X,W')_y).
\eeq
Also, since $\gr^{WY} \mc{O}_X = \eta_{\idot} \mc{O}_{\NN_{X/WY}}$ is a subalgebra of $\gr^{WY} \mc{H}_{\kappa}(X,W)$, the equality on the right hand side of (\ref{eq:decompstalk}), together with Lemma 2.3.1 of \cite{LosevSRAComplete}, implies that 
\beq{eq:Zalg2}
(\gr^{WY} \mc{H}_{\kappa}(X,W))_{\mathbf{y}} \simeq Z(W,W', \bp (\gr^{WY} \mc{H}_{\kappa}(X,W))_{\mathbf{y}} \bp),
\eeq
where $\bp \in Z(W,W', (\eta_{\idot} \mc{O}_{\NN_{X/WY}})_{\mathbf{y}} \rtimes W)$ is the idempotent defined by $(\bp f)(w) = f(w)$ if $w \in W'$ and $(\bp f)(w) = 0$ otherwise; equivalently $\bp$ is the projection map from $(\eta_{\idot} \mc{O}_{\NN_{X/WY}})_{\mathbf{y}}$ onto $\mc{O}_{\NN_{X/Y}, y}$. We may identify $\bp (\gr^{WY} \mc{H}_{\kappa}(X,W))_{\mathbf{y}} \bp$ with $\gr^{Y} \mc{H}_{\kappa'}(X,W')_y$. 

Similarly, the fact that 
$$
(\brho_{\idot} \mc{O}_{\NN_{X/Y}})_{\mathbf{y}} = \bigoplus_{w \in W_Y / W'} \mc{O}_{\NN_{X/w(Y)},w(y)}
$$
implies that 
\beq{eq:Zalg3}
(\brho_{\idot} \mc{H}_{\kappa'}( \ny, W_Y))_{\mathbf{y}} \simeq Z(W_Y, W', \mc{H}_{\kappa'}( \ny, W')_{y}).
\eeq 
Combining the isomorphisms (\ref{eq:Zalg1}), (\ref{eq:Zalg2}) and (\ref{eq:Zalg3}), we see that it is enough to show that $\gr^{Y} \mc{H}_{\kappa'}(X,W')_y$ is isomorphic to $\mc{H}_{\kappa'}( \ny, W')_{y}$. Replacing $W$ by $W'$ we may assume that $y$ is fixed by $W$ and that $X$ is affine.

\begin{prop}
Assume that $y \in Y := X^{W}$. Then $\gr^{WY} \mc{H}_{\kappa}(X,W)_y \simeq \mc{H}_{\kappa}(\NN_{X/Y},W)_y$. 
\end{prop}

\begin{proof}
We choose $x_1, \ds, x_n$ in the maximal ideal $\mf{m}_y$ of $\mc{O}_{X,y}$ such that $\overline{x}_1, \ds, \overline{x}_n$ are a basis of $\mf{m}_y / \mf{m}_y^2$, the space $\mf{h}^* := \C \{ x_1, \ds, x_n \}$ is a $W$-submodule of $\mc{O}_{X,y}$, and $Y = V(x_{k+1}, \ds, x_n)$. For each $s \in \mc{S}(X,W)$, we fix $\alpha_s \in \C \{ x_1, \ds, x_n \}$ such that $X^s = V(\alpha_s)$. Then the PBW theorem implies that we have a $\mc{O}_{X,y}$-module isomorphism 
$$
\mc{H}_{\kappa}(X,W)_y \simeq \mc{O}_{X,y} \o \C W \o \C[y_1, \ds, y_n]
$$
where $y_i$ is the Dunkl operator corresponding to the vector field $\pa_{x_i}$. By assumption, the vectors $\alpha_s$ belong to $ \C \{  x_{k+1}, \ds, x_n \}$, which means that $y_i = \pa_{x_i}$ for $i = 1, \ds, k$. The ideal $I = I(Y)$ is generated by $x_{k+1}, \ds, x_n$. It suffices to identify  
$$
V_m \mc{H}_{\kappa}(X,W)_y = \bigoplus_{i - j = m} I^j \o \C W \o \C[y_1, \ds, y_k][y_{k+1}, \ds, y_n]_i,
$$
where $\C[y_{k+1}, \ds, y_n]_i$ is the space of homogeneous polynomials of degree $i$, and show that the symbol $\bar{y}_i$ of $y_i$ in $\gr^{WY} \mc{H}_{\kappa}(X,W)_y$ equals the corresponding Dunkl operator in $\mc{H}_{\kappa}(\NN_{X/Y},W)_y$.

However, both of these facts can be shown to follow from the PBW property for $\mc{H}_{\kappa}(X,W)_y$ by considering the action of $\mc{H}_{\kappa}(X,W)_y$ on $\C[\mf{h}] \cap I^m \subset \mc{O}_{X,y}$.  
\end{proof}

%\begin{lem}\label{lem:vectorfilt}
%Let $\h$ be a vector space and $W \subset GL(\h)$ a finite group. Let $Y = \{ 0 \}$. Then, $\gr^{\{0\}} \H_{\kappa}(\h,W) = \H_{\kappa}(\h,W)$. 
%\end{lem}

%\begin{proof}
%For each $s \in \mc{S}(W)$, we fix a eigenvector $\alpha_s \in \mf{h}^*$ for the non-trivial eigenvalue of $s$. Let $\Delta_s = \frac{1}{\alpha_s}(1 - s)$ be the corresponding Demazure operator, an element in $\End_{\C}(\C[\h])$. Let $A$ be the subalgebra of $\End_{\C}(\C[\h])$ generated by $\dd({\h})$, the group $W$ and all Demazure operators. Then, $\H_{\kappa}(\h,W)$ is a subalgebra of $A$. Recall that $\H_{\kappa}(\h,W)$ is $\Z$-graded by putting $\h \subset \C[\h^*]$ in degree one, $\h^* \subset \C[\h]$ in degree $-1$ and $W$ in degree zero. By putting $\pa_v$ for $v \in \h$ and $\Delta_s$ in degree one, the algebra $A$ is also $\Z$-graded such that the embedding $\H_{\kappa}(\h,W) \hookrightarrow A$ preserves gradings. Moreover, since the $\{ 0 \}$-filtrations on $\H_{\kappa}(\h,W)$ and $A$ are defined in terms of their action on $\C[\h]$, the $\{ 0 \}$-filtration on $\H_{\kappa}(\h,W)$ is simply the restriction of the $\{ 0 \}$-filtration on $A$. Finally, it is clear that the $\{ 0 \}$-filtration on $A$ is just the $\Z$-filtration defined in terms of the $\Z$-grading i.e. $V^{\{0\}}_m A = \bigoplus_{k \le m} A_k$. Hence, the same is true of $\H_{\kappa}(\h,W)$. Thus, $\gr^{\{0\}} A = A$ and $\gr^{\{0\}}\H_{\kappa}(\h,W) = \H_{\kappa}(\h,W)$. 
%\end{proof} 

\subsection{Specialization}\label{sec:Cherespecial}

In this section we define a specialization functor for ${\mathcal H}_{\kappa}(X,W)$-modules, analogous to the specialization functor defined for $\dd$-modules in section \ref{sec:specialA}.

\begin{defn}\label{defn:specialize}
A coherent ${\mathcal H}_{\kappa}(X,W)$-module $\ms{M}$ is said to be \textit{specializable} along $WY$ if, locally on $X/W$, there exists an $WY$-good filtration $\mc{F}_{\idot} \ms{M}$ on $\ms{M}$, with respect to the $WY$-filtration on ${\mathcal H}_{\kappa}(X,W)$, such that, for some lift $\theta \in {\mathcal H}_{\kappa}(X,W)$ of $\eu$, there exists a polynomial $b$ with
$$
b(\theta) \cdot \mc{F}_0 \ms{M} \subseteq (V_{-1} {\mathcal H}_{\kappa}(X, W)) \cdot \mc{F}_0 \ms{M}.
$$ 
\end{defn}

\begin{rem}
Definition \ref{defn:specialize} is independent of the choice of lifts, since the difference of any two choices lies in $V^{WY}_{-1} {\mathcal H}_{\kappa}(X,W)$. 
\end{rem}

The category of all coherent ${\mathcal H}_{\kappa}(X,W)$-modules,
specializable along $WY$, will be denoted $\Lmod{{\mathcal H}_{\kappa}(X / WY)}$. 

Recall that we have fixed a set $C$ of coset representatives of $\Z$ in $\C$. Arguing as in the proof of \cite[Theorem 1]{KashVanishing}, one obtains the following analogue of Proposition \ref{thm:vanishing} for Cherednik algebras. 

\begin{prop}\label{prop:uniquefiltration}
Let $\ms{M}$ be specializable along $WY$. Then, there exists a unique (global) $WY$-good filtration $V^{WY}_{\idot} \ms{M}$, the \textit{$WY$-filtration}, on $\ms{M}$ and nonzero polynomial $b$ such that $b^{-1}(0) \subset C$ and 
$$
b(\theta + k) \cdot (V^{WY}_{k}\ms{M}/V^{YW}_{k-1} \ms{M}) = 0, \quad \forall \ k \in \Z.
$$ 
\end{prop}

Fix an idempotent $e_d \in Z(W,W_Y, {\mathcal H}_{\kappa'}( \NN, W_Y))$ such that $e_d Z(W,W_Y, {\mathcal H}_{\kappa'}( \NN, W_Y)) e_d$ is isomorphic to ${\mathcal H}_{\kappa'}( \NN, W_Y)$. Proposition \ref{prop:uniquefiltration} allows us to define the specialization functor 
$$
\Sp_{X/WY} :\ \Lmod{{\mathcal H}_{\kappa}(X / WY)}\ \rightarrow\ \Lmod{{\mathcal H}_{\kappa'}( \NN, W_Y)}, \quad \Sp_{X/WY}(\ms{M}) =  e_d \cdot \gr_V \ms{M}.
$$
The proposition also implies that $\Sp_{X/WY}$ is an exact functor, c.f. Proposition 5.1.3 of \cite{CyclesProches}.

For any specializable module $\mm$,
the action of the Euler element $\eu\in {\mathcal H}_{\kappa'}( \NN, W_Y)$
on $\Sp_{X/WY}(\ms{M})$ is locally finite, by construction. 
Hence this action can be exponentiated
to  a well-defined operator
$\exp(2\pi\sqrt{-1}\,\eu):\ \Sp_{X/WY}(\ms{M})\to \Sp_{X/WY}(\ms{M})$,
called the {\em monodromy} operator.  Using the defining relations of
the algebra ${\mathcal H}_{\kappa'}( \NN, W_Y)$, it is straight-forward 
to verify that the monodromy operator commutes
with the action of ${\mathcal H}_{\kappa'}( \NN, W_Y)$ on $\Sp_{X/WY}(\ms{M})$.

\begin{rem}
We note that the $\eu$-action on  $\Sp_{X/WY}(\ms{M})$ is
 not necessarily semisimple,
in general. Let  $\eu_{\operatorname{nil}}$ be
the nilpotent component in the Jordan decomposition
of the linear operator $\eu :\ \Sp_{X/WY}(\ms{M})\to \Sp_{X/WY}(\ms{M})$.
Then, the fact that the map $\exp(2\pi\sqrt{-1}\,\eu)$ 
commutes with  the ${\mathcal H}_{\kappa'}( \NN, W_Y)$-action
implies that  the map $\eu_{\operatorname{nil}}$
commutes with the ${\mathcal H}_{\kappa'}( \NN, W_Y)$-action 
 on $\Sp_{X/WY}(\ms{M})$ as well.
\end{rem}

\subsection{} For the remainder of this section, we consider the action of $W = \s_n$ on $\TSL$. For each $b \in \mf{t}$, Bezrukavnikov and Etingof, \cite{BE}, constructed a restriction functor $\mathrm{Res}_b$ from category $\mc{O}$ for the rational Cherednik algebra $\H_{\kappa}(\mf{t},W)$ to category $\mc{O}$ for the rational Cherednik algebra $\H_{\kappa}(\mf{t}',W_b)$, where $\mf{t} = \mf{t}^{W_b} \oplus \mf{t}'$ is the canonical $W_b$-module decomposition. Via the isomorphism (\ref{eq:BEiso}), one can define an analogous restriction functor for modules over the trigonometric Cherednik algebra $\H_{\kappa}(\TSL,W)$. We show that, on those ${\mathcal H}_{\kappa}(\TSL,W)$-modules that are coherent over $\mc{O}_{\TSL/W}$, the specialization functor agrees with the restriction functor. 

%The weight lattice $P$ decomposes as $P^{N_L} \oplus P_L$, where $P_L = \mf{t}_L^* \cap P$. The smooth variety $T_L$ is affine, and its closure in $\TSL$ is defined by the equations $e^x  = 1$ for all $x \in P_L$. There is a canonical isomorphism 
%$$
%\C[ \widehat{\TSL}_L] \rightsim \C[\TSL_L] [\![ \mf{t}_L ]\!], \quad e^{\lambda} \mapsto e^{\lambda}, \ \ e^{x} - 1 \mapsto \sum_{i = 1}^\infty \frac{x^i}{i !}, \quad \forall \ \lambda \in P^{N_L}, \ x \in P_L,
%$$

We use freely the notation of section \ref{jj2}. The exponential map gives a canonical isomorphism $\C[ \widehat{\TSL}_L] \rightsim \C[\TSL_L] [\![ \mf{t}_L ]\!]$, where $\widehat{\TSL}_L$ denotes the formal neighborhood of $\TSL_L$ in $\TSL$. Thus, we can canonically identify $\widehat{\TSL}_L = T_L \times \widehat{\mf{t}}_L$. Let $\widehat{\mc{O}}_{\TSL/W}$ and $\widehat{\mc{O}}_{(T_L \times \mf{t}_L)/ N_L}$ denote the sheaves of functions on the formal neighborhood of $T_L / N_L$ in $\TSL /W$ and $(T_L \times \mf{t}_L)/ N_L$, respectively. These sheaves are also isomorphic. We set $\widehat{{\mathcal H}}_{\kappa}(\TSL,W)_L = \widehat{\mc{O}}_{\TSL/W} \o_{\mc{O}_{\TSL/W}} {\mathcal H}_{\kappa}(\TSL,W)$ and 
$$
\widehat{{\mathcal H}}_{\kappa}(T_L \times \mf{t}_L, N_L ) = \widehat{\mc{O}}_{(T_L \times \mf{t}_L) / N_L} \o_{\mc{O}_{(T_L \times \mf{t}_L) / N_L}} {\mathcal H}_{\kappa}(T_L \times \mf{t}_L, N_L).
$$
Analogous to the isomorphism of \cite[Section 3.7]{BE}, we have an isomorphism
\beq{eq:BEiso}
\widehat{{\mathcal H}}_{\kappa}(\TSL,W)_L \iso Z(W,N_L, \widehat{{\mathcal H}}_{\kappa}(T_L \times \mf{t}_L, N_L ))
\eeq
of sheaves of algebras on $\TSL_L / N_L$. By Lemma \ref{lem:globalEu}, there exists a global section $\eu$ in the algebra $Z(W,N_L,{\mathcal H}_{\kappa}(T_L \times \mf{t}_L, N_L ))$, which we may consider, via the natural inclusion, as a section of $Z(W,N_L,\widehat{{\mathcal H}}_{\kappa}(T_L \times \mf{t}_L, N_L ))$. If $\ms{N}$ is a $Z(W,N_L, \widehat{{\mathcal H}}_{\kappa}(T_L \times \mf{t}_L, N_L ))$-module, then we denote by $\ms{N}^{\fin}$ the subsheaf of all sections that are locally finite with respect to the action of $\eu$. Then, the Bezrukavnikov-Etingof restriction functor is defined to be 
$$
\Res_L : \Lmod{{\mathcal H}_{\kappa}(\TSL,W)} \rightarrow \LMod{{\mathcal H}_{\kappa}(T_L \times \mf{t}_L, N_L )}, \quad \Res_L (\ms{M}) = e_d \cdot \left( \widehat{\ms{M}}_L \right)^{\fin},
$$
where $\widehat{\ms{M}}_L = \widehat{\mc{O}}_{\TSL / W} \o_{\mc{O}_{\TSL / W}} \ms{M}$.

\begin{proof}[Proof of Theorem \ref{prop:specialcoherent}]
Let $\mc{K}$ be the sheaf of ideals defining $\TSL_L /N_L$ in $\TSL / W$. The same notation will be used to denote the corresponding ideal in $\widehat{\mc{O}}_{\TSL / W } \simeq \widehat{\mc{O}}_{(T_L \times \mf{t}_L)/ N_L}$. The proof of the theorem depends on the following key claim.

\begin{claim}\label{claim:101}
Let $\ms{N}$ be a $Z(W,N_L, \widehat{{\mathcal H}}_{\kappa}(T_L \times \mf{t}_L, N_L ))$, coherent over $\widehat{\mc{O}}_{(T_L \times \mf{t}_L)/ N_L}$. For all $\ell > 0$, there exists a non-zero polynomial $d_{\ell}$ such that $d_{\ell}(\eu) \cdot (\ms{N} / \mc{K}^{\ell} \ms{N} ) = 0$. 
\end{claim}

\begin{proof}
The algebra $\ddd := \Gamma(\TSL_L,\dd_{\TSL_L}) = \Gamma(\TSL_L / N_L,\bnu_{\idot} (\rho^{-1} \dd_{\TSL_L}))$ is a subalgebra of the global sections of ${\mathcal H}_{\kappa}(T_L \times \mf{t}_L, N_L)$. Hence $Z(W,N_L, \ddd \rtimes N_L)$ is a subalgebra of the global sections of $Z(W,N_L, \widehat{{\mathcal H}}_{\kappa}(T_L \times \mf{t}_L, N_L))$ and the spaces $N_{\ell} = \Gamma(\TSL_L / N_L,\ms{N} / \mc{K}^{\ell} \ms{N})$ are $Z(W,N_L, \ddd \rtimes N_L)$-modules. Since the idempotent $e_d$ defines a Morita equivalence between $Z(W,N_L, \ddd \rtimes N_L)$ and $\ddd \rtimes N_L$, 
$$
E_{\ell} := \End_{Z(W,N_L, \ddd \rtimes N_L)} (N_{\ell}) \simeq \End_{\ddd \rtimes N_L}(e_d N_{\ell}) 
$$
is contained in $\End_{\ddd}(e_d N_{\ell})$. The $\ddd$-modules $e_d N_{\ell}$ are finitely generated over $\C[\TSL_L]^{N_L}$ and hence holonomic. Thus, $\End_{\ddd}(e_d N_{\ell})$ is finite-dimensional. The global section $\eu$ defines an element in the finite dimensional algebra $E_{\ell}$ and hence the claim follows.  
\end{proof}

The sheaf $\widehat{\ms{M}}_L$ is a $Z(W,N_L, \widehat{{\mathcal H}}_{\kappa}(T_L \times \mf{t}_L, N_L))$-module via the isomorphism (\ref{eq:BEiso}). Since $\mm$ is coherent over $\mc{O}_{\TSL / W}$, for all $\ell > 0$ the modules $\mm / \mc{K}^{\ell} \mm$ and $\widehat{\mm}_L / \mc{K}^{\ell} \widehat{\mm}_L$ are equal. Therefore $\eu \in Z(W,N_L, \widehat{{\mathcal H}}_{\kappa}(T_L \times \mf{t}_L, N_L))$ acts on $\mm / \mc{K}^{\ell} \mm$. 

Since $\mc{K} \subset V^{W\TSL_L}_{-1} {\mathcal H}_{\kappa}(\TSL,W)$, claim \ref{claim:101} implies that $d_{1} (\theta)  \cdot \mm \subset V^{W\TSL_L}_{-1} {\mathcal H}_{\kappa}(\TSL,W)) \cdot \ms{M}$. Thus, to show that the conditions of Definition \ref{defn:specialize} are satisfied we need to define a $W\TSL_L$-good filtration $\mc{F}_{\idot} \ms{M}$ on $\ms{M}$ such that $\mc{F}_0 \ms{M} = \ms{M}$. Since $\mc{O}_{\TSL/W}$ is a subsheaf of $V^{W\TSL_L}_{0} {\mathcal H}_{\kappa}(\TSL,W)$, $\ms{M}$ is a coherent $V^{W\TSL_L}_{0} {\mathcal H}_{\kappa}(\TSL,W)$-module. Thus, $\mc{F}_{k} \ms{M} := (V^{W\TSL_L}_{k} {\mathcal H}_{\kappa}(\TSL,W)) \cdot \ms{M}$ is a $W\TSL_L$-good filtration on $\ms{M}$, and hence $\ms{M} \in \Lmod{{\mathcal H}_{\kappa}(\TSL / WY)}$. 

Let $\mm^V$ denote the completion $\underset{\leftarrow}\lim \ \ms{M} / V^{W\TSL_L}_m$ of $\mm$ with respect to the $W\TSL_L$-filtration. By a well-known argument, c.f. e.g. the proof of Theorem 6.10.1 of \cite{CharCycles},  
$$
\gr^{W\TSL_L} (\mm)_i = \bigoplus_{\alpha \in C} \mm^V(\alpha - i),
$$
where $\mm^V(\alpha)$ denotes the space of generalized eigenvectors of $\eu$ with eigenvalue $\alpha$. In other words, $\Sp_{\TSL/W\TSL_L}(\mm)$ equals $e_d \cdot (\mm^V)^{\mathrm{fin}}$. Therefore, $\Sp_{\TSL/W\TSL_L}(\mm) = \Res_L(\mm)$ provided the sheaves $\widehat{\mm}_L$ and $\mm^V$ are equal. For this, it suffices to show that the $W\TSL_L$-filtration and the $\mc{K}$-adic filtration on $\mm$ are comparable. 

Since $\mm$ is coherent over $\mc{O}_{\TSL/W}$, there exists some $N \gg 0$ such that $\mm = V^{W\TSL_L}_N \mm$. Thus, $\mc{K}^m \cdot \mm \subset V^{W\TSL_L}_{N-m} \mm$ for all $m  > 0$. Hence it suffice to show that, for each $\ell$ there exist some $k(\ell)$ such that $V^{W\TSL_L}_{k(\ell)} \mm \subset \mc{K}^\ell \cdot \mm$. For this, we first remark that the eigenvalues of $\eu$ on $\mc{K} / \mc{K}^2$ are all strictly positive. Consider the filtration by $\eu$-submodules $(V_k^{W\TSL_L} \mm + \mc{K}^{\ell} \cdot \mm) /  \mc{K}^{\ell} \cdot \mm$ of $\mm / \mc{K}^{\ell} \cdot \mm$. Using claim \ref{claim:101}, we let $d_{\ell}$ be the non-zero polynomial of smallest degree such that $d_{\ell}(\eu) \cdot (\mm / \mc{K}^{\ell} \cdot \mm) = 0$. From the definition of the $WY$-filtration, there is some non-zero polynomial $b$ such that 
$$
b(\eu + k) \cdot \left( \frac{V_k^{W\TSL_L} \mm + \mc{K}^{\ell} \cdot \mm}{V_{k-1}^{W\TSL_L} \mm + \mc{K}^{\ell} \cdot \mm} \right) = 0.
$$
Moreover, the polynomials $b(t + k)$ and $b(t + k')$ have no roots in common if $k \neq k'$. This implies that the filtration $(V_k^{W\TSL_L} \mm + \mc{K}^{\ell} \cdot \mm) /  \mc{K}^{\ell} \cdot \mm$ of  $\mm / \mc{K}^{\ell} \cdot \mm$ is finite. Hence, there exists some $k(\ell) \gg 0$ such that $V^{W\TSL_L}_{k(\ell)} \mm \subset \mc{K}^\ell \cdot \mm$. 
\end{proof}

\begin{rem}
\vi Theorem \ref{prop:specialcoherent} is false if the coherence condition is dropped.   

\vii In Theorem \ref{prop:specialcoherent} we could have taken $W$ to be any Weyl group acting on the abstract maximal torus of the corresponding simple Lie group. 

\viii Similarly, Theorem \ref{prop:specialcoherent} also holds for any complex reflection group acting on its reflection representation $\mf{h}$. 
\end{rem}

Let $\mc{O}_{\kappa}$ denote category $\mc{O}$ for either the
trigonometric Cherednik algebra $\H_{\kappa}^{\trig}(\mathrm{T},W)$ or
the rational Cherednik algebra $\H_{\kappa}(\h,W)$, where in the
trigonometric case $W$ is assumed to be a Weyl group. For each parabolic
subgroup $W'$ of $W$, let $Y$ be the set of all points in $\mathrm{T}$,
resp. in $\h$, whose stabilizer is $W'$. Theorem
\ref{prop:specialcoherent} implies 

\begin{cor}\label{catO_cor} Every module in $\mc{O}_{\kappa}$ is specializable along $W Y$  
 and we have $\Sp_{X/WY}(\ms{M}) = \Res_Y(\ms{M})$. 
\end{cor}

\section{Specialization of mirabolic modules}\label{sec:nearby}

In this section we define a specialization functor for mirabolic modules. 

\subsection{}\label{sub:near1} In this subsection we let $G$ be a connected, reductive group and $H$ a reductive subgroup. Let $Y$ be a smooth, affine $H$-variety and set $X = G \times_H Y$. Denote by $\Upsilon : Y \hookrightarrow X$ the closed embedding. We assume that the embedding $\h = \mathrm{Lie} \ H \hookrightarrow \g = \mathrm{Lie} \ G$ induces an isomorphism $(\g / [\g, \g])^* \iso (\h / [\h , \h])^*$.  We fix $\chi\in (\g / [\g, \g])^*$
and,  abusing the notation, also
write $\chi$ for its  image in $(\h / [\h , \h])^*$.

\begin{prop}\label{prop:monoequiv}

\vi The functor $\Upsilon^*$ defines an equivalence between the category of $(G,\chi)$-monodromic $\dd$-modules on $X$ and $(H,\chi)$-monodromic $\dd$-modules on $Y$.

\vii The natural map $\Gamma(X, \mm)^{G} \to \Gamma(Y, \Upsilon^* \mm)^{H}$ is an isomorphism for all $(G,\chi)$-monodromic $\dd$-modules $\mm$. 
 
\end{prop}

\begin{proof}

\vi The functor $\Upsilon^*$ is exact on the category of $G$-monodromic, coherent $\dd$-modules on $X$ because all such modules are non-characteristic for $\Upsilon$. If $i_H : H \hookrightarrow G$ is the inclusion then $i^*_H \mc{O}_G^{\chi} = \mc{O}_H^{\chi}$. Functorality implies that $\Upsilon^*$ maps $(G,\chi)$-monodromic $\dd$-modules on $X$ to $(H,\chi)$-monodromic $\dd$-modules on $Y$. To show that it is an equivalence, it suffices to exhibit an inverse. Let $\pi : G \times Y \to X$ be the quotient map. As noted in example \ref{examp:one}, since $H$ act on $G$ by multiplication on the right, the module $\mc{O}_G^{\chi}$ is $(H,-\chi)$-monodromic. Hence, if $\ms{N}$ is a $(H,\chi)$-monodromic module on $Y$ then $\mc{O}_G^{\chi} \boxtimes \ms{N}$ is a $(H,0)$-monodromic $=$ $H$-equivariant $\dd$-module on $G \times Y$. In fact, it is a $(G \times H, (\chi,0))$-monodromic module. Then, $(\pi_{\idot} \mc{O}_G^{\chi} \boxtimes \ms{N})^H$ is a $(G,\chi)$-monodromic module on $X$. The fact that $\pi^* (\pi_{\idot} \mc{O}_G^{\chi} \boxtimes \ms{N})^H \simeq \mc{O}_G^{\chi} \boxtimes \ms{N}$ and that $\Upsilon = \pi \circ j$, where $j : Y \to G \times Y$, $j(y) = (e,y)$ implies that 
$$
\Upsilon^* (\pi_{\idot} \mc{O}_G^{\chi} \boxtimes \ms{N})^H \simeq \ms{N}. 
$$
Going the other way, $\pi^*$ defines an equivalence between $(G,\chi)$-monodromic $\dd$-modules on $X$ and $(G \times H, (\chi,0))$-monodromic modules on $G \times Y$. Therefore, given a $(G,\chi)$-monodromic $\dd$-module on $X$, it suffices to show that 
$$
\pi^* \mm \simeq \mc{O}_G^{\chi} \boxtimes \Upsilon^* \mm.
$$
But, noting that $a \circ (\id_G \times \Upsilon) = \pi$, we have 
$$
(\id \times \Upsilon)^* \theta : \mc{O}_G^{\chi} \boxtimes \Upsilon^* \mm = (\id \times \Upsilon)^*(\mc{O}_G^{\chi} \boxtimes \mm) \iso (\id \times \Upsilon)^* a^* \mm = \pi^* \mm. 
$$

\vii A $(G,\chi)$-monodromic $\dd$-module $\mm$ is a quasi-coherent, $G$-equivariant $\mc{O}_X$-module. It is well-known that $\Gamma(X, \mm)^{G} \to \Gamma(Y, \Upsilon^* \mm)^{H}$ is an isomorphism for any such module. 

\end{proof}

Let $\g\to \dd(X)$ and $\nu:\h\to \dd(Y)$ be the Lie algebra maps
induced by the $G$-action on $X$ and $H$-action on $Y$, respectively.
For any $\chi$ as above, we put
$\g_\chi:=(\mu - \chi)(\g)$, resp. $\h_\chi:=(\nu - \chi)(\h)$.

\begin{lem}\label{lem:pullbackUp} 
Restriction via $\Upsilon$ induces an algebra isomorphism
\[(\dd(X)/\dd(X)\g_{\chi})^G \cong (\dd(Y) / \dd(Y) \h_{\chi})^{H}.\]
\end{lem}

\begin{proof}
Let $\mu_L : \g \rightarrow \dd_G$ denote the differential of the action of $G$ on itself by left multiplication and $\mu_R : \g \rightarrow \dd_G$ the differential of the action $G \times G \rightarrow G$, $(g_1,g_2) \mapsto g_2 g_1^{-1}$ of $G$ on itself by right multiplication. We will use the following claim.

The algebra $\dd(G \times_H Y)$ is obtained from
$\dd(G \times Y)$ by quantum Hamiltonian reduction, i.e.,
there is a canonical  isomorphism, see \cite[Corollary 4.5]{Schwarz},
$$
\dd(X) \simeq \dd(G \times Y)^H  / [\dd(G \times Y) \mu_{\Delta}(\h)]^H,
$$
where $\mu_{\Delta} : \h \to \dd(G \times Y)$ is the differential of the diagonal $H$-action on $G \times Y$. This implies that 
$$
\dd(X) / \dd(X) \g_{\chi} \simeq \dd(G \times Y)^H / [\dd(G \times Y) \mu_{\Delta}(\h) + \dd(G \times Y) \g_{\chi}]^H,
$$
where $\g_{\chi} = (\mu_L - \chi)(\g)$ and we have used the fact that
$H$ is reductive. Now, it is well-known, and easy to prove, that
for any character $\chi$  of $\g$ one has the following equality of left ideals:
%\begin{claim}\label{claim:leftrightstuff}
$\dd_G (\mu_L -\chi)(\g) = \dd_G
(\mu_R + \chi)(\g)$.
It follows that we have
$\dd(G \times Y) \g_{\chi} = \dd(G \times Y) (\mu_R - \chi)(\g)$. Since
$\mu_{\Delta} = \mu_R \o 1 + 1 \o \nu$, we obtain
$$
\dd(G \times Y)  \g_{\chi} + \dd(G \times Y) \mu_{\Delta}(\h)  = \dd(G \times Y)  \g_{\chi} + \dd(G \times Y) (\nu - \chi)(\h).
$$
Therefore, 
$$\dd(X) / \dd(X) \g_{\chi} \simeq \left[ \mc{O}_{G}^{\chi} \boxtimes \left( \frac{\dd(Y)}{\dd(Y) \h_{\chi}} \right) \right]^H.
$$
The proof of Proposition \ref{prop:monoequiv} shows that this implies
%an isomorphism
$\Upsilon^*(\dd(X) / \dd(X)\g_\chi)\ \cong \ \dd(Y) / \dd(Y)\h_\chi$.
 Then part (ii) of the Proposition  completes the proof.
\end{proof}

\subsection{Specialization of monodromic modules}

Let $f : X' \to X$ be a smooth morphism and $Y \subset X$ a smooth, locally closed, subvariety. Let $Y'= f^{-1}(Y)$. The morphism $f$ induces a morphism $df : \NN_{X'/Y'} \to \NN_{X/Y}$ given by $df (y,v) = (f(y), (d_y f)(v))$, for $y \in Y'$ and $v \in T_y X' / T_y Y'$.  

\begin{lem}[Th\'eor\`eme 9.3.1, \cite{CyclesProches}]\label{lem:etalespecial}
The following diagram commutes
$$
\xymatrix{ \Lmod{\dd_X}_{\reg} \ar[rr]^{\Psi_{X/Y}} \ar[d]_{f^*} & & \Lmod{\dd_{\NN_{X/Y}}}_{\reg} \ar[d]^{(d f)^*} \\
\Lmod{\dd_{X'}}_{\reg} \ar[rr]_{\Psi_{X'/Y'}} & & \Lmod{\dd_{\NN_{X'/Y'}}}_{\reg}. 
}
$$
\end{lem} 

\begin{proof}
The result, as stated in \cite[Th\'eor\`eme 9.3.1]{CyclesProches}, is
for $\dd_{X}^V$-modules, where $\dd^V_{X}$ is the completion of $\dd_X$ with respect to the $V$-filtration, and involves $f^!$. However, as
noted in the final paragraph of subsection (9.4) of \textit{loc. cit.},
the result is also valid for $\dd_X$-modules and the functors  $f^!$ and $f^*$ only differ by a shift, since
$f$ is a smooth morphism. 
\end{proof}

Assume that $G$ acts on $X$ and $Y$ is $G$-stable. The category of $(G,\chi)$-monodromic $\dd$-modules on $X$, resp. on $\NN_{X/Y}$, that have regular singularities is denote $\mon{\dd_{X}, G,\chi}_{\mathrm{reg}}$, resp. $\mon{\dd_{\NN_{X/Y}}, G,\chi}_{\mathrm{\reg}}$.

\begin{lem}\label{lem:specialmono}
Specialization restricts to an exact functor 
$$
\mon{\dd_{X}, G,\chi}_{\mathrm{reg}} \to \mon{\dd_{\NN_{X/Y}}, G,\chi}_{\mathrm{\reg}},
$$
with image in the subcategory of $\Cs$-monodromic modules. 

An analogous statement also holds for $(G,\bq)$-monodromic modules.
\end{lem}

\begin{proof}
Write $a : G \times X \rightarrow X$ and $d a : G \times \NN_{X/ Y} \rightarrow \NN_{X/Y}$ for the action maps. Let $\mm \in \mon{\dd_{\X}, G,\chi}_{\mathrm{reg}}$. Let $\theta : \mc{O}_G^{\chi} \boxtimes \mm \iso a^* \mm$ be the isomorphism defining the monodromic structure on $\mm$. Let $\Psi'$ be the specialization functor with respect to $G \times Y \subset G \times X$. The uniqueness of the $V$-filtration implies that 
$$
\Psi'(\theta) : \Psi'(\mc{O}_G^{\chi} \boxtimes \mm) = \mc{O}_G^{\chi} \boxtimes \Psi_{X /Y}(\mm) \to \Psi'(a^* \mm)
$$
is an isomorphism. Since the action map $a$ is a smooth morphism, Lemma \ref{lem:etalespecial} implies that $\Psi' (a^* \mm) = (d a)^* \Psi_{X / Y} (\mm)$.

Arguing in the same way, one can also show that $\Psi(\theta)$ also satisfies the cocycle condition.  
\end{proof}

Clearly, there is an analogue of Lemma \ref{lem:specialmono} for
$(G,\bq)$-monodromic modules. 

We returning to the setting of section \ref{sub:near1}.
So, let $H \subset G$ be a Levi subgroup and $Y$ a smooth, affine $H$-variety. We set $X = G \times_H Y$. Assume that $Z \subset Y$ is a smooth, closed, $H$-stable subvariety. Let $Z' = G \times_H Z \hookrightarrow X$, then one can identify $\NN_{X/Z'} = G \times_H \NN_{Y/Z}$. 

\begin{lem}\label{lem:commutePsi2}
The following diagram commutes
$$
\xymatrix{
\mon{\dd_{X}, G,\chi}_{\mathrm{reg}} \ar[rr]^{\Psi_{X/Z'}} \ar[d]_{\Upsilon^*} & & \mon{\dd_{\NN_{X/Z'}}, G,\chi}_{\mathrm{\reg}} \ar[d]^{(d \Upsilon)^*} \\
\mon{\dd_{Y}, H,\chi |_{\h}}_{\mathrm{reg}} \ar[rr]_{\Psi_{Y/Z}} & & \mon{\dd_{\NN_{Y/Z}}, H,\chi |_{\h}}_{\mathrm{\reg}}.
}
$$
\end{lem}

\begin{proof}
Let $j : Y \hookrightarrow G \times Y$ and $\pi : G \times Y \to X$ so
that $\Upsilon = j \circ \pi$. Since the map $\pi$ is smooth, Lemma
\ref{lem:etalespecial} implies that $\Psi' \circ \pi^* = (d \pi)^* \circ
\Psi_{X/Z'}$, where $\Psi'$ is specialization with respect to
$\pi^{-1}(Z') = G \times Z \subset G \times Y$. The proof of Proposition
\ref{prop:monoequiv} shows that $\pi^* \mm \simeq \mc{O}_G^{\chi}
\boxtimes (\Upsilon^* \mm)$, for $\mm \in \mon{\dd_{X},
  G,\chi}_{\mathrm{reg}}$. Then, as in the proof of Lemma
\ref{lem:specialmono}, we have $\Psi'(\mc{O}_G^{\chi} \boxtimes
(\Upsilon^* \mm)) = \mc{O}_G^{\chi} \boxtimes \Psi_{Y/Z}(\Upsilon^*
\mm)$. 
We have:  $d \Upsilon = (d j) \circ (d \pi)$ where the map $d j :
\NN_{Y/Z} \hookrightarrow G \times \NN_{Y/Z}$ is  given by $(dj)(v) =
(e,v)$. Thus, we find
\begin{align*}
&(dj)^* \Psi'(\mc{O}_G^{\chi} \boxtimes (\Upsilon^* \mm)) \ = \  (dj)^* [ \mc{O}_G^{\chi} \boxtimes \Psi_{Y/Z}(\Upsilon^* \mm) ] \  = \  \Psi_{Y/Z}(\Upsilon^* \mm).\\
&(d \Upsilon)^* \Psi_{X/Z'}(\mm) \  = \  (d j)^* \Psi'(\mc{O}_G^{\chi} \boxtimes (\Upsilon^* \mm))  \ = \  \Psi_{Y/Z}(\Upsilon^* \mm),
\end{align*}
and the commutativity of the diagram of the lemma follows.
\end{proof}

\subsection{Specialization for mirabolic modules}\label{jj2}

Let $L \subset G$ be a Levi subgroup and $W_L\sset W$
an associated  parabolic subgroup. Let $\TSL_L$, resp. $\TSL^{\circ}$, be the set of all points in $\TSL$ whose stabilizer equals $W_L$ resp. is conjugate to a subgroup of $W_L$. Thus, $\TSL_L$ is a closed subset of $\TSL^{\circ}$. 

The normalizer of $W_L$ in $W$ is denoted $N_L$. It maps $T_L$ into itself. Let $\fc$, resp. $\mggW$, denote the images of $\TSL_L$, resp. $\TSL^{\circ}$, under the quotient map. Thus, $\fc$ is a closed subset of $\mggW$.

We refer the reader  to table \eqref{not} for the notation used in this section. 
The following varieties will play an important role below:
$$\Y:=\SSL \cdot (Z(\LSL)^{\circ} \times \{0\})
\cong \SSL \times_{N_{\SSL}(\LSL)} (Z(\LSL)^{\circ} \times \{0\}),\quad
\X^\circ:=\SSL \cdot \X^\circ_L\cong \SSL \times_{N_{\SSL}(\LSL)}\X^\circ_L.
$$
Thus, $\Y=\X(L,\Omega)$ is a $G$-stable locally closed
stratum associated with the $L$-orbit through the point
$\{1\}\times\{0\}\in \LSL\times V$, c.f. \S\ref{stratdef}.

Let $\varpi$ denote the composite map $\X=\SSL\times V \to \SSL\to
\SSL /\!/G\cong\TSL / W$.
The group $G = GL(V)$ acts on $\X$ and
we   let   $\X/\!/G$ 
denote the categorical quotient.
By \cite[Lemma 3.2.1]{CherednikCurves}, 
the map $\varpi$ induces an isomorphism $\X/\!/G\iso \TSL/W$.
Thus, one has a diagram
\[\X\stackrel{\varpi}\longrightarrow  \X/\!/G=\TSL/W \stackrel{\pi}{\longleftarrow} \TSL.\]
We observe that $\X^{\circ}= \varpi^{-1}(\TSL^{\circ} /W)$, and $\Y$  is \textit{properly} contained in $\varpi^{-1}(\fc)$. So, there is a commutative diagram
\beq{eq:bigmapdiag}
\xymatrix{
\mc{Y}\  \ar@{^{(}->}[rrr] \ar[dr]^{\varpi} && & \  \Xo \  \ar[dr]^{\varpi}& \\
& \fc \   \ar@{^{(}->}[rrr] && & \  \mggW \  \\
Z(\LSL)^{\circ} \  \ar@{^{(}->}[rrr] \ar@{^{(}->}[uu] \ar[ur]_{\varpi} && &\   \X^{\circ}_{\LSL} \  \ar@{^{(}->}[uu]|!{[u];;[u]}\hole \ar[ur]_{\varpi} & 
}
\eeq

\begin{rem}
The sets $\Xo$ and $\mggW$ are \textit{not} affine. However, the sets
$\Y$ and $Z(\LSL)^\circ$ are affine. If we decompose $\g = \fl \oplus
\mf{p}$ as a $Z(\LSL)$-module and further decompose $\mf{p} =
\bigoplus_{\alpha} \mf{p}_{\alpha}$, for appropriate characters $\alpha
: Z(\LSL) \rightarrow \C^\times$, then $Z(\LSL)^\circ = \{ z \ | \
\prod_{\alpha} (\alpha(z) - 1) \neq 0 \}$ is affine. The group $N_{\SSL}(\LSL)$ being reductive, it follows from Lemma \ref{lem:Giso}
below that $\Y$ is also affine.
\end{rem}

%Let $\SSL \cdot Z(\LSL)^{\circ}$ be the
%$\SSL$-saturation of the set $Z(\LSL)^{\circ} \subset \LSL$ under the adjoint $\SSL$-action. This is a smooth, locally closed subvariety of $\SSL$. Specifically, writing $N_{\SSL}(\LSL)$ for the normalizer of $\LSL$ in $\SSL$, one has an $\SSL$-equivariant isomorphism $\SSL \cdot Z(\LSL)^{\circ} \cong \SSL \times_{N_{\SSL}(\LSL)} Z(\LSL)^{\circ}$. Observe further that the normal bundle to the subvariety $\SSL \cdot Z(\LSL)^{\circ} $ is an $\SSL$-equivariant vector
%bundle which isomorphic to the vector bundle
%$$
%\SSL\times_{N_{\SSL}(\LSL)}( Z(\LSL)^{\circ} \times \fl' ) \to \SSL\times_{N_{\SSL}(\LSL)} Z(\LSL)^{\circ}.
%$$

We will use simplified notation $\NN_L$ for
the normal bundle to $Z(\LSL)^{\circ}$ in $\X^{\circ}_{\LSL}$ by $\NN_L$. Explicitly, $\NN_L = Z(\LSL)^{\circ} \times (\fl'\times V)$, is a (trivial) vector bundle over $Z(\LSL)^{\circ}$. The group $N_G(L)$ acts on $\NN_L$, making it a $N_G(L)$-vector bundle. 
The following result is clear.

\begin{lem}\label{lem:Giso} 
The following diagram commutes
%\beq{jj_diag}
$$
\xymatrix{\NN_L\  \ar@{^{(}->}[rr]^{\bi} \ar[d] &&\  
\SSL \times_{N_{\SSL}(\LSL)} \NN_L\    \ar[d]
\ar@{=}[rr]\ar[d] &&\  
 \NN_{\X/\Y} \    \ar[d]\\
 Z(\LSL)^{\circ}\   \ar@{^{(}->}[rr] &&\  \SSL \times_{N_{\SSL}(\LSL)} (Z(\LSL)^{\circ} \times \{0\})\   \ar@{=}[rr] && \  \Y
}
$$
%\eeq
\end{lem} 
%Lemma \ref{lem:Giso} gives rise to the commutative diagram
%\beq{eq:bigmapdiag3}
%\xymatrix{
%\NN_{\X / \mc{Y}} \ar@{->>}[r] & \mc{Y} \\
%\NN_{Z} \ar@{^{(}->}[u]^{\bi} \ar@{->>}[r] & Z(\LSL)^{\circ} \ar@{^{(}->}[u]
%}
%\eeq

Let $\Lambda=\NN^*_L$ be (the total space of) the
conormal to $Z(\LSL)^{\circ}$ in $\X^{\circ}_{\LSL}$.
In general, there are canonical isomorphisms
\beq{vphi} T^*(\NN_L) \cong T^*(\NN^*_L) \cong T_{T^*(\X^{\circ}_{\LSL})/\Lambda},
\eeq
 cf. sections 5 and 7 of  \cite{CharCycles}.
In our case, we have $\Lambda =\NN^*_L\cong  Z(\LSL)^{\circ} \times (\mf{l}')^* \times V^*$ and
each of the varieties in \eqref{vphi} is naturally isomorphic to 
$T^*(Z(\LSL)^{\circ}) \times  \mf{l}' \times \mf{l}' \times V \times V^*$.

Recall that, for any  algebraic cycle $S\subset T^* \Xo_{\LSL}$, there is an algebraic cycle
 $C_{\Lambda}(S)$ in $\NN_{T^*(\X^{\circ}_{\LSL})/\Lambda}$, the normal cone of $S$ at $\Lambda$.
It is defined as follows. Let $I_{\Lambda}$ and $I_S$ be the ideals in $\C[T^* \Xo_{\LSL}]$
 defining $\Lambda$ and $S$, respectively.
The $I_{\Lambda}$-adic filtration on
the algebra $\C[T^* \Xo_{\LSL}]$ induces a filtration
$I_S\cap I_{\Lambda}^k,\ k\geq 0$, on $I_S$.
Let 
\[\gr_\La\C[T^* \Xo_{\LSL}]:=\oplus_{k\geq 0}\ I_{\Lambda}^k/
I_{\Lambda}^{k+1},
\quad\text{resp.}\quad
\gr_\La (I_S):=\oplus_{k\geq 0}\ (I_S\cap I_{\Lambda}^k)/(I_S\cap
I_{\Lambda}^{k+1})
\]
be  associated graded spaces. There is a canonical graded algebra isomorphism
$\C[\NN_{T^*(\X^{\circ}_{\LSL})/\Lambda}]=\gr_\La\C[T^* \Xo_{\LSL}]$.
Using this isomorphism one can identify $\gr_\La (I_S)$
with a homogeneous ideal of the algebra $\C[\NN_{T^*(\X^{\circ}_{\LSL})/\Lambda}]$.

One defines the cycle $C_{\Lambda}(S)$ as the support cycle
of $\C[\NN_{T^*(\X^{\circ}_{\LSL})/\Lambda}]/\gr_\La (I_S)$,
viewed as
an $\C[\NN_{T^*(\X^{\circ}_{\LSL})/\Lambda}]$-module.
We let $\varphi^*C_{\Lambda}(S)$ be the pull-back of the cycle $C_{\Lambda}(S)$
via the composite isomorphism $\varphi:\ 
T^*(\NN_L) \iso T_{T^*(\X^{\circ}_{\LSL})/\Lambda}$
in \eqref{vphi}.

%, there is a natural isomorphism $\varphi :  T^* \NN_L \iso T^* \Lambda$, which sends $(x,y,v,w) \in \mf{l}' \times \mf{l}' \times V \times V^*$ to $(y,x,w,v) \in \mf{l}' \times \mf{l}' \times V^* \times V$; see (7.2.1) of \cite{CharCycles} for details. \green{If $S$ is an algebraic cycle on $T^* \Lambda$, then $\varphi^* S$ denotes the pullback of $S$ to $T^* \NN_L$.} For an algebraic cycle $S \subset T^* \Xo_{\LSL}$, it is possible to use the operation of "deformation to the conormal bundle" with respect to $\Lambda$ to produce a cycle $C_{\Lambda}(S)$ in $T^* \Lambda$. We refer the reader to sections 5 and 7 of  \cite{CharCycles} for the details of this operation. 

\begin{lem}\label{lem:inclusion} One has a set-theoretic inclusion
 $\varphi^* C_{\Lambda}(\Nnil(\LSL^{\circ})) \subset \zero \times \Nnil(\mf{l}')$.
\end{lem}

\begin{proof} 
The algebra $\C[\mf{l}]$ is a subalgebra of $\C[T^* \Xo_{\LSL}]$, via the projection $T^* \Xo_{\LSL} \to \mf{l}$. Since $\mf{l}$ decomposes as $\mf{l}' \oplus \mf{z}_{\mf{l}}$, the subspaces $\C[\mf{l}']^L_+$ and $\C[\mf{z}_{\mf{l}}]_+$ are contained in the ideal of definition of $\Nnil(\LSL^{\circ})$. We have $\C[\mf{l}']^L_+ \cap I_{\Lambda}^{k} = 0$ unless $k \le 0$ (in which case $I_{\Lambda}^{k} = \C[T^* \Xo_{\LSL}]$). Therefore the restriction of the $I_\Lambda$-adic filtration
on $\C[T^* \Xo_{\LSL}]$  to $\C[\mf{l}']^L_+$ is  the filtration whose degree $k$ part is  $(\C[\mf{l}']^L_+)_{\ge k}$, the sum of all homogeneous parts of degree at least $k$. Hence $\gr_\Lambda \C[\mf{l}']^L_+ = \C[\mf{l}']^L_+$. On the other hand, $\C[\mf{z}_{\mf{l}}]_+ \cap I_{\Lambda}^{k} = (\C[\mf{z}_{\mf{l}}]_+)^k$. Since $\C[\mf{z}_{\mf{l}}]$ is graded by putting $\mf{z}_{\mf{l}}^*$ in degree one, one can check that this implies that 
$$
V_k \C[\mf{z}_{\mf{l}}]_+ = (\C[\mf{z}_{\mf{l}}]_+)^k = \C[\mf{z}_{\mf{l}}]_{\ge k}.
$$
Again, we see that the associated graded of $\C[\mf{z}_{\mf{l}}]_+$ is naturally identified with itself.  

Since $Z(\LSL)^{\circ}$ is $L$-stable, the vector fields $\mu(\mf{l}(\g))$ are tangent to $Z(\LSL)^{\circ}$. Therefore, $\mu(\mf{l}(\g)) \subset I_\Lambda$ and hence the $\dim \mf{l}(\g)$ many functions defining the equation $g Y g^{-1} - Y + v \circ w$ belong to $I_\Lambda$. We calculate the image of these functions in $I_\Lambda/I_\Lambda^2$. Let $z \in Z(\LSL)^{\circ}$ and $(X,v) \in \mf{l}' \times V$, which is the fiber of $\NN_{Z}$ at $z$. Then,
$$
z (1 + \epsilon X) Y (1 - \epsilon X) z^{-1} - Y  + (\epsilon v) \circ w = \epsilon ([X,Y] + v \circ w). 
$$
Thus, the functions defining $[X,Y] + v \circ w$ belong to the associated graded of the ideal 
in $ \C[T^* \Xo_{\LSL}]$ generated by the functions defining $g Y g^{-1} - Y + v \circ w$. Hence, we have shown that $\varphi^* C_{\Lambda}(\Nnil(\LSL^{\circ}))$ is contained in closed subvariety of $T^* \NN_L$ defined by the vanishing of all functions in $\C[\mf{l}']^{L}_+$,\ $\C[\mf{z}_{\mf{l}}]_+$ and $[X,Y] + v \circ w$. This closed set is precisely $\zero \times \Nnil(\mf{l}')$.
\end{proof}

\subsection{} Mimicing \cite{GG}, we introduce the following

\begin{defn}\label{defn:ccZdefn}
Let $\ss_{\NN_L,c}$ be the abelian category of regular holonomic $\dd$-modules $\mm$ on $\NN_L$ that are $(N_G(L),c)$-monodromic, such that 
\begin{enumerate}
\item The action of $\mathsf{eu}_{\mf{l}'}$ on $\Gamma(\NN_L,\mm)$ is locally finite. 
\item $\Char(\mm) \subseteq \zero \times \Nnil(\mf{l}')$. 
\end{enumerate}
\end{defn}

The fact that $\mm \in \ss_{\NN_L,c}$ is $\Cs$-monodromic with respect to the $\Cs$-action on $\mf{l}' \subset \NN_L$ means that one can use the Fourier transform, as in \cite[Proposition 5.3.2]{GG}, to show that the action of $\sym (\mf{l}')^L_+ \subset \dd(\mf{l}')$ on $\Gamma(\NN_L,\mm)$ is locally nilpotent.

Associated with the subvariety $\Y \hookrightarrow \X^{\circ}$, is the $V$-filtration on $\dd_{\Xo}$. Since $\Y$ is $G$-stable, each piece $V_k \dd_{\Xo}$ of the $V$-filtration on $\dd_{\Xo}$ is a $G$-equivariant subsheaf. Hence $G$ acts on $\gr_V \dd_{\Xo}$, such that the identification $\gr_V \dd_{\Xo} \simeq
\dd_{\NN_{\X/\Y}}$ is $G$-equivariant. Thus, we have the exact specialization functor $\Psi_{\X / \Y}$, from the category of regular holonomic $\dd$-modules on $\X$, to the category of regular holonomic, $\Cs$-monodromic $\dd$-modules on $\NN_{\X/\Y}$. 

Denote by $\mon{\dd_{\X}, G,c}_{\mathrm{reg}}$, resp. $\mon{\dd_{\NN_{\X/\Y}}, G,c}_{\mathrm{\reg}}$, the category of regular holonomic, $(G,c)$-monodromic $\dd$-modules on $\X$, resp. on $\NN_{\X/\Y}$, that have regular singularities.  

We have the following
functors

\beq{eq:biequiv}
\xymatrix{
\mon{\dd_{\X}, G,c}_{\mathrm{reg}} \ar[rr]^<>(0.5){\Psi_{\X / \Y}} &&
\mon{\dd_{\NN_{\X/\Y}}, G,c}_{\mathrm{\reg}}
\ar[r]^<>(0.5){\bi^*}_<>(0.5){\cong} & \mon{\dd_{\NN_L}, N_G(L),c}_{\mathrm{reg}}.
}
\eeq
Here, the functor  $\Psi_{\X / \Y}$ is exact and its
 image is contained in the full subcategory of $\Cs$-monodromic modules, by
Lemma \ref{lem:specialmono}.
The second functor, $\bi^*$, is a  pull-back via the closed embedding $\bi$
in the diagram of Lemma \ref{lem:Giso}; this functor is an equivalence that preserves $\Cs$-monodromicity, thanks to Proposition \ref{prop:monoequiv}.

\begin{defn}
The \textit{specialization functor} $\sp_L$, from $\mon{\dd_{\X}, G,c}_{\mathrm{reg}}$ to the category of $\Cs$-monodromic modules in $\mon{\dd_{\NN_L}, N_G(L),c}_{\mathrm{reg}}$ is defined to be $\sp_L := \bi^* \circ \Psi_{\X/\Y}$.
\end{defn} 

\begin{prop}\label{prop:PsiLcatO}
The functor $\sp_L$ restricts to an exact functor $\ss_{c} \to \ss_{\NN_L,c}$. 
\end{prop}

\begin{proof}
Let $\mm$ in $\ss_{c}$. Then we know that $\sp_L(\mm)$ is $(N_G(L),c)$-monodromic, has regular singularities and is $\Cs$-monodromic. However, this is $\Cs$-monodromic with respect to the action of $\Cs$ on $\mf{l}' \times V$ by dilations and not, as in the definition of $\ss_{\NN_L,c}$, with respect to the action of $\Cs$ by dilations on $\mf{l}'$. Recall that the image of $\mathbf{1} \in \mf{l}(\g)$ under the moment map is $- \eu_V$. Since $\sp_L(\mm)$ is $(N_G(L),c)$-monodromic, this implies that $\eu_V$ acts locally finitely on $\Gamma(\NN_L,\sp_L(\mm))$. Then, the fact that $\eu_{\mf{l}'} = \eu_{\mf{l}' \times V} - \eu_{V}$, with $\eu_{\mf{l}' \times V}$ and $\eu_{V}$ commuting, locally finite endomorphisms of $\Gamma(\NN_L,\sp_L(\mm))$, implies that $\eu_{\mf{l}'}$ also acts locally finitely. Therefore, we just need to show that the singular support $\Char (\sp_L(\mm))$ of $\sp_L(\mm)$ is contained in $Z(\LSL)^{\circ} \times \Nnil(\mf{l}')$. 

One can show, using Lemma \ref{lem:etalespecial} applied to the map $p$ of Lemma \ref{lem:strongetale} and then using Lemma \ref{lem:commutePsi2}, that 
\beq{eq:psiequal}
\sp_L = \Psi_{\Xo_{\LSL} / Z(\LSL)^{\circ}} \circ \Upsilon^*. 
\eeq
We can now use (\ref{eq:psiequal}) to calculate $\Char (\sp_L(\mm))$. Since $\Upsilon$ is non-characteristic for modules in $\ss_c$, the singular support of $\Upsilon^* \mm$ equals $\rho_{\Upsilon}(\omega_{\Upsilon}^{-1}(\Char \mm))$. Then, \cite[Theorem 7.1]{CharCycles} implies that  
\beq{eq:SSformula}
\Char (\sp_L(\mm)) = \varphi^* C_{\Lambda}(\rho_{\Upsilon}(\omega^{-1}_{\Upsilon} (\Char(\mm))).
\eeq
Since the singular support $\Char(\mm)$ of $\mm$ is contained in $\Nnil(\SSL)$, formula (\ref{eq:SSformula}) together with Proposition \ref{prop:wonderfulinclusion} and Lemma \ref{lem:inclusion} imply that the singular support of $\sp_L(\mm)$ is contained in $\zero \times \Nnil (\mf{l}')$, as required.  
\end{proof}

The results \cite[Corollary 7.5.1]{CharCycles} and \cite[Corollary 7.5.2]{CharCycles} imply:

\begin{cor}
Let $\mm \in \ss_c$ and set $\Char_0 (\mm) := \rho_{\Upsilon}(\omega_{\Upsilon}^{-1} (\Char(\mm)))$.  

\vi If $\Char_0 (\mm) \cap \Lambda \neq \emptyset$ then $\sp_L(\mm) \neq 0$. 

\vii $\Supp \sp_L(\mm) = \varphi^{-1}(\Lambda \cap \Char_0 (\mm))$. 
\end{cor} 

\begin{rem} Let $\ss_{\mf{l}',c}$ be the category of all $(N_G(L),c)$-monodromic $\dd$-modules on $\mf{l}' \times V$ satisfying conditions (1) and (2) of definition \ref{defn:ccZdefn}. For $\mm$ in $\ss_{c}$, the fact that the singular support of $\sp_L(\mm)$ is contained in $\zero \times \Nnil (\mf{l}')$ means that $\sp_L(\mm)$ behaves as a local system along the base $Z(\LSL)^{\circ} \subset \NN_L$. Therefore, letting $i_b : \mf{l}' \times V \hookrightarrow \NN_L$ denote the embedding of the fiber at $b \in Z(\LSL)^{\circ}$, one can define a functor $\sp_{L,b} : \ss_{c} \to \ss_{\mf{l}',c}$ by setting $\sp_{L,b} := i_b^* \circ \sp_L$. The fact that $\sp_L(\mm)$ is non-characteristic for $i_b$ implies that $\sp_{L,b}$ is an exact functor, commuting with Verdier duality. 
\end{rem}
\subsection{An adjoint to $\sp_L$} Let $i : \mc{Y} \hookrightarrow \X$ be the locally closed embedding. We let $\dd^V_{\X}$ denote the completion of $\dd_{\X}$ with respect to the $V$-filtration. It is a sheaf of algebras on $\mc{Y}$ and we have a canonical map $i^{-1} \dd_{\X} \rightarrow \dd^V_{\X}$. Therefore adjunction gives a map of sheaves of algebras
$$
\dd_{\X} \rightarrow i_{\idot} i^{-1} \dd_{\X} \rightarrow i_{\idot} \dd^V_{\X}
$$
and hence the space $\Gamma(\X,i_{\idot} \mm)$ is a $\ddd := \Gamma(\X,\dd_{\X})$-module for any $\dd^V_{\X}$-module $\mm$. Let $\dd^V_{\NN_{\X / \mc{Y}}}$ denote the completion of $\dd_{\NN_{\X / \mc{Y}}}$ with respect to its $\Z$-grading. By construction of $V$-filtration we have an identification $\dd^V_{\X} = \dd^V_{\NN_{\X / \mc{Y}}}$ of filtered algebras on $\mc{Y}$. 

Let $\ms{N}$ be a module in $\ss_{\NN_L,c}$. Recall from (\ref{eq:biequiv}) that $\bi^*$ defines an equivalence between $\mon{\dd_{\NN_{\X/\Y}}, G,c}_{\mathrm{reg}}$ and $\mon{\dd_{\NN_L}, N_G(L),c}_{\mathrm{reg}}$. Let $\mathbf{I}( - )$ denote a quasi-inverse to $\bi^*$. The grading on $\dd_{\NN_{\X / \mc{Y}}}$ is inner in the sense that there exists an Euler vector field $\eu$ such that a section $D \in \dd_{\NN_{\X / \mc{Y}}}$ has degree $m$ if and only if $[\eu,D] = m D$. This implies:

\begin{lem}
Completion $( - )^V$ with respect to the grading defines an equivalence between coherent, monodromic $\dd_{\NN_{\X / \mc{Y}}}$-modules and coherent $\dd^V_{\NN_{\X / \mc{Y}}}$-modules. 
\end{lem}

A quasi-inverse to completion is given by taking the submodule of $\eu$-locally finite sections, or equivalently by applying $\gr_V$ with respect to the natural filtration on $\mm^V$. Thus, beginning with a module $\ms{N}$ in $\ss_{\NN_L,c}$, we have a $\dd^V_{\NN_{\X / \mc{Y}}} (= \dd^V_{\X})$-module $\mathbf{I}(\ms{N})^V$. 

\begin{defn}
Let $\Ind_L^G  : \ss_{\NN_L,c} \rightarrow \LMod{\ddd}$ be the functor that sends $\ms{N}$ to the space of $\FZ$-locally finite elements in $\Gamma(\X, i_{\idot} \mathbf{I}(\ms{N})^V)$.  
\end{defn}

We identify $\LMod{\dd_{\X}}$ with $\LMod{\ddd}$. Let $\Ind \ \ss_c(\SSL)$ denote the full subcategory of $\LMod{\dd_{\X}}$ consisting of all modules that are an inductive limit of modules in $\ss_c(\SSL)$. Based on Theorem \ref{thm:admissible2}, it is clear that the image of $\Ind_L^G$ belongs to $\Ind \ss_c(\SSL)$ provided we can show that each $\Ind_L^G$ is $(G,c)$-monodromic. The proof of Lemmata \ref{lem:etalespecial} and \ref{lem:specialmono} can be adapted to show that $\mathbf{I}(\ms{N})^V$ is a $(G,c)$-monodromic $\dd_{\X}^V$-module. Hence there is a locally finite action of $G$ on $\Ind_G^L \ms{N}$ such that differential of this action differs from the action of $\mu(\g)$ by $c \Tr$. Thus, $\Ind_G^L \ms{N}$ is an inductive limit of $(G,c)$-monodromic modules.   

\begin{prop}
The functor $\Ind_L^G$ is right adjoint to $\sp_{L}$.  
\end{prop}

\begin{proof}
Let $\mm \in \ss_c(\SSL)$ and $\ms{N} \in \ss_{\NN_L,c}$. We wish to show that 
\beq{eq:toshow1}
\Hom_{\dd_{\X}}(\mm, \Ind_L^G \ms{N}) = \Hom_{\dd_{\N_Z}}(\sp_{L} (\mm), \ms{N}). 
\eeq
If $\mm \in \ss_c(\SSL)$, then any homomorphism $\Gamma(\X,\mm) \rightarrow \Gamma(\X , i_{\idot} \mathbf{I}(\ms{N})^V)$ of $\ddd$-modules is going to factor through $\Gamma(\X,\Ind_L^G \mm)$. Thus, 
$$
\Hom_{\dd_{\X}}(\mm, \Ind_L^G \ms{N}) \ =\  \Hom_{\dd_{\X}}(\mm,i_{\idot} \mathbf{I}(\ms{N})^V))\  
  =\  \Hom_{i^{-1} \dd_{\X}}(i^{-1} \ms{N}, \mathbf{I}(\ms{N})^V).
$$
Since $\mathbf{I}(\ms{N})^V$ is complete with respect to the $V$-filtration, any morphism of $i^{-1} \dd_{\X}$-modules into $\mathbf{I}(\ms{N})^V$ extends uniquely to a $\dd^V_{\X}$-module morphism $\mm^V \rightarrow \mathbf{I}(\ms{N})^V$ i.e. 
\beq{eq:almostthere}
\Hom_{i^{-1} \dd_{\X}}(i^{-1} \mm,\mathbf{I}(\ms{N})^V) = \Hom_{\dd^V_{\X}}(\mm^V,\mathbf{I}(\ms{N})^V).
\eeq
Equality (\ref{eq:almostthere}) implies (\ref{eq:toshow1}) because $\mathbf{I}$ and $( - )^V$ are equivalences with quasi-inverses $\bi^*$ and $\gr_V$ respectively.  
\end{proof}

\begin{rem}
We expect, but are unable to prove, that the functor $\Ind_L^G$ commutes
with Verdier duality.
\end{rem}

\subsection{} In the rational case, where $\X = \mf{sl}(V) \times V$, the definition of $\Ind_L^G$ is essentially the same except that one should take $(\sym \mf{sl}(V))_+$ locally nilpotent \textit{and} $\eu_{\mf{sl}}$-locally finite vectors in $\Gamma(\X, i_{\idot}\mathbf{I}(\ms{N}))^V$, where $\eu_{\mf{sl}}$ is the Euler vector field along $\mf{sl}(V)$ in $\X$. In this case, one begins with mirabolic modules on $\NN_{\fz} := \fz_{\mf{l}}^{\circ} \times (\mf{l}' \times V)$. 

\begin{prop}\label{prop:fgInd}
Induction is a functor $\Ind_L^G : \ss_{\NN_{\fz},c} \rightarrow \ss_c(\mf{sl}(V))$ i.e. $\Ind_L^G \mm$ is finitely generated.
\end{prop}

The proof of Proposition \ref{prop:fgInd} will appear elsewhere. 

\section{Specialization of $\dd$-modules vs specialization for Cheredink algebras}\label{10}
The aim of this section is to show that the functor of Hamiltonian reduction is compatible with the specialization functor.

\subsection{Comparison of $V$-filtrations} The spherical Cherednik algebra $\htrig_{\kappa}$ localizes to a sheaf $\mc{U}_{\kappa}=e {\mathcal H}_{\kappa}(T,W) e$ on $\TSL/ W$. 
Given  a Levi subgroup $L \subset G$,
there are two ways to define a "$V$-filtration" on $\mc{U}_{\kappa}$ with respect to the locally closed subvariety $\TSL_L$ of $\TSL$, where 
 $W_L$ is the  parabolic subgroup of $W$ associated to $L$. 
The first way is to  use the $V$-filtration on the sheaf ${\mathcal H}_{\kappa}(T,W)$
coming from the closed embedding $\TSL_L \hookrightarrow \TSL^{\circ}$. This $V$-filtration is $W$-stable. Therefore it gives, by restriction,  a natural $V$-filtration $V_{\idot} \mc{U}_{\kappa}$ on the spherical subalgebra.

On the other hand, one can 
 use the $V$-filtration on $\dd_{\Xo}$ with respect to 
the  locally closed subvariety $\mc{Y}\sset \X$,
c.f.  section \ref{jj2}. We equip $\dd_{\X^{\circ}}  / (\dd_{\X^{\circ}}
\cdot \g_c)$
with the quotient filtration. The latter is $G$-stable,
hence restricting to $G$-invariants, gives
 a  well-defined filtration 
$$
\nV_{k} \mc{U}_{\kappa} := [\varpi_{\idot} (V_{k} (\dd_{\X^{\circ}}/ \dd_{\X^{\circ}} \cdot \g_c)) ]^G
$$
on $(\varpi_{\idot} \dd_{\Xo})^G$. Therefore, via the radial parts map, we get a second filtration on $\mc{U}_{\kappa}|_{\mggW}$.

We are going to show that  these two filtrations are equal, specifically, we have

\begin{thm}\label{thm:filtrationsagree}
The radial parts map identifies the $\mc{V}$-filtration on $(\varpi_{\idot} \dd_{\X^{\circ}})^G  / (\varpi \dd_{\X^{\circ}} \cdot \g_c)^G$ with the $V$-filtration on $\mc{U}_{\kappa}|_{\mggW}$. 
\end{thm}

The rest of the section is devoted to the proof of the theorem.
%Let $\mc{J}$ be the sheaf of ideals defining the smooth, closed subvariety $\mc{Y}$ in $\X^{\circ}$. For each $m \ge 0$, define $\mc{J}_m = (\varpi_{\idot} \mc{J}^m)^{G}$, a sheaf of ideals in $\mc{O}_{\mggW}$. As in section \ref{sec:notation}, let $\mc{I}$ be the sheaf of ideals in $\pi_{\idot} \mc{O}_{\TSL^{\circ}}$ defining $W \cdot \TSL^{W_L}_{\circ}$ and set $\mc{I}_m = (\mc{I}^m)^W$.  

\subsection{}\label{sec:comparison} 

%Recall that $\X^{\circ} = \SSL \cdot (\LSL^{\circ} \times V)$, an open subset of $\X$, where $\LSL^{\circ} = \{ l \in \LSL \ | \ Z_G(l) \subset L \}$. We set $Z(\LSL)^{\circ} = Z(\LSL) \cap \LSL^{\circ}$ and let $\mc{Y} = \SSL \cdot Z(\LSL)^{\circ} \subset \X^{\circ}$. The quotient map $\X \rightarrow \TSL / W$ is denoted $\varpi$ so that $\X^{\circ} = \varpi^{-1}(\mc{U})$, where $\mc{U} = \TSL^{\circ} / W$ and $\TSL^{\circ}$ is the set of all points in $\TSL$ whose stabilizer is conjugate to a subgroup of $W_L$. Then $\nV = \varpi(\mc{Y})$ is the image under $\pi : \TSL \rightarrow \TSL / W$ of $\TSL^{W_L}_{\circ}$, the set of all points whose stabilizer equals $W_L$. 

Let $X$ be a smooth affine variety, $Y$ a smooth, closed subvariety and $I$ the ideal in $\C[X]$ defining $Y$. Let $\theta$ be a vector field in $V_0 \dd(X) \subset \dd(X)$ that is a lift of the Euler vector field on $\NN_{X/Y}$. For each $f \in I^m$, we have $\theta(f) - m f \in I^{m+1}$, which implies that there is a well-defined adjoint action of $\theta$ on $\dd(X) / I^m \dd(X)$. 

\begin{lem}\label{lem:eulocallyfinite}

\vi The adjoint action of $\theta$ on $\dd (X) / I^m \dd (X)$ is locally finite.

\vii $\bigcap_{m} I^m \dd(X) = 0$.  

\end{lem}

\begin{proof}
Let $D \in \dd(X)$.  The $V$-filtration on $\dd(X)$ is exhaustive, therefore there exists some $k$ such that $D \in V_k$. Independent of the choice of lift $\theta$, we have $[\theta , D] - k D \in V_{k-1}$. Hence, by induction on $k$, there exists a non-zero polynomial $P$ such that $P(\ad(\theta))(D) \in V_{-m}$. In particular, $P(\ad(\theta))(D)(\C[X]) \subset I^n$ which implies that $P(\ad(\theta))(D) \in I^m \dd(X)$. Hence the action of $\ad(\theta)$ is locally finite. We remark for later that the roots of $P$ are contained in $\{ k,k-1, \ds, -m +1 \}$. 

For each non-zero $D \in \dd(X)$, there exists $f \in \C[X]$ such that $D(f) \neq 0$. Since $\bigcap_{m} I^m = 0$, there is some $m \gg 0$ such that $D(f) \notin I^m$. This implies that $D \notin I^m \dd(X)$.  
\end{proof}

We define an auxiliary filtration $\ak_{\idot} \dd(X)$ on $\dd(X)$ as follows. For each $m \in \mathbb{N}$ and $D \in \dd(X)$, let $P_m$ be the monic generator of the ideal $J_D = \{ Q \in \C[t] \ | \ Q(\ad(\theta))(D)  \in I^m \dd(X) \}$. The set of roots of $P_m$ is denoted $\Spec_m (D)$. We say that $D \in \ak_{\ell} \dd(X)$ if $\max \{ \bigcup_m \Spec_m (D) \} \le \ell$. If $D \in V_k \dd(X)$, then the proof of Lemma \ref{lem:eulocallyfinite} shows that $\ell \le k$. 

\begin{prop}\label{prop:WequalsV}
The filtration $\ak_{\idot} \dd(X)$ equals the $V$-filtration $V_{\idot} \dd(X)$. 
\end{prop}

\begin{proof}
Let $D \in V_k \backslash V_{k-1}$. We have already shown that $D \in \ak_l$ for some $l \le k$. Therefore, we just need to show that $l = k$. Let $\dd^V (X) = \lim_{-\infty \leftarrow k} \dd(X) / V_k \dd(X)$ be the completion of $\dd(X)$ with respect to the $V$-filtration. As noted in \cite[Section 6.8]{CharCycles}, if $\dd^V(X)(k) = \{ D \in \dd^V(X) \ | \ [\theta,D] = k D \}$, then $V_k \dd^V(X)$ is the closure of $\dd^V(X)(k) \oplus \dd^V(X)(k-1) \oplus \ds $ with respect to the grading filtration. The image of $D$ in $\dd^V(X)$ equals $D_k + D_{k-1} + \ds $, where $D_k \in \dd^V(X)(k)$ is non-zero. Now 
$$
P_m(\ad(\theta))(D) = P_m(k) D_k + P_m(k-1) D_{k-1} + \ds \in I^m \dd^V (X).
$$
Since we have $l = k$ if and only if there exists some $m$ such that $P_m(k) = 0$, we need to show that, for $m \gg 0$, there is no element $E \in I^n \dd^V (X) \cap V_k \dd^V (X)$ with $E_k = D_k$. There is a faithful action of $\dd^V (X)$ on $\widehat{\C}[X]_Y$, the completion of $\C[X]$ with respect to $I$ such that $D \in V_k \dd^V (X)$ maps $I^m \widehat{\C}[X]_Y$ into $I^{m-k} \widehat{\C}[X]_Y$. The element $D_k$ also belongs to $V_k \dd^V (X)$. Since it is assumed to be non-zero, there is a function $f \in \widehat{\C}[X]_Y$ such that $D_k(f) \neq 0$. The $I$-adic filtration on $\widehat{\C}[X]_Y$ is seperating. Therefore, there exists some $m$ such that $D_k(f) \notin I^m \widehat{\C}[X]_Y$. Since $I^m \widehat{\C}[X]_Y$ is the closure of $(\widehat{\C}[X]_Y)_m \oplus (\widehat{\C}[X]_Y)_{m+1} \oplus \ds$, and $D_k$ is homogeneous, we may assume that $f \in (\widehat{\C}[X]_Y)_l$ for some $l$. For $E \in I^m \dd^V (X) \cap V_k \dd^V (X)$, we clearly have $E(f) \in I^m \widehat{\C}[X]_Y$. The degree of $D_k(f)$ is $k - l > -m$. On the other hand, $E(f) = E_k(f) + E_{k-1}(f) + \ds$ belongs to $I^m \widehat{\C}[X]_Y$, and the fact that $f$ homogeneous implies that $E_k(f) = 0$. Hence $E_k \neq D_k$ as required. 
\end{proof}

We now return to the setting of \S\ref{jj2}.
Let $\mc{J}$ be the sheaf of ideals defining the smooth, closed subvariety $\mc{Y}$ in $\X^{\circ}$. For each $m \ge 0$, define $\mc{J}_m = (\varpi_{\idot} \mc{J}^m)^{G}$, a sheaf of ideals in $\mc{O}_{\mggW}$. As in section \ref{sec:notation}, let $\mc{I}$ be the sheaf of ideals in $\pi_{\idot} \mc{O}_{\TSL^{\circ}}$ defining $W \cdot \TSL^{W_L}_{\circ}$ and set $\mc{I}_m = (\mc{I}^m)^W$. 
\begin{lem}\label{claim:keyinclusionlemma}
For each $m \ge 1$, we have $\mc{J}_m = \mc{I}_m$ as subsheaves of $\mc{O}_{\mggW}$.  
\end{lem}

\begin{proof}
It suffices to check the equality locally on $\mggW$. If $x \in \mggW \sminus \fc $ then $\mc{J}_m = \mc{I}_m = \mc{O}_{\TSL /W}$ in a neighborhood of $x$. Therefore, we may assume that $x \in \fc $ and $U$ an affine open neighborhood of $x$ in $\mggW$. Since the maps $\varpi$ and $\pi$ are affine, $\varpi^{-1}(U)$ and $\pi^{-1}(U)$ are affine open subsets of $\X^{\circ}$ and $\TSL^{\circ}$ respectively. Let $J = \Gamma(\varpi^{-1}(U),\mc{J})$ and $I = \Gamma(\pi^{-1}(U),\mc{I})$. We assume that $\TSL$ is contained in $\LSL$ so that $\TSL \cap \varpi^{-1}(U) = \pi^{-1}(U)$. % and let $i : \pi^{-1}(U) \hookrightarrow \varpi^{-1}(U)$ be the closed embedding.
%\beq{eq:bigdiag4}
%\xymatrix{
%\varpi^{-1}(U) \ar@{^{(}->}[rr] \ar[dr]^{\varpi} & & \Xo \ar[dr]^{\varpi} & \\
%& \fc  \ar@{^{(}->}[rr] |!{[r];[r]}\hole & & \mggW \\
%\pi^{-1}(U) \ar@{^{(}->}[rr] \ar@{^{(}->}[uu]^{\bj} \ar[ur]_{\varpi} & & \TSL^{\circ} \ar@{^{(}->}[uu] \ar[ur]_{\varpi} & 
%}
%\eeq
Restriction of functions defines a surjection $\bj : \Gamma ( \varpi^{-1}(U), \mc{O}_{\X}) \rightarrow \Gamma(\pi^{-1}(U), \TSL)$. We claim that $\bj(J) = I$. Since $\pi^{-1}(U)$, $\mc{Y}$ and $\pi^{-1}(U) \cap \mc{Y}$ are smooth varieties, this will follow from the fact that the intersection $\pi^{-1}(U) \cap \mc{Y}$ is clean; see \cite[Lemma 5.1]{Li}. That is, we must show that 
$$
T_y \TSL_L = (T_y \TSL) \cap (T_y \mc{Y}), \quad \forall \ y \in \pi^{-1}(U) \cap \mc{Y}.
$$
Let $y \in \pi^{-1}(U) \cap \mc{Y}$. We have 
$$
T_y \TSL_L = \mf{t}_{L}, \quad T_y \TSL = \mf{t}, \quad T_y \mc{Y} = [\mf{z}_{\mf{l}'}, \g] + \mf{z}_{\mf{l}'}.
$$
Decompose $\g = \mf{l} \oplus \mf{l}^{\perp}$ with respect to the trace form. This is a $\mf{z}_{\mf{l}'}$-stable decomposition. Thus, 
$$
[\mf{z}_{\mf{l}'}, \g] + \mf{z}_{\mf{l}'} = [\mf{z}_{\mf{l}'}, \mf{l}^{\perp}] \oplus \mf{z}_{\mf{l}'} \subset \mf{l}^{\perp} \oplus \mf{z}_{\mf{l}'}.
$$
Since $\mf{z}_{\mf{l}'} = \mf{t}_{L}$, it follows that $T_y \TSL_L = (T_y \TSL) \cap (T_y \mc{Y})$. Then $\bj(I^m) = \mc{I}^m$ too.

 The restriction of $\bj$ to $\C[\varpi^{-1}(U)]^G$ is an isomorphism onto $\C[\pi^{-1}(U)]^W$. Therefore 
$$
\bj ((J^m)^G) = \bj(J^m \cap \C[\varpi^{-1}(U)]^G) = I^m \cap \C[\pi^{-1}(U)]^W = (I^m)^W
$$
as required. 
\end{proof}

\begin{lem}\label{lem:theta}
Let $\eu$ be the Euler vector field on $\NN_{\X/ \Y}$ corresponding to the $\Cs$-action along fibers and fix $\mathbf{n} = \frac{n(n+1)}{2}$. 

\vi Let $U \subset \mggW$ be an affine open set. Then, there exists a $G$-invariant vector field $\theta$ in $\Gamma(\varpi^{-1}(U) , V_0 \Theta_{\X})$ lifting $\eu$.

\vii The element $\mathfrak{R}(\theta) + \mathbf{n} c$ in $\Gamma(U,\mc{U}_{\kappa})$ is a local lift of the Euler element $\eu_{\kappa} \in \htrig_{\kappa}(W_L)$, as defined in Lemma \ref{lem:globalEu}.  

\end{lem}

\begin{proof}
\vi The $\Cs$-action on $\NN_{\X/\Y}$ clearly commutes with the action of $G$. Therefore, $\eu$ belongs to $\Gamma (\NN_{\X/\Y}, \Theta_{\NN_{\X/\Y}})^G$. The identification $\gr_V \dd_\X \iso \dd_{\NN_{\X/\Y}}$ restricts to a $G$-equivariant isomorphism $\gr_V \Theta_{\X} \iso \Theta_{\NN_{\X/\Y}}$. Therefore, there exists some section $\eu' \in \Gamma(\NN_{\X/\Y},\gr^0_V \Theta_{\X})^G$ mapping to $\eu$ under this identification. Since the map $\varpi$ is affine, $\varpi^{-1}(U)$ is an affine, $G$-stable open subset of $\X$. Then, the natural map 
$$
\Gamma(\varpi^{-1}(U) , V_0 \Theta_{\X} / V_{-1} \Theta_{\X}) = \Gamma(\varpi^{-1}(U), V_0 \Theta_{\X}) / \Gamma(\varpi^{-1}(U), V_{-1} \Theta_{\X}) \to \Gamma(\varpi^{-1}(U),\gr_V^0 \Theta_{\X})
$$
is an isomorphism. Our assumptions imply that $\Gamma(\varpi^{-1}(U) , V_0 \Theta_{\X})$ and $\Gamma(\varpi^{-1}(U), V_{-1} \Theta_{\X})$ are rational $G$-modules. Thus, we may choose a $G$-invariant lift $\theta$ of $\eu'$ in $\Gamma(\varpi^{-1}(U) , V_0 \Theta_{\X})$. 

\vii It suffices to show that $\mathfrak{R}(\theta)$ acts on $\mc{I}_m /  \mc{I}_{m+1}$ as multiplication by $m - \mathbf{n} c$.  Recall that $\mathfrak{R}(\theta)(f) = [s^{c} \theta(s^{-c} \varpi^*(f))] |_{\TSL^{\reg}}$. First we calculate the smallest $m$ such that $s \in \mc{J}^m \sminus \mc{J}^{m+1}$. Let $\mc{J}_L$ be the sheaf of ideals defining $Z(\LSL)^{\circ} \times \{ 0 \}$ in $\X_{\LSL}^{\circ}$. By Lemma \ref{lem:Giso}, $\mc{J}^m / \mc{J}^{m+1} \simeq (\mc{O}_G \boxtimes \mc{J}^m_L / \mc{J}^{m+1}_L)^{N_G(L)}$. Since $s$ is a $G$-semi-invariant, this implies that $s \in \mc{J}^m \sminus \mc{J}^{m+1}$ if and only if $s_L \in \mc{J}^m_L \sminus \mc{J}^{m+1}_L$, where $s_L$ is $s$ restricted to $\X_{\LSL}^{\circ}$. Let $z \in Z(L)^{\circ}$, $X \in \mf{l}'$ and $u \in V$. Taking $v = \epsilon u$ and $g = z(1 + \epsilon X)$, 
\begin{align}\label{eq:Js}
s_L (g,v) & = \epsilon^n \det( u, z u, \ds, z^{n-1} u) + \epsilon^{n+1}(\cdots) + \cdots \notag \\
 & = \epsilon^{\mathbf{n}} \det(u, z X u, \ds, z^{n-1} X^{n-1} u) + \cdots
\end{align}
since all terms of degree less that $\mathbf{n}$ in $\epsilon$ vanish on $\NN_L$. Hence $s \in J^{\mathbf{n}} \sminus J^{\mathbf{n}+1}$. By part \vi this implies that $\theta(s) - \mathbf{n} s \in J^{\mathbf{n}+1}$. Since $\C[\X]^G s \subset \C[\X]^{\det^{-1}}$ is contained in $\C[\X^{\mathrm{cyc}}]^{\det^{-1}}$, it follows from \cite[Corollary 5.3.4]{BFG} that  $\C[\X]^G s = \C[\X]^{\det^{-1}}$ and hence $\C[\varpi^{-1}(U)]^{\det^{-1}}$ equals $\C[U] s$. Then, the fact that $\theta(s) \in \C[\varpi^{-1}(U)]$ is still a $\det^{-1}$-semi-invariant implies that $\theta(s) = s f$ for some invariant function $f$. Thus, $\theta(s)s^{-1} - \mathbf{n} \in \mc{J}_{\mathbf{n}+1}$. To complete the proof of \vii, we identify $\mc{J}_m = \mc{I}_m$ via Lemma \ref{claim:keyinclusionlemma}. Let $f \in \mc{I}_m$. Since $\theta$ is a derivation, $\mathfrak{R}(\theta)(f) = \theta(f) - c \theta(s) s^{-1} f$. The element $(\theta(f) - c \theta(s) s^{-1} f)$ equals $(m - c\mathbf{n}) f$ modulo $\mc{I}_{m+1}$, hence $\mathfrak{R}(\theta)(\bar{f}) = (m - c \mathbf{n})\bar{f}$ for all $\bar{f} \in \mc{I}_m / \mc{I}_{m+1}$. 
\end{proof}

\subsection{Proof of Theorem \ref{thm:filtrationsagree}}

Since Theorem \ref{thm:filtrationsagree} is a local statement on $\mggW$, we fix once and for all an affine open subset $U$ of $\mggW$. Set 
$$
\htrig = \Gamma(U,\mc{U}_{\kappa}), \quad \ddd = \Gamma(\varpi^{-1}(U), \dd_{\X^{\circ}}). 
$$
By Lemma \ref{lem:theta}, we can choose a $G$-invariant lift $\theta \in \ddd^G$ of the Euler vector field on $\NN_{\X^{\circ} / \mc{Y}}$. If $P(\ad(\theta))(D) \in I^m \ddd$ and $D \in \ddd^G$, then $P(\ad(\theta))(D) \in (I^m \ddd)^G$. Therefore by Proposition \ref{prop:WequalsV}, the $\nV$-filtration on $\ddd^G$ can be defined by considering the action of $\ad(\theta)$ on $\ddd^G / (I^m \ddd)^G$.

The sheaf $\mc{U}_{\kappa}$ is a subsheaf of $\dd_{\TSL/W}$. Thus, $\htrig$ is a subalgebra of $\dd(U)$. The $V$-fltration on $\htrig$ can be defined in terms of its action on $\C[U]$ as follows. By definition, $V_{m} \htrig$ equals $\{ D \in \htrig \ | \ D( \mc{I}^j) \subset \mc{I}^{j - m}, \ \forall \ j \in \Z \}$. Since $D$ is $W$-invariant and $\mc{I}_m$ is defined to be $(\mc{I}^m)^W$, $D \in V_{m} \htrig$ implies that $D(\mc{I}_j) \subset \mc{I}_{j-m}$. On the other hand, if $D \in \htrig$ such that $D(\mc{I}_m) \subset \mc{I}_{m - k}, \ \forall \ m \in \Z$, then since $D = \Res(e D e)$, we have 
$$
e D e ( \mc{I}^m) = e D( \mc{I}_m) \subset \mc{I}_{m - k} \subset \mc{I}^{m-k},
$$
which implies that $e D e \in V_k (e {\mathcal H}_{\kappa}(U,W) e)$ and hence $D \in V_{k} \htrig$.  

 Lemma \ref{lem:theta} (2) implies that there is a well-defined action of $\ad(\mathfrak{R}(\theta))$ on $\htrig / \mc{I}_m \dd(U) \cap \htrig$. Since $\htrig / \mc{I}_m \dd(U) \cap \htrig$ embeds in $\dd(U) / \mc{I}_m \dd(U)$, the proof of Lemma \ref{lem:eulocallyfinite} shows that this action is locally finite. We mimic the earlier construction and, for
$D\in \htrig$ and any $m\in\Z$, let  $\Spec_m (D)$  be the set of roots
of the monic generator of the ideal
$\{Q\in\C[t]\mid Q(\ad\mathfrak{R}(\theta))(D)\in \mc{I}_m \dd(U) \cap \htrig\}$.
Then, we define a filtration $\ak_m \htrig$ on $\htrig$
 by 
\[\ak_\ell \htrig\ :=\ \{D\in \htrig\mid \ \Spec_m(D)\sset(-\infty,\ell+\mathbf{n}c],\enspace \forall m\}.\]
Repeating the proof of
 Proposition \ref{prop:WequalsV}  word for word,
cf. also  the proof of Theorem \ref{prop:specialcoherent},
one obtains the following result.

\begin{prop}\label{prop:UVagree}
The filtrations $V_{k} \htrig$ and $\ak_k \htrig$ on $\htrig$ are equal. \qed
\end{prop}

Propositions \ref{prop:WequalsV} and \ref{prop:UVagree} imply that in order to prove Theorem \ref{thm:filtrationsagree}, it suffices to show that $\mf{R}(\ak_m \ddd^G) = \ak_m \htrig$. Equivalently, we need to show that the filtrations on $\ddd^G$ defined by the adjoint action of $\theta$ on $\ddd^G / I_m \ddd^G$ and $\ddd^G / (I^m \ddd)^G$ are equal. For this it suffices to show that the filtrations $I_m \ddd^G$ and $(I^m \ddd)^G$ of $\ddd^G$ are comparable.  Let $\mc{F}_{\idot} \ddd$ denote the order filtration. Since $\theta$ is a vector field and the order filtration is exhaustive, we may calculate the $\ak$-filtration by considering the action of $\ad(\theta)$ on $(\mc{F}_{\ell} \ddd)^G / (I^m \mc{F}_{\ell} \ddd)^G$. Therefore it will actually suffice to show that the filtrations $I_m (\mc{F}_{\ell} \ddd)^G$ and $(I^m \mc{F}_{\ell} \ddd)^G$ of $(\mc{F}_{\ell} \ddd)^G$ are comparable. 

Thus, for each $m,\ell \ge 1$, we must find $m', m'' \gg 0$ such that 
$$
(I^{m'} \mc{F}_{\ell}  \ddd)^G \subset I_{m} (\mc{F}_{\ell}  \ddd)^G \subset (I^{m''} \mc{F}_{\ell}  \ddd)^G.
$$
We always have $I_{m} (\mc{F}_{\ell}  \ddd)^G \subset (I^{m} \mc{F}_{\ell}  \ddd)^G$ so we just need to show the existence of $m'$. We claim that there is some $m_0$ such that $I_m (I^{m_0} \mc{F}_{\ell}  \ddd)^G = (I^{m_0 + m} \mc{F}_{\ell}  \ddd)^G$ for all $m$ i.e. the filtration $\{ (I^{m} \mc{F}_{\ell}  \ddd)^G \}$ is $\{ I_m \}$-stable. By Artin-Rees theory, this is equivalent to showing the $(\oplus_{m} I_m)$-module $\oplus_m (I^{m} \mc{F}_{\ell}  \ddd)^G$ is finitely generated. Since $\mc{F}_{\ell}  \ddd$ is a finitely generated $\C[\varpi^{-1}(U)]$-module, the module $(\oplus_m I^{m} \mc{F}_{\ell}  \ddd)$ is finitely generated over the Rees algebra $(\oplus_m I^m)$. Therefore, Hilbert's Theorem implies that $(\oplus_m I^{m} \mc{F}_{\ell}  \ddd) = \oplus_m (I^{m} \mc{F}_{\ell}  \ddd)^G$ is finitely generated over $(\oplus_{m} I^m)^{G} = \oplus_{m} I_m$, as required. Since $(I^{m_0} \mc{F}_{\ell}  \ddd)^G \subset (\mc{F}_{\ell} \ddd)^G$ we may take $m'  = m_0 + m$. This completes the proof of Theorem \ref{thm:filtrationsagree}.

\subsection{}\label{sec:8.5} Recall that $N_L$ is the normalizer of $W_L$ in
$W$; we have $N_L / W_L \simeq N_G(L) / L$. Decompose $\mf{t} = \mf{t}^{W_L} \oplus \mf{t}_L$ as a $N_L$-module. All the reflections in $N_L$ are contained in $W_L$ and form a single conjugacy class. The group $N_L$ acts on $T_L \simeq Z(\LSL)^{\circ}$ and $T_L 
\times \mf{t}_{L}$. Therefore we may form the Cherednik algebra
${\mathcal H}_{\kappa}(\TSL_L \times \mf{t}_L, N_L)$, as in section
\ref{sec:Cheredefn}. It is a sheaf on $(\TSL_L \times \mf{t}_{L}) / N_L$ and its global sections $\H_{\kappa}(\TSL_L \times \mf{t}_L, N_L)$ contains $\dd(\TSL_L) \o \H_{\kappa}(\mf{t}_L,W_L)$ as a subalgebra. Let $e_N$ be the trivial idempotent in $\C N_L$ and set 
$$
\htrig_{\kappa}(W_L) := e_N \H_{\kappa}(\TSL_L \times \mf{t}_L, N_L) e_N.
$$
Category $\mc{O}$ for $\htrig_{\kappa}(W_L)$ is denoted $\mc{O}_{\kappa}(W_L)$. It is defined to be the category of all finitely generated $\htrig_{\kappa}(W_L)$-modules such that the action of $\C[\mf{t}_L^*]^{N_L}_+$ is locally nilpotent. 

\begin{lem}\label{lem:radios2}
The usual radial parts map extends to an isomorphism 
\beq{eq:radiso2}
\mathfrak{R}_L : (\dd(\NN_L) / \dd(\NN_L)\fl(\g)_c)^{N_G(L)} \iso \htrig_{\kappa}(W_L),
\eeq
where $\fl(\g) = \Lie L$ and $\fl(\g)_c = (\mu - \chi_c)(\fl(\g))$.
\end{lem}

\begin{proof}
Define the function $s_{\fl}$ on $\NN_L$ by $(z,X,v) \mapsto \det(v, z X v, \ds, z^{n-1} X^{n-1} v)$, where $z \in Z(\LSL)^{\circ}$, $X \in \fl'$ and $v \in V$. If $\varpi_L$ is the quotient map $\NN_L \rightarrow (T_L \times \mf{t}_L)/W_L$, then  
$$
\mathfrak{R}_L(D)(f) := (s_{\fl}^c D(s_{\fl}^{-c} \varpi_L^*(f)))|_{(T_L \times \mf{t}_L^{\reg})/W_L}, \quad \forall \ f \in \C[T_L \times \mf{t}_L^{\reg}]^{W_L}.
$$

To show that $\mathfrak{R}_L$ is an isomorphism, let us first consider the image of the larger algebra $(\dd(\NN_L) / \dd(\NN_L)\fl(\g)_c)^{L}$. Recall that $\NN_L = Z(\LSL)^{\circ} \times \mf{l'} \times V$. The group $L$, which is a product of general linear groups, acts trivially on $Z(\LSL)^{\circ}$. Note also that $\mf{l}'$ is a direct sum of copies of $\mf{sl}_m$ for various $m$. Then, one can check that proof of the rational analogue of Theorem \ref{thm:radialparts} given in \cite[Theorem 2.8]{GGS} still applies, giving an isomorphism  
$$
(\dd(\NN_L) / \dd(\NN_L)\fl(\g)_c)^{L} \iso \dd(Z(\LSL)^{\circ}) \o e \H_{\kappa}(\mf{t}_L,W_L)e.
$$
The action of $N_G(L)$ on the left factors through $N_G(L) / L$, and the action of $N_L$ on the right factors through $N_L / W_L$. Under the identification $N_G(L) / L = N_L / W_L$, the above isomorphism is $N_G(L) / L$-equivariant. Thus, the claim follows by taking invariants and noting that $\TSL_L = Z(\LSL)^{\circ}$.   
\end{proof}

Given $\mm \in \ss_{\NN_L,c}$, Lemma \ref{lem:radios2} implies that there is a natural action of $\htrig_{\kappa}(W_L)$ on 
$$
\Ham_L(\mm) := \Gamma(\NN_L,\mm)^{N_G(L)}.
$$
Moreover, the functor $\Ham_L$ maps $\ss_{\NN_L,c}$ into $\mc{O}_{\kappa}(W_L)$. 

Since the specialization functor defined in section \ref{sec:Cherespecial} is compatible with the action of $W$, there is a corresponding specialization functor $\Sp_{W_L} : \mc{O}_{\kappa} \rightarrow \mc{O}_{\kappa}(W_L)$ for the spherical subalgebra.  

\begin{prop}\label{cor:defXi}
The following diagram commutes 
$$
\xymatrix{
\ss_c \ar[d]_{\sp_L} & & \ar[ll]_{\Hamp} \mc{O}_{\kappa} \ar[d]^{\Sp_{W_L}} \\
\ss_{\NN_L,c} \ar[rr]^{\Ham_L} & & \mc{O}_{\kappa}(W_L).
}
$$
\end{prop}

\begin{proof}
For a fixed $E \in \mc{O}_{\kappa}$, let $\ms{E} = E |_{\mggW}$ and $\mm =  \Hamp(E)|_{\X^{\circ}}$. In order to show that the diagram commutes, we use the $V$-filtration on $\mm$ to define a filtration on $\ms{E}$. Then the commutativity of the diagram will follow from the fact that this filtration on $\ms{E}$ is the precisely the $V$-filtration used in the definition of $\Sp_{W_L}$. 

Define a filtration on $\ms{E}$ by $\nV_k \ms{E} = (\varpi_{\idot} (V_k \mm))^G$. This filtration is clearly compatible with the $\mc{V}$-filtration on $(\varpi_{\idot} \dd_{\X^{\circ}})^G / (\varpi_{\idot} \dd_{\X^{\circ}} \cdot \g_c)^G$. Therefore Theorem \ref{thm:filtrationsagree} implies that the filtration on $\ms{E}$ is compatible with the $V$-filtration on $\mc{U}_{\kappa}$. Moreover, it is clear that this is the unique $W \TSL_L$-filtration whose existence is guaranteed by Proposition \ref{prop:uniquefiltration} \textit{provided} that we can show it is a $W \TSL_L$-good filtration. This will be the case if we can show that the associated graded of $\ms{E}$ is a finitely generated $\htrig_{\kappa}(W_L)$-module. For this, and the commutativity of the diagram, it suffices to show that the associated graded of $\ms{E}$ equals $\Ham_L \left(\sp_{L}\left(\Hamp(E)\right)\right)$. 

Since $\varpi$ is an affine map (and hence $\varpi_{\idot}$ exact) and $G$ is reductive, $\nV_k \ms{E} / \nV_{k-1} \ms{E}$ is isomorphic to $(\varpi_{\idot} V_k \mm /  V_{k-1} \mm)^G$ and hence
$$
\Gamma(\TSL_L,\gr^{\nV} \ms{E}) = \Gamma(\NN_{\X/\Y}, \gr_V \mm)^G.
$$
Proposition \ref{prop:monoequiv} \vii implies that $\Gamma(\NN_{\X/\Y}, \gr_V \mm)^G$ is isomorphic to $\Gamma(\NN_L, \Upsilon^* (\gr_V \mm))^{N_G(L)}$. Hence, we have  
\begin{align*}
\Gamma(\NN_L, \Upsilon^* (\gr_V \mm))^{N_G(L)} & = \Ham_L \left(\Upsilon^* \Psi_{\X^{\circ}/\Y}(\mm) \right) \\
  & = \Ham_L \left(\Upsilon^* \Psi_{\X/\Y} \left(\Hamp(E) \right)\right)  = \Ham_L \left(\sp_{L}\left(\Hamp(E)\right)\right).
\end{align*}

Thus, we have constructed a functorial isomorphism $\Sp_{W_L}(E)\cong\Ham_L (\sp_{L}(\Hamp(E))$,
of graded vector spaces. To complete the proof, one must check that the constructed isomorphism respects the $\htrig_{\kappa}(W_L)$-actions. This follows from the fact that the isomorphisms in the diagram
$$
\xymatrix{
\gr_{\mc{V}} (\varpi_{\idot} \dd_{\X^{\circ}} / \varpi_{\idot} \dd_{\X^{\circ}} \cdot \g_c)^G \ar[rr]^-{\gr_{\mc{V}} \mathfrak{R}} \ar[d]_{\phi} & & \gr_V \mc{U}_{\kappa} \ar[d]^{\wr} \\
(\dd(\NN_L) / \dd(\NN_L)\fl(\g)_c)^{N_G(L)} \ar[rr]^-{\mathfrak{R}_L} & & \htrig_{\kappa}(W_L)
}
$$
make it commutative. Here $\phi$ is Lemma \ref{lem:pullbackUp} applied to $\dd(\NN_{\X / \mc{Y}})$, using the fact that taking $G$-invariants commutes with $\gr_{\mc{V}}$. Checking that the diagram commutes reduces to the key equation (\ref{eq:Js}) saying that the image of $s_L$ in $\mc{J}^{\mathbf{n}}_L / \mc{J}^{\mathbf{n} + 1}_L$ equals $s_{\fl}$, which is precisely the function used in the definition of $\mathfrak{R}_L$ of Lemma \ref{lem:radios2}. 
\end{proof}

Let $\Hamp_L$ be the functor $\mc{O}_{\kappa}(W_L) \to \ss_{\NN_L,c}$ of tensoring on the left by $\dd(\NN_L) / \dd(\NN_L) \mf{l}_c$. The various adjunctions imply that we have natural transformations $\Hamp_L \circ \ \Sp_{W_L} \to \sp_L \circ \Hamp$, and $\Sp_{W_L} \circ \Ham_c \to \Ham_L \circ \ \sp_L$.

\begin{proof}[Proof of Theorem \ref{thm:admissiblecommute}]
Proposition \ref{cor:defXi} says that $\Sp_{W_L} \rightarrow \Ham_L \circ \ \sp_L \circ \Hamp$ is an isomorphism. Therefore 
$$
\Sp_{W_L} \circ \Ham_c \longrightarrow \Ham_L \circ \ \sp_L \circ \Hamp \circ \Ham_c
$$
is also an isomorphism. Hence it suffices to show that the adjunction $\id \rightarrow \Hamp \circ \Ham_c$ induces an isomorphism $\Ham_L \circ \ \sp_L \iso \Ham_L \circ \ \sp_L \circ (\Hamp \circ \Ham_c)$. The adjuction $\id \rightarrow \Hamp \circ \Ham_c$ becomes an isomorphism on the quotient category $\ss_c / \Ker \Ham_c$. Therefore it suffices to show that $\Ham_L \circ \ \sp_L : \ss_c \rightarrow \mc{O}_{\kappa}(W_L)$ factors through $\ss_c / \Ker \Ham_c$ i.e $(\Ham_L \circ \ \sp_L) (\mm) = 0$ for all $\mm \in \Ker \Ham_c$. 

The idea is to use formula (\ref{eq:SSformula}) to show that if the singular support of $\mm$ is contained in $(T^* \X)^{\mathrm{unst}}$ then this implies that the singular support of $\sp_L(\mm)$ is contained in $(T^* \NN_L)^{\mathrm{unst}}$. Then, the theorem will follow from Theorem \ref{cor:unstable}.  

Recall that $(g,Y,v,w) \in (T^* \X)^{\mathrm{unst}}$ if and only if $\C \langle g,Y \rangle \cdot v \neq V$. We define $(T^* \X^{\circ}_{\LSL})^{\mathrm{unst}}$ analogously. As in the proof of Proposition \ref{prop:wonderfulinclusion}, we can choose a complete flag $\mc{F}_{\idot} \subset V$ that is stable under $g$ and $Y$. This flag can be chosen so that $i \in \mc{F}_{n-1}$. Let $Y'$ be the projection of $Y$ onto $\mf{l}_n$. Then it was shown in the proof of Proposition \ref{prop:wonderfulinclusion} that $Y' \in \mf{n}_0 \subset \mf{b} \cap \mf{l}_n$. Hence, $Y'$ also preserves the flag $\mc{F}_{\idot}$. Thus, $\C \langle g, Y' \rangle \cdot v \subset \mc{F}_{n-1} \subsetneq V$ and hence $\rho_{\Upsilon}(\omega^{-1}_{\Upsilon}((T^* \X)^{\mathrm{unst}}))$ is contained in $(T^* \X^{\circ}_{\LSL})^{\mathrm{unst}}$.

Next we show that $\varphi^* C_{\Lambda}((T^* \X^{\circ}_{\LSL})^{\mathrm{unst}})$ is contained in $(T^* \NN_L)^{\mathrm{unst}}$. Let $I$ be the ideal in $A := \C[T^* \X^{\circ}_{\LSL}]$ defining the unstable locus.  Then $I$ is generated by the space of $L$-semi-invariants $J = \oplus_{k > 0} A^{L,\mathrm{det}^k}$. Similarly, if $I'$ is the ideal in $B := \C[T^* \NN_L]$ that defines the unstable locus in $T^* \NN_L$, then $I'$ is generated by $J' = \oplus_{k > 0} B^{L,\mathrm{det}^k}$. Recall that the $V$-filtration on $\dd(\X^{\circ}_{\LSL})$ defines a $V$-filtration on $A$. Then, $I$ inherits a filtration and, by definition, $\varphi^* C_{\Lambda}((T^* \X^{\circ}_{\LSL})^{\mathrm{unst}})$ equals the zero set of $\gr_V I$. Therefore, it suffices to show that $J' \subset \gr_V I$. Let $f \in J'$ be a $\mathrm{det}^k$-semi-invariant. We may assume that $f$ is contained in $B_m$ for some $m$. Since $A$ is a rational $L$-module, the short exact sequence 
$$
0 \to V_{m-1} A \to V_m A \to B_m \to 0
$$
splits as $L$-modules. Therefore, there exists a $\mathrm{det}^k$-semi-invariant $h$ in $V_m A$ such that $f$ is the symbol of $h$. But then $h \in J \subset I$ and hence $f \in \gr_V I$ as required. 
\end{proof}

\subsection{An application of Hodge theory}

Let  $W'$ be a parabolic subgroup of $W$, $L \subset G$ a Levi such that $W_L = W'$ and let $\TSL_L$ be as in \S\ref{jj2}.

\begin{cor}\label{cor:mixedhodge}
Assume that $c \in \Q$ is admissible and let $E \in \mc{O}_{\kappa}$
be a simple object. Then $\Res_{\TSL_L} E$ is semi-simple if and only if the
action of the Euler element $\eu$ on 
$\Res_{T_L} E$ is semi-simple. 
\end{cor}

\begin{proof}
Clearly, if $\eu$ does not act semi-simply on $\Res_{\TSL_L} E$, then $\Res_{\TSL_L} E$ cannot be semi-simple. Therefore, we assume that $\eu$ acts semi-simply and will show that $\Res_{\TSL_L} E$ is semi-simple. By Theorem \ref{prop:specialcoherent}, $\Res_{\TSL_L} E = \Sp_{W_L} (E)$. Let $L$ be a Levi subgroup of $G$, whose Weyl group $W_L$ equals $W'$. The mirabolic module $E_{!*}$ is a simple, $(G,c)$-monodromic module. Since $c$ is assumed to be rational, $E_{!*}$ is equipped with a pure Hodge structure. The $L$-invariant vector field on $\NN_L$ corresponding to the action of $\Cs$ by dilation along $\mf{l}'$ is denoted $\eu_{\mf{l}'}$. Up to a shift, the radial parts map sends $\eu_{\mf{l}'}$ to $\eu$.  

Since $\Ham_c(E_{!*}) = E$, Theorem \ref{thm:admissiblecommute} implies that $\Sp_{W_L}(E) \simeq \Ham_L(\sp_L(E_{!*}))$. Thus, $\eu$ acts semi-simply on $\Ham_L(\sp_L(E_{!*}))$. The image of the canonical map $\Hamp_L \circ \Ham_L(\sp_L(E_{!*})) \rightarrow \sp_L(E_{!*})$ is denoted $\mm$. Since the adjoint action of $\eu_{\mf{l}'}$ on $\dd(\NN_L)$ is semi-simple, $\eu_{\mf{l}'}$ acts semi-simply on $\Gamma(\NN_L,\mm)$. Also, $\Ham_L(\mm) = \Sp_{W_L} (E)$. Therefore it suffices to show that $\mm$ is semi-simple.  

As in the proof of Proposition \ref{prop:PsiLcatO}, the fact that $\eu_{\mf{l}'}$ acts semi-simply on $\Gamma(\NN_L,\mm)$ implies that $\eu_{\mf{l}' \times V}$ acts semi-simply too. Since $E_{!*}$ has a pure weight structure, $\sp_L(E_{!*})$ has a weight filtration that is given by the action of the monodromy operator $\exp(2 \pi \sqrt{-1} \eu_{\mf{l}' \times V})$. In the case where $\mathrm{codim}_{\TSL} \TSL_{L} = 1$, this is explained in \cite[\S 5]{Saito}; the general case is deduced from this by standard arguments c.f. \cite[\S 8]{CyclesProches} This implies that the submodule $\mm$ of $\sp_L(E_{!*})$ is semi-simple if and only if $\exp(2 \pi \sqrt{-1} \eu_{\mf{l}' \times V})$ acts semi-simply on $\Gamma(\NN_L,\mm)$. But this is clearly the case, since $\eu_{\mf{l}' \times V}$ is assumed to act semi-simply.
%The required statement  is now a consequence of Proposition \ref{cor:defXi} and the discussion of monodromy at the end of section \ref{sec:Cherespecial}.
\end{proof}

\begin{rem}
The obvious analogue of Corollary \ref{cor:mixedhodge} holds in the rational case too. In that case the assumption that $c$ is rational may be dropped since the Corollary is obvious true for non rational values. 
\end{rem}

\section{Appendix: Shift functors}\label{app:shift}

The goal of this appendix is to adapt the results of \cite{GGS} to the setting of {\em trigonometric} Cherednik algebras.   

\subsection{The trigonometric Cherednik algebra}\label{app:shift1}

Let $\TSL$ be a maximal torus inside $\SSL$, $\tSL$ its Lie algebra. Let $P = \Hom(\TSL,\Cs)$ be the weight lattice and $Q \subset P$ the root lattice, so that $\Omega = P / Q \simeq \Z / n\Z$. We choose a set of positive roots $R_+$ in $R$, the set of all roots. Let $\rho$ be the half sum of all positive roots. For $\lambda \in P$, we write $e^{\lambda}$ for the corresponding monomial in $\C[\TSL]$. Choose $\kappa \in \C$. As in \cite[Definition 2.4]{Opdamlectures}, the \textit{Dunkl operator} associated to $\by \in \tSL$ is 
$$
T_{\by}^{\kappa} = \pa_{\by} + \kappa \sum_{\alpha \in R_+} \alpha(\by) \frac{1}{1 - e^{-\alpha}} (1 - s_{\alpha}) - \kappa \rho(\by). 
$$
The \textit{trigonometric Cherednik algebra} of type $\SSL$ is the associative subalgebra $\H_{\kappa}^{\trig}(\SSL)$ of the algebra \mbox{$\dd(\TSL^{\reg}) \rtimes W$} generated by $\C[\TSL]$, $W$ and all Dunkl operators $T_{\by}^{\kappa}$ for $\by \in \tSL$. The Dunkl operators pairwise commute. Define the divided-difference operator $\Delta_{\alpha} = \frac{1}{1 - e^{-\alpha}} (1 - s_{\alpha})$ so that $T_{\by}^{\kappa} = \pa_{\by} + \kappa \sum_{\alpha \in R_+} \alpha(\by) \Delta_{\alpha} - \kappa \rho(\by)$. We fix a $W$-invariant, symmetric non-degenerate bilinear form $( - , - )$ on $\mf{t}^*$ such that $(\alpha, \alpha) = 2$ for all $\alpha \in R$. 

\begin{lem}\label{lem:squareDunkl}
Let $\by_1, \ds, \by_{n-1}$ be an orthonormal basis of $\tSL$. Then 
\begin{align*}
\sum_{i = 1}^{n-1} (T_{\by_i}^{\kappa})^2 = & \ \Omega_{\mf{t}} + 2 \kappa \sum_{\alpha > 0} \frac{e^{- \alpha}}{1 - e^{-\alpha}} \Delta_{\alpha} + \kappa^2 \sum_{\alpha \neq \beta > 0} (\alpha,\beta) \Delta_{\alpha} \Delta_{\beta} \\
 & + \kappa^2 (\rho,\rho) + \kappa \sum_{\alpha >0} \frac{1 + e^{-\alpha}}{1 - e^{-\alpha}} \pa_{\alpha},
\end{align*}
where $\pa_{\alpha}$ is defined by $\pa_{\alpha}(\bx) = (\bx,\alpha) \bx$ and $\Omega_{\tSL} = \sum_{i = 1}^{n-1} \pa_{\by_i}^2$. 
\end{lem}

\begin{proof}
We note that 
$$
\pa_{\by} \left( \frac{1}{1 - e^{-\alpha}} \right) = -\frac{\alpha(\by) e^{- \alpha}}{( 1-e^{-\alpha})^2}
$$
and $(\rho,\alpha) = 1$ for all $\alpha > 0$. Therefore
\begin{align*}
\sum_{i = 1}^{n-1} (T_{\by_i}^{\kappa})^2 = & \ \Omega_{\mf{t}} + 2 \kappa \sum_{\alpha > 0} \frac{e^{- \alpha}}{1 - e^{-\alpha}} \Delta_{\alpha} + \kappa^2 \sum_{\alpha,\beta > 0} (\alpha,\beta) \Delta_{\alpha} \Delta_{\beta} \\
 & + \kappa^2 (\rho,\rho) + \kappa \sum_{\alpha >0} \frac{1 + e^{-\alpha}}{1 - e^{-\alpha}} \pa_{\alpha} - 2 \kappa^2 \sum_{\alpha > 0} \Delta_{\alpha} 
\end{align*}
The operators $\Delta_{\alpha}$ are idempotent i.e. $\Delta_{\alpha}^2 = \Delta_{\alpha}$ and hence the formula of the lemma follows from the above. 
\end{proof}

The group $W$ acts on $\H_{\kappa}^{\trig}(\SSL)$ by conjugation; the subalgebra of $W$-invariant elements is denoted $\H_{\kappa}^{\trig}(\SSL)^{W}$. The Dunkl operators act on the space $\C[\TSL]$, hence any $W$-invariant operator will act on $\C[\TSL]^{W}$. We define $\Res : \H_{\kappa}^{\trig}(\SSL)^{W} \to \htrig^{\reg}$ by $\Res (D) = D |_{\C[\TSL]^{W}}$, where $\htrig^{\reg} := \dd(\TSL^{\reg})^{W}$. Lemma \ref{lem:squareDunkl} implies that 
$$
\Res \left( \sum_{i = 1}^{n-1} (T_{y_i}^{\kappa})^{2} \right) \  =\ 
\Omega_{\mf{t}} + \kappa \sum_{\alpha > 0} \frac{1 + e^{-\alpha}}{1 -
  e^{-\alpha}} \pa_{\alpha} + \kappa^2 (\rho,\rho)\ =:\
 L(\kappa) + \kappa^2 (\rho,\rho). 
$$
The map $\Res$ restricts to an embedding $e \H_{\kappa}^{\trig}(\SSL) e \hookrightarrow \htrig^{\reg}$ and we denote its image by $\htrig_\kappa$. As shown in the proof of \cite[Theorem 3.3.3]{CherednikCurves}, the algebra $\htrig_\kappa$ is generated by $L(\kappa)$ and $\C[\TSL]^{W}$. We denote by $\delta$ the Weyl denominator 
$$
\prod_{\alpha \in R_+} (e^{\alpha /2} - e^{- \alpha / 2}) = e^{\rho} \prod_{\alpha \in R_+} ( 1- e^{-\alpha}),
$$
so that $\TSL^{\reg} = (\delta \neq 0)$. As in the rational case, the existence of shift operators for Cherednik algebras implies the existence of a collection of bimodules between the various $\htrig_\kappa$. These bimodules induce Morita equivalences in many cases. The following is a consequence of Opdam's theory of shift operators, see \cite[Theorem 5.11]{Opdamlectures}.  

\begin{cor}\label{cor:bimodules}
We have an equality 
$$
e \H^{\trig}_{\kappa + 1}(\SSL) e = e \delta^{-1} \H^{\trig}_{\kappa}(\SSL) \delta e
$$
in $\htrig^{\reg}$. Thus, the spaces 
$$
{}_{\kappa} \mathsf{P}_{\kappa + 1} := e \H^{\trig}_{\kappa}(\SSL) \delta e, \quad {}_{\kappa +1} \mathsf{Q}_{\kappa} := e \delta^{-1} \H^{\trig}_{\kappa}(\SSL) e,
$$
are $(\htrig_\kappa , \htrig_{\kappa+1})$ and $(\htrig_{\kappa+1}, \htrig_\kappa)$-bimodules respectively. 
\end{cor}

Recall from Definition \ref{bad}, the subset $\bad = \{ a/b \ | \ a,b \in \Z, \ 1 \le b \le n \}$ of $\C$. As in \cite{GGS}, we say that $\kappa$ is \textit{good} if $\kappa \notin \bad \cap (0,1)$. Under the equality $\kappa = -c + 1$, $\kappa$ is good if and only if $c$ is good in the sense of section \ref{sec:unstable}. Recall that we defined in (\ref{sec:bimodules}), based on Corollary \ref{cor:bimodules}, the bimodules ${}_{\kappa} \mathsf{P}_{\kappa + m}$ and ${}_{\kappa + m} \mathsf{Q}_{\kappa}$. 

\begin{thm}[\cite{GGS}, Theorem 3.9]\label{thm:projbi}
Fix $\kappa \in \C$ and an integer $m \ge 1$ such that each of $\kappa + 1, \kappa + 2, \ds, \kappa + m -1$ is good. 
\begin{enumerate}
\item ${}_{\kappa} \mathsf{P}_{\kappa + m}$ is the unique nonzero $(\htrig_{\kappa},\htrig_{\kappa+m})$-bisubmodule of $\htrig^{\reg}$ that is reflexive as either a right $\htrig_{\kappa+m}$-module or a left $\htrig_{\kappa}$-module.
\item ${}_{\kappa+m} \mathsf{Q}_{\kappa}$ is the unique nonzero
  $(\htrig_{\kappa+m},\htrig_{\kappa})$-bisubmodule of $\htrig^{\reg}$
  that is reflexive as either a right $\htrig_{\kappa}$-module or a left
  $\htrig_{\kappa+m}$-module. 
\item Either $\kappa$ or $\kappa + m$ is good. In the former case ${}_{\kappa} \mathsf{P}_{\kappa + m}$ and ${}_{\kappa+m} \mathsf{Q}_{\kappa}$ are projective $\htrig_{\kappa}$-modules, while in the latter case they are projective $\htrig_{\kappa+m}$-modules.  
\end{enumerate}
\end{thm}

\begin{rem}
It is important to note when applying the results of \cite{GGS} to the setup considered in this paper that their parameter ``$c$'' is related to $\kappa$ by $\kappa = - c$.
\end{rem} 

\begin{proof}
The proof of the theorem is simply a matter of carefully checking that the arguments employed in \cite{GGS} are applicable in our situation. We note that: the existence of the bimodules ${}_{\kappa} \mathsf{P}_{\kappa + 1}$ and ${}_{\kappa+m} \mathsf{Q}_{\kappa}$ is shown in Corollary \ref{cor:bimodules}. By \cite[Corollary 5.2.7]{CherednikCurves}, if $\kappa$ is good then $\htrig_{\kappa}$ is Morita equivalent to $\H^{\trig}_{\kappa}(\SSL)$ and hence has finite global dimension. The only gap is showing that condition (3) of \cite[Hypothesis 3.2]{GGS} holds. The argument using Fourier transforms given in \textit{loc. cit.} for condition (3) is not applicable to the trigonometric Cherednik algebra. Another way to see that condition (3) holds is as follows. Let $I$ be a proper two-sided ideal in $\htrig_\kappa$. Associated to $I$ is its characteristic variety in $(T^* \TSL) / W$. This will be a union of closures of symplectic leaves and hence even dimensional. Therefore either $\mathrm{GKdim} \ \htrig_\kappa / I = \mathrm{GKdim} \ \htrig_\kappa$ or $\mathrm{GKdim} \ \htrig_\kappa / I \le \mathrm{GKdim}\ \htrig_\kappa - 2$. However, $\htrig_\kappa [\delta^{-2}] = \htrig^{\reg}$ is a simple ring, hence the normal element $\delta^{2m}$ belongs to $I$ for some $m \le 1$ and hence $\mathrm{GKdim} \ \htrig_\kappa / I < \mathrm{GKdim} \ \htrig_\kappa$. 
\end{proof}

\subsection{The radial parts map}\label{sec:radial}

The radial parts map identifies $\htrig_\kappa$ with a certain algebra obtained by quantum Hamiltonian reduction. The usual Chevalley isomorphism induces an isomorphism $\eta : \C[\X^{\reg}]^G \rightarrow \C[\TSL^{\reg}]^{W}$. 

\begin{thm}\label{thm:radialparts}
Let $\kappa = -c + 1$. The following map 
is an isomorphism of filtered algebras. 
$$
\mathfrak{R} : (\ddd /\ddd  \g_{c})^G \longrightarrow \htrig_\kappa, \quad \mathfrak{R}(D)(f) = \bs^{c} D(\bs^{-c} \eta^* f) |_{\TSL^{\reg}/W}.
$$
\end{thm}

\begin{proof}
The proof of the theorem is standard, see e.g. \cite[Appendix]{GGS}. In order to show $\htrig_\kappa$ is contained in the image of that $\mathfrak{R}$, it suffices to check is that the image of the Casimir element lies in $\htrig_\kappa$. This is done in Lemma \ref{lem:radialparts} below. Since $\bs$ is a $\det^{-1}$-semi-invariant, we have $(\mu - c \Tr)(x) \cdot \bs^{-c} = 0$ for all $x \in \g$. This implies that $(\ddd  \g_{c})^G$ is contained in the kernel of $\mathfrak{R}$. As in \cite[Appendix]{GGS}, by an associated graded argument, these facts together with the fact that the moment map $\mu$ is flat are enough to conclude that $\mathfrak{R}$ is an isomorphism. 
\end{proof}

Recall that we have identified $\sll $ with \textit{left invariant} vector fields on $\SSL$ i.e. for $X \in \sll $ and $f \in \C[\X]$ we have 
$$
(\pa_X \cdot f)(g,v) = \frac{\pa}{\pa t} f(g(1 + t X),v).
$$
Since $e_{\alpha}$ and $e_{-\alpha}$ are a dual pair with respect to the trace form on $\sll $, the Casimir element $\Omega_{\sll }$ in $U(\sll )$ is given by 
$$
\Omega_{\sll } = \Omega_{\mf{t}} + \sum_{\alpha \in R_+} \pa_{e_{\alpha}} \pa_{e_{-\alpha}} + \pa_{e_{-\alpha}} \pa_{e_{\alpha}},
$$
where $\Omega_{\mf{t}} = \sum_{i = 1}^{n-1} \pa_{\by_i}^2$ (recall that the $\by_i$ form an orthonormal basis of $\tSL$). For $w \in \C$, define $\mathrm{Rad}_w : \dd(\X)^G \rightarrow \htrig^{\reg}$ by 
$$
\mathrm{Rad}_w(D)(f) = \bs^{-w} (D(\bs^w (\eta^* f))) |_{\TSL^{\reg}/W},
\qquad\forall f \in \C[\TSL^{\reg}]^{W}.
$$

\begin{lem}\label{lem:radialparts}
We have 
$$
\mathrm{Rad}_w(\Omega_{\sll }) = \Omega_{\mf{t}} + (w + 1) \sum_{\alpha > 0} \frac{1 + e^{-\alpha}}{1 - e^{-\alpha}} \pa_{\alpha} + w(w+2)(\rho,\rho). 
$$
\end{lem}

\begin{proof}
By definition, $e^{\alpha}$ is a character on $\TSL$, $e^{\alpha}(x) = \alpha(x)$. Let $f : \X \rightarrow \C$ be a $\det^w$-semi-invariant holomorphic function defined on an open neighborhood of $\TSL^{\reg} \times V$. One calculates that 
$$
\exp(t e_{\alpha}) \cdot f(x \exp(r e_{-\alpha}),v) = f(g(1 + r e_{-\alpha})(1 + t (1 - e^{-\alpha}) e_{\alpha} - rt e^{-\alpha} h_{\alpha}), v - t e_{\alpha} v).
$$
On the other hand, semi-invariance implies that 
$$
\exp(t e_{\alpha}) \cdot f(g \exp(r e_{-\alpha}),v) = f(g \exp(t e_{-\alpha}),v).
$$
Therefore, differentiating with respect to $t$ and setting $t = 0$ gives
$$
0 = [(1 - e^{-\alpha}) \pa_{e_{\alpha}} - r e^{-\alpha} \pa_{\alpha} + \mu_V(e_{\alpha})] f(g(1 + r e_{-\alpha}),v).
$$
Rewriting,
\begin{equation}\label{eq:eq101}
\pa_{e_{\alpha}} f(g(1 + r e_{-\alpha}),v) = \frac{1}{1 - e^{-\alpha}} \left[ r e^{-\alpha} \pa_{\alpha} - \mu_V(e_{\alpha}) \right] f(g(1 + r e_{-\alpha}),v). 
\end{equation}
Differentiating with respect to $r$ and setting $r = 0$ gives
$$
\pa_{e_{-\alpha}} \pa_{e_{\alpha}} f(g,v) = \frac{1}{1 - e^{-\alpha}} \left[ e^{-\alpha} \pa_{\alpha} f(g,v) - \pa_{e_{-\alpha}} (\mu_V(e_{\alpha}) \cdot f(g,v)) \right].
$$
Equation (\ref{eq:eq101}) with $e_{\alpha}$ replaced by $e_{-\alpha}$ and $r = 0$ gives
$$
\pa_{e_{-\alpha}} [\mu_V(e_{\alpha}) \cdot f] (g,v) = \frac{-1}{1 - e^{\alpha}} \mu_V(e_{-\alpha}) [\mu_V(e_{\alpha}) \cdot f] (g,v).
$$
Thus, 
$$
\pa_{e_{-\alpha}} \pa_{e_{\alpha}} f(g,v) = \frac{e^{-\alpha}}{1 - e^{-\alpha}} \pa_{\alpha} f(g,v) + \frac{1}{(1 - e^{\alpha})(1 - e^{-\alpha})} \mu_V(e_{-\alpha}) \mu_V(e_{\alpha}) \cdot f(g,v).
$$
As in \cite{GGS}, $\bs |_{T \times V} = \delta \cdot (v_1 \cdots v_n)$ and $\mu_V(e_{ij}) = - v_j \pa_{v_i}$. Thus,  
$$
\mu_V(e_{-\alpha}) \mu_V(e_{\alpha}) \cdot f(g,v) = w(w + 1) f(g,v).
$$
Note that 
$$
\frac{e^{-\alpha} - e^{\alpha}}{(1 - e^{-\alpha})(1 - e^{\alpha})} = \frac{1 + e^{-\alpha}}{1 - e^{-\alpha}}.
$$
Therefore
\begin{align*}
\mathrm{Rad}_w(\Omega_{\sll }) & = \delta^{-w} \Omega_{\mf{t}} \delta^w  + w(w+1) \sum_{\alpha > 0} \frac{(\alpha,\alpha)}{(1 - e^{\alpha})(1 - e^{-\alpha})} + \sum_{\alpha > 0} \frac{e^{-\alpha} - e^{\alpha}}{(1 - e^{-\alpha})(1 - e^{\alpha})} \delta^{-w} \pa_{\alpha} \delta^w. \\
 & = \delta^{-w} \Omega_{\mf{t}} \delta^w  - w(w+1) \sum_{\alpha > 0} \frac{(\alpha,\alpha)}{(e^{\alpha /2} - e^{-\alpha/2})^2} + \sum_{\alpha > 0} \frac{1 + e^{-\alpha}}{1 - e^{-\alpha}} \delta^{-w} \pa_{\alpha} \delta^w.
\end{align*}
By \cite[Theorem 2.1.1]{HeckmanBook},
$$
\delta^{u} \circ (L(u) + u^2 (\rho,\rho)) \circ \delta^{-u} = \Omega_{\mf{t}} - \sum_{\alpha > 0} \frac{u(u - 1) (\alpha,\alpha)}{(e^{\alpha /2} - e^{-\alpha/2})^2}
$$
Therefore, taking $u = 1$ implies that 
\begin{align*}
\delta^{w+1} \mathrm{Rad}_w(\Omega_{\sll}) \delta^{-w-1} & = \Omega_{\mf{t}} - w(w+1) \sum_{\alpha > 0} \frac{(\alpha,\alpha)}{(e^{\alpha /2} - e^{-\alpha/2})^2} - (\rho,\rho)\\
 & = \delta^{w+1} \circ (L(w+1) + (w+1)^2 (\rho,\rho)) \circ \delta^{-(w+1)} - (\rho,\rho).
\end{align*}
\end{proof}

\begin{rem}\label{rem:radialiso}
Lemma \ref{lem:radialparts} implies the morphism $\dd(\SSL)^{G} \rightarrow (\ddd /\ddd  \g_{c})^G \stackrel{\mathfrak{R}}{\rightarrow} \htrig_\kappa$ agrees with the radial parts map $Y \mapsto \delta^{-\kappa +1} \circ R_V(Y) \circ \delta^{\kappa-1}$ of \cite[\S 1]{SphericalFun} when $\kappa$ is a positive integer. Then, \cite[Corollary 1.10]{SphericalFun} implies that $\mf{R}$ restricts to the Harish-Chandra homomorphism $\FZ \stackrel{\sim}{\rightarrow} (\sym \tSL)^W$ when $\kappa$ is a positive integer. Since $\mf{R}(D)$ can be expressed as a polynomial expression in $\kappa$ for each $D \in \FZ$, this implies that $\mf{R}$ always restricts to the Harish-Chandra homomorphism on $\FZ$. 
\end{rem}

\bibliographystyle{amsalpha}
%\bibliography{master}

\def\cprime{$'$}
\providecommand{\bysame}{\leavevmode\hbox to3em{\hrulefill}\thinspace}
\providecommand{\MR}{\relax\ifhmode\unskip\space\fi MR }

% \MRhref is called by the amsart/book/proc definition of \MR.
\providecommand{\MRhref}[2]{
  \href{http://www.ams.org/mathscinet-getitem?mr=#1}{#2}
}
\providecommand{\href}[2]{#2}

\end{document}